\documentclass[11pt]{amsart}
\usepackage{amscd,amssymb,longtable, rotating, lscape, graphicx}
\usepackage[matrix,arrow,curve]{xy}
\usepackage{supertabular}
\usepackage{setspace}

\theoremstyle{definition}
\newtheorem{theorem}[equation]{Theorem}
\newtheorem*{theorem*}{Theorem}
\newtheorem{lemma}[equation]{Lemma}
\newtheorem{proposition}[equation]{Proposition}
\newtheorem{corollary}[equation]{Corollary}
\newtheorem{conjecture}[equation]{Conjecture}
\newtheorem{example}[equation]{Example}
\newtheorem{definition}[equation]{Definition}
\newtheorem*{definition*}{Definition}

\theoremstyle{remark}

\makeatletter\@addtoreset{equation}{section}
\makeatletter\@addtoreset{section}{part}

\makeatother

\def \P {\mathbb{P}}
\def \Q {\mathbb{Q}}

\def \eps {\varepsilon}

\def \qlineq {\sim_{\Q}}

\def \Supp {\mathrm{Supp}\,}
\def \mult {\mathrm{mult}}

\def \lct {\mathrm{lct}}

\def \le {\leqslant}
\def \geq {\geqslant}
\def \leq {\leqslant}

\def\nlb {\nolinebreak}

\author{Ivan Cheltsov, Jihun Park, Constantin Shramov}

\title{Exceptional del Pezzo hypersurfaces}

\begin{document}

\begin{abstract}
We compute global log canonical thresholds of a large class of
quasismooth well-formed del Pezzo weighted hypersurfaces in
$\mathbb{P}(a_{1},a_{2},a_{3},a_{4})$. As a corollary we obtain
the existence of orbifold K\"ahler--Einstein metrics on many of
them, and classify exceptional and weakly exceptional quasismooth
well-formed del Pezzo weighted hypersurfaces in
$\mathbb{P}(a_{1},a_{2},a_{3},a_{4})$.
\end{abstract}

\address{\emph{Ivan Cheltsov}\newline \textnormal{School of
Mathematics, The University of Edinburgh,
Edinburgh, EH9 3JZ, UK;
 \texttt{cheltsov@yahoo.com}}}
\address{ \emph{Jihun Park}\newline \textnormal{Department of
Mathematics, POSTECH, Pohang, Kyungbuk 790-784, Korea;
\texttt{wlog@postech.ac.kr}}}
\address{\emph{Constantin Shramov}\newline \textnormal{School of
Mathematics, The University of Edinburgh,
Edinburgh, EH9 3JZ, UK;
 \texttt{shramov@mccme.ru}}}
\maketitle

\tableofcontents

\part{Introduction} \label{section:intro}
All varieties are always assumed to be
complex, algebraic, projective and normal unless otherwise
stated.
\section{Background} \label{subsection:back}

The multiplicity of a nonzero polynomial
$f\in\mathbb{C}[z_1,\ldots, z_n]$ at a point $P\in \mathbb{C}^n$
is the nonnegative integer $m$ such that
$f\in\mathfrak{m}_P^m\setminus\mathfrak{m}_P^{m+1}$, where
$\mathfrak{m}_P$ is the maximal ideal of polynomials vanishing at
the point $P$ in $\mathbb{C}[z_1,\ldots, z_n]$. It can be also
defined by derivatives: the multiplicity of $f$ at the point $P$
is the number
\[\mult_P(f)=\min\left\{m \ \Big|\ \frac{\partial^m f}{\partial^{m_1}
z_1\partial^{m_2} z_2\ldots\partial^{m_n} z_n}(P)\ne0 \right\}.\]

On the other hand, we have a similar invariant that is defined by
integrations. This invariant, which is called the complex
singularity exponent of $f$ at the point $P$, is given by

\[c_P(f)=\mathrm{sup}\left\{c\ \Big|\ |f|^{-c}~\text{is locally}~L^2~\text{near the
 point $P\in\mathbb{C}^n$}\right\}.\]
 In algebraic geometry this invariant is usually called a log canonical threshold.
 Let $X$ be a variety with at most log canonical singularities, let $Z\subseteq X$ be a closed subvariety, and let
$D$ be an effective $\mathbb{Q}$-Cartier $\mathbb{Q}$-divisor on
the variety $X$. Then the number
$$
\mathrm{lct}_{Z}\big(X,D\big)=\mathrm{sup}\left\{\lambda\in\mathbb{Q}\
\Big|\ \text{the log pair}\
 \big(X, \lambda D\big)\ \text{is log canonical along}~Z\right\}%
$$
is called a log canonical threshold of the divisor $D$ along $Z$.
It follows from \cite{Ko97} that for a polynomial $f$ in $n$
variables over $\mathbb{C}$ and a point $P\in \mathbb{C}^n$
$$
\mathrm{lct}_{P}\Big(\mathbb{C}^n,D\Big)=c_P\big(f\big),
$$
where the divisor $D$ is defined by the equation $f=0$ on $\mathbb{C}^n$.
We can define the log canonical threshold of $D$ on
$X$ by
\[
\begin{split}
\mathrm{lct}_{X}\big(X,D\big)&=\mathrm{inf}\left\{\mathrm{lct}_P\big(X,D\big)\
\Big\vert\ P\in
 X\right\}\\ &=\mathrm{sup}\left\{\lambda\in\mathbb{Q}\ \Big|\ \text{the log pair}\ \big(X, \lambda
 D\big)\
 \text{is log canonical}\right\}.
\end{split}\]
For simplicity, the log canonical threshold
 $\mathrm{lct}_X(X,D)$ will be denoted by $\mathrm{lct}(X,D)$.



\begin{example}
\label{example:cubics} Let $D$ be a cubic curve on the projective plane $\mathbb{P}^2$.  Then
$$
\mathrm{lct}\big(\mathbb{P}^2,D\big)=\left\{%
\aligned
&1\ \ \text{if}\ D\ \text{is a smooth curve},\\%
&1\ \ \text{if}\ D\ \text{is a curve with ordinary double points},\\%
&\frac{5}{6}\ \ \text{if}\ D\ \text{is a curve with one cuspidal point},\\%
&\frac{3}{4}\ \ \text{if}\ D\ \text{consists of a conic and a line that are tangent},\\%
&\frac{2}{3}\ \ \text{if}\ D\ \text{consists of three lines intersecting at one point},\\%
&\frac{1}{2}\ \ \text{if}\ \mathrm{Supp}\big(D\big)\ \text{consists of two lines},\\%
&\frac{1}{3}\ \ \text{if}\ \mathrm{Supp}\big(D\big)\ \text{consists of one line}.\\%
\endaligned\right.%
$$
\end{example}

Now we suppose that $X$ is a Fano variety with at most log
terminal singularities.

\begin{definition}
\label{definition:threshold} The global log canonical threshold of
the Fano variety $X$ is the number
$$
\mathrm{lct}\big(X\big)
=\mathrm{inf}\left\{\mathrm{lct}\big(X,D\big)\ \Big\vert\ D\
\text{is an effective $\mathbb{Q}$-divisor on $X$ with}\ D\
\sim_{\mathbb{Q}} -K_{X}\right\}.%
$$
\end{definition}
The number $\mathrm{lct}(X)$ is an algebraic counterpart of the
$\alpha$-invariant introduced in \cite{Ti87} and \cite{TiYa87}
(see \cite[Appendix~A]{ChSh08c}). Because $X$ is rationally
connected (see \cite{Zh06}), we have
$$
\mathrm{lct}\big(X\big)=\mathrm{sup}\left\{\lambda\in\mathbb{Q}\ \left|\ %
\aligned
&\text{the log pair}\ \Big(X, \lambda D\Big)\ \text{is log canonical for every }\\
&\text{effective $\mathbb{Q}$-divisor numerically equivalent to $-K_{X}$}\\
\endaligned\right.\right\}.
$$
It immediately follows from Definition~\ref{definition:threshold}
that
$$
\mathrm{lct}\big(X\big)=\mathrm{sup}\left\{\eps\in\mathbb{Q}\ \left|\ %
\aligned
&\text{the log pair}\ \left(X, \frac{\eps}{n}D\right)\ \text{is log canonical for every }\\
&\text{divisor}\ D\in\big|-nK_{X}\big|\ \text{and every positive integer }\ n\\
\endaligned\right.\right\}.%
$$

\begin{example}[\cite{ChSh08c}]
\label{example:WPS} Suppose that
$\mathbb{P}(a_{0},a_{1},\ldots,a_{n})$ is a well-formed weighted
projective space with $a_0\leqslant a_1\leqslant\ldots \leqslant
a_n$ (see \cite{IF00}). Then
$$
\mathrm{lct}\Big(\mathbb{P}\big(a_{0},a_{1},\ldots,a_{n}\big)\Big)=\frac{a_{0}}{\sum_{i=0}^{n}a_{i}}.%
$$
\end{example}

\begin{example}
\label{example:Cheltsov-Park} Let $X$ be a smooth hypersurface in
$\mathbb{P}^{n}$ of degree $m\leqslant n$. The paper \cite{Ch01b}
shows that
$$
\mathrm{lct}\big(X\big)=\frac{1}{n+1-m}%
$$
if $m<n$. For the case $m=n\geqslant  2$ it also shows that
$$1-\frac{1}{n}\leqslant \mathrm{lct}(X)\leqslant 1$$
and the left equality holds if $X$ contains a cone of
dimension $n-2$. Meanwhile, the papers \cite{ChPaWo} and
\cite{Pu04d} show that
$$
1\geqslant\mathrm{lct}\big(X\big)\geqslant \left\{%
\aligned
&1\ \text{if}\ n\geqslant  6,\\%
&\frac{22}{25}\ \text{if}\ n=5,\\%
&\frac{16}{21}\ \text{if}\ n=4,\\%
&\frac{3}{4}\ \text{if}\ n=3,\\%
\endaligned\right.%
$$
if $X$ is general.
\end{example}

\begin{example}
\label{example:double-cover} Let $X$ be a smooth hypersurface in
the weighted projective space $\mathbb{P}(1^{n+1},d)$ of degree
$2d\geqslant 4$. Then
$$
\mathrm{lct}\big(X\big)=\frac{1}{n+1-d}%
$$
in the case when $d<n$ (see \cite[Proposition~20]{Ch08a}). Suppose
that $d=n$. Then the inequalities
$$
 \frac{2n-1}{2n}\leqslant \mathrm{lct}\big(X\big)\leqslant 1 %
$$
hold (see \cite{ChPaWo}). But $\mathrm{lct}(X)=1$ if $X$ is
general and $n\geqslant  3$.  Furthermore for the case $n=3$ the
papers \cite{ChPaWo} and  \cite{Pu04d} prove that
$$
\mathrm{lct}\big(X\big)\in\left\{\frac{5}{6},\frac{43}{50},\frac{13}{15},\frac{33}{38},\frac{7}{8},\frac{33}{38},\frac{8}{9}, \frac{9}{10},\frac{11}{12},\frac{13}{14},\frac{15}{16},\frac{17}{18},\frac{19}{20},\frac{21}{22},\frac{29}{30},1\right\}%
$$
and all these values are attained. For instance, if the
hypersurface $X$ is given by
$$
w^{2}=x^6+y^6+z^6+t^6+x^2y^2zt\subset\mathbb{P}\big(1,1,1,1,3\big)\cong\mathrm{Proj}\Big(\mathbb{C}\big[x,y,z,t,w\big]\Big),
$$
where
$\mathrm{wt}(x)=\mathrm{wt}(y)=\mathrm{wt}(z)=\mathrm{wt}(t)=1$
and $\mathrm{wt}(w)=3$, then $\mathrm{lct}(X)=1$ (see
\cite{ChPaWo}).
\end{example}

\begin{example}[\cite{Hw06b}]
\label{example:Hwang} Let $X$ be a rational homogeneous space such
that the Picard group of $X$ is generated by an ample Cartier divisor $D$ and  $-K_{X}\sim rD$ for some positive integer $r$.
Then $\mathrm{lct}(X)=\frac{1}{r}$.
\end{example}

\begin{example}
\label{example:IHES} Let $X$ be a quasismooth well-formed (see
\cite{IF00}) hypersurface in
$\mathbb{P}(1,a_{1},a_{2},a_{3},a_{4})$ of
degree~$\sum_{i=1}^{4}a_{i}$ with at most terminal singularities,
where $a_{1}\nlb\leqslant\nlb\ldots\nlb\leqslant\nlb a_{4}$. Then
there are exactly $95$ possibilities for the quadruple
$(a_{1},a_{2},a_{3},a_{4})$ (see \cite{IF00}, \cite{JoKo01}). For
a general hypersurface $X$, it follows from
    \cite{Ch07a},~\cite{Ch08},~\cite{Ch08d} and ~\cite{ChPaWo}
    that
$$
1\geqslant \mathrm{lct}\big(X\big)\geqslant \left\{%
\aligned
&\frac{16}{21}\ \ \ \text{if}\ a_{1}=a_{2}=a_{3}=a_{4}=1,\\%
&\frac{7}{9}\ \ \ \text{if}\ (a_{1},a_{2},a_{3},a_{4})=(1,1,1,2),\\%
&\frac{4}{5}\ \ \ \text{if}\ (a_{1},a_{2},a_{3},a_{4})=(1,1,2,2),\\%
&\frac{6}{7}\ \ \ \text{if}\ (a_{1},a_{2},a_{3},a_{4})=(1,1,2,3),\\%
&1\ \ \ \text{otherwise}. \\%
\endaligned\right.%
$$
The global log canonical threshold of the hypersurface
$$
w^{2}=t^{3}+z^{9}+y^{18}+x^{18}\subset\mathbb{P}\big(1,1,2,6,9\big)\cong\mathrm{Proj}\Big(\mathbb{C}\big[x,y,z,t,w\big]\Big)
$$
is equal to $\frac{17}{18}$, where
$\mathrm{wt}(x)=\mathrm{wt}(y)=1$, $\mathrm{wt}(z)=2$,
$\mathrm{wt}(t)=6$, $\mathrm{wt}(w)=9$  (see  \cite{Ch07a}).%
\end{example}

\begin{example}[\cite{Ch07c}]
\label{example:singular-cubics} Let $X$ be a singular cubic
surface in $\mathbb{P}^{3}$ with at most canonical
singularities. The possible singularities of $X$ are listed in
\cite{BW79}. The global log canonical threshold of $X$ is as follows:
$$
\mathrm{lct}\big(X\big)=\left\{%
\aligned
&\frac{2}{3}\ \ \ \text{if}\ \mathrm{Sing}\big(X\big)=\big\{\mathbb{A}_{1}\big\},\\%
&\frac{1}{3}\ \ \ \text{if}\ \mathrm{Sing}\big(X\big)\supseteq\big\{\mathbb{A}_{4}\big\},\
 \mathrm{Sing}\big(X\big)=\big\{\mathbb{D}_{4}\big\}\
\text{or}\ \mathrm{Sing}\big(X\big)\supseteq\big\{\mathbb{A}_{2},\mathbb{A}_{2}\big\},\\%
&\frac{1}{4}\ \ \ \text{if}\ \mathrm{Sing}\big(X\big)\supseteq\big\{\mathbb{A}_{5}\big\}\ \text{or}\ \mathrm{Sing}\big(X\big)=\big\{\mathbb{D}_{5}\big\},\\%
&\frac{1}{6}\ \ \ \text{if}\ \mathrm{Sing}\big(X\big)=\big\{\mathbb{E}_{6}\big\},\\%
&\frac{1}{2}\ \ \ \text{otherwise}.\\%
\endaligned\right.%
$$
\end{example}
So far we have not seen any single variety whose global log
canonical threshold is irrational. In general, it is unknown
whether global log canonical thresholds are rational numbers or
not(cf. Question~1 in \cite{Ti90b}). Even for  del Pezzo surfaces with log terminal
singularities the rationality of their global log canonical
thresholds is unknown.  However,
we expect more than this as follows:

\begin{conjecture}
\label{conjecture:stabilization} There is an effective
$\mathbb{Q}$-divisor $D$ on the variety $X$ such that it is $\mathbb{Q}$-linearly equivalent to $-K_X$  and
$$
\mathrm{lct}\big(X\big)=\mathrm{lct}\big(X,D\big).%
$$
\end{conjecture}

The following definition is due to \cite{Sho00} (cf.
\cite{IshiiPr01}, \cite{Kud02}, \cite{MarPr99},  \cite{Pr01}).

\begin{definition}
\label{definition:exceptional-del-Pezzo}  The Fano variety $X$ is
exceptional (resp. weakly exceptional, strongly exceptional) if
for every effective $\mathbb{Q}$-divisor $D$ on the variety $X$
such that $D\sim_{\mathbb{Q}}-K_X$ and the pair $(X,D)$ is log
terminal (resp. $\lct(X)\geqslant 1$, $\lct(X)>1$).
\end{definition}
It is easy to see the implications
\[\text{strongly exceptional} \ \ \Longrightarrow \ \ \text{exceptional}
\Longrightarrow \ \ \text{weakly exceptional}.\] However, if
Conjecture~\ref{conjecture:stabilization} holds for
    $X$, then we see that $X$ is exceptional if and only if $X$ is strongly
    exceptional.
Exceptional del Pezzo surfaces, which are called del Pezzo
surfaces without tigers~in~\cite{KeMa99}, lie in finitely many
families (see \cite{Sho00}, \cite{Pr01}). We expect that strongly
exceptional Fano varieties with quotient singularities enjoy very
interesting geometrical properties (cf. \cite[Theorem~3.3]{Rub08},
\cite[Theorem~1]{PhSeSt07}).

The main motivation for this article is that the global log
canonical threshold turns out to play important roles both in
birational geometry and in complex geometry. We have two
significant applications of the global log canonical threshold of
a Fano variety $X$. The first one is for the case when
$\lct(X)\geqslant 1$. This inequality has serious applications to
rationality problems for Fano varieties in birational geometry.
The other is for the case when
$\lct(X)>\frac{\dim(X)}{1+\dim(X)}$. This has important
applications to K\"ahler-Einstein metrics on Fano varieties in
complex geometry.

For a simple application of the first inequality, we can mention
the following.
\begin{theorem}[\cite{Ch07a} and \cite{Pu04d}]\label{theorem:prodcut-of-Fano}
Let $X_i$ be birationally super-rigid Fano variety with
$\lct(X_i)\geqslant 1$ for each $i=1, \ldots, r$. Then the variety
$X_{1}\times\ldots\times X_{r}$ is non-rational and
$$
\mathrm{Bir}\Big(X_{1}\times\ldots\times X_{r}\Big)= \mathrm{Aut}\Big(X_{1}\times\ldots\times X_{r}\Big).%
$$%
For every dominant map $\rho\colon
    X_{1}\times\ldots\times X_{r}\dasharrow Y$ whose
    general fiber is rationally connected, there is a
    subset $\{i_{1},\ldots,i_{k}\}\subseteq\{1,\ldots,r\}$
    and a commutative diagram
$$
\xymatrix{ X_{1}\times\ldots\times
X_{r}\ar@{->}[d]_{\pi}\ar@{-->}[rr]^{\sigma}&&X_{1}\times\ldots\times
 X_{r}\ar@{-->}[rrd]^{\rho}\\
X_{i_{1}}\times\ldots\times X_{i_{k}}\ar@{-->}[rrrr]_{\xi}&&&&Y,}%
$$
where $\xi$ and $\sigma$  are birational maps, and $\pi$
is the natural projection.%
\end{theorem}
This theorem may be more generalized so that we could obtain the following
\begin{example}[\cite{Ch07a}]
\label{example:CPR} Let $X_{i}$ be a threefold satisfying
hypotheses of Example~\ref{example:IHES} with $\lct(X_i)=1$ for
each $i=1,\ldots r$. Suppose, in addition, that each $X_{i}$ is
general in its deformation family. Then the variety
$X_{1}\times\ldots\times X_{r}$ is non-rational and
$$
\mathrm{Bir}\Big(X_{1}\times\ldots\times X_{r}\Big)=\Big<\prod_{i=1}^{r}\mathrm{Bir}(X_{i}),\ \mathrm{Aut}\Big(X_{1}\times\ldots\times X_{r}\Big)\Big>.%
$$%
For every dominant map $\rho\colon
    X_{1}\times\ldots\times X_{r}\dasharrow Y$ whose
    general fiber is rationally connected, there is a
    subset $\{i_{1},\ldots,i_{k}\}\subseteq\{1,\ldots,r\}$
    and a commutative diagram
$$
\xymatrix{ X_{1}\times\ldots\times
X_{r}\ar@{->}[d]_{\pi}\ar@{-->}[rr]^{\sigma}&&X_{1}\times\ldots\times
 X_{r}\ar@{-->}[rrd]^{\rho}\\
X_{i_{1}}\times\ldots\times X_{i_{k}}\ar@{-->}[rrrr]_{\xi}&&&&Y,}%
$$
where $\xi$ and $\sigma$  are birational maps, and $\pi$
is the  natural projection.%

\end{example}

The following result that gives strong connection between global
log canonical thresholds and K\"ahler-Einstein metrics was proved
in \cite{DeKo01}, \cite{Na90},\cite{Ti87} (see
\cite[Appendix~A]{ChSh08c}).

\begin{theorem}
\label{theorem:KE} Suppose that $X$ is a Fano variety with at most
quotient singularities. Then it admits an orbifold
K\"ahler--Einstein metric if
$$
\mathrm{lct}\big(X\big)>\frac{\mathrm{dim}\big(X\big)}{\mathrm{dim}\big(X\big)+1}.%
$$
\end{theorem}

Examples~\ref{example:Cheltsov-Park}, \ref{example:double-cover}
and \ref{example:IHES} are good examples to which we may apply
Theorem~\ref{theorem:KE}.

There are many known obstructions for the existence of orbifold
K\"ahler--Einstein metrics on Fano varieties with quotient
singularities (see \cite{Do02}, \cite{Fu83}, \cite{Lub83},
\cite{Mat57}, \cite{RoTh06}, \cite{Ti97}).
\begin{example}[\cite{GaMaSpaYau06}]
\label{example:Bishop-Lichnerowicz} Let $X$ be a quasismooth
hypersurface in $\mathbb{P}(a_{0},\ldots,a_{n})$ of degree
$d<\sum_{i=0}^{n}a_{i}$, where $a_{0}\leqslant\ldots\leqslant
a_{n}$. Suppose that $X$ is well-formed and has a
K\"ahler--Einstein metric. Then
$$
d\left(\sum_{i=0}^{n}a_i-d\right)^n\leqslant n^n\prod_{i=0}^{n}a_i,%
$$
and $\sum_{i=0}^{n}a_{i}\leqslant d+na_{0}$
(see \cite{Boy08}, \cite{Spa07}).
\end{example}
The problem of existence of K\"ahler--Einstein metrics on smooth
del Pezzo surfaces is completely solved by \cite{Ti90} as follows:
\begin{theorem}
\label{theorem:smooth-del-Pezzo} If $X$ is a smooth del Pezzo
surface, then the following conditions are equivalent:
\begin{itemize}
\item the automorphism group $\mathrm{Aut}(X)$ is reductive;%
\item the surface $X$ admits a K\"ahler--Einstein metric;%
\item the surface $X$ is not a blow up of $\mathbb{P}^{2}$ at one or two points.%
\end{itemize}
\end{theorem}

\vspace{5mm} \textbf{Acknowledgments.} The first author is
grateful to the Max Plank Institute for Mathematics at Bonn for
the hospitality and excellent working conditions. The first author
was supported by the grants NSF DMS-0701465 and EPSRC
EP/E048412/1, the~third~author was supported by the grants RFFI
No.~08-01-00395-a, N.Sh.-1987.1628.1 and EPSRC EP/E048412/1. The
second author has been supported by the Korea Research Foundation
Grant funded by the Korean Government (KRF-2008-313-C00024).

The authors thank I.~Kim, B.~Sea, and J.~Won for their pointing
out numerous mistakes in the first version of this paper.

\section{Results} \label{subsection:results}

Let $X_d$ be a quasismooth and well-formed hypersurface in
$\mathbb{P}(a_{0},a_{1},a_{2},a_{3})$ of degree $d$, where
$a_0\leqslant a_1\leqslant a_2\leqslant a_4$. Then the
hypersurface $X_d$ is given by a quasihomogeneous polynomial
equation $f(x,y,z,t)=0$ of degree $d$. The quasihomogeneous
equation
$$
f\big(x,y,z,t\big)=0\subset\mathbb{C}^{4}\cong\mathrm{Spec}\Big(\mathbb{C}\big[x,y,z,t\big]\Big),
$$
defines an isolated quasihomogeneous singularity $(V,O)$ with the
Milnor number $\prod_{i=0}^{n}(\frac{d}{a_{i}}-\nolinebreak 1)$,
where $O$ is the origin of $\mathbb{C}^{4}$. It follows from the
adjunction formula that
$$
K_{X_d}\sim_{\mathbb{Q}}\mathcal{O}_{\mathbb{P}(a_{0},\,a_{1},\,a_{2},\,a_{3})}\Big(d-\sum_{i=0}^{3}a_{i}\Big),
$$
and it follows from
\cite{Elk81}, \cite[Proposition~8.14]{Ko97}, \cite{Rie80} that the following conditions are equivalent:
\begin{itemize}
\item the inequality $d\leqslant\sum_{i=0}^{3}a_{i}-1$ holds;%
\item the surface $X_d$ is a del Pezzo surface; %
\item the singularity $(V,O)$ is rational; %
\item the singularity $(V,O)$ is canonical. %
\end{itemize}
Blowing up $\mathbb{C}^{4}$ at the origin $O$ with weights
$(a_{0},a_{1},a_{2},a_{3})$, we get a purely log terminal blow up
of the singularity $(V,O)$ (see \cite{Kud01}, \cite{Pr98plt}). The
paper \cite{Pr98plt} shows that the following conditions are
equivalent:
\begin{itemize}
\item the surface $X_d$ is exceptional (weakly exceptional, respectively); %
\item the singularity $(V,O)$ is exceptional\footnote{For notions
of exceptional and weakly exceptional singularities see
\cite[Definition~4.1]{Pr98plt}, \cite{Sho00}, \cite{IshiiPr01}.}
(weakly exceptional,
respectively). %
\end{itemize}

From now on we suppose that $d\leqslant\sum_{i=0}^{3}a_{i}-1$.
Then $X_d$ is a del Pezzo surface. Put $I=\sum_{i=0}^{3}a_{i}-d$.
The list of possible values of $(a_{0},a_{1},a_{2},a_{3},d)$ with
$2I<3a_{0}$ can be found in \cite{BoGaNa03} and \cite{ChSh09c}.
For the case $I=1$, we can obtain the complete list of del Pezzo
surfaces $X_d\subset\mathbb{P}(a_0,a_1, a_2,a_3)$ from
\cite{JoKo01b} as follows:
\begin{itemize}
\item smooth del Pezzo surfaces\\
$X_3\subset \mathbb{P}(1,1,1,1), \ \  X_4\subset
\mathbb{P}(1,1,1,2), \ \ X_6\subset \mathbb{P}(1,1,2,3)$,
\item singular del Pezzo surfaces\\
$X_{8n+4}\subset \mathbb{P}(2,2n+1,2n+1,4n+1)$, where $n$ is a positive integer,
\\
$X_{10}\subset \mathbb{P}(1,2,3,5), \ \  X_{15}\subset
\mathbb{P}(1,3,5,7), \ \ X_{16}\subset \mathbb{P}(1,3,5,8),  \
\ X_{18}\subset \mathbb{P}(2,3,5,9)$,
\\
$X_{15}\subset \mathbb{P}(3,3,5,5), \ \  X_{25}\subset \mathbb{P}(3,5,7,11), \ \ X_{28}\subset \mathbb{P}(3,5,7,14),  $\\
$ X_{36}\subset \mathbb{P}(3,5,11,18), \ \ X_{56}\subset \mathbb{P}(5,14,17,21), \ \  X_{81}\subset \mathbb{P}(5,19,27,31), $\\
$ X_{100}\subset \mathbb{P}(5,19,27,50),  \ \ X_{81}\subset \mathbb{P}(7,11,27,37), \ \
X_{88}\subset \mathbb{P}(7,11,27,44)$, \\
$
X_{60}\subset \mathbb{P}(9,15,17,20), \ \ X_{69}\subset \mathbb{P}(9,15,23,23),
\ \ X_{127}\subset \mathbb{P}(11,29,39,49)$,\\
$X_{256}\subset \mathbb{P}(11,49,69,128), \ \  X_{127}\subset \mathbb{P}(13,23,35,57), \ \ X_{256}\subset \mathbb{P}(13,35,81,128). $
\end{itemize}


The global log canonical thresholds of such del Pezzo surfaces
have been considered either implicitly or explicitly in
\cite{Ara02}, \cite{BoGaNa02}, \cite{Ch07b}, \cite{DeKo01},
\cite{JoKo01b}. For example, the papers \cite{Ara02},
\cite{BoGaNa02}, \cite{DeKo01} and  \cite{JoKo01b} gives us lower
bounds for global log canonical thresholds of singular del Pezzo
surfaces with $I=1$.
Meanwhile, the paper \cite{Ch07b} deals with the exact values of
the global log canonical thresholds of smooth del Pezzo surfaces
with $I=1$.
\begin{theorem}
\label{theorem:del-Pezzos-smooth} Suppose that $I=1$ and $X_d$ is
smooth. Then
$$
\mathrm{lct}\big(X_d\big)=\left\{%
\aligned
&1\ \ \ \text{if}\ \big(a_{0},a_{1},a_{2},a_{3}\big)=\big(1,1,2,3\big)\ \text{and}\ |-K_{X_6}|\ \text{contains no cuspidal curves},\\%
&\frac{5}{6}\ \ \ \text{if}\ \big(a_{0},a_{1},a_{2},a_{3}\big)=\big(1,1,2,3\big)\ \text{and}\ |-K_{X_6}|\ \text{contains a cuspidal curve},\\%
&\frac{5}{6}\ \ \ \text{if}\ \big(a_{0},a_{1},a_{2},a_{3}\big)=\big(1,1,1,2\big)\ \text{and}\ |-K_{X_4}|\ \text{contains no tacnodal curves},\\%
&\frac{3}{4}\ \ \ \text{if}\ \big(a_{0},a_{1},a_{2},a_{3}\big)=\big(1,1,1,2\big)\ \text{and}\ |-K_{X_4}|\ \text{contains a tacnodal curve},\\%
&\frac{3}{4}\ \ \ \text{if}\ X_3\ \text{is a cubic in}\ \mathbb{P}^{3}\ \text{with no Eckardt points},\\%
&\frac{2}{3}\ \ \ \text{if }\ X_3\ \text{is a cubic in}\ \mathbb{P}^{3}\ \text{with an Eckardt point}.\\%
\endaligned\right.%
$$
\end{theorem}

However, for singular del Pezzo surfaces, the exact
values of global log canonical thresholds have not been considered
seriously.

A singular del Pezzo hypersurface $X_d\subset \mathbb{P}(a_0, a_1,
a_2, a_3)$ must satisfy exclusively one of the following
properties:
\begin{enumerate}
\item $2I\geqslant 3a_{0}$ ;%
\item $2I<3a_{0}$ and
$$
\big(a_{0},a_{1},a_{2},a_{3},d\big)=\big(I-k,I+k,a,a+k,2a+k+I\big)
$$
for some non-negative integer $k<I$  and some positive integer
$a\geqslant I+k$.
\item $2I<3a_{0}$ but
$$
    \big(a_{0},a_{1},a_{2},a_{3},d\big)\ne\big(I-k,I+k,a,a+k,2a+k+I\big)
$$
for any non-negative integer $k<I$  and any positive integer
$a\geqslant I+k$.
\end{enumerate}
For the first two cases one can check that $\lct (X_d)\leq
\frac{2}{3}$ (for instance, see \cite{BoGaNa03} and
\cite{ChSh09c}). All the quintuples $(a_{0},a_{1},a_{2},a_{3},d)$
such that the hypersurface $X_d$ is singular and satisfies the
last condition are listed in Section~\ref{section:table}. They are
taken from \cite{BoGaNa03}  and \cite{ChSh09c}. Note that we
rearranged a little the quintuples taken from \cite{BoGaNa03} by
putting some cases that were contained in the infinite series
of~\cite{BoGaNa03} into the sublist of sporadic cases; on the
other hand, we removed two sporadic cases, because they are
contained in the additional infinite series found in
\cite{ChSh09c}. The completeness of this list is proved in
\cite{ChSh09c} by using \cite{YauYu03}.

We already know the global log canonical thresholds of smooth del
Pezzo surfaces. For del Pezzo surfaces satisfying one of the first
two conditions, their global log canonical thresholds are
relatively too small to enjoy the condition of
Theorem~\ref{theorem:KE}. However, the global log canonical
thresholds of del Pezzo surfaces satisfying the last condition
have not been investigated sufficiently. In the present paper we
compute all of them and then we obtain the following result.

\begin{theorem}
\label{theorem:main} Let $X_d$ be a quasismooth well-formed
singular del Pezzo surface in the weighted projective space
$\mathrm{Proj}(\mathbb{C}\big[x,y,z,t\big])$ with weights
$\mathrm{wt}(x)=a_{0}\leqslant \mathrm{wt}(y)=a_{1}\leqslant
\mathrm{wt}(z)=a_{2}\leqslant\mathrm{wt}(t)=a_{3}$ such that
$2I<3a_{0}$ but
$(a_{0},a_{1},a_{2},a_{3},d)\ne(I-k,I+k,a,a+k,2a+k+I)$ for any
non-negative integer $k<I$  and any positive integer $a\geqslant
I+k$, where $I=\sum_{i=0}^{3}a_{i}-d$. Then if $a_0\ne a_1$, then
$$
\lct(X_d)=\mathrm{min}\left\{\lct\Big(X_d,
\frac{I}{a_{0}}C_{x}\Big),\ \lct\Big(X_d,
\frac{I}{a_{1}}C_{y}\Big),\ \lct\Big(X_d,
\frac{I}{a_{2}}C_{z}\Big)\right\},
$$
where $C_x$ (resp. $C_y$, $C_z$) is the divisor on $X_d$ defined
by $x=0$ (resp. $y=0$, $z=0$). If $a_0=a_1$, then
$$\lct(X_d)=\lct\Big(X_d, \frac{I}{a_0}C\Big),$$
where $C$ is a reducible divisor in $|\mathcal{O}_{X_d}(a_0)|$.
\end{theorem}
In particular, we obtain the value of $\mathrm{lct}(X_d)$ for
every del Pezzo surface $X_d$ listed in
Section~\ref{section:table}. As a result, we obtain the following
corollaries.

\begin{corollary}
\label{corollary:1} The following assertions are equivalent:
\begin{itemize}
\item the surface $X_d$ is exceptional;%
\item $\lct(X_d)>1$ ;%
\item the quintuple $(a_{0},a_{1},a_{2},a_{3},d)$ lies in the set
$$
\left\{\aligned
&(2,3,5,9,18), (3,3,5,5,15), (3,5,7,11,25), (3,5,7,14, 28),\\%
&(3,5,11,18, 36),(5,14,17,21,56),(5,19,27,31,81),(5,19,27,50,100),\\
&(7,11,27,37,81),(7,11,27,44,88), (9,15,17,20,60), (9,15,23,23,69),\\
&(11,29,39,49,127), (11,49,69,128,256), (13,23,35,57,127), \\
&(13,35,81,128,256), (3,4,5,10,20), (3,4,10,15,30), (5,13,19,22,57),\\
&(5,13,19,35,70), (6,9,10,13,36), (7,8,19,25,57), (7,8,19,32,64),\\
&(9,12,13,16,48), (9,12,19,19,57), (9,19,24,31,81), (10,19,35,43,105),\\
&(11,21,28,47,105), (11,25,32,41,107), (11,25,34,43,111), (11,43,61,113,226),\\
&(13,18,45,61,135), (13,20,29,47,107), (13,20,31,49,111), (13,31,71,113,226),\\
&(14,17,29,41,99), (5,7,11,13,33), (5,7,11,20,40), (11,21,29,37,95),\\
&(11,37,53,98,196), (13,17,27,41,95), (13,27,61,98,196), (15,19,43,74,148),\\
&(9,11,12,17,45), (10,13,25,31,75), (11,17,20,27,71), (11,17,24,31,79), \\
&(11,31,45,83,166), (13,14,19,29,71), (13,14,23,33,79), (13,23,51,83,166), \\
&(11,13,19,25,63),(11,25,37,68,136), (13,19,41,68,136),(11,19,29,53,106),\\
&(13,15,31,53,106), (11,13,21,38,76)\\
\endaligned\right\}.
$$
\end{itemize}
\end{corollary}

\begin{corollary}
\label{corollary:2} The following assertions are equivalent:
\begin{itemize}
\item the surface $X_d$ is weakly exceptional and not exceptional;%
\item $\lct(X_d)=1$;%
\item one of the following holds
\begin{itemize}
\item the quintuple $(a_{0},a_{1},a_{2},a_{3},d)$ lies in the set
$$
\left\{\aligned
&(2,2n+1,2n+1,4n+1,8n+4),\\
& (3,3n,3n+1,3n+1, 9n+3), (3, 3n+1, 3n+2, 3n+2, 9n+6),\\
& (3,3n+1,3n+2,6n+1, 12n+5),(3,3n+1, 6n+1, 9n, 18n+3), \\%
&(3, 3n+1, 6n+1, 9n+3, 18n+6), (4, 2n+1, 4n+2, 6n+1, 12n+6),\\%
& (4, 2n+3, 2n+3, 4n+4, 8n+12), (6, 6n+3, 6n+5, 6n+5, 18n+15),\\
&(6, 6n+5, 12n+8, 18n+9, 36n+24), (6, 6n+5, 12n+8, 18n+15, 36n+30),\\%
&(8, 4n+5, 4n+7, 4n+9, 12n+23), (9, 3n+8, 3n+11, 6n+13, 12n+35),\\
&(1,3,5,8,16), (2,3,4,7,14), (5,6,8,9,24), (5,6,8,15,30)%
\endaligned\right\},
$$
where $n$ is a positive integer,  \item
$(a_{0},a_{1},a_{2},a_{3},d)=(1,1,2,3,6)$ and the pencil
$|-K_{X}|$ does not have cuspidal curves, \item
$(a_{0},a_{1},a_{2},a_{3},d)=(1,2,3,5,10)$ and $C_x=\{x=0\}$ has
an ordinary double point, \item
$(a_{0},a_{1},a_{2},a_{3},d)=(1,3,5,7,15)$ and the defining
equation of $X$ contains $yzt$, \item
$(a_{0},a_{1},a_{2},a_{3},d)=(2,3,4,5,12)$ and  the defining
equation of $X$ contains $yzt$.
\end{itemize}
\end{itemize}
\end{corollary}

\begin{corollary}
\label{corollary:3} In the notation and assumptions of
Theorem~\ref{theorem:main}, the surface $X_d$ has an orbifold
K\"ahler--Einstein metric with the following possible exceptions:
$X_{45}\subset\nolinebreak\mathbb{P}(7,10,15,19)$,
$X_{81}\subset\mathbb{P}(7,18,27,37)$,
$X_{64}\subset\mathbb{P}(7,15,19,32)$,
$X_{82}\subset\mathbb{P}(7,19,25,41)$,
$X_{117}\subset\mathbb{P}(7,26,39,55)$,
$X_{15}\subset\mathbb{P}(1,3,5,7)$ whose  defining equation does
not contain $yzt$,  and $X_{12}\subset\mathbb{P}(2,3,4,5)$ whose
defining equation does not contain $yzt$.
\end{corollary}

Corollary~\ref{corollary:1} illustrates the fact that exceptional
del Pezzo surfaces lie in finitely many families (see
\cite{Sho00}, \cite{Pr01}). On the other hand,
Corollary~\ref{corollary:1} shows that weakly-exceptional del
Pezzo surfaces do not enjoy this property. Note also that
Corollary~\ref{corollary:1} follows from \cite{Kud02}.


\section{Preliminaries}
\label{section:preliminaries}

For the basic definitions and properties concerning singularities
of pairs we refer the reader to \cite{Ko97}. To prove
Theorem~\ref{theorem:main} we need to compute the log canonical
thresholds of individual effective divisors. The following two
lemmas are rather basic properties of log canonical thresholds but
will be useful to compute  them. For the proofs the reader is
referred to \cite{Ko97} and \cite{Ku99}.

\begin{lemma}\label{lemma:basic-property}
Let $f\in\mathbb{C}[x_1,\ldots, x_n]$ and $D=(f=0)$. Suppose that
the polynomial $f$ vanishes at the origin $O$ in $\mathbb{C}^n$.
Set $d=\mathrm{mult}_O(f)$ and let $f_d$ denote the degree $d$
homogeneous part of $f$. Let $T_0D=(f_d=0)\subset\mathbb{C}^n$ be
the tangent cone of $D$ and
$\mathbb{P}(T_0D)=(f_d=0)\subset\mathbb{P}^{n-1}$ be the
projectivised tangent cone of $D$. Then
\begin{enumerate}

\item $\frac{1}{d}\leq \lct_O(\mathbb{C}^n, D)\leq\frac{n}{d}$.

 \item The log
pair $(\mathbb{P}^{n-1},\frac{n}{d}\mathbb{P}(T_0D))$ is log canonical if and
only if $\lct_O(\mathbb{C}^n, D)=\frac{n}{d}$.

\item If $\mathbb{P}(T_0D)$ is smooth (or even log canonical) then
$\lct_O(\mathbb{C}^n, D)=\mathrm{min}\{1,\frac{n}{d}\}$.

\end{enumerate}
\end{lemma}

\begin{lemma}
\label{lemma:Igusa} Let $f$ be a polynomial in
$\mathbb{C}[z_1,z_2]$. Suppose that the polynomial defines an
irreducible curve $C$ passing through the origin $O$ in
$\mathbb{C}^{2}$. We then have
$$
\lct_O(\mathbb{C}^2, C)=\mathrm{min}\left(1,\frac{1}{m}+\frac{1}{n}\right),%
$$
where $(m,n)$ is the first pair of Puiseux exponents of $f$. We also have
$$
\lct_O\left(\mathbb{C}^2, \left(z_{1}^{n_{1}}z_{2}^{n_{2}}\left(z_{1}^{m_{1}}+z_{2}^{m_{2}}\right)=0\right)\right)
=\mathrm{min}\left(\frac{1}{n_{1}},\frac{1}{n_{2}},\frac{m_1+m_2}{m_1m_2+m_1n_2+m_2n_1}\right),%
$$
where $n_{1}$, $n_{2}$, $m_{1}$, $m_{2}$ are non-negative
integers.
\end{lemma}

Throughout the proof of Theorem~\ref{theorem:main}, Inversion of
Adjunction that enables us to compute log canonical thresholds on
lower dimensional varieties will be frequently utilized. Let $X$
be a normal (but not necessarily projective) variety. Let $S$ be a
smooth Cartier divisor on $X$ and $B$ be an effective
$\mathbb{Q}$-Cartier $\mathbb{Q}$-divisor on $X$ such that
$K_X+S+B$ is $\mathbb{Q}$-Cartier and
$S\not\subseteq\mathrm{Supp}(B)$.

\begin{theorem}
\label{theorem:adjunction}
The log pair $(X, S+B)$ is log canonical along $S$ if and only if the log pair $(S, B|_S)$ is log canonical.
\end{theorem}

In the case when $X$ is a surface,
Theorem~\ref{theorem:adjunction} can be stated in terms of local
intersection numbers.

\begin{lemma}
\label{lemma:adjunction} Suppose that $X$ is a surface.  Let $P$
be a smooth point of $X$ such that it is also a smooth point of
$S$. Then the log pair $(X, S+B)$ is log canonical at the point
$P$ if and only if the local intersection number of $B$ and $S$ at
the point $P$ is at most $1$. In particular, if the log pair $(X,
mS+B)$ is not log canonical at the point $P$ for $m\leqslant 1$,
then $B\cdot S>1$.
\end{lemma}

\begin{lemma}\label{lemma:multiplicity}
Let $D$ be an effective $\mathbb{Q}$-divisor such that $K_X+D$ is $\mathbb{Q}$-Cartier. For a smooth point $P$ of $X$, the log pair $(X, D)$ is log canonical at the point $P$ if $\mult_P(D)\leqslant 1$.
\end{lemma}

Throughout the proof of Theorem~\ref{theorem:main}, we interrelate Lemma~\ref{lemma:multiplicity} with Lemma~\ref{lemma:adjunction} to get some contradictions. To do so, we need the following lemma that plays the role of a bridge between them.
\begin{lemma}\label{lemma:convexity}
Let $D_1$ and $D_2$ be effective $\mathbb{Q}$-divisors on $Y$ with
$D_1\sim_{\mathbb{Q}} D_2$. Suppose that the pair $(X,D_1)$ is not
log canonical at a point $P\in Y$ but the pair $(X,D_2)$ is  log
canonical at the point $P$. Then there is an effective
$\mathbb{Q}$-divisor  $D$ on $Y$ such that
\begin{itemize}
\item  $D\sim_{\mathbb{Q}} D_1$; \item at least one irreducible
component of $D_2$ is not contained in the support of $D$; \item
the pair $(X, D)$ is not log canonical at the point $P$.
\end{itemize}
\end{lemma}
\begin{proof}
Write $D_2=\sum_{i=1}^rb_iC_i$ where $b_i$'s are positive rational numbers and $C_i$'s are distinct irreducible and reduced divisors. Also, we write $D_1=\Delta+\sum_{i=1}^{r}e_iC_i$ where $e_i$'s are non-negative rational numbers and $\Delta$ is an effective $\mathbb{Q}$-divisor whose support contains none of $C_i$'s. Suppose that $e_i>0$ for each $i$. If not, then we put $D=D_1$.
Let
$$
\alpha=\mathrm{min}\left\{\frac{e_{i}}{b_{i}}\ \Big\vert\ i= 1,2, \ldots, r\right\}.%
$$
Then the positive rational number $\alpha$ is less than $1$ since
$D_1\sim_{\mathbb{Q}} D_2$. Put
\[\begin{split} D&=\frac{1}{1-\alpha}D_1-\frac{\alpha}{1-\alpha}D_2\\
&=\frac{1}{1-\alpha}\Delta+\sum_{i=1}^r\left(\frac{e_i-\alpha b_i}{1-\alpha}\right)C_i. \\ \end{split}\]
It is easy to see that the divisor $D$ satisfies the first two conditions. If the pair $(X, D)$ is log canonical at the point $P$, then the pair $(X, D_1)=(X, (1-\alpha)D+\alpha D_2)$ must be log canonical at the point $P$. Therefore, the divisor $D$ also satisfies the last condition.
\end{proof}

In the present paper, we deal with surfaces with at most quotient singularities. However, the statements mentioned so far require smoothness of the ambient space for us to utilize them to the fullest. Fortunately, the following proposition enables us to apply the statements with ease since we have a natural  finite morphism of a germ of the origin in $\mathbb{C}^2$ to a germ of a quotient singularity that is ramified only at a point.
\begin{proposition}[\cite{Ko97}]\label{proposition:lc-by-finite-morphism}
Let $f\colon Y\to X$ be a finite morphism between normal varieties
and assume that $f$ is unramified outside a set of codimension
two. Let $D$ be an effective $\mathbb{Q}$-Cartier
$\mathbb{Q}$-divisor. Then a log pair $(X, D)$ is log canonical
(resp. Kawamata log terminal) if and only if the log pair $(Y,
f^*D)$ is log canonical (resp. Kawamata log terminal).
\end{proposition}

The following two lemmas will be useful for this paper. The first
lemma is just a reformulation of Lemma~\ref{lemma:adjunction}
mixed with Proposition~\ref{proposition:lc-by-finite-morphism}
that we can apply to our cases immediately.

Suppose that $X$ is a quasismooth well-formed hypersurface in
$\mathbb{P}=\mathbb{P}(a_0,a_1,a_{2},a_{3})$ of degree~$d$.

\begin{lemma}
\label{lemma:handy-adjunction} Let $C$ be a reduced and
irreducible curve on $X$ and $D$ be an  effective $\Q$-divisor on
$X$. Suppose that for a given positive rational number $\lambda$
we have $\lambda \mult_C(D)\leqslant 1$. If $\lambda
(C\cdot\nolinebreak D-\nolinebreak(\mult_C(D)) C^2)\leqslant 1$,
then the pair $(X, \lambda D)$ is log canonical at each smooth
point $P$ of $C$ not in $\mathrm{Sing}(X)$. Furthermore, if the
point $P$ of $C$ is a singular point of $X$ of type
$\frac{1}{r}(a,b)$ and $r\lambda (C\cdot D-(\mult_C(D))
C^2)\leqslant 1$, then the pair $(X, \lambda D)$ is log canonical
at $P$.
\end{lemma}

\begin{proof}
We may write $D=mC+\Omega$, where $\Omega$ is an effective divisor
whose support does not contain the curve $C$. Suppose that the
pair  $(X, \lambda D)$ is not log canonical at a smooth point $P$
of $C$ not in $\mathrm{Sing}(X)$. Since $\lambda m\leqslant 1$,
the pair $(X, C+\lambda \Omega)$ is not log canonical at the point
$P$. Then by Lemma~\ref{lemma:adjunction} we obtain an absurd
inequality
\[1<\lambda \Omega\cdot C =\lambda C\cdot (D-mC)\leqslant 1.\]
Also, if the point $P$ is a singular point of $X$, then we obtain from Lemma~\ref{lemma:adjunction} and Proposition~\ref{proposition:lc-by-finite-morphism}
\[\frac{1}{r}<\lambda \Omega\cdot C =\lambda C\cdot (D-mC)\leqslant \frac{1}{r}.\]
This proves the second statement.
\end{proof}

Let $D$ be an effective $\mathbb{Q}$-divisor on $X$ such that
$$
D\sim_{\mathbb{Q}}
\mathcal{O}_{\mathbb{P}(a_0,\,a_{1},\,a_{2},\,a_{3})}\big(I\big)\Big\vert_{X}.
$$
The next lemma will be applied to show that the log pair $(X, D)$ is log canonical at some smooth points on $X$.

\begin{lemma}
\label{lemma:Carolina} Let $k$ be a positive integer. Suppose that $H^0(\mathbb{P}, \mathcal{O}_{\mathbb{P}}(k))$ contains \begin{itemize}
\item at least two different monomials of the form $x^{\alpha}y^{\beta}$,%
\item at least two different monomials of the form $x^{\gamma}z^{\delta}$.
\end{itemize}
For a smooth point $P$ of $X$ in the outside of $C_x$,
$$
\mult_{P}\big(D\big)\leqslant\frac{Ikd}{a_0a_1a_2a_{3}}
$$
if either
$H^0(\mathbb{P}, \mathcal{O}_{\mathbb{P}}(k))$ contains
at least two different monomials of the form $x^{\mu}t^{\nu}$ or the point $P$ is not contained in a curve contracted by the projection $\psi :X\dasharrow \mathbb{P}(a_0, a_1,a_2)$.
Here, $\alpha$, $\beta$, $\gamma$, $\delta$, $\mu$ and $\nu$ are non-negative
integers.
\end{lemma}

\begin{proof}
The first case follows from \cite[Lemma~3.3]{Ara02}.
Arguing as in the proof of \cite[Corollary~3.4]{Ara02}, we can also obtain
the second case.
\end{proof}

Let us conclude this section by mentioning two results that are
never used in this paper, but nevertheless can be used to give
shorter proofs of Corollaries~\ref{corollary:1} and
\ref{corollary:3}. Suppose that $X$ is given by a quasihomogeneous
equation
$$
f\big(x,y,z,t\big)=0\subset\mathbb{P}\big(a_0,a_1,a_{2},a_{3}\big)\cong\mathrm{Proj}\Big(\mathbb{C}\big[x,y,z,t\big]\Big),
$$
where $\mathrm{wt}(x)=a_{0}\leqslant \mathrm{wt}(y)=a_{1}\leqslant
\mathrm{wt}(z)=a_{2}\leqslant\mathrm{wt}(t)=a_{3}$.

\begin{lemma}
\label{lemma:Boyer} Suppose that $I=\sum_{i=0}^{3}a_{i}-d>0$. Then
$$
\mathrm{lct}\big(X\big)\geqslant\left\{%
\aligned
&\frac{a_{0}a_{1}}{dI},\\%
&\frac{a_{0}a_{2}}{dI}\ \text{if}\ f(0,0,z,t)\ne 0,\\%
&\frac{a_{0}a_{3}}{dI}\ \text{if}\ f(0,0,0,t)\ne 0.\\%
\endaligned\right.%
$$
\end{lemma}

\begin{proof}
See \cite[Corollary~5.3]{BoGaNa03} (cf.
\cite[Proposition~11]{JoKo01b}).
\end{proof}

\begin{lemma}
\label{lemma:Boyer-Kollar}Suppose that
$I=\sum_{i=0}^{3}a_{i}-d>0$, the curve $C_x=\{x=0\}$ is
irreducible and reduced. Then
$$
\mathrm{lct}\big(X\big)\geqslant\left\{%
\aligned
&\min\left(\frac{a_{1}a_{2}}{dI},\ \lct\left(X, \frac{I}{a_{0}}C_{x}\right)\right),\\%
&\min\left(\frac{a_{1}a_{3}}{dI},\ \lct\left(X, \frac{I}{a_{0}}C_{x}\right)\right)\ \text{if}\ f(0,0,0,t)\ne 0.\\%
\endaligned\right.%
$$
\end{lemma}

\begin{proof}
Arguing as in the proof of \cite[Proposition~11]{JoKo01b} and
using Lemma~\ref{lemma:convexity}, we obtain the required
assertion.
\end{proof}

\section{Notation}

We reserve the following notation that
will be used throughout the paper:
\begin{itemize}
\item $\mathbb{P}(a_0, a_1, a_2, a_3)$ denotes the well-formed
weighted projective space
$\mathrm{Proj}(\mathbb{C}\big[x,y,z,t\big])$ with weights
$\mathrm{wt}(x)=a_{0}$, $\mathrm{wt}(y)=a_{1}$,
$\mathrm{wt}(z)=a_{2}$, $\mathrm{wt}(t)=a_{3}$, where we always
assume the inequalities $a_{0}\leqslant a_1\leqslant a_2\leqslant
a_{3}$.  We may use simply $\mathbb{P}$ instead of
$\mathbb{P}(a_0, a_1, a_2, a_3)$ when this does not lead to
confusion.

 \item $X$ denotes a quasismooth and well-formed hypersurface in $\mathbb{P}(a_0, a_1, a_2, a_3)$ (see Definitions~6.3 and
6.9 in \cite{IF00}, respectively).

    \item $O_x$ is the point in $\mathbb{P}(a_0, a_1, a_2, a_3)$ defined by $y=z=t=0$. The points $O_y$, $O_z$ and $O_t$ are defined in the similar way.

    \item  $C_x$ is the curve on $X$ cut out by the equation $x=0$. The curves $C_y$, $C_z$ and $C_t$ are defined in the similar way.
     \item $L_{xy}$ is the curve in $\mathbb{P}(a_0, a_1, a_2, a_3)$ defined by
     $x=y=0$. The curves $L_{xz}$, $L_{xt}$, $L_{yz}$, $L_{yt}$
     and $L_{zt}$ are defined in the similar way.

\item Let $D$ be a divisor on $X$ and $P\in X$. Choose an orbifold
chart $\pi\colon\tilde{U}\to U$ for some neighborhood $P\in
U\subset X$. We put $\mult_P(D)=\nolinebreak\mult_Q(\pi^*D)$,
where $Q$ is a point on $\tilde{U}$ with $\pi(Q)=P$,  and refer to
this quantity as the multiplicity of $D$ at $P$.
\end{itemize}

\section{The scheme of the proof}

We have 83 families\footnote{By family we mean either
one-parameter series (which actually gives rise to an infinite
number of deformation families) or a sporadic case. We hope that
this would not lead to a confusion.} of del Pezzo hypersurfaces
in The Big Table. In the present section we explain the methods to
compute the global log canonical thresholds of the del Pezzo
hypersurfaces  in The Big Table.

Let $X\subset \mathbb{P}(a_0, a_1, a_2, a_3)$ be a del Pezzo
surface of degree $d$  in one of the 83 families (actually, one
infinite series has been treated in \cite{ChSh09c}, so we will
omit the computations in this case). Set $I=a_0+a_1+a_2+a_3-d$.
There are two exceptional cases where $a_0=a_1$. The method for
these two cases is a bit different from the other cases. Both
cases will be individually dealt with
(Lemmas~\ref{lemma:I-2-infinite-series-1-n-1}
and~\ref{lemma:3355}).

If $a_0\ne a_1$, then we will take steps as follows:

\textbf{Step 1.} Using Lemmas~\ref{lemma:basic-property} and~\ref{lemma:Igusa} with Proposition~\ref{proposition:lc-by-finite-morphism}, we compute the log canonical thresholds $\lct(X,\frac{I}{a_0}C_x)$, $\lct(X,\frac{I}{a_0}C_y)$, $\lct(X,\frac{I}{a_0}C_z)$ and $\lct(X,\frac{I}{a_0}C_t)$.
Set $$\lambda =\min\left\{\lct(X,\frac{I}{a_0}C_x), \lct(X,\frac{I}{a_0}C_y), \lct(X,\frac{I}{a_0}C_z), \lct(X,\frac{I}{a_0}C_t)\right\}.$$
Then the global log canonical threshold $\lct(X)$ is at most $\lambda$.
\bigskip

\textbf{Step 2.}  We claim that the global log canonical threshold $\lct(X)$ is equal to $\lambda$. To prove this assertion, we suppose $\lct(X)<\lambda$. Then there is an effective $\mathbb{Q}$-divisor $D$ equivalent to the anticanonical divisor $-K_X$ of $X$ such that the log pair $(X, \lambda D)$ is not log canonical at some point $P\in X$. In particular,  we obtain
$$
\mult_P(\lambda D)>\left\{%
\aligned
&1\ \text{ if the point $P$ is a smooth point of $X$,}\\%
&\frac{1}{r}\ \text{ if the point $P$ is a singular point of $X$ of type $\frac{1}{r}(a,b)$.}\\%
\endaligned\right.%
$$
from Lemma~\ref{lemma:multiplicity} and Proposition~\ref{proposition:lc-by-finite-morphism}.

\bigskip
\textbf{Step 3.}  We show that the point $P$ cannot be a singular point of $X$ using the following methods.
\medskip

\textbf{Method 3.1. (Multiplicity)} We may assume that a suitable irreducible component $C$ of $C_x$, $C_y$, $C_z$, and $C_t$  is not contained in the support of the divisor $D$. We derive a possible contradiction from the inequality
\[ C\cdot D\geqslant \mult_P(C)\cdot \frac{\mult_P(D)}{r}>\frac{\mult_P(C)}{r\lambda},\]
where $r$ is the index of the quotient singular point $P$. The
last inequality follows from the assumption that $(X, \lambda D)$
is not log canonical at $P$. This method can be applied to exclude
a smooth point.
\medskip

  \textbf{Method 3.2. (Inversion of Adjunction)}
  We consider a suitable irreducible curve $C$ smooth at $P$. We then write $D=\mu C+\Omega$, where $\Omega$ is an effective $\mathbb{Q}$-divisor whose support does not contain $C$. We check $\lambda\mu \leqslant 1$. If so, then the log pair $(X, C+\lambda \Omega)$ is not log canonical at the point $P$ either. By Lemma~\ref{lemma:handy-adjunction} we have
  \[\lambda(D-\mu C)\cdot C =\lambda C\cdot \Omega >\frac{1}{r}.\]
  We try to derive a contradiction from this inequality. The  curve $C$ is taken usually from an irreducible component of $C_x$, $C_y$, $C_z$, or $C_t$. This method can be applied to exclude a smooth point.

  \textbf{Method 3.3. (Weighted Blow Up)}
Sometimes we cannot exclude a singular point $P$ only with the
previous two methods. In such a case, we take a suitable weighted
blow up $\pi:Y\to X$ at the point $P$. We can write
  \[K_Y+D^{Y}\sim_{\mathbb{Q}}\pi^*(K_X+\lambda D),\]
where $D^Y$ is the log pull-back of $\lambda D$ by $\pi$. Using
method 3.1 we obtain that $D^{Y}$ is effective. Then we apply the
previous two methods to the pair $(Y, D^Y)$, or repeat this method
until we get a contradictory inequality.
\bigskip

\textbf{Step 4.}  We show that the point $P$ cannot be a smooth
point of $X$. To do so, we first apply Lemma~\ref{lemma:Carolina}.
However, this method does not work always. If the method fails,
then we try to find a suitable pencil $\mathcal{L}$ on $X$. The
pencil has a member $F$ which passes through the point $P$. We
show that the pair $(X, \lambda F)$ is log canonical at the point
$P$. Then, we may assume that the support of $D$ does not contain
at least one irreducible component of $F$.
  If the divisor $D$ itself is irreducible, then we use Method~3.1 to exclude the point $P$. If $F$ is reducible, then we use Method~3.2.

\part{Infinite series}
\label{section:series}

\section{Infinite series with $I=1$}
\label{subsection:series-I-1}

\begin{lemma}\label{lemma:infinite-series}
Let $X$ be a quasismooth hypersurface of degree $8n+4$ in $\mathbb{P}(2, 2n+1, 2n+1, 4n+1)$ for a natural number $n$.
 Then $\lct(X)=1$.
\end{lemma}

\begin{proof}
The surface $X$ is singular at the point $O_{t}$, which is of type $\frac{1}{4n+1}(1,1)$.
It has also four singular points $O_{1}$, $O_{2}$, $O_{3}$,
$O_{4}$, which are cut out on $X$ by $L_{xt}$. Each
$O_{i}$ is a singular point of type $\frac{1}{2n+1}(1,n)$ on the
surface $X$.

The curve $C_{x}$ is reducible. We see
$$
C_{x}=L_{1}+L_{2}+L_{3}+L_{4},
$$
where $L_{i}$ is a smooth rational curves
such that
$$
-K_{X}\cdot L_{i}=\frac{1}{(2n+1)(4n+1)},
$$
and $L_{1}\cap L_{2}\cap L_{3}\cap L_{4}=\{O_{t}\}$. Then

$$L_{i}\cdot
L_{j}=\frac{1}{4n+1}$$ for $i\ne j$. Also, we have
$$
L_{i}^2=C_x\cdot L_i-\frac{3}{4n+1}=\frac{2}{(2n+1)(4n+1)}-\frac{3}{4n+1}=-\frac{6n+1}{(2n+1)(4n+1)}.%
$$

It is easy to see $\lct(X, \frac{1}{2}C_x)=1$. Therefore, $\lct(X)\leqslant
1$. Suppose that $\lct(X)<1$. Then there is an effective
$\Q$-divisor $D\qlineq -K_X$ such that the log pair $(X,D)$ is not log
canonical at some point $P\in X$.

Since
$$
\frac{(4n+2)(8n+4)}{2(2n+1)^{2}(4n+1)}=\frac{4}{4n+1}<1
$$
and $H^0(\P,
\mathcal{O}_\P(4n+2))$ contains $x^{2n+1}$, $y^{2}$ and $z^{2}$, Lemma~\ref{lemma:Carolina} implies that  $P\in C_x$.

It follows from Lemma~\ref{lemma:convexity} that we may
assume that $L_{i}\not\subset\mathrm{Supp}(D)$ for some
$i$.
Also, $P\in L_{j}$ for some $j$. Put
$D=mL_{j}+\Omega$, where $\Omega$ is an effective
$\mathbb{Q}$-divisor such that
$L_{j}\not\subset\mathrm{Supp}(\Omega)$. Since
$$
\frac{1}{(2n+1)(4n+1)}=D\cdot L_{i}=\big(mL_{j}+\Omega\big)\cdot L_{i}\geqslant  mL_{i}\cdot L_{j}=\frac{m}{4n+1},%
$$
we have $0\leqslant m\leqslant \frac{1}{2n+1}$.
Since
\[(2n+1)\Omega\cdot L_j=(2n+1)(D-mL_j)\cdot L_j=(2n+1)\frac{1+m(6n+1)}{(2n+1)(4n+1)}\leqslant
\frac{2}{(2n+1)} <1\] Lemma~\ref{lemma:handy-adjunction} implies
the point $P$ must be $O_t$. Note that  the inequality
$$
\mult_{O_{t}}\big(D\big)\leqslant (4n+1)D\cdot L_{i}=\frac{1}{2n+1}\leqslant 1,
$$
shows that the point $P$ cannot be the point $O_t$. This is a contradiction.
\end{proof}

\section{Infinite series with $I=2$}
\label{subsection:series-I-2}

\begin{lemma}
\label{lemma:I-2-infinite-series-1}
Let $X$ be a quasismooth hypersurface of degree $8n+12$ in $\mathbb{P}(4, 2n+3, 2n+3, 4n+4)$ for a natural number $n$.
 Then $\lct(X)=1$.
\end{lemma}

\begin{proof}
The only singularities of $X$ are a singular point $O_t$ of index
$4n+4$, two singular points $P_1$, $P_2$ of index $4$ on $L_{yz}$,
 and four singular points $Q_1$, $Q_2$, $Q_3$, $Q_4$ of index $2n+3$ on $L_{xt}$.

The curve $C_x$ is reduced and splits into four irreducible
components $L_1, \ldots, L_4$. Each $L_i$ passes through $Q_i$.
They intersect each other at $O_t$. One can easily see that
$\lct(X, \frac{1}{2}C_x)=1$, and hence $\lct(X)\leqslant 1$.

Suppose that $\lct(X)<1$. Then there is an effective $\Q$-divisor
$D\qlineq -K_X$ such that the log pair $(X, D)$ is not log
canonical at some point $P\in X$.

By Lemma~\ref{lemma:convexity} we may
assume that  $L_i\not\subset\Supp(D)$ for some $i$. Since
$$
(4n+4)L_i\cdot D=\frac{4n+4}{(2n+2)(2n+3)}<1
$$
for all $n\geqslant 1$, the point $P$ cannot belong to the curve $L_i$.

For $j\ne i$, put $D=\mu L_j+\Omega$, where $\Omega$
is an effective $\Q$-divisor such that $L_j\not\subset\Supp(\Omega)$.
Since
$$
\frac{\mu}{4n+4}=\mu L_i\cdot L_j\leqslant D\cdot L_i=\frac{1}{2(n+1)(2n+3)},%
$$ we have
$$\mu\leqslant\frac{2}{2n+3}.$$
Note that
$$L_j^2=C_x\cdot L_j-3L_i\cdot L_j=\frac{2}{2(n+1)(2n+3)}-\frac{3}{4(n+1)}=-\frac{6n+5}{4(n+1)(2n+3)}.$$
By Lemma~\ref{lemma:adjunction} the inequality
$$(2n+3)\Omega\cdot L_j=(2n+3)(D-\mu L_j)\cdot L_j=\frac{2+(6n+5)\mu}{4(n+1)}\leqslant
\frac{4}{2n+3}<1$$ for all $n\geqslant 1$ shows that $P$ cannot be contained in $L_j$. Consequently, the point $P$ is located in the outside of $C_x$.

 By a suitable coordinate change we may assume
that $P_1=O_x$.  Then, the curve $C_t$ is reduced and splits into four
irreducible components $L_1', \ldots, L_4'$. Each $L_i'$ passes
through the point $Q_i$.  They intersect each other at $O_x$. We can easily see
that the log pair $(X, \frac{2}{4n+4}C_t)$ is log
canonical.  By
Lemma~\ref{lemma:convexity} we may assume that
$L_i'\not\subset\Supp(D)$. Since
$$
\mult_{O_x}(D)\leqslant 4L_i'\cdot
D=\frac{2}{2n+3}<1%
$$
for all $n\geqslant 1$, the point $P$ cannot be $O_x$. The point
$P_2$ can be  excluded in a similar way.

Therefore, $P$ is a smooth point of $X\setminus C_x$. Applying
Lemma~\ref{lemma:Carolina},
we see that
$$
1<\mult_P(D)\leqslant\frac{2(8n+12)^2}
{4(2n+3)^2(4n+4)}\leqslant 1%
$$
for $n\geqslant 1$ since $H^0(\P, \mathcal{O}_\P(8n+12))$
contains $x^{2n+3}$, $y^{4}$ and $z^{4}$. The obtained
contradiction completes the proof.
\end{proof}

\begin{lemma}
\label{lemma:I-2-infinite-series-2}
Let $X$ be a quasismooth hypersurface of degree $18n+6$ in $\mathbb{P}(3, 3n+1, 6n+1, 9n+3)$ for a natural number $n\geqslant 1$.
 Then $\lct(X)=1$.

\end{lemma}

\begin{proof}
The only singularities of $X$ are a singular point $O_z$ of index
$6n+1$, two singular points $P_1$, $P_2$ of index $3$ on $L_{yz}$, and two singular points $Q_1$, $Q_2$
of index $3n+1$ on $L_{xz}$.

The curve $C_x$ is reduced and splits into two components $L_1$
and $L_2$ that intersect at $O_z$. It is easy to see that $\lct(X,
\frac{2}{3}C_x)=1$. Therefore, $\lct(X)\leqslant 1$.
Note that
$$L_1\cdot L_2=\frac{3}{6n+1} \text{ and }
L_1^2=L_2^2=-\frac{9n-3}{(3n+1)(6n+1)}.$$

Suppose that $\lct(X)<1$. Then there is an effective $\Q$-divisor
$D\qlineq -K_X$ such that the log pair $(X, D)$ is not log
canonical at some point $P\in X$.

We may assume that $L_2$ is not contained the support of $D$.
The inequality
\[D\cdot L_2 =\frac{2}{(3n+1)(6n+1)}\leqslant \frac{1}{6n+1}\]
shows that the point $P$ cannot belong to the curve $L_2$.
Put $D=\mu L_1+\Omega$, where $\Omega$ is an effective $\Q$-divisor whose support does not contain the curve $L_1$. Since $$\frac{3\mu}{6n+1}=\mu L_1\cdot L_2\leqslant D\cdot L_2=\frac{2}{(3n+1)(6n+1)},$$
we have
$$0\leqslant \mu\leqslant \frac{2}{3(3n+1)}.$$

Lemma~\ref{lemma:handy-adjunction} and the inequality
$$\Omega\cdot L_1=(D-\mu L_1)\cdot L_1=\frac{2+\mu(9n-3)}{(3n+1)(6n+1)}<
\frac{4}{(3n+1)(6n+1)}$$ show that the point $P$ is located in the outside of $L_1$. Therefore, $P\not \in C_x$.

The curve $C_y$ is irreducible. It is easy to see that the log
pair $(X, \frac{2}{3n+1}C_y)$ is log canonical. Therefore, we may assume that the support of $D$ does not contain the curve $C_y$. Note that $P_1, P_2\in C_y$.
The inequality
\[3D\cdot C_y=\frac{4}{6n+1}\leqslant 1\]
shows that neither $P_1$ not $P_2$ can be the point $P$.

Hence $P$ is a smooth point of $X\setminus C_x$. Applying
Lemma~\ref{lemma:Carolina}, we get an absurd inequality
$$
1<\mult_P(D)\leqslant\frac{2 (18n+6)(18n+3)}
{3 (3n+1) (6n+1)(9n+3)}\leqslant 1%
$$ since $H^0(\P, \mathcal{O}_\P(18n+3))$
contains $x^{6n+1}$, $x^{3n}y^{3}$ and $z^{3}$. The obtained
contradiction completes the proof.
\end{proof}

\begin{lemma}
\label{lemma:I-2-infinite-series-3}
Let $X$ be a quasismooth hypersurface of degree $18n+3$ in $\mathbb{P}(3, 3n+1, 6n+1, 9n)$ for a natural number $n\geqslant 1$.
 Then $\lct(X)=1$.
\end{lemma}

\begin{proof}
The singularities of $X$ are a singular point $O_y$ of index
$3n+1$, a singular point $O_{t}$ of index $9n$, and two singular
points $Q_1$, $Q_2$ of index $3$ on $L_{yz}$.

The curve $C_x$ is reduced and irreducible and has the only
singularity  at $O_t$. It is easy to see that
$\lct(X, \frac{2}{3}C_x)=1$, and hence $\lct(X)\leqslant 1$.
The curve $C_y$ is quasismooth. Therefore, the log
pair $(X, \frac{2}{3n+1}C_y)$ is log canonical.

Suppose that $\lct(X)<1$. Then there is an effective $\Q$-divisor
$D\qlineq -K_X$ such that the log pair $(X, D)$ is not log
canonical at some point $P\in X$. By Lemma~\ref{lemma:convexity}
we may assume that neither $C_x$ nor $C_y$ is contained in
$\Supp(D)$.

The inequalities
$$
C_x\cdot
D<(3n+1)C_x\cdot
D=\frac{2}{3n}<1,$$
$$\mult_{O_t}(D)=\frac{\mult_{O_t}(C_{x})
\mult_{O_t}(D)}{3}\leqslant \frac{9nC_x\cdot D}{3}=\frac{2}{3n+1}<1$$
show that the point $P$ must be located in the outside of $C_x$.

Also, the inequality
$$
3C_y\cdot D=\frac{2}{3n}<1
$$
implies that neither $Q_1$ not $Q_2$ can be the point $P$. Hence
$P$ is a smooth point of $X\setminus C_x$. We see that $H^0(\P,
\mathcal{O}_\P(9n+3))$ contains $x^{3n+1}$, $y^{3}$ and $xt$.
Also, the projection of $X$ from the point $O_z$ has only finite
fibers. Therefore, Lemma~\ref{lemma:Carolina} implies a
contradictory inequality
$$
1<\mult_P(D)\leqslant\frac{2 (18n+3)(9n+3)}
{3 (3n+1) (6n+1)\cdot 9n}=\frac{2}{3n}<1.%
$$
The obtained contradiction
completes the proof.
\end{proof}

\begin{lemma}
\label{lemma:I-2-infinite-series-1-n-1}
Let $X$ be a quasismooth hypersurface of degree $12$ in $\mathbb{P}(3, 3, 4, 4)$.
 Then $\lct(X)=1$.
\end{lemma}

\begin{proof}
The surface $X$ can be defined by the quasihomogeneous equation
$$
\prod_{i=1}^{4}(\alpha_{i}x+\beta_{i}y)=\prod_{j=1}^{3}(\gamma_{j}z+\delta_{j}t),
$$
where
$[\alpha_{i}:\beta_{i}]$ define four distinct points  and $[\gamma_{j}:\delta_{j}]$ define three distinct points in $\mathbb{P}^1$.

Let $P_{i}$ be the point in $X$  given by
$z=t=\alpha_{i}x+\beta_{i}y=0$.
These are singular point of $X$ of type $\frac{1}{3}(1,1)$.
Let $Q_{j}$ be the point in $X$ that is given by
$x=y=\gamma_{j}z+\delta_{j}t=0$. Then
each of them is a singular point of $X$ of type $\frac{1}{4}(1,1)$.

Let $L_{ij}$ be the curve in $X$ defined by
$\alpha_{i}x+\beta_{i}y=\gamma_{j}z+\delta_{j}t=0$, where
$i=1,\ldots,4$ and $j=1,\ldots,3$.

The divisor $C_i$ cut out by the equation
$\alpha_{i}x+\beta_{i}y=0$ consists of three smooth curves
$L_{i1}$, $L_{i2}$, $L_{i3}$. These divisors $C_i$, $i=1,2,3,4$,
are the only reducible members in the linear system
$|\mathcal{O}_X(3)|$. Meanwhile, the divisor $B_j$ cut out by
$\gamma_{j}z+\delta_{j}t=0$ consists of four smooth curves
$L_{1j}$, $L_{2j}$, $L_{3j}$, $L_{4j}$. Note that $L_{i1}\cap
L_{i2}\cap L_{i3}=\{P_{i}\}$ and $L_{1j}\cap L_{2j}\cap L_{3j}\cap
L_{4j}=\{Q_{j}\}$.  We have $L_{ij}\cdot L_{ik}=\frac{1}{3}$ and
$L_{ji}\cdot L_{ki}=\frac{1}{4}$ if $k\ne j$. But
$L_{ij}^{2}=-\frac{5}{12}$.

Since
$
\lct\left(X, \frac{2}{3}C_i\right)=\lct\left(X, \frac{2}{4}B_j\right)=1,%
$
we have $\lct(X)\leqslant 1$.

Suppose that $\lct(X)<1$. Then there is an effective $\Q$-divisor
$D\qlineq -K_X$ such that the pair $(X,D)$ is not log canonical at
some point $P$. For every $i=1,\ldots,4$, we may assume that the
support of the divisor $D$ does not contain at least one curve
among $L_{i1},L_{i2},L_{i3}$.
Suppose $L_{ik}\not\subset \Supp(D)$. Then the inequality
\[\mult_{P_i}(D) \leqslant 3D\cdot L_{ik}=\frac{1}{2}\]
implies that none of the points $P_i$ can be the point $P$.
For every $j=1, 2, 3$, we may also
assume that the support of the divisor $D$ does not contain at
least one curve among $L_{1j},L_{2j},L_{3j},L_{4j}$. Suppose $L_{lj}\not\subset \Supp(D)$.  Then the inequality
\[\mult_{Q_j}(D) \leqslant 4D\cdot L_{lj}=\frac{2}{3}\]
implies that none of the points $Q_j$ can be the point $P$.
Therefore, the point must be a smooth point of $X$.

Write $D=\mu L_{ij}+\Omega$, where
$\Omega$ is an effective $\Q$-divisor whose support does not contain $L_{ij}$.
If $\mu>0$, then we have
$\mu L_{ij}\cdot L_{ik}\leqslant D\cdot L_{ik}$, and hence
$\mu\leqslant
\frac{1}{2}$.
Since $$
\Omega\cdot L_{ij}=\frac{2+5\mu}{12}<1,
$$
Lemma~\ref{lemma:adjunction} implies the point $P$ cannot be on the curve $L_{ij}$.
Consequently,
$$
P\not\in\bigcup_{i=1}^{4}\bigcup_{j=1}^{3}L_{ij}.
$$

There is a unique curve $C\subset X$
cut out by $\lambda x+\mu y=0$, where
$[\lambda:\mu]\in\mathbb{P}^{1}$, passing through the point $P$.  Then  the curve $C$ is irreducible and
quasismooth. Thus, we may assume that $C$ is not contained in the
support of $D$. Then
$$
1<\mult_{P}(D)\leqslant D\cdot C=\frac{1}{2}.
$$
This is a contradiction.
\end{proof}

\begin{lemma}
\label{lemma:I-2-infinite-series-4} Let $X$ be a quasismooth
hypersurface of degree $9n+3$ in $\mathbb{P}(3, 3n, 3n+1, 3n+1)$
for $n\geqslant 2$.
 Then $\lct(X)=1$.
\end{lemma}

\begin{proof}We may assume that the surface $X$ is defined by the
equation
\[xy(y-ax^n)(y-bx^n)+zt(z-ct)=0,\]
where $a$, $b$, $c$ are non-zero constants and $b\ne c$.
 The point $O_y$ is a singular point of
of index $3n$ on $X$. The three points $O_x$, $P_a=[1:a:0:0]$,
$P_b=[1:b:0:0]$ are singular points of index $3$ on $X$. Also, $X$
has three singular points $O_{z}$, $O_t$, $P_c=[0:0:c:1]$ of index
$3n+1$ on $L_{xy}$.

The curve $C_x$ consists of three irreducible components $L_{xz}$,
$L_{xt}$ and $L_c=\{x=z-ct=0\}$. These three components intersect
each other at $O_y$. It is easy to check $\lct(X,
\frac{2}{3}C_x)=1$. Thus, $\lct(X)\leqslant 1$.


Suppose that $\lct(X)<1$. Then there is an effective $\Q$-divisor
$D\qlineq -K_X$ such that the log pair $(X, D)$ is not log
canonical at some point $P\in X$.

By Lemma~\ref{lemma:convexity} we may assume that at least one of
the components of $C_x$ is not contained in $\Supp(D)$. Then, the
inequality
$$
3nL_{xz}\cdot D=3nL_{xt}\cdot D=3nL_c\cdot D=\frac{2}{3n+1}<1%
$$
implies that the point $P$ cannot be the point $O_y$.

Put $D=\mu L_{xz}+\Omega$, where $\Omega$ is an effective
$\Q$-divisor whose support does not contain the curve $L_{xz}$. We
claim that
$$\mu\leqslant\frac{2}{3n+1}.$$
Indeed, if the inequality fails, one of the curves $L_{xt}$ and
$L_c$ is not contained in $\Supp(D)$. Then either
$$
\frac{\mu}{3n}=\mu L_{xz}\cdot L_{xt}\leqslant D\cdot L_{xt}=\frac{2}{3n(3n+1)},
\mbox{ or }\ \
\frac{\mu}{3n}=\mu L_{xz}\cdot L_{c}\leqslant D\cdot L_{c}=\frac{2}{3n(3n+1)}
$$
holds. This is a contradiction. Note that
$$L_{xz}^2=-\frac{6n-1}{3n(3n+1)}.$$
The inequality
$$\Omega\cdot L_{xz}=\frac{2+(6n-1)\mu}{3n(3n+1)}<
\frac{1}{3n+1}$$ holds for all $n\geqslant 2$. Therefore,
Lemma~\ref{lemma:handy-adjunction} implies the point $P$ cannot
belong to $L_{xz}$. By the same way, we can show that $P\not\in
L_{xt}\cup L_c$.

Let $C$ be the curve on $X$ cut out by the equation $z-\alpha
t=0$, where $\alpha$ is non-zero constant different from $c$. Then
the curve $C$ is quasismooth and hence $\lct(X,
\frac{2}{3n+1}C)\geqslant 1$. Therefore, we may assume that the
support of $D$ does not contain the curve $C$. Then
\[\mult_{O_x}(D), \ \mult_{P_a}(D),\  \mult_{P_b}(D)\leqslant
3D\cdot C=\frac{2}{n}\leqslant 1 \] for $n\geqslant 2$. Therefore,
$P$ cannot be a singular point of $X$. Hence $P$ is a smooth point
of $X\setminus C_x$. Applying Lemma~\ref{lemma:Carolina}, we get
an absurd inequality
$$
1<\mult_P(D)\leqslant\frac{2 (9n+3)^2}
{3\cdot 3n (3n+1) (3n+1)}\leqslant 1%
$$
for $n\geqslant 2$ since $H^0(\P, \mathcal{O}_\P(9n+3))$ contains
$x^{3n+1}$, $xy^{3}$ and $z^{3}$. The obtained contradiction
completes the proof.
\end{proof}

\begin{lemma}
\label{lemma:I-2-infinite-series-5}
Let $X$ be a quasismooth
hypersurface of degree $9n+6$ in $\mathbb{P}(3, 3n+1, 3n+2, 3n+2)$
for $n\geqslant 1$.
 Then $\lct(X)=1$.
\end{lemma}

\begin{proof}
The only singularities of $X$ are a singular point $O_y$ of index
$3n+1$, and three singular points $P_i$, $i=1, 2, 3$, of
index $3n+2$ on $L_{xy}$.

The divisor $C_x$ consists of three distinct irreducible and reduced curves
$L_{1}$, $L_{2}$, $L_{3}$, where each $L_{i}$ contains the singular point $P_i$. Then $L_{1}\cap L_{2}\cap
L_{3}=\{O_y\}$. It is obvious that $\lct(X, \frac{2}{3}C_x)=1$, and hence $\lct(X)\leqslant 1$.

Suppose that $\lct(X)<1$. Then there is an effective $\Q$-divisor
$D\qlineq -K_X$ such that the log pair $(X, D)$ is not log
canonical at some point $P\in X$. By Lemma~\ref{lemma:convexity}
we may assume that  $L_1$ is not contained in $\Supp(D)$.

Since
$$
L_1\cdot D<(3n+1)L_1\cdot D=\frac{2}{3n+2}<1$$
for all $n\geqslant 1$,  we see
that $P\not\in L_{1}$. In particular, we see that $P\ne O_{y}$.

Put $D=\mu L_2+\Omega$, where $\Omega$
is an effective $\Q$-divisor such that $L_2\not\subset\Supp(\Omega)$.
Then the inequality
$$
\frac{\mu}{3n+1}=\mu L_1\cdot L_2\leqslant D\cdot L_1=\frac{2}{(3n+1)(3n+2)},%
$$
implies that $\mu\leqslant \frac{2}{3n+2}$. The intersection number
$$
L_1^2=-\frac{6n+1}{(3n+1)(3n+2)}
$$
shows
\[(3n+2)\Omega\cdot L_2=(3n+2)(D-\mu L_2)\cdot L_2=\frac{2+\mu(6n+1)}{(3n+1)}\leqslant \frac{6}{(3n+2)}\]
for all $n\geqslant 1$. Therefore, Lemma~\ref{lemma:handy-adjunction} excludes all the smooth point on $L_2$  in the case where $n\geqslant 1$
and the singular point $P_2$ in the case where $n\geqslant 2$.
For the case $n=1$, let $C_{2}$ be the unique curve in the pencil
$|\mathcal{O}_X(5)|$ that passes through the point
$P_2$. Then the divisor $C_2$ consists of two distinct irreducible and reduced curve $L_2$ and $R_2$. The curve $R_2$ is singular at the point $P_2$.  Moreover, the log pair $(X, \frac{2}{5}C_2)$
is log canonical at the point $P_2$. By
Lemma~\ref{lemma:convexity}, we may assume that
$R_{2}\not\subset\Supp(D)$. Then the inequality
$$
2\mult_{P_2}(D) \leqslant \mult_{P_2}(D)\mult_{P_2}(R_{2})\leqslant 5D\cdot
R_{2}=2%
$$
excludes the point $P_2$ in the case where $n=1$.
By the same method, we can show $P\not\in L_{3}$.

Hence the point $P$ must be a smooth point in $X\setminus C_x$. For the case $n\geqslant 2$, we can use
Lemma~\ref{lemma:Carolina} to get a contradiction
$$
1<\mult_P(D)\leqslant\frac{2(9n+6)^2}
{3(3n+1) (3n+2) (3n+2)}=\frac{6}{3n+1}<1,%
$$
since $H^0(\P, \mathcal{O}_\P(9n+6))$ contains $x^{3n+2}$,
$y^{3}x$ and $z^{3}$. For the case $n=1$, let $R_{P}$ be the
unique curve in the pencil $|\mathcal{O}_X(5)|$ that passes
through the point $P$. The log pair $(X, \frac{2}{5}R_{P})$ is log
canonical at the point $P$. By Lemma~\ref{lemma:convexity}, we may
assume that $\Supp(D)$ does not contain at least one irreducible
component of $R_{P}$. Note that either $R_{P}$ is irreducible or
$P_k\in R_{P}$ for some $k=1,2,3$. If $R_P$ is irreducible, then
we can obtain a contradiction
$$
1<\mult_P(D)\leqslant D\cdot R_{P}=\frac{1}{2}.%
$$
Thus, $P_k\in R_{P}$.
Then $R_P$ consists of two distinct irreducible curves $L_{k}$ and $Z$. Since we already showed that $P$ is located in the outside of $L_k$, the point $P$ must belong to the curve $Z$. We have
$$
L_{k}^2=-\frac{7}{20},\ \ L_{k}\cdot Z=\frac{3}{5},\ \  Z^2=\frac{2}{5}.%
$$
Put $D=m Z+\Delta$, where $\Delta$ is an effective $\Q$-divisor such
that $Z\not\subset\Supp(\Delta)$. If $m>0$, then
$$
\frac{3m}{5}=m Z\cdot L_k\leqslant D\cdot L_k=\frac{1}{10},%
$$
and hence $\mu\leqslant \frac{1}{6}$. Then Lemma~\ref{lemma:handy-adjunction} gives us a contradiction
$$
1<\Delta\cdot Z= \frac{2-2m}{5}<1.
$$
\end{proof}

\begin{lemma}
\label{lemma:I-2-infinite-series-6}
Let $X$ be a quasismooth
hypersurface of degree $12n+6$ in $\mathbb{P}(4, 2n+1, 4n+2, 6n+1)$
for $n\geqslant 1$.
 Then $\lct(X)=1$.
\end{lemma}

\begin{proof}
We may assume that the surface $X$ is defined by the equation
\[xt^2+x^{2n+1}z+ax^{2n+1}y^2-(z-a_1y^2)(z-a_2y^2)(z-a_3y^2)=0,\]
where $a_1$, $a_2$, $a_3$ are distinct constants and $a$ is a constant.

The only singularities of $X$ are a singular point $O_x$ of index
$4$, a singular point $O_t$ of index $6n+1$, a singular point
$Q=[1:0:1:0]$ of index $2$, and three singular
points $P_1=[0:1:a_1:0]$, $P_2=[0:1:a_2:0]$, $P_3=[0:1:a_3:0]$ of index $2n+1$.

The divisor $C_x$ consists of three distinct irreducible curves $L_i=\{x=z-a_iy^2=0\}$, $i=1,2,3$.
Note that each $L_i$ passes through the point $P_i$ and $L_1\cap L_2\cap L_3=\{O_t\}$.
We can easily check
$\lct(X, \frac{1}{2}C_x)=1$, and hence $\lct(X)\leqslant 1$.

Suppose that $\lct(X)<1$. Then there is an effective $\Q$-divisor
$D\qlineq -K_X$ such that the log pair $(X, D)$ is not log
canonical at some point $P\in X$.

By Lemma~\ref{lemma:convexity} we may
assume that  $L_1$ is not contained
in $\Supp(D)$. Since
$$
(6n+1)L_1\cdot
D=\frac{2}{2n+1}<1%
$$
the point $P$ is located in the outside of $L_1$.

Put $D=\mu L_2+\Omega$, where $\Omega$
is an effective $\Q$-divisor such that $L_2\not\subset\Supp(\Omega)$.
Then
$$
\frac{2\mu}{6n+1}=\mu L_1\cdot L_2\leqslant D\cdot L_2=\frac{2}{(2n+1)(6n+1)},%
$$
and hence
$\mu\leqslant\frac{1}{2n+1}$.
Since
$$L_2^2=-\frac{8n}{(2n+1)(6n+1)}$$
we have
$$(2n+1)\Omega\cdot L_2=(2n+1)(D-\mu L_2)\cdot L_2=\frac{2+8n\mu}{6n+1}\leqslant
\frac{2}{2n+1}<1$$ for all $n\geqslant 1$. Then Lemma~\ref{lemma:handy-adjunction} excludes all the points on $L_2$. Furthermore, the same method works for $L_3$.

The curve $C_y$ is quasismooth. Thus the log
pair $(X, \frac{2}{2n+1}C_y)$ is log canonical.
By Lemma~\ref{lemma:convexity} we may
assume that $C_y$ is not contained in $\Supp(D)$. Then the inequality
$$
4C_y\cdot D=\frac{6}{6n+1}<1%
$$
implies  that the point $P$ is neither $O_x$ nor $Q$.
Hence $P$ is a smooth point of $X\setminus C_x$.  However,
Lemma~\ref{lemma:Carolina} gives us
$$
\mult_P(D)\leqslant\frac{144 n(2n+1)}
{8(2n+1)^2 (6n+1)}<1%
$$
since  $H^0(\P, \mathcal{O}_\P(12n))$
contains $x^{3n}$, $y^{4}x^{n-1}$ and $z^{2}x^{n-1}$. This is a contradiction.
\end{proof}

\section{Infinite series with $I=4$}
\label{subsection:series-I-4}

\begin{lemma}
\label{lemma:I-4-infinite-series-1}
Let $X$ be a quasismooth
hypersurface of degree $18n+15$ in $\mathbb{P}(6, 6n+3, 6n+5, 6n+5)$
for $n\geqslant 1$.
 Then $\lct(X)=1$.
\end{lemma}

\begin{proof}

We may assume that the surface $X$ is defined by the equation
\[(z-a_1t)(z-a_2t)(z-a_3t)+xy(y^2-x^{2n+1})=0,\]
where $a_1$, $a_2$, $a_3$ are distinct constants.
The only singularities of $X$ are a singular point $O_x$ of index
$6$, a singular point $O_{y}$ of index $6n+3$, a singular point
$Q=[1:1:0:0]$ of index $3$, and three singular
points $P_i=[0:0:a_i:1]$, $i=1, 2, 3$, of index $6n+5$.

The divisor $C_x$ consists of three distinct irreducible curves $L_i=\{x=z-a_it=0\}$, $i=1,2,3$.
Note that each $L_i$ passes through the point $P_i$ and $L_1\cap L_2\cap L_3=\{O_y\}$.
We can easily check
$\lct(X, \frac{2}{3}C_x)=1$, and hence $\lct(X)\leqslant 1$.

The divisor $C_y$ consists of three distinct irreducible curves $L_i'=\{y=z-a_it=0\}$, $i=1,2,3$.
Each $L_i'$ passes through the point $P_i$ and $L_1'\cap L_2'\cap L_3'=\{O_x\}$.
The
log pair $(X, \frac{4}{6n+3}C_y)$ is log canonical.

Suppose that $\lct(X)<1$. Then there is an effective $\Q$-divisor
$D\qlineq -K_X$ such that the log pair $(X, D)$ is not log
canonical at some point $P\in X$.

For a general member $C$ in $|\mathcal{O}_X(6n+5)|$, we have
\[\mult_Q(D) \leqslant 3D\cdot C=\frac{6}{6n+3}<1.\]
Therefore the point $P$ cannot be the point $Q$.

By Lemma~\ref{lemma:convexity}
we may assume that $L_1$ and $L_{1}^{\prime}$ are not contained in
$\Supp(D)$. The two inequalities
$(6n+5)D\cdot L_1= \frac{4}{6n+3}<1$ and $6D\cdot L_1'= \frac{4}{6n+5}<1$ show that
the point $P$ is located in the outside of $L_1\cup L_1'$.

Write $D=\mu L_2+\Omega$, where $\Omega$
is an effective $\Q$-divisor such that $L_2\not\subset\Supp(\Omega)$.
Then
$$
\frac{\mu}{6n+3}=\mu L_1\cdot L_2\leqslant D\cdot L_1=\frac{4}{(6n+3)(6n+5)},%
$$
and hence $\mu\leqslant \frac{4}{6n+5}$. Note that
$$
L_2^2=-\frac{12n+4}{(6n+3)(6n+5)}.
$$
Therefore, we have
$$
(6n+5)\Omega\cdot L_2=(6n+5)(D-\mu L_2)\cdot L_2=\frac{4+(12n+4)\mu}{6n+3}\leqslant \frac{12}{6n+5}.
$$
Therefore, Lemma~\ref{lemma:handy-adjunction} excludes all the smooth point on $L_2$ in the case where $n\geqslant 1$
and the singular point $P_2$ in the case where $n\geqslant 2$.
For the case $n=1$, let $C_{2}$ be the unique curve in the pencil
$|\mathcal{O}_X(11)|$ that passes through the point
$P_2$.
Then the divisor $C_2$ consists of three distinct irreducible and reduced curve $L_2$, $L_2'$ and $R_2$. The log pair $(X, \frac{4}{11}C_2)$
is log canonical at the point $P_2$.
If $\mu=0$, then the inequality above immediately excludes the point $P_2$ for the case $n=1$. Therefore
we may assume that either
$L_{2}'\not\subset\Supp(D)$ or $R_{2}\not\subset\Supp(D)$. In the former case, the intersection number
\[D\cdot L_2'=\frac{2}{33}\]
shows that the point $P$ cannot be $P_2$. In the latter case, the intersection number
$$D\cdot
R_{2}=\frac{1}{11}%
$$
excludes the point $P_2$.
By the same method, we can show $P\not\in L_{3}$.

Hence the point $P$ must be a smooth point in $X\setminus C_x$. For the case $n\geqslant 2$, we can use
Lemma~\ref{lemma:Carolina} to get a contradiction
$$
1<\mult_P(D)\leqslant\frac{4 (18n+15)\cdot 6(6n+5)}
{6 (6n+3) (6n+5) (6n+5)}=\frac{4}{2n+1}<1,%
$$
since $H^0(\P, \mathcal{O}_\P(6(6n+5))$
contains $x^{6n+5}$, $y^{6}x^2$ and $z^{6}$.
For the case $n=1$, let $R_{P}$ be the unique curve in the pencil
$|\mathcal{O}_X(11)|$ that passes through the point $P$.
The log pair $(X, \frac{4}{11}R_{P})$ is log canonical at the point
$P$. By Lemma~\ref{lemma:convexity}, we may assume that
$\Supp(D)$ does not contain at least one irreducible component of
$R_{P}$. Note that either $R_{P}$ is irreducible or $P_k\in
R_{P}$ for some $k=1,2,3$.
However, if $R_P$ is irreducible, then we can obtain a contradiction
$$
1<\mult_P(D)\leqslant D\cdot R_{P}=\frac{2}{9}.%
$$
Thus, $P_k\in R_{P}$.
Then $R_P$ consists of three distinct irreducible curves $L_{k}$, $L_k'$ and $Z$.
We have
$$
 D\cdot L_k'=\frac{2}{33}, \ \ D\cdot Z=\frac{4}{33},\ \
 L_{k}'^2=-\frac{13}{66},\ \ Z^2=-\frac{4}{33}. %
$$
Put $D=m_1 Z+m_2L_k'+ \Delta$, where $\Delta$ is an effective $\Q$-divisor  whose support contains neither $Z$ nor $L_k'$.  Since the pair $(X, D)$ is log canonical at the point $P_k$, we have $m_1, m_2 \leqslant 1$.
 Since we already showed that $P$ is located in the outside of $L_k$, the point $P$ must belong to either $L_k'$ or $Z$.
However, Lemma~\ref{lemma:handy-adjunction} shows that the pair
$(X, D)$ is log canonical at the point $P$ since $$
(D-m_{1}Z)\cdot Z= \frac{4+4m_1}{33}<1, \ \ (D-m_2L_k')\cdot L_k'=
\frac{4+13m_2}{66}<1.
$$
This is a contradiction.
\end{proof}

\begin{lemma}
\label{lemma:I-4-infinite-series-2}
Let $X$ be a quasismooth
hypersurface of degree $36n+24$ in $\mathbb{P}(6, 6n+5, 12n+8, 18n+9)$
for $n\geqslant 1$.
 Then $\lct(X)=1$.
\end{lemma}
\begin{proof}
We may assume that the surface $X$ is defined by the equation
\[z^3+y^3t+xt^2-x^{6n+4}+ax^{2n+1}y^2z=0,\]
where $a$ is a constant.
The only singularities of $X$ are a singular point $O_y$ of index
$6n+5$, a singular point $O_{t}$ of index $18n+9$, a singular
point $Q=[1:0:0:1]$ of index $3$, and a singular point $Q'=[1:0:1:0]$ of index $2$.

The curve $C_x$ is reduced and irreducible with $\mult_{O_t}(C_x)=3$. Clearly,
$\lct(X, \frac{2}{3}C_x)=1$, and hence $\lct(X)\leqslant 1$.
The curve $C_y$ is quasismooth, and hence  the log
pair $(X, \frac{4}{6n+5}C_y)$ is log canonical.

Suppose that $\lct(X)<1$. Then there is an effective $\Q$-divisor
$D\qlineq -K_X$ such that the log pair $(X, D)$ is not log
canonical at some point $P\in X$.

Since $H^0(\P, \mathcal{O}_\P(36n+30))$ contains $x^{6n+5}$,
$y^{6}$ and $z^{3}x$, Lemma~\ref{lemma:Carolina} implies
$$
\mult_P(D)\leqslant\frac{4 (36n+24) (36n+30)}
{6 (6n+5) (12n+8) (18n+9)}<1.%
$$
Therefore, the point $P$ cannot be a smooth point in the outside of $C_x$.

By Lemma~\ref{lemma:convexity}
we may assume that neither $C_x$ nor $C_y$ is contained in
$\Supp(D)$.
Then the inequality \[3D\cdot C_y=\frac{2}{6n+3}\leqslant 1\]
implies that the point $P$ is neither $Q$ nor $Q'$. One the other hand, the inequality
\[(6n+5)D\cdot C_x=\frac{4}{6n+3}<1\]
shows that the point $P$  can be neither a smooth point on $C_x$ nor the point $O_y$. Therefore, it must be $O_t$.
However, this is a contradiction since
\[\mult_{O_t}(D)=\frac{\mult_{O_t}(D)\mult_{O_t}(C_x)}{3}\leqslant \frac{18n+9}{3}D\cdot C_x=\frac{4}{6n+5}<1.\]
 The obtained contradiction completes the
proof.
\end{proof}

\begin{lemma}
\label{lemma:I-4-infinite-series-3}
Let $X$ be a quasismooth
hypersurface of degree $36n+30$ in $\mathbb{P}(6, 6n+5, 12n+8, 18n+15)$
for $n\geqslant 1$.
 Then $\lct(X)=1$.
\end{lemma}

\begin{proof}
We may assume that the surface $X$ is defined by the equation
\[(t-a_1y^3)(t-a_2y^3)+xz^3-x^{6n+5}+ax^{2n+1}y^2z=0,\]
where $a_1\ne a_2$ and $a$ are constants.
The only singularities of $X$ are a singular point $O_z$ of index
$12n+8$, a singular point $Q=[1:0:1:0]$ of index $2$, a singular point $Q'=[1:0:0:1]$ of index $3$, and two singular points $P_1=[0:1:0:a_1]$, $P_2=[0:1:0:a_2]$ of index
$6n+5$.

The curve $C_x$ consists of two distinct irreducible curves $L_i=\{x=t-a_iy^3=0\}$, $i=1,2$. Each $L_i$ passes through the point $P_i$. These two curves meet each other at the point $O_z$.
It is easy to see $\lct(X, \frac{2}{3}C_x)=1$.

Suppose that $\lct(X)<1$. Then there is an effective $\Q$-divisor
$D\qlineq -K_X$ such that the log pair $(X, D)$ is not log
canonical at some point $P\in X$.

By Lemma~\ref{lemma:convexity} we may
assume that  $L_1$ is not
contained in $\Supp(D)$.
Then the inequality
\[ (12n+8)D\cdot L_1=\frac{4}{6n+5}<1\]
shows that the point $P$ must be located in the outside of $L_1$.

Write $D=\mu L_2+\Omega$, where $\Omega$
is an effective $\Q$-divisor such that
$L_2\not\subset\nlb\Supp(\Omega)$. Then, the inequality $$
\frac{3\mu}{12n+8}=\mu L_2\cdot L_1\leqslant D\cdot L_1=\frac{1}{(3n+2)(6n+5)},%
$$ implies
$$\mu\leqslant\frac{4}{3(6n+5)}.$$
 Note that
$$L_2^2=-\frac{18n+9}{(12n+8)(6n+5)}.$$
Since
$$(6n+5)\Omega\cdot L_2=\frac{4+(18n+9)\mu}{12n+8}<
\frac{4}{6n+5},$$ Lemma~\ref{lemma:handy-adjunction} excludes all the points of $L_2\setminus\{O_z\}$.
Consequently, the point $P$ is in the outside of $C_x$.

Meanwhile, the curve $C_y$ is quasismooth, and hence the log pair
$(X, \frac{4}{6n+5}C_y)$ is log canonical.
Lemma~\ref{lemma:convexity} enables us to assume that $C_y$ is not
contained in $\Supp(D)$. Then the inequality
$$
3C_y\cdot
D=\frac{1}{3n+2}\leqslant 1,%
$$
excludes the singular points $Q$ and $Q'$.

Hence $P$ is a smooth point of $X\setminus C_x$. Applying
Lemma~\ref{lemma:Carolina}, we see that
$$
1<\mult_P(D)\leqslant\frac{4 (36n+30) (3(12n+8)+6)}
{6 (6n+5) (12n+8) (18n+15)}<1,%
$$
because $H^0(\P, \mathcal{O}_\P(3(12n+8)+6))$ contains
$x^{6n+5}$, $y^{6}$ and $z^{3}x$. The obtained contradiction
completes the proof.
\end{proof}

\section{Infinite series with $I=6$}
\label{subsection:series-I-6}

\begin{lemma}
\label{lemma:I-6-infinite-series-1}Let $X$ be a quasismooth
hypersurface of degree $12n+23$ in $\mathbb{P}(8, 4n+5, 4n+7, 4n+9)$
for $n\geqslant 3$.
 Then $\lct(X)=1$.
\end{lemma}

\begin{proof} The surface $X$ can be given by the equation
$$
z^{2}t+yt^{2}+xy^{3}+x^{n+2}z=0.
$$
The surface $X$ is singular only at $O_x$, $O_y$, $O_z$ and
$O_t$.

The curve $C_x$ (resp. $C_y$, $C_z$, $C_t$) consists of  the irreducible curve
$L_{xt}$ (resp. $L_{yz}$, $L_{yz}$, $L_{xt}$) and a residual curve  $R_x=\{x=z^2+yt=0\}$ (resp. $R_y=\{y=x^{n+2}+zt=0\}$, $R_z=\{z=t^2+xy^2=0\}$, $R_t=\{t=y^3+x^{n+1}z=0\}$) . These two curves intersect each other at $O_y$ (resp. $O_t$, $O_x$, $O_z$).

We can easily  see that
$$\lct(X, \frac{3}{4}C_x)=1, \ \ \ \ \lct(X, \frac{6}{4n+5}C_y)=\frac{(n+3)(4n+5)}{12(n+2)}, $$
$$\lct(X, \frac{6}{4n+7}C_z)=\frac{4n+7}{9}, \ \ \ \  \lct(X, \frac{6}{4n+9}C_t)=\frac{(4n+9)(n+4)}{6(3n+6)}
.$$
Therefore, $\lct(X)\leqslant 1$.

Suppose that $\lct(X)<1$. Then there is an effective $\Q$-divisor
$D\qlineq -K_X$ such that the log pair $(X, D)$ is not log
canonical at some point $P\in X$.

We have the following intersection numbers:
$$L_{xt}\cdot D=\frac{6}{(4n+5)(4n+7)},\ \
L_{yz}\cdot D=\frac{6}{8(4n+9)}, \ \ R_x\cdot D=\frac{12}{(4n+5)(4n+9)},
$$
$$R_y\cdot D=\frac{6(n+2)}{(4n+7)(4n+9)},\ \
R_z\cdot D=\frac{12}{8(4n+5)},\ \ R_t\cdot D=\frac{18}{8(4n+7)},$$

$$L_{xt}\cdot R_x=\frac{2}{4n+5},\ \ L_{xt}\cdot R_t=\frac{3}{4n+7},
\ \ L_{yz}\cdot R_y=\frac{n+2}{4n+9}, \ \ L_{yz}\cdot R_z=\frac{1}{4}, \ \ $$

$$L_{xt}^2=-\frac{8n+6}{(4n+5)(4n+7)},\ \ L_{yz}^2=-\frac{4n+11}{8(4n+9)},\ \ R_x^2=-\frac{8n+2}{(4n+5)(4n+9)},$$

$$R_y^2=-\frac{2n+4}{(4n+7)(4n+9)},\ \
R_z^2=\frac{1}{2(4n+5)}, \ \ R_t^2=\frac{12n+3}{8(4n+7)}.$$

By Lemma~\ref{lemma:convexity} we may
assume that either $L_{xt}\not\subset\Supp(D)$ or $R_x\not\subset\Supp(D)$. Then at least one of the inequalities
\[\mult_{O_y}(D)\leqslant (4n+5)L_{xt}\cdot D=\frac{6}{4n+7},  \  \ \mult_{O_y}(D)\leqslant (4n+5)R_x\cdot D=\frac{12}{4n+9}\]
holds. Therefore, the point $P$ cannot be the point $O_y$.
We also may
assume that either $L_{yz}\not\subset\Supp(D)$ or $R_z\not\subset\Supp(D)$. Then at least one of the inequalities
\[\mult_{O_x}(D)\leqslant 8L_{yz}\cdot D=\frac{6}{4n+9},  \  \ \mult_{O_x}(D)\leqslant \frac{8}{2}R_z\cdot D=\frac{6}{4n+5}\]
holds. Note that the curve $R_z$ is singular at the point $O_x$. Therefore, the point $P$ cannot be the point $O_x$.
We also may
assume that either $L_{xt}\not\subset\Supp(D)$ or $R_t\not\subset\Supp(D)$. Then at least one of the inequalities
\[\mult_{O_z}(D)\leqslant (4n+7)L_{xt}\cdot D=\frac{6}{4n+5},  \  \ \mult_{O_z}(D)\leqslant
\frac{4n+7}{3}R_t\cdot D=\frac{3}{4}\]
holds. Note that the curve $R_t$ has multiplicity $3$  at the point $O_z$ if $n\geqslant 2$. Therefore, the point $P$ cannot be the point $O_z$.

Write $D=m_1 L_{xt}+ m_2 L_{yz}+ m_3 R_x+m_4 R_y+m_5R_z+m_6R_t+\Omega$,
where $\Omega$ is an effective $\Q$-divisor whose support contains
none of $L_{xt}$, $L_{yz}$, $R_x$, $R_y$, $R_z$, $R_t$.

If $m_1>0$, then $m_3=0$. Therefore, the inequality
\[\frac{2m_1}{4n+5}=m_1L_{xt}\cdot R_x\leqslant D\cdot R_{x}=\frac{12}{(4n+5)(4n+9)}\]
shows $0\leqslant m_1\leqslant \frac{6}{4n+9}$. By Lemma~\ref{lemma:handy-adjunction} the inequality
\[(D-m_1L_{xt})\cdot L_{xt}=\frac{6+m_1(8n+6)}{(4n+5)(4n+7)}\leqslant \frac{18}{(4n+7)(4n+9)}<1\]
implies that the point $P$ cannot be a smooth point on $L_{xt}$.

If $m_2>0$, then $R_z\not\subset\Supp(D)$. Therefore, the inequality
\[\frac{m_2}{4}=m_2L_{yz}\cdot R_z\leqslant D\cdot R_{z}=\frac{3}{2(4n+5)}\]
shows $0\leqslant m_2\leqslant \frac{6}{4n+5}$. By Lemma~\ref{lemma:handy-adjunction} the inequality
\[(D-m_2L_{yz})\cdot L_{yz}=\frac{6+m_2(4n+11)}{8(4n+9)}\leqslant \frac{6(n+2)}{(4n+5)(4n+9)}<1\]
implies that the point $P$ cannot be a smooth point on $L_{yz}$.

If $m_3>0$, then $m_1=0$, and hence
\[\frac{2m_3}{4n+5}=m_3L_{xt}\cdot R_x\leqslant D\cdot L_{xt}=\frac{6}{(4n+5)(4n+7)}.\]
Therefore, $0\leqslant m_3\leqslant \frac{3}{4n+7}$.
The inequality
\[(D-m_3R_x)\cdot R_x=\frac{12+m_3(8n+2)}{(4n+5)(4n+9)}\leqslant \frac{18}{(4n+7)(4n+9)}<1\]
implies that the point $P$ cannot be a smooth point on $R_x$. Moreover, this inequality shows that the point $P$ cannot be the point $O_t$ since $n\geqslant  3$.

If $m_4>0$, then we may assume that $m_2=0$. We then obtain
\[\frac{m_4(n+2)}{4n+9}=m_4R_{y}\cdot L_{yz}\leqslant D\cdot L_{yz}=\frac{3}{4(4n+9)}.\]
Therefore, $0\leqslant m_4\leqslant \frac{3}{4(n+2)}$.
The inequality
\[(D-m_4R_y)\cdot R_y=\frac{6(n+2)+2m_4(n+2)}{(4n+7)(4n+9)}\leqslant \frac{3}{2(4n+7)}<1\]
implies that the point $P$ cannot be a smooth point on $R_y$.

Since the pair $(X, D)$ is log canonical at the point $O_x$ and the curve $R_z$ contains the point $O_x$,
we have $m_5\leqslant 1$.
By Lemma~\ref{lemma:handy-adjunction}, the inequality
\[(D-m_5R_z)\cdot R_z\leqslant D\cdot R_z=\frac{3}{2(4n+5)}<1\]
shows that the point $P$ cannot be a smooth point on $R_z$.

The pair $(X, D)$ is log canonical at the point $O_x$ and the curve $R_t$ contains the point $O_x$. Thus $m_6\leqslant 1$.
By Lemma~\ref{lemma:handy-adjunction}, the inequality
\[(D-m_6R_t)\cdot R_t\leqslant D\cdot R_t=\frac{9}{4(4n+7)}<1\]
implies that the point $P$ cannot be a smooth point of $R_t$.

Consider the pencil $\mathcal{L}$ defined by  the equations
$\lambda xy^2+\mu t^{2}=0$, $[\lambda :\mu]\in \mathbb{P}^1$. Note
that the curve $L_{xt}$ is the only base component of the pencil
$\mathcal{L}$. There is a unique divisor $C_\alpha$ in
$\mathcal{L}$ passing through the point $P$. This divisor must be
defined an equation $xy^2+\alpha t^2=0$, where $\alpha$ is a
non-zero constant, since the point $P$ is located in the outside
of $C_x\cup C_y\cup C_z\cup C_t$. Note that the curve $C_y$ does not
contain any component of $C_{\alpha}$. Therefore, to see all the
irreducible components of $C_\alpha$, it is enough to see the
affine curve
$$
\left\{\aligned &x+\alpha t^{2}=0\\ &z^{2}t+t^{2}+x+x^{n+2}z
=0\\ \endaligned\right\}
\subset\mathbb{C}^{3}\cong\mathrm{Spec}\Big(\mathbb{C}\big[x,z,t\big]\Big).
$$
This is isomorphic to the plane affine curve defined by the
equation
$$
t\{z^2+(1-\alpha)t^2+(-\alpha)^{n+2}t^{2n+1}z\}
=0\subset\mathbb{C}^{2}\cong\mathrm{Spec}\Big(\mathbb{C}\big[z,t\big]\Big).
$$
Thus, if $\alpha\ne 1$, then the divisor $C_{\alpha}$ consists of
two reduced and  irreducible curves $L_{xt}$ and $Z_{\alpha}$. If
$\alpha=1$, then it consists of three reduced and irreducible
curves $L_{xt}$, $R_z$,  $R$. Moreover, $Z_\alpha$ and $R$ contain
the point $P$ and they are smooth at the point $P$.

Suppose that $\alpha\ne 1$. It is easy to check
\[
D\cdot Z_{\alpha}=\frac{3(12n+19)}{2(4n+5)(4n+7)}.%
\]
We also see that
\[Z_\alpha^2=C_\alpha\cdot Z_\alpha-L_{xt}\cdot Z_\alpha
\geqslant C_\alpha\cdot Z_\alpha-(L_{xt}+R_x)\cdot Z_\alpha=\frac{4n+5}{3}D\cdot Z_\alpha>0\]
since $Z_\alpha$ is different from the curve $R_x$.
 Put $D=\epsilon Z_{\alpha}+\Xi$, where $\Xi$ is
an effective $\mathbb{Q}$-divisor such that
$Z_{\alpha}\not\subset\mathrm{Supp}(\Xi)$. Since the pair $(X, D)$
is log canonical at the point $O_y$ and the curve $Z_\alpha$
passes through the point $O_y$,  we have  $ \epsilon\leqslant 1$.
But
\[
(D-\epsilon Z_\alpha)\cdot Z_\alpha \leqslant D\cdot Z_\alpha
=\frac{3(12n+19)}{2(4n+5)(4n+7)}<1
\]
and hence Lemma~\ref{lemma:handy-adjunction} implies that the
point $P$ cannot belong to the curve $Z_\alpha$.

Suppose that $\alpha=1$. Then  we have
$$
 D\cdot R=\frac{6(2n+3)}{(4n+5)(4n+7)}.%
$$
Since $R$ is different from $L_{yz}$ and $R_x$,
\[R^2=C_\alpha\cdot R-L_{xt}\cdot R -R_z\cdot R
\geqslant C_\alpha\cdot R-(L_{xt}+R_x)\cdot R -(L_{yz}+R_z)\cdot R\geqslant \frac{4n+3}{6}D\cdot R>0 \]

 Put $D=\epsilon_{1} R+\Xi'$, where $\Xi'$ is an effective
$\mathbb{Q}$-divisor such that $R\not\subset\mathrm{Supp}(\Xi')$.
Since the curve $R$ passes through the point $O_y$ at which the
pair $(X, D)$ is log canonical, we have $\epsilon_1\leqslant 1$. Since
$$
(D-\epsilon_1 R)\cdot R\leqslant D\cdot R=\frac{6(2n+3)}{(4n+5)(4n+7)}<1.
$$
Lemma~\ref{lemma:handy-adjunction} implies that the point $P$
cannot belong to $R$.
\end{proof}

\begin{lemma}
Let $X$ be a quasismooth
hypersurface of degree $47$ in $\mathbb{P}(8,13,15,17)$.
Then $\lct(X)=1$.
\end{lemma}

\begin{proof}
If we exclude the point $O_t$, then the proof of Lemma~\ref{lemma:I-6-infinite-series-1} works for this case.
Thus we suppose that $P=O_{t}$. Then $L_{yz}\not\subset\mathrm{Supp}(D)$; otherwise we would have a contradictory inequality
$$
\frac{3}{4\cdot 17}=D\cdot L_{yz}\geqslant \mult_{P}(D)>\frac{1}{17}.%
$$
 By Lemma~\ref{lemma:convexity}, we may
assume that $R_{y}\not\subset\mathrm{Supp}(D)$. Put
$$
D=mL_{yz}+cR_{x}+\Omega,
$$
where $m>0$ and $c\geqslant 0$, and $\Omega$ is an effective
$\mathbb{Q}$-divisor  whose support contains neither $L_{yz}$ nor $R_x$.
Then
$$
\frac{24}{15\cdot 17}=D\cdot R_{y}=\big(mL_{yz}+cR_{x}+\Omega\big)\cdot R_{y}\geqslant \frac{4m}{17}+\frac{\mult_{O_{t}}(D)-m}{17}>\frac{3m+1}{17},%
$$
and hence $$m<\frac{1}{5}.$$ Then it follows from
Lemma~\ref{lemma:handy-adjunction} that
$$
\frac{6+19m}{8\cdot 17}=(D-mL_{yz})\cdot L_{yz}>\frac{1}{17},%
$$
and hence
$$\frac{2}{19}<m.$$ On the other hand, if $c>0$, then
$$
\frac{6}{13\cdot 15}=D\cdot L_{xt}\geqslant cR_x\cdot L_{xt}= \frac{2c}{13}.%
$$
Therefore, $0\leqslant c\leqslant \frac{1}{5}$.

Let $\pi\colon\bar{X}\to X$ be the weighted blow up at the point $O_{t}$ with
weights $(6,7)$. Let $E$ be the exceptional curve of $\pi$. Also we let
$\bar{\Omega}$, $\bar{L}_{yz}$ and $\bar{R}_{x}$ be the proper
transforms of $\Omega$, $L_{yz}$ and $R_{x}$, respectively. Then
$$
K_{\bar{X}}\qlineq \pi^{*}(K_{X})-\frac{4}{17}E,\ %
\bar{L}_{yz}\qlineq \pi^{*}(L_{yz})-\frac{7}{17}E,\ %
\bar{R}_{x}\qlineq \pi^{*}(R_{x})-\frac{6}{17}E,\ %
\bar{\Omega}\qlineq \pi^{*}(\Omega)-\frac{a}{17}E,\ %
$$
where $a$ is a non-negative rational number.

The curve $E$ contains two singular points $Q_{7}$ and $Q_{6}$ of
$\bar{X}$. The point $Q_{7}$ is a singular point of type
$\frac{1}{7}(1,3)$ and the point $Q_{6}$ is a singular point of type
$\frac{1}{6}(1,1)$. Then the point $Q_7$ is contained in $\bar{R}_x$ but not in $\bar{L}_{yz}$, on the other hand, $Q_6$ is contained in $\bar{L}_{yz}$  but not in $\bar{R}_x$.
We also see that $\bar{L}_{yz}\cap\bar{R}_{x}=\varnothing$. The log pull back
of the log pair $(X,D)$ is the log pair
$$
\left(\bar{X},\
\bar{\Omega}+m\bar{L}_{yz}+c\bar{R}_{x}+\frac{4+a+7m+6c}{17}E\right).
$$
This pair must have non-log canonical singularity at some point $Q\in
E$. Then
\begin{gather*}
0\leqslant \bar{R}_{x}\cdot\bar{\Omega}=R_x\cdot\Omega
+\frac{6a}{17^2}E^2=\frac{12-13m+18c}{13\cdot 17}-\frac{a}{7\cdot
17}
,\\
0\le \bar{L}_{yz}\cdot\bar{\Omega}=L_{yz}\cdot\Omega +\frac{7a}{17^2}E^2=\frac{6+19m-8c}{8\cdot 17}-\frac{a}{6\cdot 17},%
\end{gather*}
and hence $0\leqslant 84-13a+126c-91m$ and $0\leqslant
18-4a-24c+57m$. In particular, we see that $a\leqslant
\frac{259}{40}$. Then $4+a+7m+6c<17$ since $m\frac{1}{5}$ and
$c\leqslant\frac{1}{5}$.

Suppose that the point $Q$ is neither $Q_6$ nor $Q_7$.  Then the point $Q$ must be located in the outside of
$\bar{L}_{yz}$ and $\bar{R}_{x}$. By
Lemma~\ref{lemma:handy-adjunction}, we have
$$
\frac{a}{42}=-\frac{a}{17}E^{2}=\bar{\Omega}\cdot E>1,
$$
and hence $a>42$. This is a contradiction since $a<\frac{259}{40}$.
Therefore, either $Q=Q_{6}$ or $Q=Q_{7}$.

Suppose that $Q=Q_{7}$. Then $Q\not\in\bar{L}_{yz}$. Hence, it
follows from Lemma~\ref{lemma:handy-adjunction} that
$$
\frac{1}{7}\leqslant \left(\bar{\Omega}+m\bar{L}_{yx}+\frac{4+a+7m+6c}{17}E\right)\cdot \bar{R}_{x}
=\frac{136+204c}{7\cdot 13\cdot 17},%
$$
and hence $c> \frac{5}{12}$. But $c\leqslant \frac{1}{5}$. This is a contradiction.

Finally, we suppose that $Q=Q_{6}$. Then $Q\not\in\bar{R}_{x}$. It
follows from Lemma~\ref{lemma:handy-adjunction} that
$$
\frac{1}{6}\leqslant \left(\bar{\Omega}+c\bar{R}_{x}+\frac{4+a+7m+6c}{17}E\right)\cdot \bar{L}_{yz}
=\frac{34+85m}{3\cdot 8\cdot 17},%
$$
and hence $m>\frac{2}{5}$. This contradiction completes the proof.
\end{proof}

\begin{lemma}
\label{lemma:I-6-infinite-series-1-n-2}Let $X$ be a quasismooth
hypersurface of degree $35$ in $\mathbb{P}(8,9,11,13)$.
Then $\lct(X)=1$.
\end{lemma}

\begin{proof}
If we exclude the points $O_z$ and $O_t$, then the proof of Lemma~\ref{lemma:I-6-infinite-series-1} works  also for this case.

Suppose that $P=O_{z}$. Then $L_{xt}\subset\mathrm{Supp}(D)$,
since otherwise we would have an absurd inequality
$$
\frac{6}{9\cdot 11}=D\cdot L_{xt}>\frac{1}{11}.%
$$
We may assume that
$M_{t}\not\subset\mathrm{Supp}(D)$ by
Lemma~\ref{lemma:convexity}. Put
$$
D=mL_{xt}+cM_{y}+\Omega,
$$
where $m>0$ and $c\geqslant 0$, and $\Omega$ is an effective
$\mathbb{Q}$-divisor whose support contains neither
$L_{xt}$ nor $R_{y}$. Then
$$
\frac{18}{8\cdot 11}=D\cdot R_{t}=\big(mL_{xt}+cR_{y}+\Omega\big)\cdot R_{t}\geqslant \frac{3m}{11}+\frac{2(\mult_{O_{z}}(D)-m)}{11}>\frac{m+2}{11},%
$$
and hence $m<\frac{1}{4}$. Note that $\mult_{O_z}(R_t)=2$. It follows from
Lemma~\ref{lemma:handy-adjunction} that
$$
\frac{6+14m}{9\cdot 11}=\big(D-mL_{xt}\big)\cdot L_{xt}>\frac{1}{11}.%
$$
Therefore, $\frac{3}{14}<m<\frac{1}{4}$. On the other hand, if $c>0$, then
$$
\frac{6}{8\cdot 13}=D\cdot L_{yz}\geqslant cR_{y}\cdot L_{yz}=\frac{3c}{13},%
$$
and hence $c\leqslant \frac{1}{4}$.

Let $\pi\colon\bar{X}\to X$ be the weighted blow up at the point $O_{z}$ with
weights $(3,2)$. Let $E$ be the exceptional curve of $\pi$ and let
$\bar{\Omega}$, $\bar{L}_{xt}$ and $\bar{R}_{y}$ be the proper
transforms of $\Omega$, $L_{xt}$ and $R_{y}$, respectively. Then
$$
K_{\bar{X}}\qlineq \pi^{*}(K_{X})-\frac{6}{11}E,\ %
\bar{L}_{xt}\qlineq \pi^{*}(L_{xt})-\frac{3}{11}E,\ %
\bar{R}_{y}\qlineq \pi^{*}(R_{y})-\frac{2}{11}E,\ %
\bar{\Omega}\qlineq \pi^{*}(\Omega)-\frac{a}{11}E.%
$$
where $a$ is a non-negative rational number.

The curve $E$ contains two singular points $Q_{2}$ and $Q_{3}$ of
$\bar{X}$. The point $Q_{2}$ is a singular point of type
$\frac{1}{2}(1,1)$. It is contained in $\bar{L}_{xt}$ but not in $\bar{R}_y$. On the other hand, the point $Q_{3}$ is a singular point of type
$\frac{1}{3}(2,1)$. It is contained in $\bar{R}_{y}$ but not in $\bar{L}_{xt}$. But $\bar{L}_{xt}\cap\bar{R}_{y}=\varnothing$.

The log pull back
of the log pair $(X,D)$ is the log pair
$$
\left(\bar{X},\
\bar{\Omega}+m\bar{L}_{xt}+c\bar{R}_{y}+\frac{6+a+3m+2c}{11}E\right),
$$
which must have non-log canonical singularity at some point $Q\in
E$. We have
\begin{gather*}
0\leqslant \bar{\Omega}\cdot
\bar{R}_{y}=\frac{18+6c}{11\cdot
13}-\frac{m}{11}-\frac{a}{33},\\
0\le\bar{\Omega}\cdot \bar{L}_{xt}=\frac{6+14m}{9\cdot 11}-\frac{c}{11}-\frac{a}{22}.%
\end{gather*}
Then, $a\leqslant \frac{12+28m}{9}< \frac{19}{9}$ since  $m<
\frac{1}{4}$. Also, we obtain $6+a+3m+2c<11$ since  $c\leqslant \frac{1}{4}$.

Suppose that the point $Q$ is neither $Q_2$ nor $Q_3$. Then
$Q\not\in\bar{L}_{xt}\cup\bar{R}_{y}$. By
Lemma~\ref{lemma:adjunction}, we have
$$
\frac{a}{2\cdot 3}=-\frac{a}{11}E^{2}=\bar{\Omega}\cdot E>1,
$$
and hence $a>6$. This contradicts to the inequality $a<\frac{19}{9}$.
Therefore, we see that either $Q=Q_{2}$ or $Q=Q_{3}$.

Suppose that $Q=Q_{2}$. Then $Q\not\in\bar{R}_{y}$. Lemma~\ref{lemma:handy-adjunction} shows that
$$
\frac{1}{2}<\left(\bar{\Omega}+c\bar{R}_y+\frac{6+a+3m+2c}{11}E\right)\cdot \bar{L}_{xt}=\frac{66+55m}{2\cdot 9\cdot 11}
$$
and hence $m>\frac{3}{5}$. But $m<\frac{1}{4}$. This is a
contradiction.

Thus, the point $Q$ must be $Q_3$. Then $Q\not\in\bar{L}_{xt}$. It
follows from Lemma~\ref{lemma:handy-adjunction} that
$$
\frac{1}{3}<
\left(\bar{\Omega}+m\bar{L}_{xt}+\frac{6+a+3m+2c}{11}E\right)\cdot \bar{R}_{y}=\frac{132+44c}{13\cdot 33}.%
$$
Therefore, $c>\frac{1}{4}$. But we have seen $c\leqslant \frac{1}{4}$. The obtained
contradiction shows that $P\ne O_{z}$.
The point $P$ must be the point $O_t$. Then $L_{yz}\not\subset\mathrm{Supp}(D)$
since otherwise we would have
$$
\frac{6}{8\cdot 13}=D\cdot L_{yz}>\frac{1}{13}.%
$$
By Lemma~\ref{lemma:convexity}, we may
assume that $R_{y}\not\subset\mathrm{Supp}(D)$. Put
$$
D=mL_{yz}+cR_{x}+\Omega,
$$
where $m>0$ and $c\geqslant 0$, and $\Omega$ is an effective
$\mathbb{Q}$-divisor whose support contains neither
$L_{yz}$ nor $R_{x}$. Then
$$
\frac{18}{11\cdot 13}=D\cdot R_{y}=\big(mL_{yz}+cR_{x}+\Omega\big)\cdot R_{y}\geqslant \frac{3m}{13}+\frac{\mult_{O_{t}}(D)-m}{13}>\frac{2m+1}{13},%
$$
and hence  $m<\frac{7}{22}$. On the other hand,
Lemma~\ref{lemma:handy-adjunction} implies
$$
\frac{6+15m}{8\cdot 13}=\big(D-mL_{yz}\big)\cdot L_{yz}>\frac{1}{13},%
$$
and hence $\frac{2}{15}<m<\frac{7}{22}$. If $c>0$, then
$$
\frac{6}{9\cdot 11}=D\cdot L_{xt}=cR_{x}\cdot L_{xt}\geqslant \frac{2c}{9}.%
$$
Therefore, $c\leqslant \frac{3}{11}$.

Let $\pi\colon\bar{X}\to X$ be the weighted blow up at the point $O_{t}$ with
weights $(5,2)$. Let $E$ be the exceptional curve of $\pi$. Let
$\bar{\Omega}$, $\bar{L}_{yz}$ and $\bar{R}_{x}$ be the proper
transforms of $\Omega$, $L_{yz}$ and $R_{x}$, respectively. Then
$$
K_{\bar{X}}\qlineq \pi^{*}(K_{X})-\frac{6}{13}E,\ %
\bar{L}_{yz}\qlineq \pi^{*}(L_{yz})-\frac{2}{13}E,\ %
\bar{R}_{x}\qlineq \pi^{*}(R_{x})-\frac{5}{13}E,\ %
\bar{\Omega}\qlineq \pi^{*}(\Omega)-\frac{a}{13}E,\ %
$$
where $a$ is a non-negative rational number.

The curve $E$ contains two singular points $Q_{5}$ and $Q_{2}$ of
$\bar{X}$.  The point $Q_{5}$ is a singular point of type
$\frac{1}{5}(1,1)$. It belongs to $\bar{L}_{yz}$ but not to $\bar{R}_{x}$. The point $Q_{2}$ is a singular point of type
$\frac{1}{2}(1,1)$.  It belongs to $\bar{R}_{x}$ but not to $\bar{L}_{yz}$. Note that $\bar{L}_{yz}\cap\bar{R}_{x}=\varnothing$.

The log pull back
of the log pair $(X,D)$ is the log pair
$$
\left(\bar{X},\
\bar{\Omega}+m\bar{L}_{yz}+c\bar{R}_{x}+\frac{6+a+2m+5c}{13}E\right).
$$
It must have non-log canonical singularity at some point $Q\in
E$. We have
\begin{gather*}
0\leqslant \bar{\Omega}\cdot
\bar{R}_{x}=\frac{12+10c}{9\cdot
13}-\frac{m}{13}-\frac{a}{26},\\
0\le\bar{\Omega}\cdot \bar{L}_{yz}=\frac{6+15m}{8\cdot 13}-\frac{c}{13}-\frac{a}{65}.%
\end{gather*}
Therefore, $30+75m\geqslant 40c+8a$ and $24+20c\geqslant
18m+9a$. In particular, we see that $a\leqslant \frac{240}{77}$. Then
$6+a+2m+5c<13$ since $c\leqslant \frac{3}{11}$ and $m\leqslant \frac{7}{22}$.

Suppose that $Q\ne Q_{2}$ and $Q\ne Q_{5}$. Then
$Q\not\in\bar{L}_{yz}\cup\bar{R}_{x}$. By
Lemma~\ref{lemma:handy-adjunction}, we have
$$
\frac{a}{10}=-\frac{a}{13}E^{2}=\bar{\Omega}\cdot E>1,
$$
and hence $a>10$. This is a contradiction since $a<\frac{240}{77}$.
Therefore, the point $Q$ is either the point $Q_{2}$ or the point $Q_{5}$.

Suppose that $Q=Q_{2}$. Then $Q\not\in\bar{L}_{yz}$. It
follows from Lemma~\ref{lemma:handy-adjunction} that
$$
\frac{1}{2}<\left(\bar{\Omega}+m\bar{L}_{yz}+\frac{6+a+2m+5c}{13}E\right)\cdot \bar{R}_{x}=\frac{78+65c}{9\cdot 26},%
$$
and hence $c>\frac{3}{5}$. However,  $c\leqslant \frac{3}{11}$.
Thus, the point $Q$ must be $Q_5$. Then $Q\not\in\bar{R}_{x}$. Again,
Lemma~\ref{lemma:handy-adjunction} shows that
\begin{gather*}
\frac{1}{5}<\left(\bar{\Omega}+c\bar{R}_x+\frac{6+a+2m+5c}{13}E\right)\cdot\bar{L}_{yz}=\frac{78+91m}{5\cdot 8\cdot
13},\\
\frac{1}{5}<\left(\bar{\Omega}+m\bar{L}_{yz}\right)\cdot E=\frac{a}{10}+\frac{m}{5}.%
\end{gather*}
Therefore, $m>\frac{2}{7}$ and $a+2m>2$. In particular, $\frac{2}{7}<m<\frac{7}{22}$.

Let $\psi\colon\tilde{X}\to \bar{X}$ be the weighted blow up at the point
$Q_{5}$ with weights $(1,1)$. Let $G$ be the exceptional curve of
$\psi$ and let $\tilde{\Omega}$, $\tilde{L}_{yz}$, $\tilde{R}_{x}$
and $\tilde{E}$ be the proper transforms of $\Omega$, $L_{yz}$,
$R_{x}$ and $E$, respectively. Then
$$
K_{\tilde{X}}\qlineq \psi^{*}(K_{\bar{X}})-\frac{3}{5}G,\ %
\tilde{L}_{yz}\qlineq \psi^{*}(\bar{L}_{yz})-\frac{1}{5}G,\ %
\tilde{E}\qlineq \psi^{*}(E)-\frac{1}{5}G,\ %
\tilde{\Omega}\qlineq\psi^{*}(\bar{\Omega})-\frac{b}{5}G,%
$$
where $b$ is a non-negative rational number.

The surface is smooth along $G$. The log pull back of $(X,D)$ is
the log pair
$$
\left(\tilde{X},\
\tilde{\Omega}+m\tilde{L}_{yz}+c\tilde{R}_{x}+\frac{6+a+2m+5c}{13}\tilde{E}+\theta G\right),%
$$
where
$$\theta =\frac{15m+45+a+13b+5c}{65}.$$
Then
 the log pair  is not log canonical at some point $O\in
G$. We have
\begin{gather*}
0\leqslant \tilde{E}\cdot\tilde{\Omega}=\frac{a}{10}-\frac{b}{5},\\
0\le\tilde{L}_{yz}\cdot\tilde{\Omega}=\frac{6+15m}{8\cdot 13}-\frac{c}{13}-\frac{a}{65}-\frac{b}{5},%
\end{gather*}
and hence $30+75m\geqslant 8(a+13b+5c)$ and $a\geqslant
2b$.
In particular,  we obtain
\[\theta =\frac{15m+45+a+13b+5c}{65}\leqslant \frac{195m+390}{8\cdot 65}\leqslant \frac{195\cdot 7+390\cdot 22}{8\cdot 22\cdot 65}<1\]
since $m\leqslant\frac{7}{22}$.

Suppose that $O\not\in\tilde{E}\cup\tilde{L}_{yz}$. Then it
follows from Lemma~\ref{lemma:handy-adjunction} that
$$
b=-\frac{b}{5}G^{2}=\tilde{\Omega}\cdot G>1.%
$$
However, this gives an absurd inequality $104<104b\leqslant 30+75m-8a-40c\leqslant 30+75m <104$ since $m\leqslant\frac{7}{22}$. Therefore,
$O\in\tilde{E}\cup\tilde{L}_{yz}$. Note that
$\tilde{E}\cap\tilde{L}_{yz}=\varnothing$.

Suppose that $O\in \tilde{L}_{yz}$. Then it follows from
Lemma~\ref{lemma:handy-adjunction} that
$$
1<\big(\tilde{\Omega}+c\tilde{R}_{x}+\frac{6+a+2m+5c}{13}\tilde{E}+\theta
G\big)\cdot \tilde{L}_{yz}=\big(\tilde{\Omega}+\theta G\big)\cdot
\tilde{L}_{yz}
=\frac{3m+6}{8},%
$$
and  hence $m>\frac{2}{3}$. But $m\leqslant\frac{7}{22}$.
Thus, we see that $O\in\tilde{E}$.
Lemma~\ref{lemma:handy-adjunction} implies that
\begin{gather*}
1<
\left(\tilde{\Omega}+\frac{6+a+2m+5c}{13}\tilde{E}\right)\cdot
G=b+\frac{6+a+2m+5c}{13},\\
1<\big(\tilde{\Omega}+\theta G\big)\cdot \tilde{E}=\frac{a}{10}-\frac{b}{5}+\theta.%
\end{gather*}
Therefore, we obtain $13b+a+2m+5c>7$ and $3a+2c+6m>8$.

Let $\phi\colon\hat{X}\to\tilde{X}$ be the blow up at the point $O$.
Let $F$ be the exceptional curve of $\phi$. Let $\hat{\Omega}$,
$\hat{L}_{yz}$, $\hat{R}_{x}$, $\hat{E}$ and $\hat{G}$ be the
proper transforms of $\Omega$, $L_{yz}$, $R_{x}$, $E$ and $G$,
respectively. Then
$$
K_{\hat{X}}\qlineq \phi^{*}(K_{\tilde{X}})+F,\ %
\hat{G}\qlineq \phi^{*}(G)-F,\ %
\hat{E}\qlineq \phi^{*}(\tilde{E})-F,\ %
\hat{\Omega}\qlineq\phi^{*}(\tilde{\Omega})-dF,%
$$
where $d$ is a non-negative rational number. The log pull back of
$(X,D)$ is the log pair
$$
\left(\hat{X},\
\hat{\Omega}+m\hat{L}_{yz}+c\hat{R}_{x}+\frac{6+a+2m+5c}{13}\hat{E}+\theta\hat{G}+\nu F\right),%
$$
where
$$\nu =\frac{65d+25m+6a+13b+30c+10}{65}.$$
It is not log canonical at some point $A\in
F$. We have
\begin{gather*}
0\leqslant\hat{E}\cdot\hat{\Omega}=\frac{a}{10}-\frac{b}{5}-d,\\
0\le\hat{G}\cdot\hat{\Omega}=b-d,%
\end{gather*}
 and hence $b\geqslant d$ and $a\geqslant 2b+10d$. In particular,
\[\begin{split}\nu &=\frac{65d+25m+6a+13b+30c+10}{65}=\\
&=\frac{13(5d+b)+25m+6a+30c+10}{65}\leqslant\\
&\leqslant \frac{5a+10m+12c+4}{26}\leqslant\\
&\leqslant \frac{6+8c}{9}<1\\
\end{split}\]
since we have $24+20c\geqslant
18m+9a$ and $c\leqslant \frac{3}{11}$.

Suppose that $A\not\in\hat{E}\cup\hat{G}$. Then
Lemma~\ref{lemma:handy-adjunction} shows that $d=\hat{\Omega}\cdot F>1$. This
is impossible since
\[10d\leqslant a-2b\leqslant a\leqslant \frac{240}{77}.\]
Thus, we see that $A\in\hat{E}\cup\hat{G}$. Note
that $\hat{E}\cap\hat{G}=\varnothing$.

Suppose that $A\in\hat{E}$. Then it follows from
Lemma~\ref{lemma:handy-adjunction} that
$$
\frac{a}{10}-\frac{b}{5}-d+\nu=\big(\hat{\Omega}+\nu F\big)\cdot \hat{E}>1,%
$$
which implies that $5a+10m+12c>22$. However,
this inequality with $24+20c\geqslant 18m+9a$ gives
\[\frac{9}{5}(22-12c)<\frac{9}{5}(5a+10m)\leqslant 24+20c,\]
and hence $\frac{3}{8}<c$.
But $c\leqslant \frac{3}{11}$.
Thus, the point $A$ cannot belong to  $\hat{E}$. Then
$A\in\hat{G}$. By Lemma~\ref{lemma:handy-adjunction}, we see that
$$
b-d+\nu=\big(\hat{\Omega}+\nu F\big)\cdot \hat{G}>1,%
$$
and hence $6a+25m+30c+78b>55$. But
\[55<25m+6a+78b+3c=25m+\frac{3}{4}(8a+104b+40c)\leqslant 25m+\frac{3}{4}(30+75m)<55\]
since $8a+104b+40c\leqslant 30+75m$ and $m\leqslant\frac{7}{22}$. The obtained contradiction completes the proof.
\end{proof}

\begin{lemma}
\label{lemma:I-6-infinite-series-2} Let $X$ be a quasismooth
hypersurface of degree $12n+35$ in $\mathbb{P}(9, 3n+8, 3n+11,
6n+13)$ for $n\geqslant 1$.
 Then $\lct(X)=1$.
\end{lemma}

\begin{proof} The surface $X$ can be defined by the equation
$$
z^{2}t+y^{3}z+xt^{2}+x^{n+3}y=0.
$$
It is singular only at the points $O_x$, $O_y$, $O_z$ and $O_t$.

The curve $C_x$ (resp. $C_y$, $C_z$, $C_t$) consists of two
irreducible and reduced curves $L_{xz}$ (resp. $L_{yt}$, $L_{xz}$,
$L_{yt}$) and $R_x=\{x=zt+y^3=0\}$ (resp. $R_y=\{y=z^2+xt=0\}$,
$R_z=\{z=t^2+x^{n+2}y=0\}$, $R_t=\{t=y^2z+x^{n+3}=0\}$). These two
curves intersect at the point $O_t$ (resp. $O_x$, $O_y$, $O_z$).

It is easy to see that $\lct(X, \frac{2}{3}C_x)=1$ is less than
each of the numbers
\[
\lct(X, \frac{6}{3n+8} C_y),\ \ \lct(X, \frac{6}{3n+11} C_z), \ \
\lct(X, \frac{6}{6n+13} C_t).\]

We have the following intersection numbers.
$$-L_{xz}\cdot K_X=\frac{6}{(3n+8)(6n+13)},\ \ \ -L_{yt}\cdot K_X=\frac{2}{3(3n+11)},\ \ \
-R_x\cdot K_X=\frac{18}{(3n+11)(6n+13)},$$
$$ -R_y\cdot K_X=\frac{4}{3(6n+13)}, \ \ \
-R_z\cdot K_X=\frac{4}{3(3n+8)}, \ \ \ -R_t\cdot K_X=\frac{6(n+3)}{(3n+8)(3n+11)},$$
$$ L_{xz}\cdot R_x=\frac{3}{6n+13}, \ \ \
L_{yt}\cdot R_y=\frac{2}{9}, \ \ \ L_{xz}\cdot R_z=\frac{2}{3n+8}, \ \ \ L_{yt}\cdot R_t=\frac{n+3}{3n+11}, $$
$$L_{xz}^2=-\frac{9n+15}{(3n+8)(6n+13)}, \ \ \  L_{yt}^2=-\frac{3n+14}{9(3n+11)}, \ \ \ R_x^2=-\frac{9n+6}{(3n+11)(6n+13)},$$
$$ R_y^2=-\frac{6n+10}{9(6n+13)}, \ \ \
R_z^2=\frac{6n+4}{9(3n+8)}, \ \ \ R_t^2=\frac{(n+3)(3n+5)}{(3n+8)(3n+11)}.$$

Now we suppose that $\lct(X)<1$. Then there is an effective
$\Q$-divisor $D\qlineq -K_X$ such that the log pair $(X, D)$ is
not log canonical at some point $P\in X$.

 By Lemma~\ref{lemma:convexity} we may
assume that $\Supp(D)$ does not contain either the curve $L_{yt}$
or the curve $R_y$. Since these two curves intersect at the point
$O_x$, the inequalities
$$L_{yt}\cdot D =\frac{2}{3(3n+11)}<\frac{1}{9},$$
$$R_y\cdot D=\frac{4}{3(6n+13)}<\frac{1}{9}$$
show that the point $P$ cannot be the point $O_x$.

By Lemma~\ref{lemma:convexity} we may assume that $\Supp(D)$ does
not contain either the curve $L_{xz}$ or the curve $R_z$.
Therefore, one of the following inequalities must hold:
$$\mult_{O_y}( D)\leqslant (3n+8)L_{xz}\cdot D =\frac{6}{6n+13}<1,$$
$$\mult_{O_y}( D)\leqslant \frac{3n+8}{2}R_z\cdot D =\frac{2}{3}.$$
Therefore, the point $P$ cannot be the point $O_y$.

Suppose that $P=O_{z}$. If $L_{yt}\not\subset\Supp(D)$, then we get an absurd inequality
$$\frac{6}{9(3n+11)}=L_{yt}\cdot D>\frac{1}{3n+11}.$$
Therefore
$\Supp(D)$ must contain the curve $L_{yt}$. By Lemma~\ref{lemma:convexity} we may
assume that $M_{t}\not\subset\Supp(D)$. Put $D=\mu L_{yt}+\Omega$,
where $\Omega$ is an effective $\mathbb{Q}$-divisor whose support does not contain the curve $L_{yt}$. Then
$$
\frac{6(n+3)}{(3n+8)(3n+11)}=D\cdot R_t\geqslant\mu L_{yt}\cdot
R_{t}+\frac{(\mult_P(D)-\mu)\mult_{P}(R_{t})}{3n+11}>
\frac{\mu(n+3)}{3n+11}+\frac{2(1-\mu)}{3n+11},
$$
and hence
$$
\mu<\frac{2}{(3n+8)(n+1)}.
$$
On the other hand, Theorem~\ref{theorem:adjunction} shows
$$\frac{1}{3n+11}<\Omega\cdot L_{yt}=D\cdot L_{yt}-\mu L_{yt}^2=\frac{6+\mu(3n+14)}{9(3n+11)}.$$
It implies $\frac{3}{3n+14}<\mu$. Consequently, the point $P$ cannot be the point $O_z$.

Suppose that $P=O_{t}$. Since $L_{xz}\cdot D< \frac{1}{6n+13}$,
the curve $L_{xz}$ must be contained in $\Supp(D)$. Then, we may
assume that $R_x\not\subset\Supp(D)$.
 Put $D=\mu L_{xz}+\Omega$,
where $\Omega$ is an effective $\mathbb{Q}$-divisor whose support
does not contain the curve $L_{xz}$. Then
$$
\frac{18}{(3n+11)(6n+13)}=D\cdot R_x \geqslant\mu L_{xz}\cdot
R_x+\frac{\mult_{P}(D)-\mu}{6n+13}> \frac{1+2\mu}{6n+13},
$$
and hence
$$
\mu<\frac{7-3n}{6n+22}.
$$

However,
Theorem~\ref{theorem:adjunction} implies
$$\frac{1}{6n+13}<\Omega\cdot L_{xz}=D\cdot L_{xz}-\mu L_{xz}^2=\frac{6+(9n+15)\mu}{(3n+8)(6n+13)},$$
and hence $\frac{3n+2}{9n+15}<\mu$. This is a contradiction. Therefore, the point $P$ cannot be the point $O_t$.

Write $D=aL_{xz}+b R_x+\Delta$, where $\Delta$ is an  effective
$\mathbb{Q}$-divisor whose support contains neither $L_{xz}$ nor
$R_x$. Since the log pair $(X, D)$ is log canonical at the point
$O_t$, we have $0\leqslant a, b \leqslant 1$. Then by
Theorem~\ref{theorem:adjunction} the following two inequalities
$$(bR_x+\Delta)\cdot L_{xz}=(D-aL_{xz})\cdot L_{xz}=\frac{6+a(9n+15)}{(3n+8)(6n+13)}<1,$$
$$(aL_{xz}+\Delta)\cdot R_x=(D-bR_{x})\cdot R_{x}=\frac{18+b(9n+6)}{(3n+11)(6n+13)}<1$$
show that the point $P$ cannot belong to the curve $C_x$. By the
same way, we can show $P\not\in C_{y}\cup C_{z}\cup C_{t}$.

Consider the pencil $\mathcal{L}$ defined by  the equations
$\lambda xt+\mu z^{2}=0$, $[\lambda :\mu]\in \mathbb{P}^1$. Note
that the curve $L_{xz}$ is the only base component of the pencil
$\mathcal{L}$. There is a unique divisor $C_\alpha$ in
$\mathcal{L}$ passing through the point $P$. This divisor must be
defined an equation $xt+\alpha z^2=0$, where $\alpha$ is a
non-zero constant, since the point $P$ is located in the outside
of $C_x\cup C_z\cup C_t$. Note that the curve $C_t$ does not
contain any component of $C_{\alpha}$. Therefore, to see all the
irreducible components of $C_\alpha$, it is enough to see the
affine curve
$$
\left\{\aligned &x+\alpha z^{2}=0\\ &z^{2}+y^{3}z+x+x^{n+3}y
=0\\ \endaligned\right\}
\subset\mathbb{C}^{3}\cong\mathrm{Spec}\Big(\mathbb{C}\big[x,y,z\big]\Big).
$$
This is isomorphic to the plane affine curve defined by the
equation
$$
z\{(1-\alpha)z+y^3+(-\alpha)^{n+3}yz^{2n+5}\}
=0\subset\mathbb{C}^{2}\cong\mathrm{Spec}\Big(\mathbb{C}\big[y,z\big]\Big).
$$
Thus, if $\alpha\ne 1$, then the divisor $C_{\alpha}$ consists of
two reduced and  irreducible curves $L_{xz}$ and $Z_{\alpha}$. If
$\alpha=1$, then it consists of three reduced and irreducible
curves $L_{xz}$, $R_y$,  $R$. Moreover, $Z_\alpha$ and $R$ are
smooth at the point $P$.

Suppose that $\alpha\ne 1$. Then we have
\[D\cdot Z_{\alpha}=\frac{2(24n+61)}{3(3n+8)(6n+13)}.%
\]
Since $Z_\alpha$ is different from $R_x$,
\[
Z_{\alpha}^2=C_\alpha\cdot Z_\alpha- L_{xz}\cdot Z_\alpha
\geqslant C_\alpha\cdot Z_\alpha- (L_{xz}+R_x)\cdot Z_\alpha=\frac{6n+13}{6}D\cdot Z_\alpha>0.\]

Put $D=\epsilon Z_{\alpha}+\Xi$, where $\Xi$ is an effective
$\mathbb{Q}$-divisor such that
$Z_{\alpha}\not\subset\mathrm{Supp}(\Xi)$. Since the pair $(X, D)$
is log canonical at the point $O_t$ and the curve $Z_\alpha$
passes through the point $O_t$,  we have  $ \epsilon\leqslant 1$.
But
\[
(D-\epsilon Z_\alpha)\cdot Z_\alpha \leqslant D\cdot Z_\alpha
=\frac{2(24n+61)}{3(3n+8)(6n+13)}<1
\]
and hence Lemma~\ref{lemma:handy-adjunction} implies that the
point $P$ cannot belong to the curve $Z_\alpha$.

Suppose that $\alpha=1$. We have
$$
 D\cdot R=\frac{6(2n+5)}{(3n+8)(6n+13)}.%
$$
Since $R$ is different from $R_x$ and $L_{yt}$,
\[ R^2=C_\alpha\cdot R-L_{xz}\cdot R- R_y\cdot R\geqslant C_\alpha\cdot R-(L_{xz}+R_x)\cdot R- (L_{yt}+R_y)\cdot R=\frac{3n+5}{6}D\cdot D>0.\]
Put $D=\epsilon_{1} R+\Xi'$, where $\Xi'$ is an effective
$\mathbb{Q}$-divisor such that $R\not\subset\mathrm{Supp}(\Xi')$.
Since the curve $R$ passes through the point $O_t$ at which the
pair $(X, D)$ is log canonical, $\epsilon_1\leqslant 1$. Since
$$
(D-\epsilon_1 R)\cdot R\leqslant D\cdot
R=\frac{6(2n+5)}{(3n+8)(6n+13)}<1,
$$
Lemma~\ref{lemma:handy-adjunction} implies that the point $P$
cannot belong to $R$.
\end{proof}

\part{Sporadic cases}

\section{Sporadic cases with $I=1$}
\label{section:index-1}

\begin{lemma}
\label{lemma:123510}
Let $X$ be a quasismooth
hypersurface of degree $10$ in $\mathbb{P}(1,2,3,5)$.
Then
$$
\mathrm{lct}\big(X\big)=\left\{%
\aligned
&1\ \text{if $C_x$ has an ordinary double point},\\%
&\frac{7}{10}\ \text{if $C_x$ has a non-ordinary double point}.\\%
\endaligned\right.%
$$
\end{lemma}
\begin{proof} The surface $X$ is singular only at the point $O_{z}$.
The curve $C_{x}$ is reduced and irreducible. Moreover, we have
$$
\mathrm{lct}\big(X, C_{x}\big)=\left\{%
\aligned
&1\ \text{if the curve $C_x$ has an ordinary double point at the point $O_{z}$},\\%
&\frac{7}{10}\ \text{if the curve $C_x$ has a non-ordinary double point at the point $O_{z}$}.\\%
\endaligned\right.%
$$

Suppose that $\lct\big(X\big)<\mathrm{lct}\big(X, C_{x}\big)$. Then there is
 an  effective $\Q$-divisor $D\qlineq -K_X$
such that the log pair
$(X, D)$ is not log canonical at some point $P\in X$.  By Lemma~\ref{lemma:convexity} we may assume that the support of $D$ does not contain the curve $C_x$. Also Lemma~\ref{lemma:Carolina} shows that $P\in
C_{x}$.
However,  we obtain absurd inequalities
$$
\frac{1}{3}=D\cdot C_{x}\geqslant \left\{%
\aligned
&\mult_{P}\big(D\big)>1\ \text{ if $P\ne O_{z}$},\\%
&\frac{\mult_{P}\big(D\big)}{3}>\frac{1}{3}\ \text{ if $P=O_{z}$}.\\%
\endaligned\right.
$$
Therefore, $\lct\big(X\big)=\mathrm{lct}\big(X, C_{x}\big)$.
\end{proof}

\begin{lemma}
\label{lemma:1357}
Let $X$ be the quasismooth
hypersurface defined by a quasihomogeneous polynomial $f(x,y,z,t)$  of degree $15$ in $\mathbb{P}(1,3,5,7)$.
Then
$$
\mathrm{lct}\big(X\big)=\left\{%
\aligned
&1\ \text{ if $f(x,y,z,t)$ contains  $yzt$},\\%
&\frac{8}{15}\ \text{ if $f(x,y,z,t)$ does not contain  $yzt$}.\\%
\endaligned\right.%
$$
\end{lemma}

\begin{proof} The surface $X$ is singular only at the point $O_t$.
The curve $C_{x}$ is reduced and irreducible. It is easy to check
$$
\mathrm{lct}\big(X, C_{x}\big)=\left\{%
\aligned
&1\ \text{ if $f(x,y,z,t)$ contains  $yzt$},\\%
&\frac{8}{15}\ \text{ if $f(x,y,z,t)$ does not contain  $yzt$}.\\%
\endaligned\right.%
$$
The proof is exactly the same as  the proof of Lemma~\ref{lemma:123510}. The contradictory inequalities

$$
\frac{1}{7}=D\cdot C_{x}\geqslant \left\{%
\aligned
&\mult_{P}\big(D\big)>1 \ \text{ if $P\ne O_{t}$},\\%
&\frac{\mult_{P}\big(D\big)}{7}> \frac{1}{7}\ \text{ if $P=O_{t}$}.\\%
\endaligned\right.%
$$
complete the proof.
\end{proof}

\begin{lemma}
\label{lemma:1358} Let $X$ be a quasismooth
hypersurface of degree $16$ in $\mathbb{P}(1,3,5,8)$. Then
$\mathrm{lct}(X)=1$.
\end{lemma}

\begin{proof}
The surface $X$ is singular only at the points $O_{y}$ and $O_z$.
The former is  a singular point of type $\frac{1}{3}(1,1)$ and the
latter is of type $\frac{1}{5}(1,1)$.

The curve $C_{x}$ consists of two distinct irreducible curves $L_1$ and $L_2$.  It is easy to see that $\lct(X, C_x)=1$.

Suppose that $\lct\big(X\big)<1$. Then there is
 an  effective $\Q$-divisor $D\qlineq -K_X$
such that the log pair
$(X, D)$ is not log canonical at some point $P\in X$.  By Lemma~\ref{lemma:convexity} we may assume that the support of $D$ does not contain the curve $L_1$ without loss of generality. Moreover,  Lemma~\ref{lemma:Carolina} implies $P\in C_{x}$.

We have
$$
D\cdot L_{1}=D\cdot L_{2}=\frac{1}{15},
$$
and $L_{1}\cap L_{2}=\{O_{y}, O_{z}\}$. We also have
$$
L_{1}^2=L_{2}^2=-\frac{7}{15}, \ \ L_{1}\cdot L_{2}=\frac{8}{15}.
$$

Since $5D\cdot L_1=\frac{1}{3}$, the point $P$ cannot belong to $L_1$.
Therefore, the point $P$ is a smooth point on $L_2$. Put
$$
D=mL_{2}+\Omega,
$$
where $\Omega$ is an effective $\mathbb{Q}$-divisor such that
$L_{2}\not\subset\mathrm{Supp}(\Omega)$. Since the log pair $(X,D)$ is log canonical at $O_y$, we must have $m\leqslant 1$. Then it follows from Lemma~\ref{lemma:handy-adjunction} that
$$
1<\Omega\cdot
L_{2}=\big(D-mL_{2}\big)\cdot L_{2}=\frac{1+7m}{15}.
$$
This gives us $m>2$. This is a contradiction. Consequently, $\lct(X)=1$.
\end{proof}

\begin{lemma}
\label{lemma:2359} Let $X$ be a quasismooth hypersurface of degree
$18$ in $\mathbb{P}(2,3,5,9)$.  Then
$$
\mathrm{lct}\big(X\big)=\left\{%
\aligned
&2\ \text{if $C_y$ has a tacnodal point},\\%
&\frac{11}{6}\ \text{if $C_y$ has no tacnodal points}.\\%
\endaligned\right.%
$$
\end{lemma}

\begin{proof}
The surface $X$ is singular at the point $O_{z}$. This is a
singular point of type $\frac{1}{5}(1,2)$. The
surface $X$ also has two singular points $O_{1}$ and $O_{2}$ that
are cut out by the equations $x=z=0$. These are of type $\frac{1}{3}(1,1)$
on the surface $X$.

The curves $C_{x}$ and $C_{y}$ are reduced and irreducible. The
curve $C_y$ is always singular at the point $O_z$. We can see
$\lct(X, C_{x})=1$ and
$$
\mathrm{lct}\big(X, C_{y}\big)=\left\{%
\aligned
&\frac{3}{4}\ \text{if $C_y$ has a tacnodal singularity at the point $O_{z}$},\\%
&\frac{11}{18}\ \text{if $C_y$ has a non-tacnodal singularity at the point $O_{z}$}.\\%
\endaligned\right.%
$$
Therefore, if $C_y$ has a tacnodal singularity at the point
$O_{z}$, then
$$
2=\lct\left(X, \frac{1}{2}C_{x}\right)<\lct\left(X,
\frac{1}{3}C_{y}\right)=\frac{9}{4}.
$$
If $C_y$ has a non-tacnodal singularity at the point
$O_{z}$, then
$$
2=\lct\left(X, \frac{1}{2}C_{x}\right)>\lct\left(X,
\frac{1}{3}C_{y}\right)=\frac{11}{6}.
$$

Let $\epsilon=\min\left\{\lct\left(X, \frac{1}{2}C_{x}\right),  \lct\left(X,
\frac{1}{3}C_{y}\right)\right\}$.
Then $\lct(X)\leqslant\epsilon$.

Suppose that
$\lct(X)<\epsilon$. Then there is an effective $\Q$-divisor
$D\qlineq -K_X$ such that the log pair $(X,\epsilon D)$ is not log
canonical at some point $P\in X$. By
Lemma~\ref{lemma:convexity}, we may assume that the support
of the divisor $D$  contains neither the curve $C_{x}$ nor the curve
$C_{y}$.

The inequalities
\[\mult_{O_z}(\epsilon D)\leqslant
\frac{\epsilon}{2}\mult_{O_z}(D)\mult_{O_z}(C_y)\leqslant 5D\cdot
C_y=1\] imply that the point $P$ cannot be the point $O_z$. If the
point $P$ is a smooth point on $C_y$, then we have obtain a
contradictory inequalities
$$
\frac{1}{5}=D\cdot C_{y}\geqslant
\mult_{P}\left(D\right)>\frac{1}{\epsilon}\geqslant
\frac{1}{2}.
$$
Therefore, the point $P$ is located in the outside of the curve
$C_y$.

Suppose that $P\in C_{x}$. Then we obtain the following
contradictory inequalities
$$
\frac{2}{15}=D\cdot C_{x}\geqslant \left\{%
\aligned
&\mult_{P}\big(D\big)>\frac{1}{2}\ \text{ if}\ P\in X\setminus\mathrm{Sing}(X),\\%
&\frac{\mult_{P}\big(D\big)}{3}>\frac{1}{6}\ \text{ if}\ P=O_{1}\ \text{or}\ P=O_{2}.%
\endaligned\right.
$$
Therefore, $P\not\in C_{x}\cup C_{y}$. Then $P$ is a smooth
point. There is a unique curve $C$ in the pencil $|-5K_{X}|$
passing through the point $P$. The curve $C$ is a hypersurface in
$\mathbb{P}(1,2,3)$ of degree $6$ such that the natural projection
$$
C\longrightarrow\mathbb{P}(1,2)\cong\mathbb{P}^{1}
$$
is a double cover. Thus, we have $\mult_{P}(C)\leqslant
2$. In particular, the log pair $(X, \frac{\epsilon}{5}C)$ is log
canonical. Thus, it follows from Lemma~\ref{lemma:convexity}
that we may assume that the support of the divisor $D$ does not
contain one of the irreducible components of the curve $C$. Then
$$
\frac{1}{3}=D\cdot C\geqslant
\mult_{P}\big(D\big)>\frac{1}{2}
$$
in the case when $C$ is irreducible (but possibly non-reduced).
Therefore, the curve $C$ must be reducible and reduced. Then
$$
C=C_{1}+C_{2},
$$
where $C_{1}$ and $C_{2}$ are irreducible and reduced smooth
rational curves such that
$$
C_{1}^2=C_{2}^2=-\frac{2}{3}, \ \ C_{1}\cdot C_{2}=\frac{3}{2}.
$$
 Without loss of generality we may
assume that $P\in C_{1}$. Put
$$
D=mC_{1}+\Omega,
$$
where $\Omega$ is an effective $\mathbb{Q}$-divisor such that
$C_{1}\not\subset\mathrm{Supp}(\Omega)$. If $m\ne 0$, then
$C_{2}\not\subset\mathrm{Supp}(\Omega)$ and
$$
\frac{1}{6}=D\cdot C_{2}=\big(mC_{1}+\Omega\big)\cdot C_{2}\geqslant  mC_{1}\cdot C_{2}=\frac{3m}{2},%
$$
and hence $m\leqslant \frac{1}{9}$. Thus, it follows from
Lemma~\ref{lemma:handy-adjunction} that
$$
\frac{1+4m}{6}=\big(D-mC_{1}\big)\cdot C_{1}=\Omega\cdot
C_{1}>\frac{1}{\epsilon}\geqslant \frac{1}{2}.
$$
Therefore, $m>\frac{1}{2}$. But $m\leqslant \frac{1}{9}$.
Consequently, $\lct(X)=\epsilon$.

\end{proof}

\begin{lemma}
\label{lemma:3355} Let $X$ be a quasismooth hypersurface of degree
$15$ in $\mathbb{P}(3,3,5,5)$.  Then $\lct(X)=2$.
\end{lemma}

\begin{proof}
The surface $X$ has five singular points $O_{1},\ldots,O_{5}$ of
type $\frac{1}{3}(1,1)$. They are cut out  by the equations
$z=t=0$. The surface also has three singular points
$Q_{1},Q_{2},Q_{3}$ of type $\frac{1}{5}(1,1)$. These three points
are cut out by the equations $x=y=0$.


Let $C_{i}$ be the curve in the pencil $|-3K_{X}|$ passing through
the point $O_i$, where $i=1,\ldots,5$. The curve $C_i$ consists of
three reduced and irreducible smooth rational curves
$$
C_{i}=L^{i}_{1}+L^{i}_{2}+L^{i}_{3}.
$$
The curve $L^i_j$ contains the point $Q_j$. Furthermore,
$L^{i}_{1}\cap L^{i}_{2}\cap L^{i}_{3}=\{O_{i}\}$. We see that
$$
-K_{X}\cdot L^{i}_{j}=\frac{1}{15}, \ \ \left(L^{i}_j\right)^2 =-\frac{7}{15},
\ \ L^{i}_{j}\cdot L^{i}_{k}=\frac{1}{3}
$$
where $j\ne k$.

Note that $\lct(X, C_i)=\frac{2}{3}$. Thus $\lct(X)\leqslant 2$.

Suppose that $\lct(X)<2$. Then there is an effective $\Q$-divisor
$D\qlineq -K_X$ such that the log pair $(X,2D)$ is not log
canonical at some point $P\in X$. Then,
$\mult_{P}(D)>\frac{1}{2}$.

Suppose that $P\not\in C_{1}\cup C_{2}\cup C_{3}\cup C_{4}\cup
C_{5}$. Then $P$ is a smooth point of $X$. There is a unique curve
$C\in |-3K_{X}|$ passing through point $P$. Then $C$ is different
from the curves $C_{1},\ldots,C_{5}$ and hence $C$ is irreducible.
Furthermore, the log pair $(X,C)$ is log canonical. Thus, it
follows from Lemma~\ref{lemma:convexity} that we may assume that
$C\not\subset\mathrm{Supp}(D)$. Then we obtain an absurd
inequality
$$
\frac{1}{5}=D\cdot C\geqslant
\mult_{P}\big(D\big)>\frac{1}{2},
$$
since the log pair $(X,2D)$ is not log canonical at the point $P$.
Therefore, $P\in C_{1}\cup C_{2}\cup C_{3}\cup C_{4}\cup C_{5}$.
However, we may assume that $P\in C_{1}$ without loss of
generality. Furthermore, by Lemma~\ref{lemma:convexity}, we may
assume that $L^{1}_{i}\not\subset\mathrm{Supp}(D)$ for some
$i=1,2,3$.

Since
$$
\frac{1}{5}=3D\cdot L^{1}_{i}\geqslant
\mult_{O_{1}}\big(D\big),
$$
the point $P$ cannot be the point $O_1$.

Without loss of generality, we may assume that $P\in L^{1}_{1}$.

Let $Z$ be the curve in the pencil $|-5K_{X}|$ passing through the
point $Q_{1}$. Then
$$
Z=Z_{1}+Z_{2}+Z_{3}+Z_{4}+Z_{5},
$$
where $Z_{i}$ is a reduced and  irreducible smooth rational curve.
The curve $Z_i$ contains the point $O_i$. Moreover, $Z_{1}\cap
Z_{2}\cap Z_{3}\cap Z_{4}\cap Z_{5}=\{Q_{1}\}$. It is easy to
check $\lct(X,Z)=\frac{2}{5}$. By Lemma~\ref{lemma:convexity}, we
may assume that $Z_{k}\not\subset\mathrm{Supp}(D)$ for some
$k=1,\ldots,5$. Then
$$
\frac{1}{3}=5D\cdot Z_{k}\geqslant
\mult_{Q_{1}}\big(D\big),
$$
and hence the point $P$ cannot be the point $Q_1$.

Thus, the point $P$ is a smooth point on $L^1_1$. Put
$$
D=mL^{1}_{1}+\Omega,
$$
where $\Omega$ is an effective
$\mathbb{Q}$-divisor such that
$L^{1}_{1}\not\subset\mathrm{Supp}(\Omega)$. If $m\ne 0$, then
$$
\frac{1}{15}=D\cdot L^{1}_{i}=\big(mL^{1}_{1}+\Omega\big)\cdot L^{1}_{i}\geqslant  mL^{1}_{1}\cdot L^{1}_{i}=\frac{m}{3},%
$$
and hence $m\leqslant \frac{1}{5}$. Then it follows from
Lemma~\ref{lemma:handy-adjunction} that
$$
\frac{1+7m}{15}=\big(D-mL^{1}_{1}\big)\cdot
L^{1}_{1}=\Omega\cdot L^{1}_{1}>\frac{1}{2}.%
$$
This implies that $m>\frac{13}{14}$. But $m\leqslant \frac{1}{5}$.
The obtained contradiction completes the proof.
\end{proof}

\begin{lemma}
\label{lemma:3571125}  Let $X$ be a quasismooth hypersurface of
degree $25$ in $\mathbb{P}(3,5,7, 11)$.  Then
$\lct(X)=\nlb\frac{21}{10}$.
\end{lemma}
\begin{proof}
The curve $C_x$ is irreducible and reduced. It is easy to see that
$\lct(X, \frac{1}{3}C_x)=\frac{21}{10}$. Therefore,
$\lct(X)\leqslant \frac{21}{10}$.

Suppose that $\lct(X)<\frac{21}{10}$. Then there is an effective
$\Q$-divisor $D\qlineq -K_X$ such that the log pair
$(X,\frac{21}{10}D)$ is not log canonical at some point $P\in X$.
We may assume that the support of $D$ does not contain the curve
$C_x$ by Lemma~\ref{lemma:convexity}.

Since $H^0(\P, \mathcal{O}_\P(21))$ contains $x^{7}, x^2y^3, z^3$,
by Lemma~\ref{lemma:Carolina} we have
$$
\mult_{P}(D)\leqslant \frac{21\cdot 25}{3\cdot 5\cdot 7\cdot 11}<\frac{10}{21}%
$$
if $P$ is a smooth point in the outside of the curve $C_x$. Thus,
either $P=O_x$ or $P\in C_x$.

If $P\in C_x$, then we obtain a contradictory inequalities
$$
\frac{5}{77}=D\cdot C_{x}\geqslant \left\{%
\aligned
&\mult_{P}(D)\mult_{P}(C_{x})=\mult_{P}(D)>\frac{10}{21} \ \text{ if}\ P\in X\setminus\mathrm{Sing}(X),\\%
&\frac{\mult_{P}(D)\mult_{P}(C_{x})}{7}=\frac{\mult_{P}(D)}{7}>\frac{10}{147}\ \text{ if}\ P=O_{z},\\%
&\frac{\mult_{P}(D)\mult_{P}(C_{x})}{11}=\frac{2\mult_{P}(D)}{11}>\frac{20}{231}\ \text{ if}\ P=O_{t}.\\%
\endaligned\right.
$$
Therefore, we see that $P=O_{x}$.

Since the curve $C_y$ is irreducible and the log pair $(X,
\frac{1}{5}C_y)$ is log canonical at the point $O_x$, we may
assume that the support of $D$ does not contain the curve $C_y$.
Then
$$
\frac{10}{63}<\frac{\mult_{O_x}(D)}{3}\leqslant D\cdot C_y=\frac{25}{231}<\frac{10}{63}.%
$$
This is a contradiction.
\end{proof}

\begin{lemma}
\label{lemma:35714} Let $X$ be a quasismooth hypersurface of
degree $28$ in $\mathbb{P}(3,5,7,14)$. Then
$\lct(X)=\nlb\frac{9}{4}$.
\end{lemma}

\begin{proof}
The surface $X$ is singular at the point $O_{x}$ and the point
$O_y$. The former is a singular point of type $\frac{1}{3}(1,1)$
and the latter is  of type $\frac{1}{5}(1,2)$.  Let $O_{1}$ and
$O_{2}$ be the two points cut out on $X$ by the equations $x=y=0$.
The points $O_{1}$ and $O_{2}$ are singular points of type
$\frac{1}{7}(3,5)$ on the surface $X$.

The curve $C_x$ consists of two reduced and irreducible smooth
rational curves $L_1$ and $L_2$. These two curves intersect each
other only at the point $O_y$. Each curve $L_i$ contains the
singular point $O_i$. We have
$$
-K_{X}\cdot L_{i}=\frac{1}{35},\ \  L_{1}\cdot L_{2}=\frac{2}{5},
\ \ L_{1}^2=L_{2}^2=-\frac{11}{35}.
$$
Since $\lct(X, C_x)=\frac{3}{4}$, $\lct(X)\leqslant \frac{9}{4}$.

Suppose that $\lct(X)<\frac{9}{4}$. Then there is an effective
$\Q$-divisor $D\qlineq -K_X$ such that the log pair
$(X,\frac{9}{4}D)$ is not log canonical at some point $P\in X$.

If $P$ is a smooth point in the outside of $C_x$, then
$$
\mult_{P}\big(D\big)\leqslant\frac{588}{1470}<\frac{4}{9}
$$
by Lemma~\ref{lemma:Carolina} since $H^0(\P,
\mathcal{O}_\P(21))$ contains $x^{7}, z^3, x^2y^3$. Therefore,
either $P$ belongs to the curve $C_x$ or $P=O_x$.

By Lemma~\ref{lemma:convexity}, we may assume that
$L_{i}\not\subset\mathrm{Supp}(D)$ for some $i=1,2$. Similarly, we
may assume that $C_{y}\not\subset\mathrm{Supp}(D)$ since $(X,
C_{y})$ is log canonical and the curve $C_{y}$ is
irreducible.

The inequalities
$$\mult_{O_x}(D)\leqslant 3D\cdot C_y=\frac{2}{7}<\frac{4}{9}$$
show that the point $P$ cannot be the point $O_x$. Therefore, the
point $P$ belongs to the curve $C_x$.

The inequalities
$$\mult_{O_y}(D)\leqslant 5D\cdot L_i=\frac{1}{7}<\frac{4}{9}$$
show that the point $P$ cannot be the point $O_y$.

Without loss of generality, we may assume that $P\in L_{1}$. Put
$D=mL_{1}+\Omega$, where $\Omega$ is an effective
$\mathbb{Q}$-divisor such that
$L_{1}\not\subset\mathrm{Supp}(\Omega)$. If $m\ne 0$, then
$$
\frac{1}{35}=D\cdot L_{2}=\big(mL_{1}+\Omega\big)\cdot L_{2}\geqslant  mL_{1}\cdot L_{2}=\frac{2m}{5},%
$$
and hence  $m\leqslant \frac{1}{14}$. Then
Lemma~\ref{lemma:handy-adjunction} implies an absurd inequality
$$
\frac{5}{98}\geqslant\frac{1+11m}{35}=\big(D-mL_{1}\big)\cdot L_{1}=\Omega\cdot
L_{1}>\left\{%
\aligned
&\frac{4}{9}\ \text{if}\ P\ne O_{1},\\%
&\frac{4}{63}\ \text{if}\ P=O_{1}.\\%
\endaligned\right.%
$$
The obtained contradiction completes the proof.
\end{proof}

\begin{lemma}
\label{lemma:35111836}Let $X$ be a quasismooth hypersurface of
degree $36$ in $\mathbb{P}(3,5,11, 18)$. Then
 $\lct(X)=\nlb\frac{21}{10}$.
\end{lemma}
\begin{proof}
The surface $X$ is singular at the points $O_y$ and $O_z$. It is
also singular at two points $P_1$ and $P_2$ on the curve $L_{yz}$.
These two points $P_1$ and $P_2$ are contained in $C_y$.

The curve $C_x$ is irreducible and reduced. It is easy to see that
$\lct(X, \frac{1}{3}C_x)=\frac{21}{10}$. Also, the curve $C_y$ is
always irreducible and the pair $(X, \frac{21}{5\cdot 10}C_y)$ is
log canonical. We see that $\lct(X)\leqslant \frac{21}{10}$.

Suppose that $\lct(X)<\frac{21}{10}$. Then there is an effective
$\Q$-divisor $D\qlineq -K_X$ such that the pair
$(X,\frac{21}{10}D)$ is not log canonical at some point $P\in X$.
By Lemma~\ref{lemma:convexity}, we may assume that the support of
$D$ contains neither the curve $C_x$ nor $C_y$.

 Then the following inequalities
$$
\mult_{O_y}(D)\leqslant 5D\cdot C_x=\frac{5\cdot 3\cdot 36}{ 3\cdot 5\cdot 11\cdot 18}
<\frac{10}{21},%
$$
$$
\mult_{O_z}(D)\leqslant 11D\cdot C_x=\frac{11\cdot 3\cdot 36}{ 3\cdot 5\cdot 11\cdot 18}
<\frac{10}{21},%
$$
$$
\mult_{P_i}(D)\leqslant 3D\cdot C_y=\frac{3\cdot 5\cdot 36}{3\cdot 5\cdot 11\cdot 18}
<\frac{10}{21},%
$$
show that the point $P$ is a smooth point $P$ on $X$. Furthermore,
the first two inequalities also show that the point $P$ cannot
belong to the curve $C_x$. Therefore, the point $P$ is a smooth
point in the outside of the curve $C_x$.

 However, since  $H^0(\P, \mathcal{O}_\P(39))$
contains $x^{13}, x^3y^6, x^2z^3$, by
Lemma~\ref{lemma:Carolina} we have
$$
\frac{10}{21}<\mult_{P}(D)\leqslant \frac{36\cdot 39}{3\cdot 5\cdot 11\cdot 18}
<\frac{10}{21}.%
$$
  The obtained contradiction completes the
proof.
\end{proof}

\begin{lemma}
\label{lemma:5141721} Let $X$ be a quasismooth hypersurface of
degree $56$ in $\mathbb{P}(5,14,17, 21)$. Then
$\mathrm{lct}(X)=\frac{25}{8}$.
\end{lemma}

\begin{proof}
The surface $X$ is singular at the points $O_{x}$, $O_z$ and
$O_t$. The first point is a singular point of type
$\frac{1}{5}(2,1)$, the second of type $\frac{1}{17}(7,2)$, the
last of type $\frac{1}{21}(5,17)$. There is one more singular
point $O$ of type $\frac{1}{7}(5,3)$ on $L_{xz}$ that is different
from the singular point $O_t$.

The curve $C_x$ (resp. $C_y$) consists of two reduced and
irreducible curves $L_{xy}$ and $R_{x}$ (resp. $R_y$). The curve
$L_{xy}$ intersects the curve $R_x$ at the point $O_z$. The curve
$R_x$ is singular at the point $O_z$. On the other hand, it
intersects the curve $R_y$ at the point $O_t$. The curve $R_y$ is
singular at $O_t$. We have
$$
L_{xy}^2=-\frac{37}{357},\ \ L_{xy}\cdot R_{x}=\frac{2}{17},\ \ R_x^2=-\frac{9}{119},
\ \ L_{xy}\cdot R_{y}=\frac{1}{7},\  \ R_y^2=\frac{9}{35}.%
$$

It is easy to check $\lct(X,C_{x})=\frac{5}{8}$ and
$\lct(X,C_{y})=\frac{3}{7}$, and hence $\mathrm{lct}(X)\leqslant
\frac{25}{8}$.

Suppose that $\lct(X)<\frac{25}{8}$. Then there is an effective
$\Q$-divisor $D\qlineq -K_X$ such that the log pair
$(X,\frac{25}{8} D)$ is not log canonical at some point $P\in X$.
By Lemma~\ref{lemma:convexity} we may assume that either the
support of the divisor $D$ does not contain the curve $L_{xy}$ or
it contains neither  $R_{x}$ nor $R_{y}$.

Suppose that $P\not\in C_{x}\cup C_{y}$. Then $P$ is a smooth
point and
$$
\mult_{P}\big(D\big)\leqslant\frac{4}{21}<\frac{8}{25}
$$
by Lemma~\ref{lemma:Carolina} since  the natural
projection $X\dasharrow\mathbb{P}(5,14,17)$ is a finite morphism
outside of the curve $C_{y}$ and $H^0(\P, \mathcal{O}_\P(85))$
contains monomials $x^{17}, z^5, x^3y^5$. This is a contradiction.
Thus, the point $P$ must belong to $C_{x}\cup C_{y}$.

The curve $C_{z}$ is irreducible and the log pair $(X,
\frac{25}{8\cdot 17}C_{z})$ is log canonical. By
Lemma~\ref{lemma:convexity} we may assume that
$C_{z}\not\subset\mathrm{Supp}(D)$. Then
$$
\frac{8}{25}>\frac{4}{21}=5D\cdot C_{z}\geqslant
\mult_{O_{x}}\big(D\big),
$$
and hence the point $P$ cannot be $O_x$.

Suppose that $P\in L_{xy}$. Put $D=mL_{xy}+\Omega$, where $\Omega$
is an effective $\mathbb{Q}$-divisor such that
$L_{xy}\not\subset\mathrm{Supp}(\Omega)$. If $m\ne 0$, then
$$
\frac{1}{119}=D\cdot R_{x}=\big(mL_{xy}+\Omega\big)\cdot R_{x}\geqslant
mL_{xy}\cdot R_{x}=\frac{2m}{17},%
$$
and hence  $m\leqslant \frac{1}{14}$. Then it follows from
Lemma~\ref{lemma:handy-adjunction} that
$$
\frac{1+37m}{357}=\big(D-mL_{xy}\big)\cdot L_{xy}=\Omega\cdot
L_{xy}>\left\{%
\aligned
&\frac{8}{525}\ \text{if}\ P=O_{t},\\%
&\frac{8}{425}\ \text{if}\ P=O_{z},\\%
&\frac{8}{25}\ \text{if}\ P\ne O_{z}\ \text{and}\ P\ne O_{t}.\\%
\endaligned\right.%
$$
This implies $m>\frac{3}{25}$. But $m\leqslant \frac{1}{14}$. The
obtained contradiction implies that $P\not\in L_{xy}$.

Suppose that $P\in R_{x}$. Put $D=aR_{x}+\Upsilon$, where
$\Upsilon$ is an effective $\mathbb{Q}$-divisor such that
$R_{x}\not\subset\mathrm{Supp}(\Upsilon)$. If $a\ne 0$, then
$$
\frac{1}{357}=D\cdot L_{xy}=\big(aR_{x}+\Upsilon\big)\cdot L_{xy}\geqslant
 aL_{xy}\cdot R_{x}=\frac{2a}{17},%
$$
and hence $a\leqslant \frac{1}{42}$. Then it follows from
Lemma~\ref{lemma:handy-adjunction} that
$$
\frac{1+9a}{119}=\big(D-aR_x\big)\cdot R_{x}=\Upsilon\cdot
R_{x}>\left\{%
\aligned
&\frac{8}{175}\ \text{if}\ P=O,\\%
&\frac{8}{25}\ \text{if}\ P\ne O.\\%
\endaligned\right.%
$$
This is impossible because $a\leqslant \frac{1}{42}$. Thus, we see
that $P\not\in C_{x}$.

We see that $P\in R_{y}$ and $P\in X\setminus\mathrm{Sing}(X)$.
Put $D=bR_{y}+\Delta$, where $\Delta$ is an effective
$\mathbb{Q}$-divisor such that
$R_{y}\not\subset\mathrm{Supp}(\Delta)$. If $b\ne 0$, then
$$
\frac{1}{357}=D\cdot L_{xy}=\big(bR_{y}+\Delta\big)\cdot L_{xy}\geqslant  bL_{xy}\cdot
R_{y}=\frac{b}{7},%
$$
and hence $b\leqslant \frac{1}{51}$. Then it follows from
Lemma~\ref{lemma:handy-adjunction} that
$$
\frac{1+9b}{35}=\big(D-bR_{y}\big)\cdot R_{y}=\Delta\cdot
R_{y}>\frac{8}{25}.%
$$
This is impossible  because $b\leqslant \frac{1}{51}$. The
obtained contradiction completes the proof.
\end{proof}

\begin{lemma}
\label{lemma:519273181} Let $X$ be a quasismooth hypersurface of
degree $81$ in $\mathbb{P}(5,19,27, 31)$. Then
$\lct(X)=\frac{25}{6}$.
\end{lemma}

\begin{proof}
The curve $C_x$ is irreducible and reduced. Moreover, the curve
$C_x$ is smooth outside of the singular locus of the surface $X$.
It is easy to see that $\lct(X, \frac{1}{5}C_x)=\frac{25}{6}$.
Hence, we have $\lct(X)\leqslant \frac{25}{6}$. The curve $C_y$ is
irreducible  and reduced.  The log pair $(X, \frac{1}{19}C_y)$ is
log canonical.

Suppose that $\lct(X)<\frac{25}{6}$. Then there is an effective
$\Q$-divisor $D\qlineq -K_X$ such that the pair
$(X,\frac{25}{6}D)$ is not log canonical at some point $P\in X$.
We may assume that the support of $D$ contains neither $C_x$ nor
$C_y$ by Lemma~\ref{lemma:convexity}.

The inequality
\[31D\cdot C_x=\frac{3}{19}<\frac{6}{25}\]
shows that the point $P$ cannot be on the curve $C_x$. On the
other hand, the inequality
\[5D\cdot C_y=\frac{3}{31}<\frac{6}{25}\] shows that the point $P$
 cannot be on the curve $C_y$. In particular, the point $P$ cannot
 be the point $O_x$.

 Therefore, the point $P$ must be a smooth point in the outside of
 $C_x$. However, Lemma~\ref{lemma:Carolina} implies
$$
\mult_{P}(D)\leqslant \frac{190\cdot 81}{5\cdot 19\cdot 27\cdot 31}<\frac{6}{25}%
$$
since  $H^0(\P, \mathcal{O}_\P(190))$ contains $x^{38}, x^{11}z,
y^{10}$. This is a contradiction.
\end{proof}

\begin{lemma}
\label{lemma:5192750100} Let $X$ be a quasismooth hypersurface of
degree $100$ in $\mathbb{P}(5,19,27, 50)$. Then
$\lct(X)=\frac{25}{6}$.
\end{lemma}

\begin{proof}
The surface $X$ is singular at the points $O_y$ and $O_z$. Also,
it is singular at two points $P_1$ and $P_2$ on $L_{yz}$. The
point $O_y$ is a singular point of type $\frac{1}{19}(2,3)$ on
$X$. The point $O_z$ is of type $\frac{1}{27}(5,23)$. The last two
points are of type $\frac{1}{5}(2,1)$.

The curve $C_x$ is irreducible and reduced. It is easy to see that
$\lct(X, \frac{1}{5}C_x)=\frac{25}{6}$. Therefore,
$\lct(X)\leqslant \frac{25}{6}$. The curve $C_z$ is irreducible
and reduced.  The log pair $(X, \frac{25}{6\cdot 27}C_z)$ is log
canonical.

Suppose that $\lct(X)<\frac{25}{6}$. Then it follows from
Lemma~\ref{lemma:convexity} that there is an effective
$\Q$-divisor $D\qlineq -K_X$ such that $C_{x},
C_z\not\subset\mathrm{Supp}(D)$ and the pair $(X,\frac{25}{6}D)$
is not log canonical at some point $P\in X$.

The inequality
\[27D\cdot C_x=\frac{2}{19}<\frac{6}{25}\]
shows that the point $P$ cannot be on the curve $C_x$. On the
other hand, the inequality
\[5D\cdot C_z=\frac{2}{19}<\frac{6}{25}\] shows that the point $P$
 cannot be on the curve $C_z$. In particular, the point $P$ can be
 neither the point $P_1$ nor the point $P_2$.

Consequently, the point $P$ must be a smooth point in the outside
of $C_x$. However,  $H^0(\P, \mathcal{O}_\P(270))$ contains
$x^{54}, x^{16}y^{10}, z^{10}$. Then,
Lemma~\ref{lemma:Carolina}  implies a contradictory
inequality
$$
\frac{6}{25}<
\mult_{P}(D)\leqslant \frac{270\cdot 100}{5\cdot 19\cdot 27\cdot 50}<\frac{6}{25}.%
$$
\end{proof}

\begin{lemma}
\label{lemma:711273781} Let $X$ be a quasismooth hypersurface of
degree $81$ in $\mathbb{P}(7,11,27, 37)$. Then
$\lct(X)=\frac{49}{12}$.
\end{lemma}

\begin{proof}
The surface $X$ is singular only at the points $O_{x}$, $O_{y}$
and $O_{t}$.

The curve $C_x$ is irreducible and reduced. It is easy to see that
$\lct(X, \frac{1}{7}C_x)=\frac{49}{12}$, and hence
$\lct(X)\leqslant \frac{49}{12}$. The curve $C_y$ is irreducible
and reduced.  Moreover, the log pair $(X, \frac{49}{11\cdot
12}C_y)$ is log canonical.

Suppose that $\lct(X)<\frac{49}{12}$. By
Lemma~\ref{lemma:convexity}, there is an effective $\Q$-divisor
$D\qlineq -K_X$ such that the support of $D$ contains neither  the
curve $C_x$ nor the curve $C_y$, and the log pair
$(X,\frac{49}{12}D)$ is not log canonical at some point $P\in X$.

The three inequalities
\[11D\cdot C_x=\frac{3}{37}<\frac{12}{49},\]
\[7D\cdot C_y=\frac{3}{37}<\frac{12}{49},\]
\[\mult_{O_t}(D)=\frac{\mult_{O_t}(D)\mult_{O_t}(C_{x})}{3}\leqslant
\frac{37}{3}D\cdot C_x=\frac{1}{ 11}<\frac{12}{49}\] show that the
point $P$ is a smooth point in the outside of $C_x$.

However, since $H^0(\P, \mathcal{O}_\P(189))$ contains $x^{27},
x^{16}y^7, z^{7}$, Lemma~\ref{lemma:Carolina} implies
an absurd inequalities
$$
\frac{12}{49}<\mult_{P}(D)\leqslant \frac{189\cdot 81}{7\cdot 11\cdot 27\cdot 37}
<\frac{12}{49}.%
$$
Therefore, $\lct(X)=\frac{49}{12}$.
\end{proof}

\begin{lemma}
\label{lemma:7112744}
Let $X$ be a quasismooth hypersurface of
degree $88$ in $\mathbb{P}(7,11,27, 44)$. Then
$\lct(X)=\frac{35}{8}$.
\end{lemma}

\begin{proof}
The surface $X$ is singular at the points $O_{x}$ and $O_z$. The former is a
singular point of type $\frac{1}{7}(3,1)$ and the latter is of type $\frac{1}{27}(11,17)$. The surface is also singular  at the
points $O_{1}$ and $O_{2}$ on $L_{xz}$. They are of type
$\frac{1}{11}(7,5)$.

The curve $C_{x}$ consists of two smooth rational curves $L_1$ and $L_2$. Each curve $L_i$ contains the singular point $O_i$. The curves $L_1$ and $L_2$ intersects each other only at the point $O_z$.  We have
$$
L_{1}^2=L_{2}^2=-\frac{37}{297},\ \  L_{1}\cdot L_{2}=\frac{4}{27}.%
$$
It is easy to check $\lct(X,\frac{1}{7}C_{x})=\frac{35}{8}$. Meanwhile, the curve $C_y$ is irreducible and reduced. Also, the log pair $(X, \frac{35}{88}C_y)$ is log canonical.

Suppose that $\lct(X)<\frac{35}{8}$. Then there is an effective
$\Q$-divisor $D\qlineq -K_X$ such that the log pair
$(X,\frac{35}{8} D)$ is not log canonical at some point $P\in X$.
By Lemma~\ref{lemma:convexity}  we may assume that the support of
$D$ contains neither $C_y$ nor $L_{2}$ without loss of generality.

The inequality
\[27D\cdot L_2=\frac{1}{11}<\frac{8}{35}\]
shows that the point $P$ is located in the outside of $L_2$. The inequality
\[ 7D\cdot C_y=\frac{2}{27}<\frac{8}{35}\]
implies that the point $P$ cannot be $O_x$.
Write
$$
D=mL_{1}+\Omega,
$$
where $\Omega$ is an effective $\mathbb{Q}$-divisor such that
$L_{1}\not\subset\mathrm{Supp}(\Omega)$. If $m\ne 0$, then
$$
\frac{1}{297}=D\cdot L_{2}=\big(mL_{1}+\Omega\big)\cdot L_{2}\geqslant  mL_{1}\cdot L_{2}=\frac{4m}{27},%
$$
and hence $m\leqslant \frac{1}{44}$. Then
$$
(D-mL_1)\cdot L_1=\frac{1+37m}{297}\leqslant \frac{3}{484}<\frac{8}{35\cdot 11},
$$
and hence Lemma~\ref{lemma:handy-adjunction} implies that the point $P$ cannot be on the curve $L_1$.
Therefore, the point $P$ is a smooth point in the outside of $C_x$.
However, Lemma~\ref{lemma:Carolina} shows
$$
\mult_{P}\big(D\big)\leqslant\frac{2}{11}<\frac{8}{35}%
$$
  since  $H^0(\P,
\mathcal{O}_\P(189))$ contains monomials $x^{27}, z^7, x^{16}y^7$.
This is a contradiction.
\end{proof}

\begin{lemma}\label{lemma:915172060}
Let $X$ be a quasismooth hypersurface of
degree $60$ in $\mathbb{P}(9,15,17, 20)$. Then
$\lct(X)=\frac{21}{4}$.
\end{lemma}

\begin{proof}
We may assume that the surface $X$ is defined by the quasihomogeneous equation
$$xz^3+x^5y-y^4+t^3=0.$$
Note that the surface $X$ is singular at $O_x$ and $O_z$. It is also singular at the point $P_1=[1:1:0:0]$ and  the point $P_2=[0:1:0:1]$.

The curves $C_x$, $C_y$, and $C_z$  are  irreducible and reduced. We have
$$
\lct(X, \frac{1}{9}C_x)=\frac{21}{4},\ \ \lct(X,
\frac{1}{15}C_y)=10,\ \ \lct(X,
\frac{1}{17}C_z)=17.
$$
The curve $C_x$ is singular at the point $O_z$ with multiplicity $3$.

Suppose that $\lct(X)<\frac{21}{4}$. Then there is an effective
$\Q$-divisor $D\qlineq -K_X$ such that the pair $(X,\frac{21}{4}D)$ is
not log canonical at some point $P$. By
Lemma~\ref{lemma:convexity}, we may assume that the support of
$D$ contains none of the curves $C_x$, $C_y$, $C_z$.

The three inequalities
$$
\frac{17}{3}D\cdot C_x=\frac{17\cdot 9\cdot 60}{3\cdot9\cdot 15\cdot 17\cdot 20}<\frac{4}{21},%
$$
$$
9D\cdot C_y=\frac{9\cdot 15\cdot 60}{9\cdot 15\cdot 17\cdot 20 }<\frac{4}{21},%
$$
$$
3D\cdot C_z=\frac{3\cdot 17\cdot 60}{9\cdot 15\cdot 17\cdot 20}<\frac{4}{21}%
$$
imply that the point $P$ is located in the outside of $C_x\cup C_y\cup C_z$.

Let $\mathcal{L}$ be the pencil on $X$ that is cut out by the equations
$$
\lambda z^3+\mu x^4y=0,
$$
where $[\lambda :\mu]\in\P^1$. Then the base locus of the pencil
$\mathcal{L}$ consists of the  points $P_{2}$ and $O_{x}$.
Let $C$ be the unique curve in $\mathcal{L}$ that passes through
the point $P$. Then $C$ is cut out on $X$ by an equation
$$
x^4y=\alpha z^3,
$$
where $\alpha$ is a non-zero constant, since the point $P$
is located in the outside of $C_x\cup C_y\cup C_z$. The curve $C$ is smooth
outside of the points $P_{2}$ and $O_{x}$ by the Bertini theorem
because $C$ is isomorphic to a general curve in the pencil
$\mathcal{L}$ unless $\alpha=-1$. In the case when $\alpha=-1$,
the curve $C$ is smooth outside the points $P_{2}$ and $O_{x}$ as
well.

We claim that the curve $C$ is irreducible. If so, then we may
assume that the support of $D$ does not contain the curve $C$ and
hence we obtain a contradiction
$$
\frac{4}{21}<\mult_{Q}(D)\leqslant D\cdot C=\frac{ 51\cdot
60}{9\cdot 15\cdot 17\cdot 20}<\frac{4}{21}.
$$

For the irreducibility of the curve $C$, we may consider the curve
$C$ as a surface in $\mathbb{C}^4$ defined by the equations
$t^3+y^4+(1+\alpha)xz^3=0$ and  $x^4y=\alpha z^3$. This
surface is isomorphic to the surface in $\mathbb{C}^4$ defined by
the equations $t^3+y^4+\beta xz^3=0$ and  $x^4y=z^3$, where
$\beta=1$ or $0$. Then, we consider the surface in $\P^4$ defined
by the equations $t^3w+y^4+\beta xz^3=0$ and $x^4y=z^3w^2$. We
take the affine piece defined by $t\ne 1$. This affine
piece is isomorphic to the surface defined by the equation
$x^4y+z^3(y^4+\beta xz^3)^2=0$ in $\mathbb{C}^3$. If $\beta=1$,
the surface is irreducible. If $\beta=0$, then it has an extra
component defined by $y=0$. However, this component originates
from the hyperplane $w=0$ in $\P^4$. Therefore, the surface  in
$\mathbb{C}^4$ defined by the equations $t^3+y^4=0$ and $x^4y=z^3$
is also irreducible.
\end{proof}

\begin{lemma}
\label{lemma:915232369} Let $X$ be a quasismooth hypersurface of
degree $69$ in $\mathbb{P}(9,15,23, 23)$. Then
 $\lct(X)=6$.
\end{lemma}

\begin{proof}
We may assume that the surface $X$ is defined by the
quasihomogeneous equation
$$
zt(z-t)+xy(y^3-x^5)=0.
$$
The surface $X$ is singular at three distinct points $O_x$,
$O_y$, $P_1=[1:1:0:0]$. Also, it is singular at three distinct points $O_z$, $O_t$, $Q_1=[0:0:1:1]$.

The curve $C_x$ consists of three distinct curves $L_{xz}$,
$L_{xt}$ and $R_x=\{x=z-t=0\}$ that intersect altogether at the
point $O_y$. Similarly, the curve $C_y$ consists of three curves
$L_{yz}$, $L_{yt}$ and $R_y=\{y=z-t=0\}$ that intersect altogether
at the point $O_x$. The curve $C_z$ consists of $L_{xz}$,
$L_{yz}$, and $R_z=\{z=y^3-x^5=0\}$. The curve $R_z$ is singular
at the point $O_t$ with multiplicity $3$.  The curve $C_t$
consists of $L_{xt}$, $L_{yt}$ and $R_t=\{t=y^3-x^5=0\}$. The
curve $R_t$ is singular at the point $O_z$ with multiplicity $3$.

Note that $\lct(X, \frac{1}{9}C_x)=6$. The log pairs $(X,
\frac{6}{15}C_y)$,  $(X, \frac{6}{23}C_z)$ and  $(X,
\frac{6}{23}C_t)$ are log canonical.

Suppose that $\lct(X)<6$. Then there is an effective $\Q$-divisor
$D\qlineq -K_X$ such that the pair $(X,6D)$ is not log canonical
at some point $P\in X$. Lemma~\ref{lemma:convexity} implies that
we may assume that the support of $D$ contains neither $R_x$ nor
$R_y$ by a linear coordinate change. Furthermore, we may assume
that the support of $D$ does not  contain at least one component
of $C_z$. Also, it may be assumed not to contain at least one
component of $C_t$.

The inequalities
 \[15D\cdot R_x=\frac{15\cdot 23\cdot 9}{9\cdot 15\cdot 23\cdot 23}
 =\frac{1}{23}<\frac{1}{6},\]
\[23D\cdot R_y
=\frac{23\cdot 23\cdot 15}{9\cdot 15\cdot 23\cdot 23
}=\frac{1}{9}<\frac{1}{6}\] show that the point $P$ is located in
the outside of $R_x\cup R_y$.

 Then the inequalities
$$
23D\cdot L_{xz}=\frac{1}{15}<\frac{1}{6},\ \ 23D\cdot
L_{yz}=\frac{1}{9}<\frac{1}{6},\ \ \frac{23}{3}D\cdot
R_z=\frac{1}{9}<\frac{1}{6}
$$
show that $\mult_{O_t}(D)<\frac{1}{6}$, and hence the point $P$
cannot be the point $O_t$. By the same way, we can show that $P\ne
O_z$.

Write $D=mR_z+\Omega$, where $\Omega$ is an effective
$\mathbb{Q}$-divisor such that
$R_z\not\subset\mathrm{Supp}(\Omega)$. Then $m\leqslant
\frac{1}{6}$ since $(X,6D)$ is log canonical at $O_{t}$. We have
$$
R_{z}\cdot(L_{xz}+L_{yz})=\frac{8}{23},\ \ R_z\cdot D=\frac{1}{69},%
$$
and hence $R_z^2=-\frac{1}{69}$. Then
$$
\Omega\cdot R_z=D\cdot R_z-mR_z^2=\frac{1+m}{3\cdot 23}
\leqslant \frac{7}{6\cdot 3\cdot 23}<\frac{1}{3\cdot 6}.%
$$
Lemma~\ref{lemma:handy-adjunction} implies that the point $P$
cannot belong to $R_z$. In particular, the point $P$ cannot be the
point $P_1$.

 Write $D=a L_{xz}+\Delta$, where $\Delta$ is an
effective $\mathbb{Q}$-divisor whose support does not contain the
curve $L_{xz}$. Then $a\leqslant \frac{1}{6}$. Then
$$
\Omega\cdot L_{xz}=6(D\cdot L_{xz}-aL_{xz}^2)=\frac{6\cdot(1+37a)}{345}\leqslant \frac{6+37}{345}=\frac{43}{345}<1,%
$$
because $L_{xz}^2=-\frac{37}{345}$. Thus, we see that $P\not\in
L_{xz}$. Similarly, we can show that $P\not\in L_{yz}$. Thus, we
see that $P\not\in C_z$. In the same way, we can see that $P$ is
not contained in the curves $C_t$ and $\{z-t=0\}$.

Therefore, the point $P$ is a smooth point in the outside of
$C_z\cup C_t\cup\{z-t=0\}$. Let $E$ be the unique curve on $X$
such that $E$ is given by the equation $z=\lambda t$ and $P\in E$,
where $\lambda$ is a non-zero constant different from $1$. Then
$E$ is quasismooth and hence irreducible. Therefore, we may assume
that the support of $D$ does not contain the curve $E$. Then
$$
\mult_{P}(D)\leqslant D\cdot E=\frac{23\cdot 69}{9\cdot 15\cdot 23\cdot 23}<\frac{1}{6}.%
$$
This is a contradiction.
\end{proof}

\begin{lemma}
\label{lemma:11293949127} Let $X$ be a quasismooth hypersurface of
degree $127$ in $\mathbb{P}(11,29,39, 49)$. Then
$\lct(X)=\frac{33}{4}$.
\end{lemma}

\begin{proof}
We may assume that the hypersurface $X$  is defined by the
equation
$$
z^2t+yt^2+xy^4+x^8z=0.
$$
The singularities of $X$ consist of a singular point of type
$\frac{1}{11}(7,5)$ at $O_x$, a singular point of type
$\frac{1}{29}(1, 2)$ at $O_y$, a singular point of type
$\frac{1}{39}(11, 29)$ at $O_z$, and a singular point of type
$\frac{1}{49}(11, 39)$ at $O_t$.

The curve $C_x$ (resp. $C_y$, $C_z$, $C_t$) consists of two
irreducible curves $L_{xt}$ (resp. $L_{yz}$, $L_{yz}$, $L_{xt}$)
and $R_x=\{x=z^2+yt=0\}$ (resp. $R_y=\{y=x^8+zt=0\}$,
$R_z=\{z=t^2+xy^3=0\}$, $R_t=\{t=y^4+x^7z=0\}$). We can see that
\[L_{xt}\cap R_x=\{O_y\}, \ L_{yz}\cap R_y=\{O_t\}, \ \
L_{yz}\cap R_z=\{O_x\}, \ L_{xt}\cap R_t=\{O_z\}. \]

It is easy to check $\lct(X, \frac{1}{11} C_x)=\frac{33}{4}$. The
log pairs $(X, \frac{33}{4\cdot 29}C_y)$, $(X, \frac{33}{4\cdot
39}C_z)$ and $(X, \frac{33}{4\cdot 49}C_t)$ are log canonical.

Suppose that $\lct(X)<\frac{33}{4}$. Then there is an effective
$\Q$-divisor $D\qlineq -K_X$ such that the log pair $(X,
\frac{33}{4} D)$ is not log canonical at some point $P\in X$.

By Lemma~\ref{lemma:convexity}, we may assume that the support of
$D$ does not contain $L_{xt}$ or $R_x$. Then one of the following
two inequalities must hold:
$$
\frac{4}{33}>\frac{1}{39}=29L_{xt}\cdot D\geqslant
\mult_{O_{y}}(D),
$$
$$
\frac{4}{33}>\frac{2}{49}=29 R_x\cdot
D\geqslant \mult_{O_{y}}(D).
$$
Therefore, the point $P$ cannot be the point $O_y$. For the same
reason, one of  two inequalities
$$
\frac{4}{33}>\frac{1}{49}=11L_{yz}\cdot
D\geqslant \mult_{O_{x}}(D),
$$
$$
\frac{4}{33}>\frac{2}{29}=11 R_z\cdot
D\geqslant \mult_{O_{x}}(D)
$$
must hold, and hence the point $P$ cannot be the point $O_x$.
Since $R_t$ is singular at the point $O_z$ with multiplicity $4$,
we can apply the same method to $C_t$, i.e., one of the following
inequalities must be satisfied:
\[\frac{4}{33}>\frac{1}{29}=39L_{xt}\cdot
D\geqslant \mult_{O_{z}}(D),\]
\[\frac{4}{33}>\frac{1}{11}=\frac{39}{4}R_t\cdot D
\geqslant\frac{1}{4}\mult_{O_{z}}(D)\mult_{O_{z}}(R_{t})=\mult_{O_{z}}(D).\]
Thus, the point $P$ cannot be $O_z$.

Write $D=\mu R_x+\Omega$, where $\Omega$ is an effective
$\Q$-divisor such that $R_x\not\subset\Supp(\Omega)$. If $\mu>0$,
then $L_{xt}$ is not contained in the support of $D$. Thus,
$$
\frac{2}{29}\mu=\mu R_x\cdot L_{xt}\leqslant D\cdot
L_{xt}=\frac{1}{29\cdot 39},
$$
and hence $\mu\leqslant \frac{1}{78}$. We have
$$
49\Omega\cdot R_x=49(D\cdot R_x-\mu R_x^2)=\frac{2+76\mu}{29}<\frac{4}{33}.%
$$
Then Lemma~\ref{lemma:handy-adjunction} shows that the point $P$
cannot belong to $R_x$. In particular, the point $P$ cannot be
$O_t$.

Put $D=\epsilon L_{xt}+\Delta$, where $\Delta$ is an effective
$\Q$-divisor such that $L_{xt}\not\subset\Supp(\Delta)$.  Since
$(X, \frac{33}{4}D)$ is log canonical at the point $O_y$,
$\epsilon\leqslant \frac{4}{33}$ and hence
$$
\Delta\cdot L_{xt}=D\cdot L_{xt}-\epsilon L_{xt}^2
=\frac{1+67\epsilon}{29\cdot 39}<\frac{4}{33}.%
$$
Then Lemma~\ref{lemma:handy-adjunction} implies that the point $P$
cannot belong to $L_{xt}$.

Consequently, the point $P$ must be a smooth point in the outside
of $C_x$. Then an absurd inequality
$$
\frac{4}{33}<\mult_P(D)\leqslant\frac{539\cdot 127}{11\cdot
29\cdot 39\cdot 49}<\frac{4}{33}
$$
follows from Lemma~\ref{lemma:Carolina} since $H^0(\P,
\mathcal{O}_\P(539))$ contains $x^{20}y^{11}$, $x^{49}$,
$x^{10}z^{11}$ and $t^{11}$. The obtained contradiction completes
the proof.
\end{proof}

\begin{lemma}
\label{lemma:114969128256} Let $X$ be a quasismooth hypersurface
of degree $256$ in $\mathbb{P}(11,49,69, 128)$. Then
$\lct(X)=\frac{55}{6}$.
\end{lemma}

\begin{proof}
The curve $C_x$ is irreducible and
reduced. Moreover, it is easy to see $\lct(X,
\frac{1}{11}C_x)=\frac{55}{6}$. The curve $C_y$ is also irreducible and reduced and the log pair $(X,
\frac{1}{49}C_y)$ is log canonical.

Suppose that $\lct(X)<\frac{55}{6}$. By Lemma~\ref{lemma:convexity},
there is an effective $\Q$-divisor $D\qlineq -K_X$ such that
$C_{x}, C_y\not\subset\mathrm{Supp}(D)$ and the log pair
$(X,\frac{55}{6}D)$ is not log canonical at some point $P\in X$.

The inequalities
$$
69D\cdot C_x=\frac{69\cdot 11\cdot 256}{ 11\cdot 49\cdot 69\cdot 128}<\frac{6}{55},%
$$
$$11 D\cdot C_y=\frac{11\cdot 49\cdot 256}{11\cdot 49\cdot 69\cdot 128}<\frac{6}{55},%
$$
imply that the point $P$ is a smooth point in the outside of $C_x$.
However, since  $H^0(\P,
\mathcal{O}_\P(759))$ contains $x^{69}, x^{20}y^{11}, z^{11}$, we obtain
$$
\mult_{P}(D)\leqslant \frac{759\cdot 256}{11\cdot 49\cdot 69\cdot 128}<\frac{6}{55}%
$$
from Lemma~\ref{lemma:Carolina}. This is a contradiction.
\end{proof}

\begin{lemma}
\label{lemma:13233557127} Let $X$ be a quasismooth hypersurface of
degree $127$ in $\mathbb{P}(13,23,35, 57)$. Then
$\lct(X)=\frac{65}{8}$.
\end{lemma}

\begin{proof}
We may assume that the hypersurface $X$ is  given by the equation
$$
z^2t+y^4z+xt^2+x^8y=0.
$$
The only singularities of $X$ are a singular point of type
$\frac{1}{13}(9,5)$ at $O_x$, a singular point of type
$\frac{1}{23}(13,11)$ at $O_y$, a singular point of type
$\frac{1}{35}(13,23)$ at $O_z$, and a singular point of type
$\frac{1}{57}(23, 35)$ at $O_t$.

The curve $C_x$ (resp. $C_y$, $C_z$, $C_t$) consists of two
irreducible curves $L_{xz}$ (resp. $L_{yt}$, $L_{xz}$, $L_{yt}$)
and $R_x=\{x=y^4+zt=0\}$ (resp. $R_y=\{y=z^2+xt=0\}$,
$R_z=\{z=t^2+x^7y=0\}$, $R_t=\{t=y^3z+x^8=0\}$). We can see that
\[L_{xt}\cap R_x=\{O_t\}, \ L_{yz}\cap R_y=\{O_x\}, \ \
L_{yz}\cap R_z=\{O_y\}, \ L_{xt}\cap R_t=\{O_z\}. \]

It is easy to check $\lct(X, \frac{1}{13} C_x)=\frac{65}{8}$. The
log pairs $(X, \frac{65}{8\cdot 23}C_y)$, $(X, \frac{65}{8\cdot
35}C_z)$ and $(X, \frac{65}{8\cdot 57}C_t)$ are log canonical.

Suppose that $\lct(X)<\frac{65}{8}$. Then there is an effective
$\Q$-divisor $D\qlineq -K_X$ such that the log pair $(X,
\frac{65}{8} D)$ is not log canonical at some point $P\in X$.

By Lemma~\ref{lemma:convexity}, we may assume that the support of
$D$ does not contain $L_{xz}$ or $R_x$. Then one of the following
two inequalities must hold:
$$
\frac{8}{65}>\frac{1}{23}=57L_{xz}\cdot D\geqslant
\mult_{O_{t}}(D),
$$
$$
\frac{8}{65}>\frac{4}{35}=57R_x\cdot
D\geqslant \mult_{O_{t}}(D).
$$
Therefore, the point $P$ cannot be the point $O_t$. For the same
reason, one of  two inequalities
$$
\frac{8}{65}>\frac{1}{35}=13L_{yt}\cdot
D\geqslant \mult_{O_{x}}(D),
$$
$$
\frac{8}{65}>\frac{2}{57}=13 R_y\cdot
D\geqslant \mult_{O_{x}}(D)
$$
must hold, and hence the point $P$ cannot be the point $O_x$.

To apply the same method to $C_z$ and $C_t$, we note that $R_z$ is
singular at $O_y$ with multiplicity $2$ and $R_t$ is singular at
$O_z$ with multiplicity $3$. Then we can see that one inequality
from each of the pairs
$$
\frac{8}{65}>\frac{1}{13}=35L_{yt}\cdot D\geqslant \mult_{O_{z}}(D),
$$
$$
\frac{8}{65}>\frac{8}{23\cdot 3}=\frac{35}{3}R_{t}\cdot D
\geqslant\frac{1}{3}\mult_{O_{z}}(D)\mult_{O_{z}}(R_{t})=\mult_{O_{z}}(D);
$$

$$
\frac{8}{65}>\frac{1}{57}=23L_{xz}\cdot D\geqslant \mult_{O_{y}}(D),
$$
$$
\frac{8}{65}>\frac{1}{13}=\frac{23}{2}R_{z}\cdot D\geqslant
\frac{1}{2}\mult_{O_{y}}(D)\mult_{O_{y}}(R_{z})=\mult_{O_{y}}(D)
$$
must be satisfied. Therefore, the point $P$ can be neither $O_y$
nor $O_z$.

To apply Lemma~\ref{lemma:handy-adjunction} to $L_{xz}$ and $R_x$,
we compute
$$L_{xz}^2=-\frac{79}{23\cdot 57},\ \ R_x^2=-\frac{88}{35\cdot 57}.
$$
Put $D=aL_{xz}+bR_x+\Omega$, where $\Omega$ is an effective
$\Q$-divisor such that $L_{xz}, R_x\not\subset\Supp(\Omega)$. Then
$a, b\leqslant \frac{8}{65}$ since the log pair $(X,
\frac{65}{8}D)$ is log canonical at the point $O_t$. Therefore,
$$
D\cdot L_{xz}-a L_{xz}^2=\frac{1+79a}{23\cdot 57}<\frac{8}{65},%
$$
$$
D\cdot R_x-b R_x^2=\frac{4+88b}{35\cdot 57}<\frac{8}{65}.%
$$
Then, Lemma~\ref{lemma:handy-adjunction} implies that the point
$P$ is a smooth point in the outside of $C_x$.

Applying Lemma~\ref{lemma:Carolina}, we see that
$$\frac{8}{65}<\mult_P(D)\leqslant\frac{741\cdot 127}{13\cdot 23\cdot 35\cdot 57}<\frac{8}{65},%
$$
since $H^0(\P, \mathcal{O}_\P(455))$ contains $x^{35}, x^{12}y^{13}, z^{13}$ and the point $P$ is in the outside of $L_{xz}$. The obtained
contradiction completes the proof.
\end{proof}

\begin{lemma}
\label{lemma:133581128256}Let $X$ be a quasismooth hypersurface of
degree $256$ in $\mathbb{P}(13,35,81, 128)$. Then
$\lct(X)=\frac{91}{10}$.
\end{lemma}

\begin{proof}
We may assume that the surface $X$ is given by the equation
$$
t^2+y^5z+xz^3+x^{17}y=0.
$$
It has  a singular point of type $\frac{1}{13}(3, 11)$
at $O_x$, a singular point of type $\frac{1}{35}(13, 23)$ at $O_y$, and a
singular point of type $\frac{1}{81}(35, 47)$ at $O_z$.

The curve $C_x$ is reduced and irreducible. The curve is singular at the point $O_z$. It is easy to check
that $\lct(X, C_x)=\frac{7}{10}$. Therefore, $\lct(X)\leqslant \frac{91}{10}$.
The curve $C_y$ is also reduced and irreducible. The
curve $C_y$ is singular only at $O_x$. Moreover, the log pair $(X, \frac{91}{10\cdot 35}C_y)$ is log canonical.

Suppose that $\lct(X)<\frac{91}{10}$. Then there is an effective
$\Q$-divisor $D\qlineq -K_X$ such that the log pair $(X,
\frac{91}{10} D)$ is not log canonical at some point $P\in X$. By
Lemma~\ref{lemma:convexity} we may assume neither $C_x$ nor
$C_y$ is contained in $\Supp(D)$.

The following two inequalities show that the point $P$ is located in the outside of $C_x\cup C_y$:
\[\frac{81}{2}C_x\cdot D=\frac{1}{35}<\frac{10}{91},\]
\[\frac{13}{2}C_x\cdot D=\frac{1}{81}<\frac{10}{91}.\]
However, applying Lemma~\ref{lemma:Carolina}, we can obtain
$$
\mult_P(D)\leqslant\frac{1053\cdot 256}{13\cdot 35\cdot 81\cdot 128}<\frac{10}{91},%
$$
since $H^0(\P, \mathcal{O}_\P(1053))$ contains $x^{81}$,
$x^{11}y^{26}$ and $z^{13}$. This is a contradiction.
\end{proof}

\section{Sporadic cases with $I=2$}
\label{section:index-2}

\begin{lemma}
\label{lemma:I-2-W-2-3-4-5-D-12}
Let $X$ be the quasismooth
hypersurface defined by a quasihomogeneous polynomial $f(x,y,z,t)$  of degree $12$ in $\mathbb{P}(2,3,4,5)$.
Then
$$
\mathrm{lct}\big(X\big)=\left\{%
\aligned
&1\ \text{ if $f(x,y,z,t)$  contains the term $yzt$},\\%
&\frac{7}{12}\ \text{ if $f(x,y,z,t)$  does not contain the term $yzt$}.\\%
\endaligned\right.%
$$
\end{lemma}
\begin{proof}
We may assume
\[f(x,y,z,t)=z(z-x^2)(z-\epsilon x^2)+y^4 +xt^2+ayzt+bxy^2z+cx^2yt+dx^3y^2,\]
where $\epsilon$ $(\ne 0, 1)$, $a$, $b$, $c$, $d$ are constants.
Note that $X$ is singular at the point $O_t$ and three points
$Q_1=[1:0:0:0]$, $Q_2=[1:0:1:0]$, $Q_3=[1:0:\epsilon:0]$. The curve $C_x$ always is irreducible and reduced. We can easily check that
$$
\mathrm{lct}\big(X, C_x\big)=\left\{%
\aligned
&1\ \text{ if $a\ne0$},\\%
&\frac{7}{12}\ \text{ if $a=0$}.\\%
\endaligned\right.%
$$

 Suppose that $\lct(X)<\lambda:=\lct (X, C_x)$. Then
there is an effective $\Q$-divisor $D\qlineq -K_X$ such that the
log pair $(X, \lambda D)$ is not log canonical at some point
$P\in X$. We may assume that the curve $C_x$ is not contained in the support of $D$.

First, we consider the case where $a=0$. Since $H^0(\P, \mathcal{O}_{\P}(6))$ contains $x^3$,
$y^2$, and $xz$, Lemma~\ref{lemma:Carolina} implies
that for a smooth point $O\in X\setminus C_x$
\[\mult_O(D)<\frac{2\cdot 12\cdot 6}{2\cdot 3\cdot 4\cdot 5}<\frac{12}{7}.\]
Therefore, the point $P$ cannot be a smooth point in $X\setminus
C_x$. Since the curve $C_x$ is not contained in the support of $D$ and it is singular at $O_t$ with multiplicity $3$,  the inequality
\[\frac{5}{3}D\cdot C_x=\frac{5\cdot 2\cdot 2\cdot 12}{3\cdot 2\cdot 3\cdot 4\cdot 5}<\frac{12}{7}\]
implies that the point $P$ is located in the outside of $C_x$.
Thus, the point $P$ must be one of the point $Q_1$, $Q_2$, $Q_3$.
The curve $C_y$ is quasismooth. Therefore, we may assume that the
support of $D$ does not contain the curve $C_y$. Then the
inequality
\[\mult_{Q_i}(D)\leqslant 2D\cdot C_y=\frac{2\cdot 2\cdot 3\cdot 12}{2\cdot 3\cdot 4\cdot 5}<\frac{12}{7}\]
gives us a contradiction.

From now we consider the case where $a\ne 0$. Note that the curve
$C_x$ is not contained in the support of $D$ and it is singular at
$O_t$ with multiplicity $2$. Since
\[\frac{5}{2}D\cdot C_x=\frac{5\cdot 2\cdot 2\cdot 12}{2\cdot 2\cdot 3\cdot 4\cdot 5}=1\]
the point $P$ is located in the outside of $C_x$.

The curve $C_z$ is irreducible and the log pair $(X,
\frac{1}{2}C_z)$ is log canonical. Therefore, we may assume that
the support of $D$ does not contain the curve $C_z$. The curve
$C_z$ is singular at the point $Q_1$. The inequality
\[\mult_{Q_1}(D)\leqslant D\cdot C_z = \frac{2\cdot 4\cdot 12}{2\cdot 3\cdot 4\cdot 5}<1\]
implies that $P$ cannot be the point $Q_1$. We consider the curves
$C_{z-x^2}$ defined by $z=x^2$ and $C_{z-\epsilon x^2}$ defined by
$z=\epsilon x^2$. Then by coordinate changes we can see that they
have the same properties as that of $C_z$. Moreover, we can see
that the point $P$ can be neither $Q_2$ nor $Q_3$. Therefore, the
point $P$ must be located in the outside of $C_x\cup C_z\cup
C_{z-x^2}\cup C_{z-\epsilon x^2}$.

Let $\mathcal{L}$ be the pencil on $X$ defined by $\lambda x^2+\mu
z=0$, where $[\lambda:\mu]\in \P^1$. Let $C$ the curve in
$\mathcal{L}$ that passes through the point $P$. Then it is cut
out by $z=\alpha x^2$, where $\alpha\ne  0, 1, \epsilon$. The
curve $C$ is isomorphic to the curve in $\P(2,3,5)$ defined by
\[x^6+y^4+xt^2+\beta x^2yt+\gamma x^3y^2=0,\]
where $\beta$ and $\gamma$ are constants. We can easily see that the curve $C$
is irreducible. Moreover, we can check $\mult_P(C)\leqslant 2$ and hence the log pair $(X, \frac{1}{2}C)$ is log canonical. Therefore, we may assume that the support of $D$ does not contain the curve $C$. Then, the inequality
\[\mult_P(D)\leqslant D\cdot C=\frac{2\cdot 4\cdot 12}{2\cdot 3\cdot 4\cdot 5}<1\]
gives us a contradiction.
\end{proof}

\begin{lemma}
\label{lemma:I-2-W-2-3-4-7-D-14}
Let $X$ be a quasismooth hypersurface of
degree $14$ in $\mathbb{P}(2,3,4,7)$. Then
$\lct(X)=1$.
\end{lemma}

\begin{proof}
We may assume that $X$ is defined by the quasihomogeneous equation
\[t^2-y^2z^2+x(z-\beta_1x^2)(z-\beta_2x^2)(z-\beta_3x^2)+\epsilon xy^2(y^2-\gamma x^3)\]
where $\epsilon\ne 0$, $\beta_1$, $\beta_2$, $\beta_3$, $\gamma$
are constants. Note that $X$ is singular at the points $O_y$,
$O_z$ and three points $Q_1=[1:0:\beta_1:0]$,
$Q_2=[1:0:\beta_2:0]$, $Q_3=[1:0:\beta_3:0]$. The constants
$\beta_1$, $\beta_2$ and $\beta_3$ are distinct since $X$ is
quasismooth. The curve $C_x$ consists of two irreducible reduced
curves $C_{-}$ and $C_{+}$. However, the curves $C_y$ and $C_z$
are irreducible. We can easily see that $\lct(X, C_x)=1$, $\lct
(X, \frac{2}{3}C_y)=\frac{3}{2}$ and $\lct(X, \frac{1}{2}C_z)>1$.

Suppose that $\lct(X)<1$. Then there is an effective $\Q$-divisor
$D\qlineq -K_X$ such that the log pair $(X, D)$ is not log
canonical at some point $P\in X$. Since $H^0(\P,
\mathcal{O}_{\P}(6))$ contains $x^3$, $y^2$ and $xz$,
Lemma~\ref{lemma:Carolina} implies that the point $P$
is either a singular point of $X$ or a point of $C_x$.
Furthermore, $C_y$ is irreducible and hence we may assume that the
support of $D$ does not contain the curve $C_y$. Hence the
equality
\[2C_y\cdot D=\frac{2\cdot 3\cdot 2\cdot 14}{ 2\cdot 3\cdot  4\cdot 5}=1\]
implies that $P\ne Q_i$ for each $i=1,2,3$. In particular, the
point $P$ must belong to $C_x$.

We have the following intersection numbers:
\[C_x\cdot C_-=C_x\cdot C_+=\frac{1}{6}, \ \  C_-\cdot C_+=\frac{7}{12}, \ \ C_-^2=C_+^2=-\frac{5}{12}.\]
We may assume that the support of $D$ cannot contain either $C_-$
or $C_+$. If $D$ does not contain the curve $C_+$, then we obtain
\begin{gather*}
\mult_{O_y}(D) \leqslant 4D\cdot C_+=\frac{2}{3}<1,\\
\mult_{O_z}(D) \leqslant 4D\cdot C_+=\frac{2}{3}<1.
\end{gather*}
On the other hand, if $D$ does not contain the curve $C_-$, then
we obtain
\begin{gather*}
\mult_{O_y}(D)\leqslant 4D\cdot C_-=\frac{2}{3}<1,\\
\mult_{O_z}(D) \leqslant 4D\cdot C_-=\frac{2}{3}<1.
\end{gather*}
Therefore, the point $P$ must be in
$C_x\setminus\mathrm{Sing}(X)$.

We write $D=mC_++\Omega$, where the support of $\Omega$ does not
contain the curve $C_+$.  Then $m\geqslant \frac{2}{7}$ since
$D\cdot C_-\geq mC_+\cdot C_-$. Then we see $C_+\cdot D-mC_+^2<1$.
By the same method, we also obtain $C_-\cdot D-mC_-^2<1$. Then
Lemma~\ref{lemma:handy-adjunction} completes the proof.
\end{proof}

\begin{lemma}
\label{lemma:I-2-W-3-4-5-10-D-20}
Let $X$ be a quasismooth hypersurface of
degree $20$ in $\mathbb{P}(3,4,5, 10)$. Then
$\lct(X)=\frac{3}{2}$.
\end{lemma}

\begin{proof}
The surface $X$ can be defined by the quasihomogeneous equation
$$
t^2=y^{5}+z^4+x^{5}z+\epsilon_{1}xy^{3}z+\epsilon_{2}x^{2}yz^{2}+\epsilon_{3}x^{4}y^{2},
$$
where $\epsilon_{i}\in\mathbb{C}$. Note that the surface $X$ is singular only at
the point $O_x$, $O=[0:1:0:1]$, $P_1=[0:0:1:1]$ and $P_2=[0:0:1:-1]$.

The curves $C_x$, $C_y$ and $C_{z}$ are irreducible. Moreover, we
have
$$
\frac{3}{2}=\lct(X, \frac{2}{3}C_x)<\lct(X, \frac{2}{4}C_y)=2,%
$$
and hence  $\lct(X)\leqslant \frac{3}{2}$. We also see that $\lct(X, \frac{2}{5}C_z)>\frac{3}{2}$.

Suppose that $\lct(X)<\frac{3}{2}$. Then there is an effective $\Q$-divisor
$D\qlineq -K_X$ such that the pair $(X,\frac{3}{2}D)$ is not log
canonical at some point $P$. By Lemma~\ref{lemma:convexity}, we
may assume that the support of the divisor $D$ does not contain
the curves $C_x$, $C_y$ and $C_{z}$.

Suppose that $P\not\in C_{x}\cup C_{y}\cup C_{z}$. Then we
consider the pencil $\mathcal{L}$ on $X$ cut out by the equations
$\lambda y^2+\mu xz=0$, $[\lambda:\mu]\in \mathbb{P}^1$. There is
a unique member $Z$ in the pencil $\mathcal{L}$ with  $P\in Z$.
The curve $Z$ is cut out by an equation of the form $\alpha
y^2+xz$, where $\alpha$ is a non-zero constant. There is a natural
double cover $\omega\colon Z\to C$, where $C$ is the curve in
$\mathbb{P}(3,4,5)$ given by the equation $\alpha y^2+xz$. The
curve $C$ is quasismooth and $\omega(P)$ is a smooth point of
$\mathbb{P}(3,4,5)$. Thus, we see that $\mult_{P}(Z)\leqslant 2$,
the curve $Z$ consists of at most $2$ components, each component
of $Z$ is a smooth rational curve. In particular, $(X,
\frac{3}{8}Z)$ is log canonical. Therefore, we may assume that
$\Supp(D)$ does not contain at least one irreducible component of
$Z$. Thus, if $Z$ is irreducible, then we obtain an absurd
inequality
$$
\frac{8}{15}=D\cdot
Z\geqslant\mult_{P}(D)\geqslant\frac{2}{3}.
$$
So, we see that $Z=Z_{1}+Z_{2}$, where $Z_{1}$ and $Z_{2}$ are
smooth irreducible rational curves. Then
$$
Z_1^2=Z_2^2=-\frac{4}{15},\ \ \ Z_{1}\cdot Z_{2}=\frac{4}{3}.%
$$
Without loss of generality we may assume that $P\in Z_{1}$. Put
$D=m Z_{1}+\Omega$, where $\Omega$ is an effective
$\mathbb{Q}$-divisor such that
$Z_{1}\not\subset\mathrm{Supp}(\Omega)$. If $m\ne 0$, then
$$
\frac{4}{15}=D\cdot Z_{2}\geqslant  m Z_{1}\cdot Z_{2}=\frac{4m}{3},%
$$
and hence $m\leqslant \frac{1}{5}$. On the other hand,
Lemma~\ref{lemma:handy-adjunction} shows that
$$
\frac{4+4m}{15}=\big(D-mZ_{1}\big)\cdot
Z_{1}=\Omega\cdot Z_{1}>\frac{2}{3},
$$
and hence $m>\frac{3}{2}$. This is a contradiction. Therefore,
$P\in C_{x}\cup C_{y}\cup C_{z}$.

 The inequalities
$$
D\cdot C_{x}=\frac{1}{5}<\frac{2}{3}, \ \
D\cdot C_{y}=\frac{4}{15}<\frac{2}{3}, \ \
D\cdot C_{z}=\frac{1}{3}<\frac{2}{3}
$$
imply that the point $P$ must be a singular point of $X$.

The curve $C_{z}$ is singular at the point $O_{x}$. Thus, we have
$$
\frac{1}{2}=\frac{3}{2}D\cdot C_{z}\geqslant\frac{\mult_{O_x}(D)\mult_{O_x}(C_{z})}{2}
=\mult_{O_x}(D).
$$
Therefore, the point $P$ cannot be $O_x$.

Also, we have
$$
\frac{2}{5}=2D\cdot C_{x}\geqslant \mult_{O}(D).
$$
This inequality shows that the point $P$ cannot be the point $O$.
Consequently, the point $P$ must be either $P_1$ or $P_2$.

Without loss of generality we may assume that $P=P_{1}$. Note that
$C_{x}\cap C_{y}=\{P_{1},P_{2}\}$.

Let $\pi\colon\bar{X}\to X$ be the weighted blow up at the point
$P_{1}$ with weights $(3,4)$. Let $E$ be the exceptional curve of
$\pi$ and  let $\bar{D}$, $\bar{C}_{x}$ and $\bar{C}_{y}$ be the
proper transforms of $D$, $C_{x}$ and $C_{y}$, respectively. Then
$$
K_{\bar{X}}\qlineq \pi^{*}(K_{X})+\frac{2}{5}E,\ %
\bar{C}_{x}\qlineq \pi^{*}(C_{x})-\frac{3}{5}E,\ %
\bar{C}_{y}\qlineq \pi^{*}(C_{y})-\frac{4}{5}E,\ %
\bar{D}\qlineq \pi^{*}(D)-\frac{a}{5}E,%
$$
where $a$ is a non-negative  rational number. The curve $E$
contains one singular point $Q_{3}$ of type $\frac{1}{3}(1,1)$ and
one singular point of $Q_{4}$ of type $\frac{1}{4}(1,1)$ on the
surface $\bar{X}$. The point $Q_3$ is contained in $\bar{C}_y$ but
not in $\bar{C}_x$. On the other hand, the point $Q_4$ is
contained in $\bar{C}_x$ but not in $\bar{C}_y$. The intersection
$\bar{C}_{x}\cap\bar{C}_{y}$ consists of a single point that
dominates the point $P_{2}$.

The log pull back of the log pair $(X,\frac{3}{2}D)$ is the log
pair
$$
\left(\bar{X},
\frac{3}{2}\bar{D}+\frac{3a-4}{10}E\right).
$$
This is not log canonical at some point $Q\in E$. We see that
$$
0\leqslant \bar{C}_{x}\cdot\bar{D}=C_{x}\cdot D+\frac{3a}{25}E^{2}=\frac{1}{5}-\frac{a}{20},%
$$
and hence $a\leqslant 4$. In particular,
$$
\frac{3a-4}{10}<1.
$$
This implies that the log pull back of the log pair
$(X,\frac{3}{2}D)$ is log canonical in a punctured neighborhood of
the point $Q$.

If $a\leqslant \frac{4}{3}$, then the log pair $(\bar{X},
\frac{3}{2}\bar{D})$ is not log canonical at $Q$ as well. We then
obtain
$$
\frac{a}{12}=\bar{D}\cdot E>
\left\{%
\aligned
&\frac{2}{3}\ \text{if}\ Q\ne Q_{3}\ \text{and}\ Q\ne Q_{4},\\%
&\frac{2}{3}\cdot\frac{1}{3}\ \text{if}\ Q=Q_{3},\\%
&\frac{2}{3}\cdot\frac{1}{4}\ \text{if}\ Q=Q_{4}.\\%
\endaligned\right.%
$$
In particular, we have $a>2$. This contradicts the assumption
$a\leqslant \frac{4}{3}$. Therefore, $a>\frac{4}{3}$ and the log
pull back of the log pair $(X,\frac{3}{2}D)$ is effective.
Then
$$
\mult_{Q}(\bar{D})>\frac{2}{3}\left(1-\frac{3a-4}{10}\right)=\frac{14-3a}{15}.%
$$
Since
$
\bar{D}\cdot E =\frac{a}{12}\leqslant \frac{2}{3},
$
Lemma~\ref{lemma:handy-adjunction} implies  that the point $Q$
cannot be a smooth point. Therefore, the point $Q$ is either
$Q_{3}$ or $Q_{4}$. However, two inequalities
$$
\frac{4}{5}-\frac{a}{5}=4\bar{D}\cdot\bar{C}_{x}
\geqslant\mult_{Q_{4}}(\bar{D})>\frac{14-3a}{15},
$$
$$
\frac{4}{5}-\frac{a}{5}=3\bar{D}\cdot\bar{C}_{y}\geqslant\mult_{Q_{3}}(\bar{D})>\frac{14-3a}{15}
$$
give us a contradiction.
\end{proof}

\begin{lemma}
\label{lemma:I-2-W-3-4-10-15-D-30}
Let $X$ be a quasismooth hypersurface of
degree $30$ in $\mathbb{P}(3,4,10, 15)$. Then
$\lct(X)=\frac{3}{2}$.
\end{lemma}

\begin{proof}
The surface $X$ can be defined by the quasihomogeneous equation
$$
t^{2}=z^{3}-y^{5}z-x^{10}+\epsilon_{1}x^{2}yz^{2}+\epsilon_{2}x^{2}y^{6}+\epsilon_{3}x^{4}y^{2}z+\epsilon_{4}x^{6}y^{3},
$$
where $\epsilon_{i}\in\mathbb{C}$. The surface $X$ is singular at
the points $O_y$, $O_{2}=[0:1:1:0]$, $O_{5}=[0:0:1:1]$, $P_1=[1:0:0:1]$ and
$P_{2}=[1:0:0:-1]$.

The curves $C_x$ and $C_y$ are irreducible. Moreover, we have
$$
\frac{3}{2}=\lct\left(X, \frac{2}{3}C_x\right)<\lct\left(X, \frac{2}{4}C_y\right)=2.%
$$

Suppose that $\lct(X)<\frac{3}{2}$. Then there is an effective $\Q$-divisor
$D\qlineq -K_X$ such that the pair $(X,\frac{3}{2}D)$ is not log
canonical at some point $P$. By Lemma~\ref{lemma:convexity}, we
may assume that the support of the divisor $D$ contains
neither $C_x$ nor $C_y$.

Since $H^0(\P, \mathcal{O}_{\P}(20))$ contains the monomials
$y^5$, $y^2x^{4}$, $z^{2}$, it follows from
Lemma~\ref{lemma:Carolina} that the point $P$ is either a singular point of $X$ or a smooth point in $C_y$. However,
the point $P$ cannot belong to  $C_y$ since
$
\frac{2}{3}=5D\cdot C_{y}%
$.
Therefore, the point $P$ must be either the point $O_y$ or $O_2$.
On the other hand, we have $4D\cdot C_x=\frac{2}{5}$.
This means that the pair $(X, \frac{3}{2}D)$ is log canonical at the points $O_y$ and $O_2$.
Consequently, $\lct(X)=\frac{3}{2}$.
\end{proof}

\begin{lemma}
\label{lemma:I-2-W-5-13-19-22-D-57}
Let $X$ be a quasismooth hypersurface of
degree $57$ in $\mathbb{P}(5,13,19, 22)$. Then
$\lct(X)=\frac{25}{12}$.
\end{lemma}

\begin{proof}
 The surface $X$ can be defined by the
quasihomogeneous equation
$$
z^{3}+yt^{2}+xy^{4}+x^{7}t+\epsilon x^{5}yz=0,
$$
where $\epsilon\in\mathbb{C}$. The surface $X$ is singular only at the
points $O_x$, $O_y$ and $O_t$.

The curves $C_x$ and $C_y$ are irreducible. Moreover, we have
$$
\frac{25}{12}=\lct\left(X, \frac{2}{5}C_x\right)<\lct\left(X, \frac{2}{13}C_y\right)=\frac{65}{21}.%
$$

Suppose that $\lct(X)<\frac{25}{12}$. Then there is an effective
$\Q$-divisor $D\qlineq -K_X$ such that the pair $(X,\frac{25}{12}D)$ is
not log canonical at some point $P$. By
Lemma~\ref{lemma:convexity}, we may assume that the support of
the divisor $D$ contains neither $C_x$ nor $C_y$.

Since $H^0(\P, \mathcal{O}_{\P}(110))$ contains the monomials
$x^{9}y^{5}$, $x^{22}$ and $t^{5}$, it follows from
Lemma~\ref{lemma:Carolina} that the point $P$ is either a singular point of $X$ or a smooth point on $C_x$.
However, this is impossible since
$22D\cdot C_x=\frac{6}{13}<\frac{12}{25}$ and $5D\cdot C_y=\frac{3}{11}<\frac{12}{25}$.
\end{proof}

\begin{lemma}
\label{lemma:I-2-W-5-13-19-35-D-70}
Let $X$ be a quasismooth hypersurface of
degree $70$ in $\mathbb{P}(5,13,19, 35)$. Then
$\lct(X)=\frac{25}{12}$.
\end{lemma}

\begin{proof}
 The surface $X$ can be defined by the
quasihomogeneous equation
$$
t^{2}+yz^{3}+xy^{5}-x^{14}+\epsilon x^{5}y^{2}z=0,
$$
where $\epsilon\in\mathbb{C}$. The surface $X$ is singular at the
points $O_y$ and $O_z$. It is also singular at two points $P_1=[1:0:0:1]$
and $P_2=[1:0:0:-1]$.

The curves $C_x$ is irreducible. On the other hand, the curve
$C_y$ consists of two smooth curves $C_1=\{y=x^7-t=0\}$ and
$C_2=\{y=x^7+t=0\}$. Moreover, we have
$$
\frac{25}{12}=\lct\left(X, \frac{2}{5}C_x\right)<\lct\left(X, \frac{2}{13}C_y\right)=\frac{26}{7}.%
$$

Suppose that $\lct(X)<\frac{25}{12}$. Then there is an effective
$\Q$-divisor $D\qlineq -K_X$ such that the pair
$(X,\frac{25}{12}D)$ is not log canonical at some point $P$. By
Lemma~\ref{lemma:convexity}, we may assume that the support of the
divisor $D$  dose not contains $C_x$. Also, we may assume that the
support of $D$ does not contain either $C_1$ or $C_2$.

Since $19D\cdot C_x=\frac{4}{13}<\frac{12}{25}$, the point $P$ cannot belong to $C_x$.

We put $m_1C_1+m_2C_2+\Omega$, where $\Omega$ is an effective $\mathbb{Q}$-divisor whose support contains neither $C_1$ nor $C_2$. Since the pair $(X, \frac{25}{12}D)$ is log canonical at the point $O_z$, we see that
$m_i\leqslant \frac{12}{25}$.
Since
\[5(D-m_iC_i)\cdot C_i=\frac{2-m_i}{19}<\frac{12}{25}\]
 for each $i$, Lemma~\ref{lemma:handy-adjunction} implies that the point $P$ can be neither $P_1$ nor $P_2$.
 Therefore, the point $P$ is a smooth point of $X$ in the outside of $C_x$.
 However,
since $H^0(\P, \mathcal{O}_{\P}(95))$ contains the monomials
$x^{6}y^{5}$, $x^{19}$ and $z^{5}$, it follows from
Lemma~\ref{lemma:Carolina} that the point $P$ is either a singular point of $X$ or a smooth point on $C_x$. This is a contradiction.
\end{proof}

\begin{lemma}
\label{lemma:I-2-W-6-9-10-13-D-36} Let $X$ be a quasismooth
hypersurface of degree $36$ in $\mathbb{P}(6,9,10, 13)$. Then
$\lct(X)=\frac{25}{12}$.
\end{lemma}

\begin{proof}
 The surface $X$ can be defined by the
quasihomogeneous equation
$$
zt^{2}+y^{4}+xz^{3}+x^{6}+\epsilon x^{3}y^{2}=0,
$$
where $\epsilon$ is a constant different from $\pm 2$. The surface $X$ is singular at the
points $O_z$ and $O_t$. It is also singular at two points $P_1$
and $P_2$ on $L_{zt}$. The
surface $X$ is also singular at one point $Q$ on $L_{yt}$.

The curves $C_x$ and $C_y$ are irreducible and reduced. However, the curve $C_z$ consists of two irreducible and reduced curves $C_1$ and $C_2$. The curve $C_1$ contains the point $P_1$ but not $P_2$. On the other hand, $C_2$ contains the point $P_2$ but not $P_1$. We also see
$$
C_{1}^2=C_{2}^2=-\frac{8}{39},\ \ C_{1}\cdot C_{2}=\frac{6}{13}.%
$$
It is easy to check
$$
\frac{25}{12}=\lct\left(X, \frac{2}{10}C_z\right)<\frac{9}{4}=\lct\left(X, \frac{2}{6}C_x\right)<\frac{9}{2}=\lct\left(X, \frac{2}{9}C_y\right).%
$$

Suppose that $\lct(X)<\frac{25}{12}$. Then there is an effective
$\Q$-divisor $D\qlineq -K_X$ such that the pair
$(X,\frac{25}{12}D)$ is not log canonical at some point $P$. By
Lemma~\ref{lemma:convexity}, we may assume that the support of the
divisor $D$  contains neither $C_x$ nor $C_y$. In addition,  we
may assume that it cannot contain either $C_{1}$ or $C_{2}$.

Since $H^0(\P, \mathcal{O}_{\P}(30))$ contains the monomials
$x^{2}y^{2}$, $x^{5}$ and $z^{3}$, it follows from
Lemma~\ref{lemma:Carolina} that
$P\in\mathrm{Sing}(X)\cup C_x\cup C_{z}$. However, $2D\cdot
C_y=\frac{12}{65}<\frac{12}{25}$ and hence the point $P$ cannot be
the point $Q$. Note that the curve $C_x$ passes through the point
$O_z$ with multiplicity $2$. Then the inequality $5D\cdot
C_x=\frac{4}{13}<\frac{12}{25}$ shows that the point $P$ cannot be
a point on $C_x\setminus \{O_t\}$.

Put $D=mC_{1}+\Omega$, where $\Omega$ is an effective
$\mathbb{Q}$-divisor such that
$C_{1}\not\subset\mathrm{Supp}(\Omega)$. If $m\ne 0$, then
$$
\frac{2}{39}=D\cdot C_{2}=\big(mC_{1}+\Omega\big)\cdot C_{2}\geqslant  mC_{1}\cdot C_{2}=\frac{6m}{13},%
$$
and hence $m\leqslant \frac{1}{9}$. Then
$$
3\big(D-mC_{1}\big)\cdot C_{1}=\frac{2+8m}{13}\leqslant \frac{12}{25}.
$$
Therefore, it follows from Lemma~\ref{lemma:handy-adjunction} that
the point $P$  cannot be a point on $C_1\setminus \{O_t\}$. By the
same method, we can show that the point $P$  cannot be a point on
$C_2\setminus \{O_t\}$. Therefore, the point $P$ must be the point
$O_t$.

Let $\pi\colon\bar{X}\to X$ be the weighted blow up at the point
$O_t$ with weights $(2,3)$. Let $E$ be the exceptional curve of
$\pi$ and  let $\bar{D}$, $\bar{C}_{x}$ and $\bar{C}_{y}$ be the
proper transforms of $D$, $C_{x}$ and $C_{y}$, respectively. Then
$$
K_{\bar{X}}\qlineq \pi^{*}(K_{X})-\frac{8}{13}E,\ %
\bar{C}_{x}\qlineq \pi^{*}(C_{x})-\frac{2}{13}E,\ %
\bar{C}_{y}\qlineq \pi^{*}(C_{y})-\frac{3}{13}E,\ %
\bar{D}\qlineq \pi^{*}(D)-\frac{a}{13}E,%
$$
where $a$ is a non-negative rational number. The curve $E$
contains one singular point $Q_{3}$ of type $\frac{1}{3}(1,1)$ and
one singular point of $Q_{2}$ of type $\frac{1}{2}(1,1)$ on the
surface $\bar{X}$. The point $Q_2$ is contained in $\bar{C}_y$ but
not in $\bar{C}_x$. On the other hand, the point $Q_3$ is
contained in $\bar{C}_x$ but not in $\bar{C}_y$.

The log pull back of the log pair $(X,\frac{25}{12}D)$ is the log
pair
$$
\left(\bar{X},
\frac{25}{12}\bar{D}+\frac{25a+96}{12\cdot 13}E\right).
$$
This is not log canonical at some point $Q\in E$. We see that
$$
0\leqslant \bar{C}_{y}\cdot\bar{D}=C_{y}\cdot
D+\frac{3a}{169}E^{2}
=\frac{6}{5\cdot 13}-\frac{a}{2\cdot 13},%
$$
and hence $a\leqslant \frac{12}{5}$. In particular,
$$
\frac{25a+96}{12\cdot 13}\leqslant 1.
$$
This implies that the log pull back of the log pair
$(X,\frac{25}{12}D)$ is log canonical in a punctured neighborhood
of the point $Q$. Then
$$
\mult_{Q}(\bar{D})>\frac{12}{25}\left(1-\frac{25a+96}{12\cdot 13}\right)
=\frac{12}{5\cdot 13}-\frac{a}{13}.%
$$
Since $ \bar{D}\cdot E =\frac{a}{6}\leqslant \frac{12}{25}, $
Lemma~\ref{lemma:handy-adjunction} implies  that the point $Q$
cannot be a smooth point. Therefore, the point $Q$ is either
$Q_{2}$ or $Q_{3}$. However, two inequalities
$$
\frac{12}{5\cdot 13}-\frac{a}{13}=3\bar{D}\cdot\bar{C}_{x}
\geqslant\mult_{Q_{3}}(\bar{D})>\frac{12}{5\cdot 13}-\frac{a}{13},
$$
$$
\frac{12}{5\cdot 13}-\frac{a}{13}=2\bar{D}\cdot\bar{C}_{y}
\geqslant\mult_{Q_{2}}(\bar{D})>\frac{12}{5\cdot 13}-\frac{a}{13}
$$
give us a contradiction.
\end{proof}

\begin{lemma}
\label{lemma:I-2-W-7-8-19-25-D-57}
Let $X$ be a quasismooth
hypersurface of degree $57$ in $\mathbb{P}(7,8,19, 25)$. Then
$\lct(X)=\frac{49}{24}$.
\end{lemma}

\begin{proof}
 The surface $X$ can be defined by the
quasihomogeneous equation
$$
z^{3}+y^{4}t+xt^{2}+x^{7}y+\epsilon x^{2}y^{3}z=0,
$$
where $\epsilon\in\mathbb{C}$. The surface $X$ is singular at the
points $O_x$, $O_{y}$ and $O_t$. The curves $C_x$, $C_y$ and
$C_{z}$ are irreducible. We have
$$
\frac{49}{24}=\lct\left(X, \frac{2}{7}C_x\right)<\lct\left(X, \frac{2}{8}C_y\right)
=\frac{10}{3}<\lct\left(X, \frac{2}{19}C_z\right)=\frac{19}{2}.%
$$
Thus, $\lct(X)\leqslant \frac{49}{24}$.

Suppose that $\lct(X)<\frac{49}{24}$. Then there is an effective
$\Q$-divisor $D\qlineq -K_X$ such that the pair
$(X,\frac{49}{24}D)$ is not log canonical at some point $P$. By
Lemma~\ref{lemma:convexity}, we may assume that the support of the
divisor $D$  contains none of the curves $C_x$, $C_y$ and $C_{z}$. The curve $C_x$ is singular at the point $O_t$.
Since $\frac{25}{2}D\cdot C_x=\frac{3}{8}<\frac{24}{49}$, $7D\cdot
C_y=\frac{6}{25}<\frac{24}{49}$ and  $D\cdot
C_z=\frac{57}{700}<\frac{24}{49}$, the point $P$ cannot belong to
the set $C_x\cup C_y\cup C_z$.

Consider the pencil $\mathcal{L}$ on $X$ defined by the equations
$\lambda y^2z+\mu x^5=0$, $[\mu, \lambda]\in \mathbb{P}^1$. Then
there is a unique curve $Z$ in the pencil $\mathcal{L}$ passing
through the point $P$. Then the curve $Z$ is defined by an
equation of the form $y^2z-\alpha x^5=0$, where $\alpha$ is a
non-zero constant.

We see that $C_{y}\not\subset\mathrm{Supp}(Z)$. But the open
subset $Z\setminus C_{y}$ of the curve $Z$ is a
$\mathbb{Z}_{8}$-quotient  of the affine curve
$$
z-\alpha x^5=z^{3}+t+xt^{2}+x^{7}+\epsilon
x^{2}z=0\subset\mathbb{C}^{3}\cong\mathrm{Spec}\Big(\mathbb{C}\big[x,z,t\big]\Big)
$$
that is isomorphic to the plane affine curve defined by the
equation
$$
\alpha^{3}x^{15}+t+xt^2+x^7+\epsilon\alpha
x^7=0\subset\mathbb{C}^{2}\cong\mathrm{Spec}\Big(\mathbb{C}\big[x,t\big]\Big).
$$
This curve is irreducible and hence the curve $Z$ is also irreducible. Thus
$\mult_{P}(Z)\leqslant 14$. We may assume that
$\mathrm{Supp}(D)$ does not contain the curve $Z$ by
Lemma~\ref{lemma:convexity}. Then we obtain an absurd inequality
$$
\frac{3}{20}=D\cdot Z\geqslant\mult_{P}\big(D\big)>\frac{24}{49}.%
$$
\end{proof}

\begin{lemma}
\label{lemma:I-2-W-7-8-19-32-D-64}
Let $X$ be a quasismooth
hypersurface of degree $64$ in $\mathbb{P}(7,8,19, 32)$. Then
$\lct(X)=\frac{35}{16}$.
\end{lemma}

\begin{proof}
 The surface $X$ can be defined by the
quasihomogeneous equation
$$
t^{2}-y^{8}+xz^{3}+x^{8}y+\epsilon x^{3}y^{3}z,
$$
where $\epsilon\in\mathbb{C}$. Note that $X$ is singular at the
points $O_x$ and $O_{z}$. The surface $X$ also has two singular
points $P_{1}=[0:1:0:1]$ and $P_{2}=[0:1:0:-1]$ of type $\frac{1}{8}(7,3)$.

The curve $C_{x}$ is reducible. We have $C_{x}=C_{1}+C_{2}$, where
$C_{1}$ and $C_{2}$ are irreducible and  reduced curves. The curve $C_1$ contains the point $P_1$ but not the point $P_2$. On the other hand, the curve $C_2$ contains the point $P_2$ but not the point $P_1$. However, these two curves meet each other only at the point $O_z$. We also have
$$
C_{1}^2=C_{2}^2=-\frac{25}{8\cdot 19},\ \ \ C_{1}\cdot C_{2}=\frac{4}{19}.%
$$
 The curve $C_y$ is irreducible. It is easy to check
$$
\lct\left(X, \frac{2}{7}C_x\right)=\frac{35}{16}<\lct\left(X, \frac{2}{8}C_y\right)=\frac{10}{3}.%
$$
Therefore, $\lct(X)\leqslant \frac{35}{16}$.

Suppose that $\lct(X)<\frac{35}{16}$. Then there is an effective
$\Q$-divisor $D\qlineq -K_X$ such that the pair $(X,\frac{35}{16}D)$ is
not log canonical at some point $P$. By
Lemma~\ref{lemma:convexity}, we may assume that the support of
$D$ does not contain the curve $C_y$. Moreover, we may assume that
the support of $D$ does not contain either the curve $C_{1}$ or
the curve $C_{2}$.

Since $C_{i}\not\subset\mathrm{Supp}(D)$ for either $i=1$ or $2$,
we have
$$
\mult_{O_z}(D)\leqslant 19D\cdot C_i=\frac{1}{4}<\frac{16}{35},%
$$
and hence $P\ne O_{z}$. Meanwhile, the inequality $7D\cdot C_y=\frac{4}{19}<\frac{16}{25}$ implies that the point $P$ cannot belong to $C_y$.

Suppose that $P\in C_{1}$. Then we write $D=mC_{1}+\Omega$, where $\Omega$
is an effective $\mathbb{Q}$-divisor such that
$C_{1}\not\subset\mathrm{Supp}(\Omega)$. If $m\ne 0$, then
$$
\frac{1}{4\cdot 19}=D\cdot C_{2}=\big(mC_{1}+\Omega\big)\cdot C_{2}\geqslant  mC_{1}\cdot C_{2}=\frac{4m}{19},%
$$
and hence $m\leqslant \frac{1}{16}$. Then it follows from
Lemma~\ref{lemma:handy-adjunction} that
$$
\frac{2+25m}{8\cdot 19}=\big(D-mC_{1}\big)\cdot
C_{1}=\Omega\cdot
C_{1}>\left\{%
\aligned
&\frac{16}{35}\ \text{if}\ P\ne P_{1},\\%
&\frac{16}{35}\cdot\frac{1}{8}\ \text{if}\ P=P_{1}.\\%
\endaligned\right.%
$$
This is impossible since  $m\leqslant \frac{1}{16}$. Thus,
$P\not\in C_{1}$. Similarly, we can show that $P\not\in C_{2}$.

Consequently, the point $P$ is located in the outside of $
C_x\cup C_y$. In particular, it is a smooth point of $X$. But $H^0(\P, \mathcal{O}_\P(64))$ contains monomials
$y^{8}$, $x^{8}y$, $y^4t$ and $t^{2}$. This is impossible by
Lemma~\ref{lemma:Carolina}. The obtained
contradiction completes the proof.
\end{proof}

\begin{lemma}
\label{lemma:I-2-W-9-12-13-16-D-48}
Let $X$ be a quasismooth
hypersurface of degree $48$ in $\mathbb{P}(9,12,13, 16)$. Then
$\lct(X)=\frac{63}{24}$.
\end{lemma}

\begin{proof}
 The surface $X$ can be defined by the
quasihomogeneous equation
$$
t^{3}-y^{4}+xz^{3}+x^{4}y=0.
$$
The surface $X$ is singular at the points $O_x$, $O_z$, $Q_{4}=[0:1:0:1]$ and
$Q_{3}=[1:1:0:0]$.

The curves $C_x$, $C_y$, $C_{z}$ and $C_{t}$ are irreducible and reduced. We
have
$$
\frac{63}{24}=\lct\left(X, \frac{2}{9}C_x\right)<\lct\left(X, \frac{2}{12}C_y\right)=4<\lct\left(X, \frac{2}{13}C_z\right)=\frac{13}{2}<\lct\left(X, \frac{2}{16}C_t\right)=\frac{16}{2}.%
$$
Therefore, $\lct(X)\leqslant \frac{63}{24}$.

Suppose that $\lct(X)<\frac{63}{24}$. Then there is an effective
$\Q$-divisor $D\qlineq -K_X$ such that the pair $(X,\frac{63}{24}D)$ is
not log canonical at some point $P$. By
Lemma~\ref{lemma:convexity}, we may assume that the support of
the divisor $D$ contains none of the curves $C_x$, $C_y$, $C_{z}$
and $C_{t}$.

Note that the curve $C_x$ is singular at $O_z$ with multiplicity $3$ and the curve $C_y$ is singular at $O_x$ with multiplicity $3$. Then the inequalities
\[\frac{13}{3}D\cdot C_x=\frac{1}{6}<\frac{24}{63},\ \ \frac{9}{3}D\cdot C_y=\frac{2}{13}<\frac{24}{63},\ \
3D\cdot C_z=\frac{1}{6}<\frac{24}{63},\ \ D\cdot C_t=\frac{8}{9\cdot 13}<\frac{24}{63}\]
show that  the point $P$ must be located in the outside of $C_x\cup C_y\cup C_z\cup C_t$.

Consider the pencil $\mathcal{L}$ on $X$ defined by the equations
$\lambda xt+\mu yz=0$, $[\mu, \lambda]\in \mathbb{P}^1$. Then
there is a unique curve $Z$ in the pencil $\mathcal{L}$ passing
through the point $P$. Then the curve $Z$ is defined by an
equation of the form $xt-\alpha yz=0$, where $\alpha$ is a
non-zero constant.
We see that
$C_{x}\not\subset\mathrm{Supp}(Z)$. But the open subset
$Z\setminus C_{x}$ of the curve $Z$ is a
$\mathbb{Z}_{9}$-quotient of the affine curve
$$
t-\alpha
yz=t^{3}+y^{4}+z^{3}+y=0\subset\mathbb{C}^{3}\cong\mathrm{Spec}\Big(\mathbb{C}\big[y,z,t\big]\Big),
$$
which is isomorphic to the plane affine curve  given
by the equation
$$
\alpha^{3}y^{3}z^{3}+y^{4}+z^{3}+y=0\subset\mathbb{C}^{2}\cong\mathrm{Spec}\Big(\mathbb{C}\big[y,z\big]\Big).
$$
Then, it is easy to see that the curve
$Z$ is irreducible and $\mult_{P}(Z)\leqslant 4$.
Thus, we may assume that $\mathrm{Supp}(D)$ does not contain the
curve $Z$ by Lemma~\ref{lemma:convexity}. However,
$$
\frac{25}{18\cdot 13}=D\cdot Z\geqslant\mult_{P}\big(D\big)>\frac{24}{63}.%
$$
Consequently, $\lct(X)=\frac{63}{24}$.
\end{proof}

\begin{lemma}
\label{lemma:I-2-W-9-12-19-19-D-57}
Let $X$ be a quasismooth
hypersurface of degree $57$ in $\mathbb{P}(9,12,19, 19)$. Then
$\lct(X)=3$.
\end{lemma}

\begin{proof}
 We may assume that the surface $X$ is defined by
the quasihomogeneous equation
$$
zt(z-t)-xy^{4}+x^{5}y=0.
$$
The surface $X$ is singular at three distinct points $O_x$,
$O_y$, $Q_3=[1:1:0:0]$ . Also, it is singular at three distinct points $O_z$, $O_t$, $Q_{19}=[0:0:1:1]$.

The curve $C_x$ consists of three distinct curves $L_{xz}$,
$L_{xt}$ and $R_x=\{x=z-t=0\}$ that intersect altogether at
the point $O_y$. Similarly, the
curve $C_y$ consists of three curves $L_{yz}$,
$L_{yt}$ and $R_y=\{y=z-t=0\}$ that intersect altogether
at the point $O_x$.
The curve $C_z$ consists of three distinct curves $L_{xz}$, $L_{yz}$
and $R_{z}=\{z=x^4-y^3=0\}$ that intersect altogether at
the point $O_t$. The curve $C_t$ consists
of three distinct curves $L_{xt}$, $L_{yt}$ and $R_{t}=\{t=x^4-y^3=0\}$ that intersect altogether at
the point $O_z$.
Let $C_{z-t}$ be the curve cut out on $X$ by the equation
$z=t$. Then $C_{z-t}$ consists of three distinct curves $R_x$,
$R_y$ and $R_{z-t}=\{z-t=x^4-y^3=0\}$ that intersect altogether at
the point $Q_{19}$.

We have the following intersection numbers:
$$
L_{xz}^{2}=L_{xt}^{2}=R_{x}^{2}=-\frac{29}{19\cdot 12},\ \ \
L_{yz}^{2}=L_{yt}^{2}=R_{y}^{2}=-\frac{26}{19\cdot 9}, \ \ \ R_z^2=R_t^2=R_{z-t}^2=-\frac{2}{19\cdot 3}$$
$$-K_X\cdot L_{xz}=-K_X\cdot L_{xt}=-K_X\cdot R_{x}=\frac{1}{19\cdot 6}, \ \ \ -K_X\cdot L_{yz}=-K_X\cdot L_{yt}=-K_X\cdot R_{y}=\frac{2}{19\cdot 9},$$
$$-K_X\cdot R_{z}=-K_X\cdot R_{t}=-K_X\cdot R_{z-t}=\frac{2}{19\cdot 3}.$$

Since $\lct(X, \frac{2}{9}C_x)=3$, we have
$\lct(X)\leqslant 3$. Suppose that $\lct(X)<3$. Then there is an
effective $\Q$-divisor $D\qlineq -K_X$ such that the pair $(X,3D)$
is not log canonical at some point $P\in X$.

The pairs $(X, \frac{6}{9}C_x)$ and  $(X, \frac{6}{12}C_y)$ are
log canonical. By Lemma~\ref{lemma:convexity}, we may assume
that the support of $D$ does not contain at least one component of
$C_x$. Then one of the inequalities
\[\mult_{O_y}(D)\leqslant 12D\cdot L_{xz}=\frac{6}{57}<\frac{1}{3}, \]
 \[\mult_{O_y}(D)\leqslant 12D\cdot L_{xt}=\frac{6}{57}<\frac{1}{3},\]
  \[\mult_{O_y}(D)\leqslant 12D\cdot R_{x}=\frac{6}{57}<\frac{1}{3}\]
  must hold, and hence the point $P$ cannot be the point $O_y$.
Also, we may assume that the support of $D$ does not
contain at least one component of $C_y$. By the same reason, the point $P$ cannot be the point $O_x$.

We have
$$
\lct\left(X,\frac{2}{19}C_z\right)=\lct\left(X,\frac{2}{19}C_t\right)=\lct\left(X,\frac{2}{19}C_1\right)=\frac{7}{2}.
$$
 By
Lemma~\ref{lemma:convexity}, we may assume that the support of
$D$ does not contain at least one component of each curve $C_z$,
$C_t$ and $C_{z-t}$. Since the curve $R_z$  is singular at the point $O_t$ with multiplicity $3$, Then one of the inequalities
\[\mult_{O_t}(D)\leqslant 19D\cdot L_{xz}=\frac{1}{6}<\frac{1}{3}, \]
 \[\mult_{O_t}(D)\leqslant 19D\cdot L_{yz}=\frac{2}{9}<\frac{1}{3},\]
  \[\mult_{O_t}(D)\leqslant \frac{19}{3}D\cdot R_{z}=\frac{2}{9}<\frac{1}{3}\]
  must hold, and hence the point $P$ cannot be the point $O_t$. By applying the same method to $C_t$ and $C_{z-t}$, we see that the point $P$  can neither $O_z$ not $Q_{19}$.

The three curves $R_z$, $R_t$, and $R_{z-t}$ intersects only at the point $Q_3$.
The log pair
$$
\left(X,\ \frac{3}{18}\Big(R_{z}+R_{t}+R_{z-t}\Big)\right)
$$
is log canonical at $Q_3$, and $R_{z}+R_{t}+R_{z-t}\sim -18K_{X}$.
By Lemma~\ref{lemma:convexity}, we may assume that the support
of $D$ does not contain at least one curve among $R_{z}$, $R_{t}$
and $R_{z-t}$. Without loss of generality, we may assume that the
support of $D$ does not contain the curve $R_{z}$. Then
$$
\mult_{Q_3}(D)\leqslant 3D\cdot R_z=\frac{2}{19}<\frac{1}{3},
$$
and hence the point $P$ cannot be $Q_{3}$.

Write $D=m_1L_{xz}+m_2L_{yz}+m_3R_z+\Delta$, where $\Delta$ is an effective $\mathbb{Q}$-divisor whose support contains none of the curves $L_{xz}$, $L_{yz}$, $R_z$. Since the pair $(X, 3D)$ is log canonical at the point $O_t$, we have $m_i\leqslant \frac{1}{3}$ for each $i=1$, $2$, $3$. By Lemma~\ref{lemma:handy-adjunction}, the inequalities
\[(D-m_1L_{xz})\cdot L_{xz}=\frac{2+29m_1}{12\cdot 19}<\frac{1}{3},\]
\[(D-m_2L_{yz})\cdot L_{yz}=\frac{2+26m_2}{9\cdot 19}<\frac{1}{3},\]
\[(D-m_3R_{z})\cdot R_{z}=\frac{2+2m_3}{3\cdot 19}<\frac{1}{3}\]
show that the point $P$ cannot belong to $C_z$. By the same way,
we can show that the point $P$ is not contained in $C_t\cup
C_{z-t}$. Therefore, the point $P$ is a smooth point of $X$ in the
outside of the set $C_z\cup C_{t}\cup C_{z-t}$. Then there is a
unique quasismooth irreducible curve $E\subset X$ passing through
the point $P$ and  defined by the equation $z=\lambda t$, where
$\lambda$ is a non-zero constant different from~$1$. By
Lemma~\ref{lemma:convexity}, we may assume that the support of $D$
does not contain the curve $E$. Then
$$
\frac{1}{3}<\mult_{P}(D)\leqslant D\cdot E=\frac{1}{18}.%
$$
This is a contradiction.
\end{proof}

\begin{lemma}
\label{lemma:I-2-W-9-19-24-31-D-81}Let $X$ be a quasismooth
hypersurface of degree $81$ in $\mathbb{P}(9,19,24, 31)$. Then
$\lct(X)=3$.
\end{lemma}

\begin{proof}
 The surface $X$ can be defined by the
quasihomogeneous equation
$$
yt^{2}+y^{3}z+xz^{3}-x^{9}=0.
$$
It is singular at the point $O_y$, $O_z$ and $O_t$. The
surface $X$ is also singular at the point $Q_3=[1:0:1:0]$.

The curve $C_x$ (resp. $C_y$) consists of two irreducible curves $L_{xy}$ and $R_x=\{x=t^2+y^2z=0\}$
(resp. $R_y=\{y=z^3-x^8=0\}$). The curve $L_{xy}$ intersects $R_x$ (resp. $R_y$) only at the point $O_z$ (resp. $O_t$). We have the following intersection numbers:
$$
-K_X\cdot L_{xy}=\frac{1}{12\cdot 31}, \ \ -K_X\cdot R_{x}=\frac{1}{6\cdot 19},  \ \ -K_X\cdot R_{y}=\frac{2}{3\cdot 31},\ \  L_{xy}\cdot R_{x}=\frac{1}{12},%
$$
$$
L_{xy}\cdot R_{y}=\frac{3}{31},\ \  L_{xy}^2=-\frac{53}{24\cdot 31},\ \ R_x^2=-\frac{5}{6\cdot 19},\ \ R_y^2=\frac{10}{3\cdot 31}.%
$$
Meanwhile, the curve $C_z$ is irreducible. We
see that $\lct(X)\leqslant 3$ since
$$
3=\lct\left(X, \frac{2}{9}C_x\right)<\lct\left(X, \frac{2}{19}C_y\right)=\frac{209}{54}<\lct\left(X, \frac{2}{24}C_z\right)=\frac{22}{3}.%
$$

Suppose that $\lct(X)<3$. Then there is an effective $\Q$-divisor
$D\qlineq -K_X$ such that the pair $(X,3D)$ is not log canonical at
some point $P$.
We may assume that the support of $D$ does not contain at least one component of each of $C_x$ and $C_y$ by Lemma~\ref{lemma:convexity}.  One of the inequalities
\[\mult_{O_z}D\leqslant 24D\cdot L_{xy}=\frac{2}{31}<\frac{1}{3}, \ \ \mult_{O_z}D\leqslant 24D\cdot R_{x}=\frac{4}{19}<\frac{1}{3}\] must hold, and hence the point $P$ cannot be the point $O_z$. Since the curve $R_y$ is singular at the point $O_t$ with multiplicity $3$, one of the inequalities
\[\mult_{O_t}D\leqslant 31D\cdot L_{xy}=\frac{1}{12}<\frac{1}{3}, \ \ \mult_{O_t}D\leqslant \frac{31}{3}D\cdot R_{y}=\frac{2}{9}<\frac{1}{3}\] must hold, and hence the point $P$ cannot be the point $O_t$.

By Lemma~\ref{lemma:convexity}, we may also assume
that the curve $C_z$ is not contained in the support of $D$. The curve $C_z$ is singular at the point $O_y$.
Then the inequality $$\frac{19}{2}D\cdot C_z=\frac{9}{31}<\frac{1}{3}$$ shows that the point $P$ cannot be the point $O_y$.

Write $D=m_0L_{xy}+m_1R_x+m_2R_y+\Omega$, where $\Omega$ is an effective $\mathbb{Q}$-divisor whose support contains none of $L_{xy}$, $R_x$, $R_y$.
If $m_0\ne 0$, then we obtain
\[\frac{1}{6\cdot 19}=D\cdot R_x\geqslant m_0L_{xy}\cdot R_x=\frac{m_0}{12},\]
and hence $m_0\leqslant \frac{2}{19}$. Similarly, we see that $m_1\leqslant\frac{1}{31} $ and $m_2\leqslant\frac{1}{36}$.
Since we have
\[(D-m_0L_{xy})\cdot L_{xy}=\frac{2+53m_0}{24\cdot 31}<\frac{1}{3},\]
\[(D-m_1R_{x})\cdot R_{x}=\frac{1+5m_1}{6\cdot 19}<\frac{1}{3},\]
\[3(D-m_2R_{y})\cdot R_{y}=\frac{2-10m_2}{31}<\frac{1}{3},\]
it follows from Lemma~\ref{lemma:handy-adjunction} that the point $P$ is located in the outside of $C_x$ and $C_y$. Therefore, the point $P$ is a smooth point in the outside of $C_x$ and $C_y$.
However, since $H^0(\P, \mathcal{O}_{\P}(171))$ contains the monomials
$y^{9}$, $x^{19}$, $x^{3}z^{6}$ and $x^{11}z^{3}$, it follows from
Lemma~\ref{lemma:Carolina} that
the point $P$ must be either a singular point of $X$ or a point in $C_x\cup C_y$. This is a contradiction.
\end{proof}

\begin{lemma}
\label{lemma:I-2-W-10-19-35-43-D-105}Let $X$ be a quasismooth
hypersurface of degree $105$ in $\mathbb{P}(10,19,35, 43)$. Then
$\lct(X)=\frac{57}{14}$.
\end{lemma}

\begin{proof}
  The surface $X$ can be defined by the
quasihomogeneous equation
$$
z^{3}+yt^{2}+xy^{5}-x^{7}z=0.
$$
The surface $X$ is singular at the points $O_x$, $O_y$, $O_t$ and
$Q_5=[1:0:1:0]$.

The curve $C_x$ is irreducible. However, the curve $C_y$ consists of two irreducible curves $L_{yz}$ and $R_y=\{y=z^2-x^7=0\}$. The curve $L_{yz}$ intersects $R_y$ at the point $O_t$. We have
$$
L_{yz}^2=-\frac{51}{10\cdot 43},\ R_{y}^2=-\frac{16}{5\cdot 43},\ L_{yz}\cdot R_{y}=\frac{7}{43}.%
$$
We also have $\lct(X)\leqslant \frac{57}{14}$ since
$$
\frac{57}{14}=\lct\left(X, \frac{2}{19}C_y\right)<\lct\left(X, \frac{2}{10}C_x\right)=\frac{25}{6}.%
$$

Suppose that $\lct(X)<\frac{57}{14}$. Then there is an effective
$\Q$-divisor $D\qlineq -K_X$ such that the pair
$(X,\frac{57}{14}D)$ is not log canonical at some point $P$. By
Lemma~\ref{lemma:convexity}, we may assume that the support of the
divisor $D$ does not contain the curve $C_x$. Similarly, we may
assume that the support of the divisor $D$ does not contain either
$L_{yz}$ or $R_y$.

Since the support of the divisor $D$ does not contain either
$L_{yz}$ or $R_y$ and the curve $R_y$ is singular at the point
$O_t$, one of the inequalities
\[ \mult_{O_t}(D)\leqslant 43D\cdot L_{yz}=\frac{1}{5}<\frac{14}{57},  \ \
\mult_{O_t}(D)\leqslant \frac{43}{2}D\cdot R_{y}=\frac{1}{5}<\frac{14}{57}\]
must hold, and hence the point $P$ cannot be $O_t$.

We write $D=m_0L_{yz}+m_1R_y+\Omega$, where $\Omega$ is an effective $\mathbb{Q}$-divisor whose support contains neither $L_{yz}$ nor $R_y$. If $m_0\ne 0$, then $m_1=0$ and hence
\[\frac{2}{5\cdot 43}=D\cdot R_y\geqslant m_0L_{yz}\cdot R_y=\frac{7m_0}{43}.\]
Therefore, $m_0\leqslant \frac{2}{35}$. Similarly, we have $m_1\leqslant \frac{1}{35}$.
Since
\[10(D-m_0L_{yz})\cdot L_{yz}=\frac{2+51m_0}{43}<\frac{14}{57},\]
\[5(D-m_1R_{y})\cdot R_{y}=\frac{2+16m_1}{43}<\frac{14}{57},\]
it follows from Lemma~\ref{lemma:handy-adjunction} that the point $P$ is located in the outside of $C_y$.

Since the divisor $D$ does not contain the curve $C_x$, $\mult_{O_y}(D)\leqslant 19D\cdot C_x=\frac{6}{43}<\frac{14}{57}$, and hence the point $P$ cannot belong to the curve $C_x$. Therefore, the point $P$ is a smooth point in the outside of $C_x\cup C_y$.
However,
since $H^0(\P, \mathcal{O}_{\P}(190))$ contains $x^{19}$,
$y^{10}$, $x^{5}z^{4}$ and $x^{12}z^{2}$, it follows from
Lemma~\ref{lemma:Carolina} that the point $P$ must be either a singular point of $X$ or a point in
$C_x\cup C_{y}$. This is a contradiction.
\end{proof}

\begin{lemma}
\label{lemma:I-2-W-11-21-28-47-D-105}Let $X$ be a quasismooth
hypersurface of degree $105$ in $\mathbb{P}(11,21,28, 47)$. Then
$\lct(X)=\frac{77}{30}$.
\end{lemma}

\begin{proof}
 The surface $X$ can be defined by the
quasihomogeneous equation
$$
yz^{3}-y^{5}+xt^{2}+x^{7}z=0.
$$
The surface $X$ is singular at the point $O_x$, $O_z$, $O_t$ and $Q_7=[0:1:1:0]$.

The curve $C_x$ (resp. $C_y$) consists of two irreducible curves $L_{xy}$ and $R_x=\{x=z^3-y^4=0\}$
(resp. $R_y=\{y=t^2+x^6z=0\}$. The curve $L_{xy}$ intersects $R_x$ (resp. $R_y$) only at the point $O_t$ (resp. $O_z$).
We have the following intersection numbers:
$$
-K_X\cdot L_{xy}=\frac{1}{14\cdot 47}, \ \ -K_X\cdot R_{x}=\frac{2}{7\cdot 47},\ \  -K_X\cdot R_{y}=\frac{1}{7\cdot 11}, \ \ L_{xy}\cdot R_{x}=\frac{3}{47},%
$$
$$L_{xy}\cdot R_{y}=\frac{1}{14},\ \ %
L_{xy}^2=-\frac{73}{28\cdot 47},\ \ R_{x}^2=-\frac{10}{7\cdot 47},\ \ R_y^2=\frac{5}{7\cdot 11}.
$$

We see that $\lct(X)\leqslant
\frac{77}{30}$ since
$$
\frac{77}{30}=\lct\left(X, \frac{2}{11}C_x\right)<\lct\left(X, \frac{2}{21}C_y\right)=6.%
$$

Suppose that $\lct(X)<\frac{77}{30}$. Then there is an effective
$\Q$-divisor $D\qlineq -K_X$ such that the pair $(X,\frac{77}{30}D)$ is
not log canonical at some point $P$. By
Lemma~\ref{lemma:convexity}, we may assume that the support of $D$ does not contain at least one component of each of $C_x$ and $C_y$. Note that the curve $R_x$ is singular at the point $O_t$ with multiplicity $3$ and the curve $R_y$ is singular at the point $O_z$.
Then one of two inequalities
\[ \mult_{O_t}(D)\leqslant 47D\cdot L_{xy}=\frac{1}{14}<\frac{30}{77},  \ \
\mult_{O_t}(D)\leqslant \frac{47}{3}D\cdot R_{x}=\frac{2}{21}<\frac{30}{77}\]
must hold, and hence the point $P$ cannot be $O_t$. Applying the same method to $C_y$, we show that the point $P$ cannot be the point $O_z$.

Write $D=m_0L_{xy}+m_1R_x+m_2R_y+\Omega$, where $\Omega$ is an effective $\mathbb{Q}$-divisor whose support contains none of $L_{xy}$, $R_x$, $R_y$.
If $m_0\ne 0$, then we obtain
\[\frac{2}{7\cdot 47}=D\cdot R_x\geqslant m_0L_{xy}\cdot R_x=\frac{3m_0}{47},\]
and hence $m_0\leqslant \frac{2}{21}$. Similarly, we see that $m_1\leqslant\frac{1}{42} $ and $m_2\leqslant\frac{1}{47}$.
Since we have
\[(D-m_0L_{xy})\cdot L_{xy}=\frac{2+73m_0}{28\cdot 47}<\frac{30}{77},\]
\[7(D-m_1R_{x})\cdot R_{x}=\frac{2+10m_1}{47}<\frac{30}{77},\]
\[11(D-m_2R_{y})\cdot R_{y}=\frac{1-5m_2}{7}<\frac{30}{77},\]
it follows from Lemma~\ref{lemma:handy-adjunction} that the point $P$ is located in the outside of $C_x$ and $C_y$. Therefore, the point $P$ is a smooth point in the outside of $C_x$.
However, since $H^0(\P, \mathcal{O}_{\P}(517))$ contains $x^{5}y^{22}$,
$x^{26}y^{11}$, $x^{47}$, $x^{19}z^{11}$, $x^{47}$, $t^{11}$, it follows from
Lemma~\ref{lemma:Carolina} that
the point $P$ must be either a singular point of $X$ or a point in $C_x$. This is a contradiction.
\end{proof}

\begin{lemma}
\label{lemma:I-2-W-11-25-32-41-D-107}Let $X$ be a quasismooth
hypersurface of degree $107$ in $\mathbb{P}(11,25,32, 41)$. Then
$\lct(X)=\frac{11}{3}$.
\end{lemma}

\begin{proof}
  The surface $X$ can be defined by the
quasihomogeneous equation
$$
yt^{2}+y^{3}z+xz^{3}+x^{6}t=0.
$$
The surface $X$ is singular at the points $O_x$, $O_y$, $O_z$, $O_t$. Each of the divisors $C_x$, $C_y$, $C_z$, and $C_t$ consists of two irreducible and reduced components. The divisor $C_x$ (resp. $C_y$, $C_z$, $C_t$) consists of $L_{xy}$ (resp. $L_{xy}$, $L_{zt}$, $L_{zt}$)
and $R_x=\{x=t^2+y^2z=0\}$ (resp. $R_y=\{y=z^3+x^5t=0\}$, $R_z=\{z=x^6+yt=0\}$, $R_t=\{t=y^3+xz^2=0\}$).
Also, we see that
\[L_{xy}\cap R_x=\{O_z\}, \ L_{xy}\cap R_y=\{O_t\}, \ L_{zt}\cap R_z=\{O_y\}, \ L_{zt}\cap R_t=\{O_x\}.\]
We have the following intersection numbers:
$$
-K_X\cdot L_{xy}=\frac{1}{16\cdot 41},\ \ -K_X\cdot L_{zt}=\frac{2}{11\cdot 25},\ \ -K_X\cdot R_{x}=\frac{1}{8\cdot 25},\ \ -K_X\cdot R_{y}=\frac{6}{11\cdot 41},%
$$
$$
-K_X\cdot R_{z}=\frac{12}{25\cdot 41}, \ \ -K_X\cdot R_{t}=\frac{3}{11\cdot 16},\ \ L_{xy}\cdot R_{x}=\frac{1}{16},\ L_{xy}\cdot R_{y}=\frac{3}{41},\ \ L_{zt}\cdot R_{z}=\frac{6}{25}%
$$
$$
L_{xy}^2=-\frac{71}{32\cdot 41},\ \ L_{zt}^2=-\frac{34}{11\cdot 25}, \ \ R_x^2=-\frac{7}{8\cdot 25},\ \ R_y^2=\frac{42}{11\cdot 41}%
$$

We see $\lct(X)\leqslant \frac{11}{3}$ since
$$
\frac{11}{3}=\lct\left(X, \frac{2}{11}C_x\right)<\frac{50}{9}=\lct\left(X, \frac{2}{25}C_y\right)<\frac{28}{3}=\lct\left(X, \frac{2}{32}C_z\right)<\frac{205}{18}=\lct\left(X, \frac{2}{41}C_t\right).%
$$

Suppose that $\lct(X)<\frac{11}{3}$. Then there is an effective
$\Q$-divisor $D\qlineq -K_X$ such that the pair $(X,\frac{11}{3}D)$ is
not log canonical at some point $P$. By
Lemma~\ref{lemma:convexity}, we may assume that either
$\mathrm{Supp}(D)$ does not contain at least one irreducible
component of each of $C_{x}$, $C_{y}$, $C_{z}$ and $C_{t}$.
Since the curve $R_y$ is singular at the point $O_t$ with multiplicity $3$, one of the inequalities
\[\mult_{O_t}(D)\leqslant 41D\cdot L_{xy}=\frac{1}{16}<\frac{3}{11}, \ \
\mult_{O_t}(D)\leqslant \frac{41}{3}D\cdot R_{y}=\frac{2}{11}<\frac{3}{11}\]
must hold, and hence the point $P$ cannot be the point $O_t$.
Applying the same method to each of $C_x$ and $C_t$, we can show that the point $P$ can be neither $O_z$ nor $O_x$.

Since $H^0(\P, \mathcal{O}_{\P}(352))$ contains the monomials
$x^{7}y^{11}$, $x^{32}$ and $z^{11}$, it follows from
Lemma~\ref{lemma:Carolina} that the point $P$ is either the point $O_t$ or a smooth point on $C_x$.

Write $D=m_0L_{xy}+m_1R_x+m_2L_{zt}+\Omega$, where $\Omega$ is an effective $\mathbb{Q}$-divisor
whose support contains none of $L_{xy}$, $R_x$, $R_z$.
If $m_0\ne 0$, then $m_1=0$ and hence we obtain
\[\frac{1}{8\cdot 25}=D\cdot R_x\geqslant m_0L_{xy}\cdot R_x=\frac{m_0}{16}.\]
Therefore, $m_0\leqslant \frac{2}{25}$. Similarly, we get $m_1\leqslant \frac{1}{41}$.
Since we have
\[(D-m_0L_{xy})\cdot L_{xy}=\frac{2+71m_0}{32\cdot 41}<\frac{3}{11},\]
\[(D-m_1R_{x})\cdot R_{x}=\frac{1+7m_1}{8\cdot 25}<\frac{3}{11}\]
it follows from Lemma~\ref{lemma:handy-adjunction} that the point $P$ must be the point $O_y$.

Suppose that $m_2=0$. Then the inequality
\[\mult_{O_y}(D)\leqslant25D\cdot L_{zt}=\frac{2}{11}<\frac{3}{11}\]
gives us a contradiction. Therefore, $m_2\ne 0$ and hence the curve $R_z$ is not contained in the support of $D$. Then
\[ \frac{12}{25\cdot 41}=D\cdot R_z\geqslant m_2L_{zt}\cdot R_z+\frac{\mult_{O_y}(D)-m_2}{25}>\frac{5m_2}{25}+\frac{3}{11\cdot 25},\]
and hence $m_2<\frac{9}{5\cdot 11\cdot 41}$. Since
\[25(D-m_2L_{zt})\cdot L_{zt}=\frac{2+34m_2}{11}<\frac{3}{11}\]
the pair $(X, \frac{11}{3}D)$ is log canonical at the point $O_y$ by Lemma~\ref{lemma:handy-adjunction}. This is a contradiction.
\end{proof}

\begin{lemma}
\label{lemma:I-2-W-11-25-34-43-D-111}Let $X$ be a quasismooth
hypersurface of degree $111$ in $\mathbb{P}(11,25,34,43)$. Then
$\lct(X)=\frac{33}{8}$.
\end{lemma}

\begin{proof}
We may assume that the surface $X$ is defined by the quasihomogeneous equation
\[t^2y+tz^2+xy^4+x^7z=0.\]
The surface $X$ is singular at the points $O_x$, $O_y$, $O_z$, $O_t$. Each of the divisors $C_x$, $C_y$, $C_z$, and $C_t$ consists of two irreducible and reduced components. The divisor $C_x$ (resp. $C_y$, $C_z$, $C_t$) consists of $L_{xt}$ (resp. $L_{yz}$, $L_{yz}$, $L_{xt}$)
and $R_x=\{x=yt+z^2=0\}$ (resp. $R_y=\{y=zt+x^7=0\}$, $R_z=\{z=xy^3+t^2=0\}$, $R_t=\{t=y^4+x^6z=0\}$).
Also, we see that
\[L_{xt}\cap R_x=\{O_y\}, \ L_{yz}\cap R_y=\{O_t\}, \ L_{yz}\cap R_z=\{O_x\}, \ L_{xt}\cap R_t=\{O_z\}.\]
The intersection numbers among the divisors $D$, $L_{xt}$, $L_{yz}$, $R_x$, $R_y$, $R_z$, $R_t$ are as follows:

\[-K_X\cdot L_{xt}=\frac{1}{17\cdot 25}, \ \ -K_X\cdot R_x=\frac{4}{25\cdot 43}, \ \ -K_X\cdot R_y=\frac{7}{17\cdot 43}, \]
\[-K_X\cdot L_{yz}=\frac{2}{11\cdot 43}, \ \ -K_X\cdot R_z=\frac{4}{11\cdot 25}, \ \ -K_X\cdot R_t=\frac{4}{11\cdot 17}, \]
\[L_{xt}\cdot R_x=\frac{2}{25}, \ \ L_{yz}\cdot R_y=\frac{7}{43}, \ \ L_{yz}\cdot R_z=\frac{2}{11}, \ \ L_{xt}\cdot R_t=\frac{2}{17},\]

\[L_{xt}^2=-\frac{57}{34\cdot 25}, \ \ R_x^2=-\frac{64}{25\cdot 43}, \ \ R_y^2=-\frac{63}{34\cdot 43},\]

\[L_{yz}^2=-\frac{52}{11\cdot 43}, \ \ R_z^2=\frac{18}{11\cdot 25}, \ \ R_t^2=\frac{64}{11\cdot 17}.\]

We can easily see that $\lct(X, \frac{2}{11}C_x)=\frac{33}{8}$
is less than each of the numbers $\lct(X, \frac{2}{25}C_y)$,
$\lct(X, \frac{2}{34}C_z)$ and $\lct(X, \frac{2}{43}C_t)$.
Therefore, $\lct(X)\leq\frac{33}{8}$.

Suppose that
$\lct(X)<\frac{33}{8}$. Then there is an effective $\Q$-divisor
$D\qlineq -K_X$ such that the log pair $(X, \frac{33}{8}D)$ is not
log canonical at some point $P\in X$.

By Lemma~\ref{lemma:convexity} we may assume that the support of $D$ does not contain at least one component of each divisor $C_x$, $C_y$, $C_z$, $C_t$.
The inequalities
\[25D\cdot L_{xt}=\frac{1}{17}<\frac{8}{33}, \ \ \  25D\cdot R_x=\frac{4}{43}<\frac{8}{33}\]
imply that $P\ne O_y$.
The inequalities
\[11D\cdot L_{yz}=\frac{2}{43}<\frac{8}{33}, \ \ \  11D\cdot R_z=\frac{4}{25}<\frac{8}{33}\]
imply that $P\ne O_x$.
Since the curve $R_t$ is singular at the point $O_z$, the inequalities
\[34D\cdot L_{xt}=\frac{34}{17\cdot 25}<\frac{8}{33}, \ \ \  \frac{34}{4}D\cdot R_t=\frac{2}{11}<\frac{8}{33}\]
imply that $P\ne O_z$.

We write $D=a_1L_{xt}+a_2L_{yz}+a_3R_x+a_4R_y+a_5R_z+a_6R_t+\Omega$, where  $\Omega$ is an effective divisor whose support contains none of the curves $L_{xt}$, $L_{yz}$, $R_x$, $R_y$, $R_z$, $R_t$. Since the pair $(X, \frac{33}{8}D)$ is log canonical at the points $O_x$, $O_y$, $O_z$, the numbers $a_i$ are at most $\frac{8}{33}$. Then by Lemma~\ref{lemma:handy-adjunction} the following inequalities enable us to conclude that either the point $P$ is in the outside of $C_x\cup C_y\cup C_z\cup C_t$ or $P=O_t$:
\[\frac{33}{8}D\cdot L_{xt}-L_{xt}^2=\frac{261}{8\cdot 17\cdot 25}<1, \ \ \ \frac{33}{8}D\cdot L_{yz}-L_{xt}^2=\frac{241}{4\cdot 11\cdot 43}<1, \]

\[\frac{33}{8}D\cdot R_x-R_x^2=\frac{161}{2\cdot 25\cdot 43}<1, \ \ \ \frac{33}{8}D\cdot R_y-R_y^2=\frac{483}{4\cdot 34\cdot 43}<1,\]

\[ \frac{33}{8}D\cdot R_z-R_z^2\leq \frac{33}{8}D\cdot R_z=\frac{3}{2\cdot 25}<1, \ \ \ \frac{33}{8}D\cdot R_t-R_t^2\leq \frac{33}{8}D\cdot R_t=\frac{3}{34}<1.\]
Suppose that $P\ne O_t$. Then we consider the pencil $\mathcal{L}$ defined by $\lambda yt+\mu z^2=0$, $[\lambda:\mu]\in\P^1$. The base locus of the pencil consists of the curve $L_{yz}$ and the point $O_y$.
Let $E$ be the unique divisor in $\mathcal{L}$ that passes through the point $P$. Since $P\not\in C_x\cup C_y\cup C_z\cup C_t$, the divisor $E$ is defined by the equation $z^2=\alpha yt$, where $\alpha\ne 0$.

Suppose that $\alpha\ne -1$. Then the curve $E$ is isomorphic to the curve defined by the equations $yt=z^2$ and $t^2y+xy^4+x^7z=0$. Since the curve $E$ is isomorphic to a general curve in $\mathcal{L}$, it is smooth at the point $P$. The affine piece of $E$ defined by $t\ne 0$ is the curve given by $z(z^3+xz^7+x^7)=0$. Therefore, the divisor $E$ consists of two irreducible and reduced curves $L_{yz}$ and $C$. We have
\[D\cdot C=D\cdot E-D\cdot L_{yz}=\frac{394}{11\cdot 25\cdot 43}.\]
Also, we see
\[C^2=E\cdot C- C\cdot L_{yz}\geqslant E\cdot C- (L_{yz}+R_y)\cdot C=\frac{43}{2}D\cdot C>0.\]
By Lemma~\ref{lemma:handy-adjunction} the inequality $D\cdot C<\frac{8}{33}$ gives us a contradiction.

Suppose that $\alpha=-1$. Then divisor $E$ consists of three irreducible and reduced curves $L_{yz}$, $R_x$, and $M$.  Note that the curve $M$ is different from the curves $R_y$ and $L_{xt}$. Also, it is smooth at the point $P$. We have
\[D\cdot M=D\cdot E- D\cdot L_{yz}-D\cdot R_x=\frac{14}{11\cdot 43},\]
\[M^2=E\cdot M-L_{yz}\cdot M- R_x\cdot M \geq E\cdot M-C_y\cdot M- C_x\cdot M>0.\]
By Lemma~\ref{lemma:handy-adjunction} the inequality $D\cdot M<\frac{8}{33}$ gives us a contradiction.
Therefore, $P=O_t$.

Put $D=bR_{x}+\Delta$, where $\Delta$ is an effective divisor
whose support does not contain $R_x$. By
Lemma~\ref{lemma:convexity}, we may assume that
$R_{x}\not\subseteq\mathrm{Supp}(\Delta)$ if $b>0$. Thus, if
$b>0$, then
$$
\frac{2}{25\cdot 34}=D\cdot L_{xt}\geqslant b R_x \cdot L_{xt}=\frac{2b}{25},%
$$
and hence $b\leqslant \frac{1}{34}$. On the other hand, it
follows from Lemma~\ref{lemma:handy-adjunction} that
$$
\frac{4+64b}{25\cdot 43}=\Delta\cdot R_{x}>\frac{8}{33\cdot 43}.
$$
Therefore, $b>\frac{17}{528}$. Since $\frac{17}{528}>\frac{1}{34}$, this is a
contradiction.
\end{proof}

\begin{lemma}
\label{lemma:I-2-W-11-43-61-113-D-226}Let $X$ be a quasismooth
hypersurface of degree $226$ in $\mathbb{P}(11,43,61,113)$. Then
$\lct(X)=\frac{55}{12}$.
\end{lemma}

\begin{proof}
 The surface $X$ can be defined by the
quasihomogeneous equation
$$
t^{2}+yz^{3}+xy^{5}+x^{15}z=0.
$$
The surface $X$ is singular at the points $O_x$, $O_{y}$ and
$O_z$. The curves $C_x$ and $C_y$ are irreducible. We have
$$
\frac{55}{12}=\lct\left(X, \frac{2}{11}C_x\right)<\lct\left(X, \frac{2}{43}C_y\right)=\frac{17\cdot 43}{60}.%
$$
Therefore, $\lct(X)\leqslant \frac{55}{12}$.

Suppose that $\lct(X)<\frac{55}{12}$. Then there is an effective
$\Q$-divisor $D\qlineq -K_X$ such that the pair $(X,\frac{55}{12}D)$ is
not log canonical at some point $P$. By
Lemma~\ref{lemma:convexity}, we may assume that the support of
the divisor $D$ contains neither $C_x$ nor $C_y$. Then the inequalities
\[61D\cdot C_x=\frac{4}{43}<\frac{12}{55}, \ \ 11D\cdot C_y=\frac{4}{61}< \frac{12}{55}\]
show that the point $P$ must be a smooth point of $X$ in the outside of $C_x$. However,
since $H^0(\P, \mathcal{O}_{\P}(671))$ contains the monomials
$x^{18}y^{11}$, $x^{61}$ and $z^{11}$, it follows from
Lemma~\ref{lemma:Carolina} that
the point $P$ is either a singular point of $X$ or a point on $C_x$. This is a contradiction.
\end{proof}

\begin{lemma}
\label{lemma:I-2-W-13-18-45-61-D-135}Let $X$ be a quasismooth
hypersurface of degree $135$ in $\mathbb{P}(13,18,45,61)$. Then
$\lct(X)=\frac{91}{30}$.
\end{lemma}

\begin{proof}
 The surface $X$ can be defined by the
quasihomogeneous equation
$$
z^{3}-y^{5}z+xt^{2}+x^{9}y=0.
$$
The surface $X$ is singular at the points $O_x$, $O_y$, $O_t$, $Q_9=[0:1:1:0]$.

The curve $C_x$  consists of two irreducible and reduced curves $L_{xz}$ and $R_x=\{x=z^2-y^5=0\}$. The curve $L_{xz}$ intersects $R_x$ at the point $O_t$. It is easy to check
$$
L_{xz}^2=-\frac{77}{18\cdot 61},\ R_x^2=-\frac{32}{9\cdot 61},\ L_{xz}\cdot R_{x}=\frac{5}{61}.%
$$
 Meanwhile, the curve $C_y$ is irreducible.  We have
$$
\frac{91}{30}=\lct\left(X, \frac{2}{13}C_x\right)<\lct\left(X, \frac{2}{18}C_y\right)=\frac{15}{2}.%
$$
Therefore, $\lct(X)\leqslant \frac{91}{30}$.

Suppose that $\lct(X)<\frac{91}{30}$. Then there is an effective
$\Q$-divisor $D\qlineq -K_X$ such that the pair $(X,\frac{91}{30}D)$ is
not log canonical at some point $P$. By
Lemma~\ref{lemma:convexity}, we may assume that the support of
the divisor $D$ does not contain the curve $C_y$. Similarly, we
may assume that either $L_{xz}\not\subset\mathrm{Supp}(D)$ or
$R_{x}\not\subset\mathrm{Supp}(D)$.

Since  the support of $D$ cannot contain either $L_{xz}$ or $R_x$
one of the inequalities
\[\mult_{O_t}(D)\leqslant 61D\cdot L_{xz}=\frac{1}{9}<\frac{30}{91}, \ \
\mult_{O_t}(D)\leqslant 61D\cdot R_{x}=\frac{2}{9}<\frac{30}{91}\]
must hold, and hence the point $P$ cannot be $O_t$. Also, the inequality
\[13D\cdot C_y=\frac{6}{61}<\frac{30}{91}\]
implies that the point $P$ cannot be $O_x$.

We write $D=m_0L_{xz}+m_1R_x+\Omega$, where $\Omega$ is an effective $\mathbb{Q}$-divisor whose support contains neither $L_{xz}$ nor $R_x$.
If $m_0\ne 0$, then we obtain
\[\frac{2}{9\cdot 61}=D\cdot R_x\geqslant m_0L_{xz}\cdot R_x=\frac{5m_0}{61}\]
and hence $m_0\leqslant \frac{2}{45}$. By the same way, we get $m_1\leqslant \frac{1}{45}$.
Since
\[18(D-m_0L_{xz})\cdot L_{xz} =\frac{2+77m_0}{61}<\frac{30}{91}, \ \ 9(D-m_1R_{x})\cdot R_{x} =\frac{2+32m_0}{61}<\frac{30}{91}\]
it follows from
Lemma~\ref{lemma:handy-adjunction} that the point $P$ is a smooth point in the outside of $C_x$.
However, since $H^0(\P, \mathcal{O}_{\P}(585))$ contains $x^{45}$,
$x^{27}y^{13}$, $z^{13}$,
 this is impossible by Lemma~\ref{lemma:Carolina}.
\end{proof}

\begin{lemma}
\label{lemma:I-2-W-13-20-29-47-D-107}Let $X$ be a quasismooth
hypersurface of degree $107$ in $\mathbb{P}(13,20,29,47)$. Then
$\lct(X)=\frac{65}{18}$.
\end{lemma}

\begin{proof}
  The surface $X$ can be defined by the
quasihomogeneous equation
$$
yz^{3}+y^{3}t+xt^{2}+x^{6}z=0.
$$
The surface $X$ is singular at the points $O_x$, $O_{y}$, $O_z$ and $O_t$.
 Each of the divisors $C_x$, $C_y$, $C_z$, and $C_t$ consists of two irreducible and reduced components. The divisor $C_x$ (resp. $C_y$, $C_z$, $C_t$) consists of $L_{xy}$ (resp. $L_{xy}$, $L_{zt}$, $L_{zt}$)
and $R_x=\{x=z^3+y^2t=0\}$ (resp. $R_y=\{y=t^2+x^5z=0\}$, $R_z=\{z=y^3+xt=0\}$, $R_t=\{t=x^6+yz^2=0\}$). The curve $L_{xy}$ intersects $R_x$ (resp. $R_y$) only at the point $O_t$ (resp. $O_z$). Also, the curve $L_{zt}$ intersects $R_z$ (resp. $R_t$) only at the point $O_x$ (resp. $O_y$).
It is easy to check

$$-K_X\cdot L_{xy}=\frac{2}{29\cdot 47},\ \ -K_X\cdot L_{zt}=\frac{1}{13\cdot 10},\ \ -K_X\cdot R_{x}=\frac{3}{10\cdot 47},$$
$$-K_X\cdot R_{y}=\frac{4}{13\cdot 29},\ \ -K_X\cdot R_{z}=\frac{6}{13\cdot 47},\ \ -K_X\cdot R_{t}=\frac{3}{5\cdot 29},$$
$$
L_{xy}^2=-\frac{74}{29\cdot 47},\ \ R_{x}^2=-\frac{21}{20\cdot 47},\ \ L_{xy}\cdot R_{x}=\frac{3}{47}.%
$$
We see $\lct(X)\leqslant
\frac{65}{18}$ since
$$
\frac{65}{18}=\lct\left(X, \frac{2}{13}C_x\right)<\frac{70}{12}=\lct\left(X, \frac{2}{20}C_y\right)<\frac{29}{3}=\lct\left(X, \frac{2}{29}C_z\right)<\frac{94}{9}=\lct\left(X, \frac{2}{47}C_t\right).%
$$

Suppose that $\lct(X)<\frac{65}{18}$. Then there is an effective
$\Q$-divisor $D\qlineq -K_X$ such that the pair
$(X,\frac{65}{18}D)$ is not log canonical at some point $P$. By
Lemma~\ref{lemma:convexity}, we may assume that the support of $D$
 does not contain at least one irreducible
component of each of the curves $C_{x}$, $C_{y}$, $C_{z}$ and
$C_{t}$. The curve $R_x$ (resp. $R_y$, $R_t$) is singular at the point $O_t$ (resp. $O_z$, $O_y$).
Then in each of the following pairs of inequalities, at least one of two must hold:
\[\mult_{O_t}(D)\leqslant 47D\cdot L_{xy}=\frac{2}{29}<\frac{18}{65}, \ \
\mult_{O_t}(D)\leqslant \frac{47}{2}D\cdot R_x=\frac{3}{20}<\frac{18}{65};\]
\[\mult_{O_z}(D)\leqslant 29D\cdot L_{xy}=\frac{2}{47}<\frac{18}{65}, \ \
\mult_{O_z}(D)\leqslant \frac{29}{2}D\cdot R_y=\frac{2}{13}<\frac{18}{65};\]
\[\mult_{O_x}(D)\leqslant 13D\cdot L_{zt}=\frac{1}{10}<\frac{18}{65}, \ \
\mult_{O_x}(D)\leqslant 13D\cdot R_z=\frac{6}{47}<\frac{18}{65};\]
\[\mult_{O_y}(D)\leqslant 20D\cdot L_{zt}=\frac{2}{13}<\frac{18}{65}, \ \
\mult_{O_y}(D)\leqslant \frac{20}{2}D\cdot R_t=\frac{6}{29}<\frac{18}{65}.\]
Therefore, the point $P$ must be a smooth point of $X$.

We write $D=m_0L_{xy}+m_1R_x+\Omega$, where $\Omega$ is an effective $\mathbb{Q}$-divisor
whose support contains none of $L_{xy}$, $R_x$.
If $m_0\ne 0$, then $m_1=0$ and hence we obtain
\[\frac{3}{10\cdot 47}=D\cdot R_x\geqslant m_0L_{xy}\cdot R_x=\frac{3m_0}{47}.\]
Therefore, $m_0\leqslant \frac{1}{10}$. Similarly, we get $m_1\leqslant \frac{2}{87}$.
Since
\[(D-m_0L_{xy})\cdot L_{xy}=\frac{2+74m_0}{29\cdot 47}<\frac{18}{65},\]
\[(D-m_1R_{x})\cdot R_{x}=\frac{6+21m_1}{20\cdot 47}<\frac{18}{65}\]
it follows from Lemma~\ref{lemma:handy-adjunction} that the point $P$ is a smooth point in the outside of $C_x$. However,
since $H^0(\P, \mathcal{O}_{\P}(377))$ contains the monomials
$x^{9}y^{13}$, $x^{29}$ and $z^{13}$, this is impossible by
Lemma~\ref{lemma:Carolina}.
\end{proof}

\begin{lemma}
\label{lemma:I-2-W-13-20-31-49-D-111}Let $X$ be a quasismooth
hypersurface of degree $111$ in $\mathbb{P}(13,20,31,49)$. Then
$\lct(X)=\frac{65}{16}$.
\end{lemma}

\begin{proof}
  The surface $X$ can be defined by the
quasihomogeneous equation
$$
z^{2}t+y^{4}z+xt^{2}+x^{7}y=0.
$$
It is singular at the point $O_x$, $O_{y}$, $O_z$ and $O_t$. Each of the divisors $C_x$, $C_y$, $C_z$, and $C_t$ consists of two irreducible and reduced components. The divisor $C_x$ (resp. $C_y$, $C_z$, $C_t$) consists of $L_{xz}$ (resp. $L_{yt}$, $L_{xz}$, $L_{yt}$)
and $R_x=\{x=y^4+zt=0\}$ (resp. $R_y=\{y=z^2+xt=0\}$, $R_z=\{z=t^2+x^6y=0\}$, $R_t=\{t=x^7+y^3z=0\}$). The curve $L_{xz}$ intersects $R_x$ (resp. $R_z$) only at the point $O_t$ (resp. $O_y$). Also, the curve $L_{yt}$ intersects $R_y$ (resp. $R_t$) only at the point $O_x$ (resp. $O_z$).
It is easy to check
$$-K_X\cdot L_{xz}=\frac{1}{10\cdot 49},\ -K_X\cdot L_{yt}=\frac{2}{13\cdot 31},\ -K_X\cdot R_{x}=\frac{8}{31\cdot 49},%
$$
$$
 -K_X\cdot R_{y}=\frac{4}{13\cdot 49},\ -K_X\cdot R_{z}=\frac{1}{5\cdot 13},\
 -K_X\cdot R_{t}=\frac{7}{10\cdot 31},%
$$
$$
L_{xz}^2=-\frac{67}{20\cdot 49},\  R_x^2=-\frac{72}{31\cdot 49},\ L_{xz}\cdot R_{x}=\frac{4}{49}.
$$

We have $\lct(X)\leqslant \frac{65}{16}$ since
$$
\frac{65}{16}=\lct\left(X, \frac{2}{13}C_x\right)<\frac{30}{4}=\lct\left(X, \frac{2}{20}C_y\right)<\frac{245}{28}=\lct\left(X, \frac{2}{49}C_t\right)<\frac{62}{7}=\lct\left(X, \frac{2}{31}C_z\right).%
$$

Suppose that $\lct(X)<\frac{65}{16}$. Then there is an effective
$\Q$-divisor $D\qlineq -K_X$ such that the pair $(X,\frac{65}{16}D)$ is
not log canonical at some point $P$. By
Lemma~\ref{lemma:convexity}, we may assume that the support of $D$
 does not contain at least one irreducible
component of each of the curves $C_{x}$, $C_{y}$, $C_{z}$ and
$C_{t}$. The curve $R_z$ is singular at the point $O_y$. The curve
$R_t$ is singular at $O_z$ with multiplicity~$3$. Then in each of
the following pairs of inequalities, at least one of two must
hold:
\[\mult_{O_x}(D)\leqslant 13D\cdot L_{yt}=\frac{2}{31}<\frac{16}{65}, \ \
\mult_{O_x}(D)\leqslant 13D\cdot R_y=\frac{4}{49}<\frac{16}{65};\]
\[\mult_{O_y}(D)\leqslant 20D\cdot L_{xz}=\frac{2}{49}<\frac{16}{65}, \ \
\mult_{O_y}(D)\leqslant \frac{20}{2}D\cdot R_z=\frac{2}{13}<\frac{16}{65};\]
\[\mult_{O_z}(D)\leqslant 31D\cdot L_{yt}=\frac{2}{13}<\frac{16}{65}, \ \
\mult_{O_z}(D)\leqslant \frac{31}{3}D\cdot R_t=\frac{7}{30}<\frac{16}{65}.\]
Therefore, the point $P$ can be none of $O_x$, $O_y$, $O_z$.

Since $H^0(\P, \mathcal{O}_{\P}(403))$ contains the monomials
$x^{11}y^{13}$, $x^{31}$ and $z^{13}$, it follows from
Lemma~\ref{lemma:Carolina} that the point $P$ is either the point $O_t$ or a smooth point of $X$ in $C_x$.

Write $D=m_0L_{xz}+m_1R_x+\Omega$, where $\Omega$ is an effective $\mathbb{Q}$-divisor
whose support contains none of $L_{xz}$, $R_x$.
If $m_0\ne 0$, then $m_1=0$ and hence we obtain
\[\frac{8}{31\cdot 49}=D\cdot R_x\geqslant m_0L_{xz}\cdot R_x=\frac{4m_0}{49}.\]
Therefore, $m_0\leqslant \frac{2}{31}$. Similarly, we get $m_1\leqslant \frac{1}{40}$.
Since we have
\[(D-m_0L_{xz})\cdot L_{xz}=\frac{2+67m_0}{20\cdot 49}<\frac{16}{65},\]
\[(D-m_1R_{x})\cdot R_{x}=\frac{8+72m_1}{31\cdot 49}<\frac{16}{65}\]
it follows from Lemma~\ref{lemma:handy-adjunction} that the point $P$ must be the point $O_t$.

Suppose that $m_0=0$. Then the inequality
\[\mult_{O_t}(D)\leqslant49D\cdot L_{xz}=\frac{1}{10}<\frac{16}{65}\]
gives us a contradiction. Therefore, $m_0\ne 0$ and hence the curve $R_x$ is not contained in the support of $D$. Then
\[ \frac{8}{31\cdot 49}=D\cdot R_x\geqslant m_0L_{xz}\cdot R_x+\frac{\mult_{O_t}(D)-m_0}{49}>\frac{3m_0}{49}+\frac{16}{65\cdot 49},\]
and hence $m_0<\frac{8}{31\cdot 65}$. Since
\[49(D-m_0L_{xz})\cdot L_{xz}=\frac{2+67m_0}{20}<\frac{16}{65}\]
the pair $(X, \frac{65}{16}D)$ is log canonical at the point $O_t$ by Lemma~\ref{lemma:handy-adjunction}. This is a contradiction.
\end{proof}

\begin{lemma}
\label{lemma:I-2-W-13-31-71-113-D-226}
Let $X$ be a quasismooth
hypersurface of degree $226$ in $\mathbb{P}(13,31,71,113)$. Then
$\lct(X)=\frac{91}{20}$.
\end{lemma}

\begin{proof}
The surface $X$ can be defined by the quasihomogeneous equation
$$
t^{2}+y^{5}z+xz^{3}+x^{15}y=0.
$$
It is singular at the points $O_x$, $O_{y}$ and
$O_z$. The curves $C_x$ and $C_y$ are irreducible. We have
$$
\frac{91}{20}=\lct\left(X, \frac{2}{13}C_x\right)
<\lct\left(X, \frac{2}{31}C_y\right)=\frac{155}{12}.%
$$
Therefore,  $\lct(X)\leqslant \frac{91}{20}$.

Suppose that $\lct(X)<\frac{91}{20}$. Then there is an effective
$\Q$-divisor $D\qlineq -K_X$ such that the pair $(X,\frac{91}{20}D)$ is
not log canonical at some point $P$. By
Lemma~\ref{lemma:convexity}, we may assume that the support of
the divisor $D$  contains neither  $C_x$ nor $C_y$. Then the inequalities
\[71D\cdot C_x=\frac{4}{31}<\frac{20}{91}, \ \ 13D\cdot C_y=\frac{4}{71}<\frac{20}{91}\]
show that the point $P$ is a smooth point in the outside of $C_x$.
However,
since $H^0(\P, \mathcal{O}_{\P}(923))$ contains $x^{71}$,
$y^{26}x^{9}$, $y^{13}x^{40}$ and $z^{13}$, it follows from
Lemma~\ref{lemma:Carolina} that the point $P$ is either a singular point of $X$ or a point on $C_x$. This is a contradiction.
\end{proof}

\begin{lemma}
\label{lemma:I-2-W-14-17-29-41-D-99}
Let $X$ be a quasismooth
hypersurface of degree $99$ in $\mathbb{P}(14,17,29,41)$. Then
$\lct(X)=\frac{51}{10}$.
\end{lemma}

\begin{proof}
We may assume that the surface $X$ is defined by the quasihomogeneous equation
\[t^2y+tz^2+xy^5+x^5z=0.\]
The surface $X$ is singular at the points $O_x$, $O_y$, $O_z$, $O_t$. Each of the divisors $C_x$, $C_y$, $C_z$, and $C_t$ consists of two irreducible and reduced components. The divisor $C_x$ (resp. $C_y$, $C_z$, $C_t$) consists of $L_{xt}$ (resp. $L_{yz}$, $L_{yz}$, $L_{xt}$)
and $R_x=\{x=yt+z^2=0\}$ (resp. $R_y=\{y=zt+x^5=0\}$ , $R_z=\{z=xy^4+t^2=0\}$ , $R_t=\{t=y^5+x^4z=0\}$ ).
Also, we see that
\[L_{xt}\cap R_x=\{O_y\}, \ L_{yz}\cap R_y=\{O_t\}, \ L_{yz}\cap R_z=\{O_x\}, \ L_{xt}\cap R_t=\{O_z\}.\]
We can easily check that $\lct(X, \frac{2}{17}C_y)=\frac{51}{10}$ is less than
each of the numbers $\lct(X, \frac{2}{14}C_y)$, $\lct(X, \frac{2}{29}C_z)$
and $\lct(X, \frac{2}{41}C_t)$.
Therefore, $\lct(X)\leq\frac{51}{10}$. Suppose
$\lct(X)<\frac{51}{10}$. Then, there is an effective $\Q$-divisor
$D\qlineq -K_X$ such that the log pair $(X, \frac{51}{10}D)$ is not
log canonical at some point $P\in X$.

The intersection numbers among the divisors $D$, $L_{xt}$, $L_{yz}$, $R_x$, $R_y$, $R_z$, $R_t$ are as follows:

\[D\cdot L_{xt}=\frac{2}{17\cdot 29}, \ \ D\cdot R_x=\frac{4}{17\cdot 41}, \ \ D\cdot R_y=\frac{10}{29\cdot 41}, \]
\[D\cdot L_{yz}=\frac{1}{7\cdot 41}, \ \ D\cdot R_z=\frac{2}{7\cdot 17}, \ \ D\cdot R_t=\frac{5}{7\cdot 29}, \]
\[L_{xt}\cdot R_x=\frac{2}{17}, \ \ L_{yz}\cdot R_y=\frac{5}{41}, \ \ L_{yz}\cdot R_z=\frac{1}{7}, \ \ L_{xt}\cdot R_t=\frac{5}{29},\]

\[L_{xt}^2=-\frac{44}{17\cdot 29}, \ \ R_x^2=-\frac{54}{17\cdot 41}, \ \ R_y^2=-\frac{60}{29\cdot 41},\]

\[L_{yz}^2=-\frac{53}{14\cdot 41}, \ \ R_z^2=\frac{12}{7\cdot 17}, \ \ R_t^2=\frac{135}{14\cdot 29}.\]
By Lemma~\ref{lemma:convexity} we may assume that the support of $D$ does not contain at least one component of each divisor $C_x$, $C_y$, $C_z$, $C_t$.
The inequalities
\[17D\cdot L_{xt}=\frac{2}{29}<\frac{10}{51}, \ \ \  17D\cdot R_x=\frac{4}{41}<\frac{10}{51}\]
imply that $P\ne O_y$.
The inequalities
\[14D\cdot L_{yz}=\frac{2}{41}<\frac{10}{51}, \ \ \  7D\cdot R_z=\frac{2}{17}<\frac{10}{51}\]
imply that $P\ne O_x$. The curve $R_z$ is singular at the point $O_x$.
The inequalities
\[29D\cdot L_{xt}=\frac{2}{17}<\frac{10}{51}, \ \ \  \frac{29}{4}D\cdot R_t=\frac{5}{28}<\frac{10}{51}\]
imply that $P\ne O_z$. The curve $R_t$ is singular at the point $O_z$.

We write $D=m_0L_{xt}+m_1L_{yz}+m_2R_x+m_3R_y+m_4R_z+m_5R_t+\Omega$, where  $\Omega$ is an effective divisor whose support contains none of the curves $L_{xt}$, $L_{yz}$, $R_x$, $R_y$, $R_z$, $R_t$. Since the pair $(X, \frac{51}{10}D)$ is log canonical at the points $O_x$, $O_y$, $O_z$, the numbers $m_i$ are at most $\frac{10}{51}$. Then by Lemma~\ref{lemma:handy-adjunction} the following inequalities enable us to conclude that either the point $P$ is in the outside of $C_x\cup C_y\cup C_z\cup C_t$ or $P=O_t$:
\[(D-m_0L_{xt})\cdot L_{xz}=\frac{2+44m_0}{17\cdot 29}\leqslant \frac{10}{51},\ \
(D-m_1L_{yz})\cdot L_{yt}=\frac{2+53m_1}{14\cdot 41}\leqslant \frac{10}{51},\]
\[(D-m_2R_{x})\cdot R_{x}=\frac{4+54m_2}{17\cdot 41}\leqslant \frac{10}{51},\ \
(D-m_3R_{y})\cdot R_{y}=\frac{10+60m_3}{29\cdot 41}\leqslant \frac{10}{51},\]
\[(D-m_4R_{z})\cdot R_{z}=\frac{2-12m_4}{7\cdot 17}\leqslant \frac{10}{51},\ \
(D-m_5R_{t})\cdot R_{t}=\frac{10-135m_5}{14\cdot 29}\leqslant \frac{10}{51}.\]

Suppose that $P\ne O_t$. Then we consider the pencil $\mathcal{L}$ on $X$
defined by $\lambda yt+\mu z^2=0$, $[\lambda:\mu]\in\P^1$. The
base locus of the pencil $\mathcal{L}$ consists of the curve
$L_{yz}$ and the point $O_y$. Let $E$ be the unique divisor in
$\mathcal{L}$ that passes through the point $P$. Since $P\not\in
C_x\cup C_y\cup C_z\cup C_t$, the divisor $E$ is defined by the
equation $z^2=\alpha yt$, where $\alpha\ne 0$.

Suppose that $\alpha\ne -1$. Then the curve $E$ is isomorphic to the curve defined by the equations $yt=z^2$ and $t^2y+xy^5+x^5z=0$. Since the curve $E$ is isomorphic to a general curve in $\mathcal{L}$, it is smooth at the point $P$. The affine piece of $E$ defined by $t\ne 0$ is the curve given by $z(z+xz^9+x^5)=0$. Therefore, the divisor $E$ consists of two irreducible and reduced curves $L_{yz}$ and $C$. We have the intersection number
\[D\cdot C=D\cdot E-D\cdot L_{yz}=\frac{181}{7\cdot 17\cdot 41}.\]
Also, we see
\[C^2=E\cdot C- C\cdot L_{yz}\geq E\cdot C-C_y\cdot C>0\]
since $C$ is different from $R_y$.
By Lemma~\ref{lemma:handy-adjunction} the inequality $D\cdot C<\frac{10}{51}$ gives us a contradiction.

Suppose that $\alpha=-1$. Then divisor $E$ consists of three irreducible and reduced curves $L_{yz}$, $R_x$, and $M$.  Note that the curve $M$ is different from the curves $R_y$ and $L_{xt}$. Also, it is smooth at the point $P$. We have
\[D\cdot M=D\cdot E- D\cdot L_{yz}-D\cdot R_x=\frac{153}{7\cdot 17\cdot 41},\]
\[M^2=E\cdot M-L_{yz}\cdot M- R_x\cdot M \geq E\cdot M-C_y\cdot M- C_x\cdot M>0.\]
By Lemma~\ref{lemma:handy-adjunction} the inequality $D\cdot M<\frac{10}{51}$ gives us a contradiction.
Therefore, the log pair $(X,\frac{51}{10}D)$ is not log canonical at the point $O_t$.

Put $D=aL_{yz}+bR_{x}+\Delta$, where $\Delta$ is an effective
$\Q$-divisor whose support contains neither $L_{yz}$ nor $R_x$.
Then $a>0$ since otherwise we would have a contradictory inequality
$$
\frac{1}{7}=41D\cdot
L_{yz}\geqslant \mult_{O_t}(D)>\frac{10}{51}.
$$
Therefore, we may assume that
$R_{y}\not\subset\mathrm{Supp}(\Delta)$ by
Lemma~\ref{lemma:convexity}. Similarly, we may assume that
$L_{xt}\not\subset\mathrm{Supp}(\Delta)$ if $b>0$.

If $b>0$, then
$$
\frac{2}{17\cdot 29}=D\cdot L_{xt}\geqslant b R_x \cdot L_{xt}=\frac{2b}{17},%
$$
and hence $b\leqslant \frac{1}{29}$. Similarly, we have
$$
\frac{10}{29\cdot 41}=D\cdot R_y\geqslant \frac{5a}{41}+\frac{b}{41}+\frac{\mult_{O_t}(D)-a-b}{41}>\frac{4a}{41}+\frac{4}{21\cdot 41}.%
$$
Therefore, $a<\frac{47}{2\cdot 21\cdot 29}$.

Let $\pi\colon\bar{X}\to X$ be the weighted blow up at the point
$O_t$ with weights $(9,4)$ and let $F$ be the exceptional curve
of the morphism $\pi$. Then $F$ contains two singular points
$Q_{9}$ and $Q_{4}$ of $\bar{X}$ such that $Q_{9}$ is a singular point of type
$\frac{1}{9}(1,1)$, and $Q_{4}$ is a singular point of type
$\frac{1}{4}(3,1)$. Then
\[
K_{\bar{X}}\qlineq\pi^*(K_X)-\frac{28}{41}F,\ \
\bar{L}_{yz}\qlineq\pi^*(L_{yz})-\frac{4}{41}F, \ \
\bar{R}_x\qlineq\pi^*(R_x)-\frac{9}{41}F,\ \
\bar{\Delta}\qlineq\pi^*(\Delta)-\frac{c}{41}F,
\]
where $\bar{L}_{yz}$, $\bar{R}_x$ and $\bar{\Delta}$
are the proper transforms of $L_{yz}$, $R_x$ and $\Delta$
by $\pi$, respectively, and $c$ is a non-negative rational number.
Note that $F\cap\bar{R}_x=\{Q_{4}\}$ and $F\cap\bar{L}_{yz}=\{Q_{9}\}$.

The log pull-back of the log pair $(X,\frac{51}{10}D)$ by $\pi$ is
the log pair
$$
\left(\bar{X},\ \frac{51a}{10}\bar{L}_{yz}+ \frac{51b}{10}\bar{R}_x+ \frac{51}{10}\bar{\Delta}+\theta_1 F\right),%
$$
where
$$\theta_1=\frac{280+51(4a+9b+c)}{10\cdot 41}.$$
This  is not log canonical at some point $Q\in F$.

We have
\[0\leqslant\bar{\Delta}\cdot\bar{R}_x=\frac{4+54b}{17\cdot 41}-\frac{a}{41}-\frac{c}{4\cdot 41}.\]%

This inequality shows $4a+c\leqslant \frac{4}{17}(4+54b)$. Since $b\leqslant \frac{1}{29}$, we obtain $$\theta_{1}=\frac{280+51(4a+9b+c)}{10\cdot 41}\leqslant \frac{4760+51(16+369b)}{10\cdot 17\cdot 41}<1.$$

Suppose that  $Q\not\in \bar{R}_x\cup \bar{L}_{yz}$. Then the log pair $(F, \frac{51}{10}\bar{\Delta}|_F)$ is not log canonical at the point $Q$, and hence
$$
\frac{17c}{120}=\frac{51}{10}\bar{\Delta}\cdot F>1
$$
by Lemma~\ref{lemma:handy-adjunction}. Thus, we see that $c>\frac{120}{17}$.
However, since $b\leqslant \frac{1}{29}$, we obtain
\[c\leqslant 4a+c\leqslant \frac{4}{17}(4+54b)<\frac{120}{17}.\]
Therefore, the point $Q$ must be either $Q_4$ or $Q_9$.

Suppose that $Q=Q_{4}$. The pair $(\bar{R}_x, (\frac{51}{10}\bar{\Delta}+\theta_1 F)|_{\bar{R}_x})$ is not log canonical at $Q$. It then follows
from Lemma~\ref{lemma:handy-adjunction} that
\[1<4\left(\frac{51}{10}\bar{\Delta}+\theta_{1}F\right)\cdot\bar{R}_{x}
=\frac{4\cdot 51}{10}\left(\frac{4+54b}{17\cdot
41}-\frac{a}{41}-\frac{c}{4\cdot
41}\right)+\theta_{1}.\]
However,
\[\frac{4\cdot 51}{10}\left(\frac{4+54b}{17\cdot
41}-\frac{a}{41}-\frac{c}{4\cdot
41}\right)+\theta_{1}= \frac{4760+51(16+369b)}{10\cdot 17\cdot 41}<1.
\]
 This is a contradiction.
Consequently, the point $Q$ must be $Q_9$.

Let $\psi\colon\tilde{X}\to \bar{X}$ be the blow up at the point
$Q_9$ and let $E$ be the exceptional curve
of the morphism $\psi$. The surface $\tilde{X}$ is smooth along the exceptional divisor $E$. Then
\[
K_{\tilde{X}}\qlineq\psi^*(K_{\bar{X}})-\frac{7}{9}E,\ \
\tilde{L}_{yz}\qlineq\psi^*(\bar{L}_{yz})-\frac{1}{9}E, \ \
\tilde{F}\qlineq\psi^*(F)-\frac{1}{9}E,
\ \
\tilde{\Delta}\qlineq\psi^*(\bar{\Delta})-\frac{d}{9}E,
\]
where $\tilde{L}_{yz}$, $\tilde{F}$ and $\tilde{\Delta}$
are the proper transforms of $\bar{L}_{yz}$, $F$ and $\bar{\Delta}$
by $\psi$, respectively, and $d$ is a non-negative rational number.

The log pull-back of the log pair $(X,\frac{51}{10}D)$ by $\pi\circ\psi$ is
the log pair
$$
\left(\tilde{X},\ \frac{51a}{10}\tilde{L}_{yz}+ \frac{51b}{10}\tilde{R}_x+ \frac{51}{10}\tilde{\Delta}+\theta_1 \tilde{F}+\theta_2E\right),%
$$
where $\tilde{R}_x$ is the proper transform of $\bar{R}_x$ by $\psi$ and
$$\theta_2=\frac{70+51(a+d)+10\theta_1}{90}=\frac{3150+51(45a+9b+c+41d)}{90\cdot 41}.$$
This  is not log canonical at some point $O\in E$.

We have
\[0\leqslant\tilde{\Delta}\cdot\tilde{L}_{yx}= \bar{\Delta}\cdot\bar{L}_{yz}-\frac{d}{9}
=\frac{2+53a}{14\cdot 41}-\frac{b}{41}-\frac{c}{9\cdot
41}-\frac{d}{9},\]
and hence $9b+c+41d\leqslant \frac{9}{14}(2+53a)$.
Therefore, this inequality together with $a<\frac{47}{2\cdot 21\cdot 29}$ gives us
\[\begin{split}\theta_2 &=\frac{3150+51(45a+9b+c+41d)}{90\cdot 41}=\\
&=\frac{3150+2295a}{90\cdot 41}+\frac{51(9b+c+41d)}{90\cdot 41}\leqslant \\
&\leqslant \frac{5002+6273a}{10\cdot 14\cdot 41}<1.\\
\end{split}
\]

Suppose that the point $O$ is in the outside of $\tilde{L}_{yz}$ and $\tilde{F}$.
Then the log pair $(E,  \frac{51}{10}\tilde{\Delta}|_{E})$ is not log canonical at the point $O$ and hence
\[1<\frac{51}{10}\tilde{\Delta}\cdot E=\frac{51d}{10}.\]
However,
\[41d\leqslant 9b+c+41d\leqslant \frac{9}{14}(2+53a)<\frac{10\cdot 41}{51}\]
since $a<\frac{47}{2\cdot 21\cdot 29}$.
This is a contradiction.

Suppose that the point $O$ belongs to $\tilde{L}_{yz}$
Then the log pair $(E,  (\frac{51a}{10}\tilde{L}_{yz}+ \frac{51}{10}\tilde{\Delta})|_{E})$ is not log canonical at the point $O$ and hence
\[1<(\frac{51a}{10}\tilde{L}_{yz}+ \frac{51}{10}\tilde{\Delta})\cdot E=\frac{51}{10}(a+d).\]
However,
\[ \frac{51}{10}(a+d)\leqslant \frac{51}{10}\left(a+\frac{9}{14\cdot 41}\left(2+53a\right)\right)<1\]
since $a<\frac{47}{2\cdot 21\cdot 29}$.
This is a contradiction. Therefore, the point $O$ is the intersection point of $\tilde{F}$ and $E$.

Let $\xi\colon\hat{X}\to \tilde{X}$ be the blow up at the
point $O$ and let $H$ be the exceptional
divisor of $\xi$. We also let $\hat{L}_{yz}$, $\hat{R}_x$, $\hat{\Delta}$,
$\hat{E}$, and $\hat{F}$ be the proper transforms of $\tilde{L}_{yz}$,
$\tilde{R}_x$, $\tilde{\Delta}$,  $E$ and $\tilde{F}$ by $\xi$, respectively. We have
$$
K_{\hat{X}}\qlineq\xi^*(K_{\tilde{X}})+H,\
\hat{E}\qlineq\xi^*(E)-H,\
\hat{F}\qlineq\xi^*(\tilde{F})-H,\
\hat{\Delta}\qlineq\xi^*(\tilde{\Delta})-eH,%
$$
where $e$ is a non-negative rational number. The log pull-back of the
log pair $(X, \frac{51}{10}D)$ via $\pi\circ \phi\circ \xi$ is
$$
\left(\hat{X},\frac{51a}{10}\hat{L}_{yz}+\frac{51b}{10}\hat{R}_x+\frac{51}{10}\hat{\Delta}+\theta_1\hat{F}+\theta_2 \hat{E}+\theta_{3}H\right),%
$$
where
$$
\theta_3=\theta_{1}+\theta_{2}+\frac{51e}{10}-1
=\frac{1980+51(81a+90b+10c+41d+369e)}{90\cdot 41}.$$
This log pair is not log canonical at some point $A\in H$. We have
\[
\frac{c}{9\cdot 4}-\frac{d}{9}-e=\hat{\Delta}\cdot\hat{F}\geqslant
0.\] Therefore, $4d+36e\leqslant c$. Then
\[\begin{split}\theta_3&=\frac{1980+51(81a+90b+10c)}{90\cdot 41}+\frac{51(d+9e)}{90}\leqslant \\
&\leqslant \frac{7920+51(324a+360b+81c)}{4\cdot 90\cdot 41}=\\
&=\frac{22+51b}{41}+\frac{51\cdot 81(4a+c)}{4\cdot 90\cdot 41}\leqslant \\
&\leqslant \frac{22+51b}{41}+\frac{9\cdot51(2+27b)}{5\cdot17\cdot 41}<1\\
\end{split}
\]
since $b\leqslant \frac{1}{29}$ and $4a+c\leqslant \frac{4}{17}(4+54b)$.

Suppose that $A\not\in\hat{F}\cup\hat{E}$. Then the log pair $
\left(\hat{X},\frac{51}{10}\hat{\Delta}+\theta_{3}H\right)
$  is not log canonical at the point $A$. Applying
Lemma~\ref{lemma:adjunction}, we get
$$
1<\frac{51}{10}\hat{\Delta}\cdot H=\frac{51e}{10}.
$$
However, $$e\leqslant \frac{1}{36}(4d+36e)\leqslant \frac{c}{36}\leqslant \frac{1}{36}(4a+c)\leqslant
\frac{4+54b}{17\cdot 9}< \frac{10}{51}.$$
Therefore, the point $A$ must be either in $\hat{F}$ or in $\hat{E}$.

Suppose that $A\in\hat{F}$. Then the log pair  $
\left(\hat{X},\frac{51}{10}\hat{\Delta}+\theta_1\hat{F}+\theta_{3}H\right)%
$ is not log canonical at the point $A$. Applying Lemma~\ref{lemma:adjunction},
we get
\[
1<\left(\frac{51}{10}\hat{\Delta}+\theta_{3}
H\right)\cdot\hat{F}=\frac{51}{10}\left(\frac{c}{9\cdot 4}-\frac{d}{9}-e\right)+
\theta_{3}=\frac{7920+51(324a+360b+81c)}{4\cdot 90\cdot 41}.
\]
However,
\[\frac{7920+51(324a+360b+81c)}{4\cdot 90\cdot 41}<1
\]
as seen in the previous. Therefore, the point $A$ is the
intersection point of $H$ and $\hat{E}$. Then the log pair $
\left(\hat{X},\frac{51}{10}\hat{\Delta}+\theta_2
\hat{E}+\theta_{3}H\right) $ is not log canonical at the point
$A$. From Lemma~\ref{lemma:adjunction}, we obtain
\[
1<\left(\frac{51}{10}\hat{\Delta}+\theta_{3}H\right)\cdot\hat{E}=\frac{51}{10}\left(d-e\right)+\theta_{3}=\frac{1980+51(81a+90b+10c+410d)}{90\cdot
41}.
\]
However,
\[\begin{split}\frac{1980+51(81a+90b+10c+410d)}{90\cdot 41}&=\frac{220+459a}{10\cdot 41}+\frac{51(9b+c+41d)}{9\cdot 41}\leqslant\\
&\leqslant \frac{220+459a}{10\cdot 41}+\frac{51(2+53a)}{14\cdot 41}<1\\ \end{split}
\]
since $9b+c+41d\leqslant \frac{9}{14}(2+53a)$ and $a<\frac{47}{2\cdot 21\cdot 29}$.
The obtained contradiction completes the
proof.
\end{proof}

\section{Sporadic cases with $I=3$}
\label{section:index-3}

\begin{lemma}
\label{lemma:I-3-W-5-7-11-13-D-33}Let $X$ be a quasismooth
hypersurface of degree $33$ in $\mathbb{P}(5,7,11,13)$. Then
$\lct(X)=\nlb\frac{49}{36}$.
\end{lemma}

\begin{proof}
 The surface $X$ can be defined by the
quasihomogeneous equation
$$
z^{3}+yt^{2}+xy^{4}+x^{4}t+\epsilon x^{3}yz=0,
$$
where $\epsilon\in\mathbb{C}$. Note that the surface $X$ is singular at $O_x$,
$O_{y}$ and $O_t$.

The curves $C_x$, $C_y$ are  irreducible. Moreover, we have
$$
\frac{25}{18}=\lct(X, \frac{3}{5}C_x)>\lct(X, \frac{3}{7}C_y)=\frac{49}{36}.%
$$
Therefore, $\lct(X)\leqslant \frac{49}{36}$.

Suppose that $\lct(X)<\frac{49}{36}$. Then there is an effective
$\Q$-divisor $D\qlineq -K_X$ such that the pair $(X,\frac{49}{36}D)$ is
not log canonical at some point $P$. By
Lemma~\ref{lemma:convexity}, we may assume that the support of
$D$ contains neither $C_x$ nor $C_y$. Since the curve $C_y$ is singular at the point $O_t$, the three inequalities
\[5D\cdot C_y=\frac{9}{13}<\frac{36}{49}, \ \ 7D\cdot C_x=\frac{63}{91}<\frac{36}{49}\]
show that the point $P$ is located in the outside of the set $C_x\cup C_y$.

Let $\mathcal{L}$ be the pencil on $X$ that is cut out by the
equations
$$
\lambda x^{7}+\mu y^{5}=0,
$$
where $[\lambda :\mu]\in\P^1$. Then the base locus of the pencil
$\mathcal{L}$ consists of the  point $O_{t}$.
Let $C$ be the unique curve in $\mathcal{L}$ that passes through
the point $P$. Since the point $P$ is in the outside of the set $C_x\cup C_y$, the curve $C$ is defined by an equation of the form $y^5-\alpha x^7=0$, where $\alpha$ is a non-zero constant.
Suppose that $C$ is irreducible and reduced. Then
$\mult_{P}(C)\leqslant 3$ since the curve $C$ is a triple cover
of the curve
$$
y^5-\alpha x^7=0\subset\mathbb{P}\big(5,7,13\big)\cong\mathrm{Proj}\Big(\mathbb{C}\big[x,y,t\big]\Big).
$$
 In particular,  $\lct(X,\frac{3}{35}C)>\frac{49}{36}$. Thus, we may assume
that the support of $D$ does not contain the curve $C$ and hence
we obtain
$$
\frac{36}{49}<\mult_{P}(D)\leqslant D\cdot
C=\frac{9}{13}<\frac{36}{49}.
$$
This is a contradiction. Thus, to conclude the proof it suffices to prove that the curve $C$ is irreducible and reduced.

Let $S\subset\mathbb{C}^{4}$ be the affine variety  defined
by the equations
$$
y^{5}-\alpha x^{7}=z^{3}+yt^{2}+xy^{4}+x^{4}t+\epsilon x^{3}yz=0\subset\mathbb{C}^{4}\cong\mathrm{Spec}\Big(\mathbb{C}\big[x,y,z,t\big]\Big).%
$$
 To conclude the proof, it is enough to prove
that the variety $S$ is irreducible.

Consider the projectivised surface $\bar{S}$ of $S$ defined by
the homogeneous equations
$$
y^{5}w^{2}-\alpha x^{7}
=z^{3}w^{2}+yt^{2}w^{2}
+xy^{4}+x^{4}t+\epsilon x^{3}yz=0
\subset\mathbb{P}^{4}\cong\mathrm{Proj}\Big(\mathbb{C}\big[x,y,z,t,w\big]\Big).%
$$
Then we consider the affine piece $S'$ of $\bar{S}$ defined by $y\ne0$. The affine surface $S'$  is defined
by the equations
$$
w^{2}-\alpha x^{7}
=z^{3}w^{2}+t^{2}w^{2}+x+x^{4}t
+\epsilon x^{3}z=0
\subset\mathbb{C}^{4}\cong\mathrm{Spec}\Big(\mathbb{C}\big[x,z,t,w\big]\Big).%
$$
This is isomorphic to the affine hypersurface defined by
$$
x(\alpha x^{6}z^{3}+\alpha x^{6}t^{2}+1
+x^{3}t+\epsilon x^{2}z)=0
\subset\mathbb{C}^{3}\cong\mathrm{Spec}\Big(\mathbb{C}\big[x,z,t\big]\Big).%
$$
This affine hypersurface has two irreducible components. However, the component defined by $x=0$
originates from the hyperplane section of $\bar{S}$ by $w=0$. Therefore, the original affine surface $S$ must be irreducible and reduce.
\end{proof}

\begin{lemma}
\label{lemma:I-3-W-5-7-11-20-D-40}Let $X$ be a quasismooth
hypersurface of degree $40$ in $\mathbb{P}(5,7,11,20)$. Then
$\lct(X)=\frac{25}{18}$.
\end{lemma}

\begin{proof}
  The surface $X$ can be defined by the
quasihomogeneous equation
$$
t(t- x^4)+yz^{3}+xy^{5}+\epsilon x^{3}y^{2}z,
$$
where $\epsilon\in\mathbb{C}$. Note that $X$ is singular at the
points $O_x$, $O_y$, $O_{z}$ and $Q_5=[1:0:0:1]$.

The curve $C_x$ is irreducible. We have
$$
\lct(X, \frac{3}{5}C_x)=\frac{25}{18}.%
$$
Therefore, $\lct(X)\leqslant \frac{25}{18}$. Meanwhile, the curve $C_{y}$ is
reducible. It consists of two irreducible components $L_{yt}$ and $R_y=\{y=t-x^4=0\}$. The curve $L_{yt}$ intersects $R_y$ only at the point $O_z$.
It is easy to see
$$
L_{yt}^2=R_y^2=-\frac{13}{55},\ L_{yt}\cdot R_{y}=\frac{4}{11}.%
$$

Suppose that $\lct(X)<\frac{25}{18}$. Then there is an effective
$\Q$-divisor $D\qlineq -K_X$ such that the pair
$(X,\frac{25}{18}D)$ is not log canonical at some point $P$. By
Lemma~\ref{lemma:convexity}, we may assume that the support of $D$
does not contain the curve $C_x$. Moreover, we may assume that the
support of $D$ does not contain either $L_{yt}$ or $R_y$ since
$$
\lct(X, \frac{3}{7}C_y)=\frac{35}{24}>\frac{25}{18}.%
$$
Then one of the inequalities
\[\mult_{O_z}(D)\leqslant 11D\cdot L_{yt}=\frac{3}{5}<\frac{18}{25}, \ \ \mult_{O_z}(D)\leqslant 11D\cdot R_{y}=\frac{3}{5}<\frac{18}{25}\]
must hold, and hence the point $P$ cannot be the point $O_z$. Also, since $7D\cdot C_x=\frac{6}{11}<\frac{18}{25}$, the point $P$ cannot belong to the curve $C_x$.

We write $D=aL_{yt}+bR_y+\Omega$, where $\Omega$ is an effective $\mathbb{Q}$-divisor whose support contains neither $L_{yt}$ nor $R_y$. If $a\ne 0$, then we have
\[\frac{3}{55}=D\cdot R_y\geqslant aL_{yt}\cdot R_y=\frac{4a}{11}.\]
Therefore, $a\leqslant \frac{3}{20}$. By the same way, we also obtain $b\leqslant \frac{3}{20}$.

Since we have
\[5(D-aL_{yt})\cdot L_{yt}=\frac{3+13a}{11}<\frac{18}{25},\ \  5(D-bR_{y})\cdot R_{y}=\frac{3+13a}{11}<\frac{18}{25}\]
Lemma~\ref{lemma:handy-adjunction} implies that the point $P$ is in the outside of $C_y$.
Consequently, the point $P$ is located in the outside of $C_x\cup C_y$. However,
since $H^0(\P, \mathcal{O}_\P(40))$
contains monomials $x^{8}, xy^5, x^4t$ and the natural
projection $X\dasharrow\mathbb{P}(5,7,20)$ is a finite morphism
outside of the curve $C_{y}$, Lemma~\ref{lemma:Carolina} shows that the point $P$ must belong to
the set $C_x\cup C_y$. This is a contradiction.
\end{proof}

\begin{lemma}
\label{lemma:I-3-W-11-21-29-37-D-95} Let $X$ be a quasismooth
hypersurface of degree $95$ in $\mathbb{P}(11,21,29,37)$. Then
$\lct(X)=\frac{11}{4}$.
\end{lemma}

\begin{proof}
We may assume that the surface $X$ is defined by the quasihomogeneous equation
\[t^2y+tz^2+xy^4+x^6z=0.\]
The surface $X$ is singular at the points $O_x$, $O_y$, $O_z$, $O_t$. Each of the divisors $C_x$, $C_y$, $C_z$, and $C_t$ consists of two irreducible and reduced components. The divisor $C_x$ (resp. $C_y$, $C_z$, $C_t$) consists of $L_{xt}$ (resp. $L_{yz}$, $L_{yz}$, $L_{xt}$)
and $R_x=\{x=yt+z^2=0\}$ (resp. $R_y=\{y=zt+x^6=0\}$, $R_z=\{z=xy^3+t^2=0\}$, $R_t=\{t=y^4+x^5z=0\}$).
Also, we see that
\[L_{xt}\cap R_x=\{O_y\}, \ L_{yz}\cap R_y=\{O_t\}, \ L_{yz}\cap R_z=\{O_x\}, \ L_{xt}\cap R_t=\{O_z\}.\]
It is easy to check that $\lct(X, \frac{3}{11}C_x)=\frac{11}{4}$
is less than each of the numbers $\lct(X, \frac{3}{21}C_y)$,
$\lct(X, \frac{3}{29}C_z)$ and $\lct(X, \frac{3}{37}C_t)$.
Therefore, $\lct(X)\leq\frac{11}{4}$.

Suppose that
$\lct(X)<\frac{11}{4}$. Then, there is an effective $\Q$-divisor
$D\qlineq -K_X$ such that the log pair $(X, \frac{11}{4}D)$ is not
log canonical at some point $P\in X$.

The intersection numbers among the divisors $D$, $L_{xt}$, $L_{yz}$, $R_x$, $R_y$, $R_z$, $R_t$ are as follows:

\[D\cdot L_{xt}=\frac{1}{7\cdot 29}, \ \ D\cdot R_x=\frac{2}{7\cdot 37}, \ \ D\cdot R_y=\frac{18}{29\cdot 37}, \]
\[D\cdot L_{yz}=\frac{3}{11\cdot 37}, \ \ D\cdot R_z=\frac{2}{7\cdot 11}, \ \ D\cdot R_t=\frac{12}{11\cdot 29}, \]
\[L_{xt}\cdot R_x=\frac{2}{21}, \ \ L_{yz}\cdot R_y=\frac{6}{37}, \ \ L_{yz}\cdot R_z=\frac{2}{11}, \ \ L_{xt}\cdot R_t=\frac{4}{29},\]

\[L_{xt}^2=-\frac{47}{21\cdot 29}, \ \ R_x^2=-\frac{52}{21\cdot 37}, \ \ R_y^2=-\frac{48}{29\cdot 37},\]

\[L_{yz}^2=-\frac{45}{11\cdot 37}, \ \ R_z^2=\frac{16}{11\cdot 21}, \ \ R_t^2=\frac{104}{11\cdot 29}.\]
By Lemma~\ref{lemma:convexity} we may assume that the support of $D$ does not contain at least one component of each divisor $C_x$, $C_y$, $C_z$, $C_t$.
The inequalities
\[21D\cdot L_{xt}=\frac{3}{29}<\frac{4}{11}, \ \ \  21D\cdot R_x=\frac{6}{37}<\frac{4}{11}\]
imply that $P\ne O_y$.
The inequalities
\[11D\cdot L_{yz}=\frac{3}{37}<\frac{4}{11}, \ \ \  11D\cdot R_z=\frac{2}{7}<\frac{4}{11}\]
imply that $P\ne O_x$.
The inequalities
\[29D\cdot L_{xt}=\frac{1}{7}<\frac{4}{11}, \ \ \  \frac{29}{4}D\cdot R_t=\frac{3}{11}<\frac{4}{11}\]
imply that $P\ne O_z$. The curve $R_t$ is singular at the point $O_z$ with multiplicity $4$.

We write $D=m_0L_{xt}+m_1L_{yz}+m_2R_x+m_3R_y+m_4R_z+m_5R_t+\Omega$, where  $\Omega$ is an effective divisor whose support contains none of the curves $L_{xt}$, $L_{yz}$, $R_x$, $R_y$, $R_z$, $R_t$. Since the pair $(X, \frac{11}{4}D)$ is log canonical at the points $O_x$, $O_y$, $O_z$, the numbers $m_i$ are at most $\frac{4}{11}$. Then by Lemma~\ref{lemma:handy-adjunction} the following inequalities enable us to conclude that either the point $P$ is in the outside of $C_x\cup C_y\cup C_z\cup C_t$ or $P=O_t$:
\[(D-m_0L_{xt})\cdot L_{xt}=\frac{3+47m_0}{21\cdot 29}\leqslant \frac{4}{11},\ \
(D-m_1L_{yz})\cdot L_{yz}=\frac{3+45m_1}{11\cdot 37}\leqslant \frac{4}{11},\]
\[(D-m_2R_{x})\cdot R_{x}=\frac{6+52m_2}{21\cdot 37}\leqslant \frac{4}{11},\ \
(D-m_3R_{y})\cdot R_{y}=\frac{18+48m_3}{29\cdot 37}\leqslant \frac{4}{11},\]
\[(D-m_4R_{z})\cdot R_{z}=\frac{6-16m_4}{11\cdot 21}\leqslant \frac{4}{11},\ \
(D-m_5R_{t})\cdot R_{t}=\frac{12-104m_5}{11\cdot 29}\leqslant \frac{4}{11}.\]

Suppose that $P\ne O_t$. Then we consider the pencil $\mathcal{L}$
defined by $\lambda yt+\mu z^2=0$, $[\lambda:\mu]\in\P^1$. The
base locus of the pencil $\mathcal{L}$ consists of the curve
$L_{yz}$ and the point $O_y$. Let $E$ be the unique divisor in
$\mathcal{L}$ that passes through the point $P$. Since $P\not\in
C_x\cup C_y\cup C_z\cup C_t$, the divisor $E$ is defined by the
equation $z^2=\alpha yt$, where $\alpha\ne 0$.

Suppose that $\alpha\ne -1$. Then the curve $E$ is isomorphic to the curve defined by the equations $yt=z^2$ and $t^2y+xy^4+x^6z=0$. Since the curve $E$ is isomorphic to a general curve in $\mathcal{L}$, it is smooth at the point $P$. The affine piece of $E$ defined by $t\ne 0$ is the curve given by $z(z+xz^7+x^6)=0$. Therefore, the divisor $E$ consists of two irreducible and reduced curves $L_{yz}$ and $C$. We have the intersection number
\[D\cdot C=D\cdot E-D\cdot L_{yz}=\frac{169}{7\cdot 11\cdot 37}.\]
Also, we see
\[C^2=E\cdot C- C\cdot L_{yz}\geq E\cdot C-C_y\cdot C>0\]
since $C$ is different from $R_y$. The multiplicity of $D$ along the curve $C$ is at most $\frac{4}{11}$ since the intersection number $C\cdot C_t$ is positive and the pair $(X, \frac{11}{4}D)$ is log canonical along the curve $C_t$.
Then by Lemma~\ref{lemma:handy-adjunction} the inequality $D\cdot C<\frac{4}{11}$ gives us a contradiction.

Suppose that $\alpha=-1$. Then divisor $E$ consists of three irreducible and reduced curves $L_{yz}$, $R_x$, and $M$.  Note that the curve $M$ is different from the curves $R_y$ and $L_{xt}$. Also, it is smooth at the point $P$. We have
\[D\cdot M=D\cdot E- D\cdot L_{yz}-D\cdot R_x=\frac{147}{7\cdot 11\cdot 37},\]
\[M^2=E\cdot M-L_{yz}\cdot M- R_x\cdot M \geq E\cdot M-C_y\cdot M- C_x\cdot M>0.\]
The multiplicity of $D$ along the curve $M$ is at most $\frac{4}{11}$ since the intersection number $M\cdot C_t$ is positive and the pair $(X, \frac{11}{4}D)$ is log canonical along the curve $C_t$.
By Lemma~\ref{lemma:handy-adjunction} the inequality $D\cdot M<\frac{4}{11}$ gives us a contradiction.
Therefore, $P=O_t$.

Put $D=aL_{yz}+bR_{x}+\Delta$, where $\Delta$ is an effective
$\Q$-divisor whose support contains neither $L_{yz}$ nor $R_x$.
Then $a>0$ since  otherwise we would obtain an absurd inequality
$$
\frac{3}{11}=37D\cdot
L_{yz}\geqslant\mult_{O_t}(D)>\frac{4}{11}.
$$
Therefore, we may assume that
$R_{y}\not\subset\mathrm{Supp}(\Delta)$ by
Lemma~\ref{lemma:convexity}.

If $b>0$, the curve $L_{xt}$ is not contained in the support of $D$, and hence
$$
\frac{3}{21\cdot 29}=D\cdot L_{xt}\geqslant b R_x \cdot L_{xt}=\frac{2b}{21}.%
$$
Therefore, $b\leqslant \frac{3}{58}$. Similarly, we have
$$
\frac{18}{29\cdot 37}=D\cdot R_y\geqslant \frac{6a}{37}+\frac{b}{37}+\frac{\mult_{O_t}(D)-a-b}{37}>\frac{5a}{37}+\frac{4}{11\cdot 37},%
$$
and hence $a<\frac{82}{5\cdot 11\cdot 29}$.

Let $\pi\colon\bar{X}\to X$ be the weighted blow up of the point
$O_t$ with weights $(13,4)$ and let $F$ be the exceptional curve
of the morphism $\pi$. Then $F$ contains two singular points
$Q_{13}$ and $Q_{4}$ of $\bar{X}$ such that $Q_{13}$ is a singular point of
type $\frac{1}{13}(1,2)$ and $Q_{4}$ is a singular point of type
$\frac{1}{4}(3,1)$. Then
$$
K_{\bar{X}}\qlineq\pi^*(K_X)-\frac{20}{37}F,\
\bar{L}_{yz}\qlineq\pi^*(L_{yz})-\frac{4}{37}F,\
\bar{R}_x\qlineq\pi^*(R_x)-\frac{13}{37}F,\
\bar{\Delta}\qlineq\pi^*(\Delta)-\frac{c}{37}F,
$$
where $\bar{L}_{yz}$, $\bar{R}_x$ and $\bar{\Delta}$ are the
proper transforms of $L_{yz}$, $R_x$ and $\Delta$ by $\pi$,
respectively, and $c$ is a non-negative rational number.

The log pull-back of the log pair $(X,\frac{11}{4}D)$ by $\pi$ is
the log pair
$$
\left(\bar{X},\ \frac{11a}{4}\bar{L}_{yz}+ \frac{11b}{4}\bar{R}_x+ \frac{11}{4}\bar{\Delta}+\theta_1 F\right),%
$$
where
$$
\theta_1=\frac{11(4a+13b+c)+80}{4\cdot37}.
$$
This pair is not log canonical at some point $Q\in F$. We have
\[0\leqslant\bar{\Delta}\cdot\bar{R}_x=\frac{6+52b}{21\cdot 37}-\frac{a}{37}-\frac{c}{4\cdot 37}.\]
This inequality shows $4a+c\leqslant\frac{4}{21}(6+52b)$. Then
\[\theta_1=\frac{11(4a+c)}{4\cdot 37}+\frac{143b}{4\cdot37}+\frac{20}{37}
\leqslant\frac{11}{ 21\cdot 37}(6+52b)+\frac{143b}{4\cdot37}+\frac{20}{37}=\frac{1944+5291b}{4\cdot21\cdot37}<1\]
since $b\leqslant\frac{3}{58}$. Note that
$F\cap \bar{R}_x=\{Q_{4}\}$ and $F\cap\bar{L}_{yz}=\{Q_{13}\}$.

Suppose that  the point $Q$ is neither $Q_4$ nor $Q_{13}$. Then the pair $\left(\bar{X},\frac{11}{4}\bar{\Delta}+F\right)$ is not log canonical at the point $Q$. Then
$$
\frac{11c}{16\cdot 13}=\frac{11}{4}\bar{\Delta}\cdot F>1
$$
by Lemma~\ref{lemma:adjunction}. However, $c\leqslant 4a+c\leqslant\frac{4}{21}(6+52b)$. This is a contradiction since $b\leqslant \frac{3}{58}$.
Therefore, the point $Q$ is either $Q_4$ or $Q_{13}$.

Suppose that the point $Q$ is the point $Q_{4}$. Then the log pair
$\left(\bar{X},\frac{11b}{4}\bar{R}_x+ \frac{11}{4}\bar{\Delta}+\theta_1 F\right)$ is not log canonical at the point $Q$. It then follows
from Lemma~\ref{lemma:adjunction} that
\[1<4\left(\frac{11}{4}\bar{\Delta}+\theta_{1}F\right)\cdot\bar{R}_{x}
=11\left(\frac{6+52b}{21\cdot 37}-\frac{a}{37}-\frac{c}{4\cdot 37}\right)+\theta_{1}.\]
However,
\[11\left(\frac{6+52b}{21\cdot 37}-\frac{a}{37}-\frac{c}{4\cdot 37}\right)+\theta_{1}=\frac{1944+5291b}{4\cdot21\cdot37}<1.\]
 Therefore, the point $Q$ must be
the point $Q_{13}$.

Let $\phi\colon\tilde{X}\to \bar{X}$ be the weighted blow up at
the point $Q_{13}$ with weights $(1, 2)$. Let $G$ be the
exceptional divisor of the morphism $\phi$. Then $G$ contains one
singular point $Q_{2}$ of the surface $\tilde{X}$ that is a
singular point of type $\frac{1}{2}(1,1)$. Let $\tilde{L}_{yz}$,
$\tilde{R}_x$, $\tilde{\Delta}$ and $\tilde{F}$ be the proper
transforms of $L_{yz}$, $R_x$, $\Delta$ and $F$ by $\phi$,
respectively. We have
$$
K_{\tilde{X}}\qlineq\phi^*(K_{\bar{X}})-\frac{10}{13}G,\ \tilde{L}_{yz}\qlineq\phi^*(\bar{L}_{yz})-\frac{2}{13}G, \ \tilde{F}\qlineq\phi^*(F)-\frac{1}{13}G,\ \tilde{\Delta}\qlineq\phi^*(\bar{\Delta})-\frac{d}{13}G,%
$$
where $d$ is a non-negative rational number. The log pull-back of the
log pair $(X, \frac{11}{4}D)$ via $\pi\circ \phi$ is
$$
\left(\tilde{X},\frac{11a}{4}\tilde{L}_{yz}+\frac{11b}{4}\tilde{R}_x+\frac{11}{4}\tilde{\Delta}+\theta_1\tilde{F}+\theta_2 G\right),%
$$
where
$$\theta_2=\frac{11}{4\cdot 13}(2a+d)+\frac{\theta_1}{13}+\frac{10}{13}=\frac{1560+11(78a+13b+c+37d)}{4\cdot 13\cdot 37}.$$
This log pair is not log canonical at some point $O\in G$. We have
\[0\leqslant \tilde{\Delta}\cdot\tilde{L}_{yz}=\frac{3+45a}{11\cdot 37}-\frac{b}{37}-\frac{c}{13\cdot 37}-\frac{d}{13}.\]
We then obtain $13b+c+37d\leqslant \frac{13}{11}(3+45a)$.
Since $a<\frac{82}{5\cdot 11\cdot 29}$, we see
\[\theta_2=\frac{1560+11(78a+13b+c+37d)}{4\cdot 13\cdot 37}\leqslant
\frac{1560+858a}{4\cdot 13\cdot 37}+\frac{3+45a}{4\cdot 37}<1. \]

 Note that $\tilde{F}\cap G=Q_2$ and
$Q_{2}\not\in\tilde{L}_{yz}$. Suppose that
$O\not\in\tilde{F}\cup\tilde{L}_{yz}$. The log pair
$\left(\tilde{X},\frac{11}{4}\tilde{\Delta}+G\right)$  is not log
canonical at the point $O$. Applying Lemma~\ref{lemma:adjunction},
we get
$$
1<\frac{11}{4}\tilde{\Delta}\cdot G=\frac{11d}{4\cdot 2},%
$$
and hence $d>\frac{8}{11}$. However, $d\leqslant \frac{1}{37}(13b+c+37d)\leqslant \frac{13}{11\cdot 37}(3+45a)$. This is a contradiction since $a<\frac{82}{5\cdot 11\cdot 29}$.
Therefore, the point $O$ is either the point $Q_2$ or the intersection point of $G$ and $\tilde{L}_{yz}$. In the latter case, the pair $
\left(\tilde{X},\frac{11a}{4}\tilde{L}_{yz}+\frac{11}{4}\tilde{\Delta}+\theta_2 G\right)%
$ is not log canonical at the point $O$. Then,
applying
Lemma~\ref{lemma:adjunction}, we get
\[
1<\left(\frac{11}{4}\tilde{\Delta}+\theta_{2}G\right)\cdot\tilde{L}_{yz}=\frac{11}{4}\left(\frac{3+45a}{11\cdot
37}-\frac{b}{37}-\frac{c}{13\cdot
37}-\frac{d}{13}\right)+\theta_{2}.\] However,
\[\frac{11}{4}\left(\frac{3+45a}{11\cdot
37}-\frac{b}{37}-\frac{c}{13\cdot
37}-\frac{d}{13}\right)+\theta_{2}=\frac{11}{4}\left(\frac{3+45a}{11\cdot
37}\right)+\frac{1560+858a}{4\cdot 13\cdot 37}<1\]
since $a<\frac{82}{5\cdot 11\cdot 29}$.
Therefore, the point $O$ must be the point $Q_2$.

Let $\xi\colon\hat{X}\to \tilde{X}$ be the blow up at the
point $Q_{2}$  and let $H$ be the exceptional
divisor of $\xi$. We also let $\hat{L}_{yz}$, $\hat{R}_x$, $\hat{\Delta}$,
$\hat{G}$, and $\hat{F}$ be the proper transforms of $\tilde{L}_{yz}$,
$\tilde{R}_x$, $\tilde{\Delta}$,  $G$ and $\tilde{F}$ by $\xi$, respectively. Then
$\hat{X}$ is smooth along the exceptional divisor $H$. We have
$$
K_{\hat{X}}\qlineq\xi^*(K_{\tilde{X}}),\
\hat{G}\qlineq\xi^*(G)-\frac{1}{2}H,\
\hat{F}\qlineq\xi^*(\tilde{F})-\frac{1}{2}H,\
\hat{\Delta}\qlineq\xi^*(\tilde{\Delta})-\frac{e}{2}H,%
$$
where $e$ is a non-negative rational number. The log pull-back of the
log pair $(X, \frac{11}{4}D)$ via $\pi\circ \phi\circ \xi$ is
$$
\left(\hat{X},\frac{11a}{4}\hat{L}_{yz}+\frac{11b}{4}\hat{R}_x+\frac{11}{4}\hat{\Delta}+\theta_1\hat{F}+\theta_2 \hat{G}+\theta_{3}H\right),%
$$
where
$$
\theta_3=\frac{\theta_{1}+\theta_{2}}{2}+\frac{11e}{8}
=\frac{2600+11(130a+182b+14c+37d+481e)}{8\cdot 13\cdot 37}.$$
This log pair is not log canonical at some point $A\in H$. We have
\[
\frac{c}{13\cdot 4}-\frac{d}{13\cdot
2}-\frac{e}{2}=\hat{\Delta}\cdot\hat{F}\geqslant 0.\] Therefore,
$2d+26e\leqslant c$.
Then
\[\begin{split}\theta_3&=\frac{2600+11(130a+182b+14c)}{8\cdot 13\cdot 37}
+\frac{11(d+13e)}{8\cdot 13}\leqslant\\
&\leqslant \frac{5200+11(260a+364b+65c)}{16\cdot 13\cdot 37}=\\
&=\frac{5200+4004b}{16\cdot 13\cdot 37}+\frac{11\cdot 65(4a+c)}{16\cdot 13\cdot 37}\leqslant\\
&\leqslant \frac{100+77b}{4\cdot 37}+\frac{5\cdot 11(6+52b)}{4\cdot 21 \cdot 37}=\frac{2430+4477b}{4\cdot21\cdot 37}<1\\
\end{split}
\]
since $b\leqslant \frac{3}{58}$ and $4a+c\leqslant \frac{4}{21}(6+52b)$.

Suppose that $A\not\in\hat{F}\cup\hat{G}$. Then the log pair $
\left(\hat{X},\frac{11}{4}\hat{\Delta}+\theta_{3}H\right)
$  is not log canonical at the point $A$. Applying
Lemma~\ref{lemma:adjunction}, we get
$$
1<\frac{11}{4}\hat{\Delta}\cdot H=\frac{11e}{4}.
$$
However, $$e\leqslant \frac{1}{26}(2d+26e)\leqslant \frac{c}{26}\leqslant \frac{1}{26}(4a+c)\leqslant
\frac{4(6+52b)}{21\cdot 26}\leqslant \frac{4}{11}.$$
Therefore, the point $A$ must be either in $\hat{F}$ or in $\hat{G}$.

Suppose that $A\in\hat{F}$. Then the log pair  $
\left(\hat{X},\frac{11}{4}\hat{\Delta}+\theta_1\hat{F}+\theta_{3}H\right)%
$ is not log canonical at the point $A$. Applying Lemma~\ref{lemma:adjunction},
we get
\[
1<\left(\frac{11}{4}\hat{\Delta}+\theta_{3}
H\right)\cdot\hat{F}=\frac{11}{4}\left(\frac{c}{4\cdot 13}-\frac{d}{2\cdot
13}-\frac{e}{2}\right)+
\theta_{3}=\frac{5200+11(260a+364b+65c)}{16\cdot 13\cdot 37}.
\]
However,
\[\frac{5200+11(260a+364b+65c)}{16\cdot 13\cdot 37}=\frac{400+11\cdot 28b}{16\cdot 37}+\frac{11\cdot 5(4a+c)}{16\cdot 37}\leqslant \frac{2430+4477b}{4\cdot 21\cdot 37}<1.
\]
Therefore, the point $A$ is the intersection point of $H$ and
$\hat{G}$. Then the log pair $
\left(\hat{X},\frac{11}{4}\hat{\Delta}+\theta_2
\hat{G}+\theta_{3}H\right) $ is not log canonical at the point
$A$. From Lemma~\ref{lemma:adjunction}, we obtain
\[
1<\left(\frac{11}{4}\hat{\Delta}+\theta_{3}H\right)\cdot\hat{G}=\frac{11}{4}\left(\frac{d}{2}-\frac{e}{2}\right)+\theta_{3}=\frac{2600+11(130a+182b+14c)}{8\cdot
13\cdot 37}
+\frac{77d}{4\cdot 13}.%
\]
However,
\[\frac{2600+11(130a+182b+14c)}{8\cdot 13\cdot 37}
+\frac{77d}{4\cdot 13}=\frac{100+55a}{4\cdot 37}
+\frac{77(13b+c+37d)}{4\cdot 13\cdot 37}\leqslant\frac{121+370a}{4\cdot 37}<1
\]
since $a<\frac{82}{5\cdot 11\cdot 29}$ and $13b+c+37d\leqslant \frac{13}{11}(3+45a)$.
The obtained contradiction completes the
proof.
\end{proof}

\begin{lemma}
\label{lemma:I-3-W-11-37-53-98-D-196}Let $X$ be a quasismooth
hypersurface of degree $196$ in $\mathbb{P}(11,37,53,98)$. Then
$\lct(X)=\frac{55}{18}$.
\end{lemma}

\begin{proof}
 The surface $X$ can be defined by the
quasihomogeneous equation
$$
t^{2}+yz^{3}+xy^{5}+x^{13}z=0.
$$
It is singular at the points $O_x$, $O_{y}$ and $O_z$.
The curves $C_x$ and $C_y$ are irreducible. We have
$$
\frac{55}{18}=\lct\left(X, \frac{3}{11}C_x\right)<\lct\left(X, \frac{3}{37}C_y\right)=\frac{37\cdot 5}{26},%
$$
and hence $\lct(X)\leqslant \frac{55}{18}$.

Suppose that $\lct(X)<\frac{55}{18}$. Then there is an effective
$\Q$-divisor $D\qlineq -K_X$ such that the pair $(X,\frac{55}{18}D)$ is
not log canonical at some point $P$.
 By
Lemma~\ref{lemma:convexity}, we may assume that the support of
the divisor $D$  contains neither  $C_x$ nor $C_y$. Then the inequalities
\[53D\cdot C_x=\frac{6}{37}<\frac{18}{55}, \ \ 11D\cdot C_y=\frac{6}{53}<\frac{18}{55}\]
show that the point $P$ is a smooth point in the outside of $C_x$.
However,
since $H^0(\P, \mathcal{O}_{\P}(583))$ contains the monomials
$x^{53}$, $y^{11}x^{16}$ and $z^{11}$, it follows from
Lemma~\ref{lemma:Carolina} that the point $P$ is either a singular point of $X$ or a point on $C_x$. This is a contradiction.
\end{proof}

\begin{lemma}
\label{lemma:I-3-W-13-17-27-41-D-95} Let $X$ be a quasismooth
hypersurface of degree $95$ in $\mathbb{P}(13,17,27,41)$. Then
$\lct(X)=\frac{65}{24}$.
\end{lemma}

\begin{proof}
 The surface $X$ can be defined by the
quasihomogeneous equation
$$
z^{2}t+y^{4}z+xt^{2}+x^{6}y=0.
$$
The surface $X$ is singular at the point $O_x$, $O_{y}$, $O_z$ and $O_t$.
Each of the divisors $C_x$, $C_y$, $C_z$, and $C_t$ consists of two irreducible and reduced components. The divisor $C_x$ (resp. $C_y$, $C_z$, $C_t$) consists of $L_{xz}$ (resp. $L_{yt}$, $L_{xz}$, $L_{yt}$)
and $R_x=\{x=y^4+zt=0\}$ (resp. $R_y=\{y=z^2+xt=0\}$, $R_z=\{z=t^2+x^5y=0\}$, $R_t=\{t=x^6+y^3z=0\}$). The curve $L_{xz}$ intersects $R_x$ (resp. $R_z$) only at the point $O_t$ (resp. $O_y$). Also, the curve $L_{yt}$ intersects $R_y$ (resp. $R_t$) only at the point $O_x$ (resp. $O_z$).

It is easy to check
$$-K_X\cdot L_{xz}=\frac{3}{17\cdot 41},\ -K_X\cdot L_{yt}=\frac{1}{9\cdot 13},\ -K_X\cdot R_{x}=\frac{4}{9\cdot 41},%
$$
$$
 -K_X\cdot R_{y}=\frac{6}{13\cdot 41},\ -K_X\cdot R_{z}=\frac{6}{13\cdot 17},\
 -K_X\cdot R_{t}=\frac{2}{3\cdot 17},%
$$
$$
L_{xz}^2=-\frac{55}{17\cdot 41},\  L_{yt}^2=-\frac{37}{13\cdot 27},\
R_x^2=-\frac{56}{27\cdot 41},\  R_y^2=-\frac{48}{13\cdot 41},\
R_{z}^2=\frac{28}{13\cdot 17}
$$
$$
R_{t}^2=\frac{16}{3\cdot 17},\ L_{xz}\cdot R_{x}=\frac{4}{41}, \ L_{yt}\cdot R_{y}=\frac{2}{13}, \
L_{xz}\cdot R_{z}=\frac{2}{17},\
L_{yt}\cdot R_{t}=\frac{2}{9}.
$$

We have $\lct(X)\leqslant \frac{65}{24}$ since
$$
\frac{65}{24}=\lct\left(X, \frac{3}{13}C_x\right)<\frac{51}{12}=\lct\left(X, \frac{3}{17}C_y\right)<\frac{41}{8}=\lct\left(X, \frac{3}{41}C_t\right)<\frac{21}{4}=\lct\left(X, \frac{3}{27}C_z\right).%
$$

Suppose that $\lct(X)<\frac{65}{24}$. Then there is an effective
$\Q$-divisor $D\qlineq -K_X$ such that the pair
$(X,\frac{65}{24}D)$ is not log canonical at some point $P$. By
Lemma~\ref{lemma:convexity}, we may assume that the support of $D$
 does not contain at least one irreducible
component of each of the curves $C_{x}$, $C_{y}$, $C_{z}$ and
$C_{t}$. The curve $R_z$  is singular at the point $O_y$. The
curve $C_t$ is singular at $O_z$ with multiplicity $3$. Then in
each of the following pairs of inequalities, at least one of two
must hold:
\[\mult_{O_x}(D)\leqslant 13D\cdot L_{yt}=\frac{1}{9}<\frac{24}{65}, \ \
\mult_{O_x}(D)\leqslant 13D\cdot R_y=\frac{6}{41}<\frac{24}{65};\]
\[\mult_{O_y}(D)\leqslant 17D\cdot L_{xz}=\frac{3}{41}<\frac{24}{65}, \ \
\mult_{O_y}(D)\leqslant \frac{17}{2}D\cdot R_z=\frac{3}{13}<\frac{24}{65};\]
\[\mult_{O_z}(D)\leqslant 27D\cdot L_{yt}=\frac{3}{13}<\frac{24}{65}, \ \
\mult_{O_z}(D)\leqslant \frac{27}{3}D\cdot
R_t=\frac{6}{17}<\frac{24}{65};\] Therefore,
the point $P$ can be none of $O_x$, $O_y$, $O_t$.

Put $D=m_0L_{xz}+m_1L_{yt}+m_2R_x+m_3R_y+m_4R_z+m_5R_t+\Omega$, where  $\Omega$ is an effective $\mathbb{Q}$-divisor whose support contains none of  $L_{xz}$, $L_{yt}$, $R_x$, $R_y$, $R_z$, $R_t$. Since the pair $(X, \frac{65}{24}D)$ is log canonical at the points $O_x$, $O_y$, $O_z$, we have $m_i\leqslant \frac{24}{65}$ for each $i$. Since
\[(D-m_0L_{xz})\cdot L_{xz}=\frac{3+55m_0}{17\cdot 41}\leqslant \frac{24}{65},\ \
(D-m_1L_{yt})\cdot L_{yt}=\frac{3+37m_1}{13\cdot 27}\leqslant \frac{24}{65},\]
\[(D-m_2R_{x})\cdot R_{x}=\frac{12+56m_2}{27\cdot 41}\leqslant \frac{24}{65},\ \
(D-m_3R_{y})\cdot R_{y}=\frac{6+48m_3}{13\cdot 41}\leqslant \frac{24}{65},\]
\[(D-m_4R_{z})\cdot R_{z}=\frac{6-28m_4}{13\cdot 17}\leqslant \frac{24}{65},\ \
(D-m_5R_{t})\cdot R_{t}=\frac{2-16m_5}{3\cdot 17}\leqslant \frac{24}{65}\]
Lemma~\ref{lemma:handy-adjunction}
implies that the point $P$ cannot be a smooth point of $X$ on $C_x\cup C_y\cup C_z\cup C_t$. Therefore, the point $P$ is either
a point  in the outside of  $C_x\cup C_y\cup C_z\cup C_t$ or the point $O_t$.

Suppose that the point $P$ is not the point $O_t$.
We consider the pencil $\mathcal{L}$ on $X$ defined by the equations
$\lambda xt+\mu z^2=0$, $[\lambda:\mu]\in\mathbb{P}^1$. Then
there is a unique curve $Z_{\alpha}$ in the pencil $\mathcal{L}$ passing through the point $P$.
 Since the point $P$ is located in the outside of $C_{x}\cup C_{z}\cup C_{t}$, the curve $Z_\alpha$ is defined by an equation of the form
$$
xt+\alpha z^{2}=0,
$$
where $\alpha$ is a non-zero constant.
Note that any component of $C_t$ is not contained in $Z_\alpha$.
The open subset
$Z_{\alpha}\setminus C_{t}$
is a $\mathbb{Z}_{41}$-quotient of the
affine curve
$$
x+\alpha
z^{2}=z^{2}+y^{4}z+x+x^6y=0\subset\mathbb{C}^{3}\cong\mathrm{Spec}\Big(\mathbb{C}\big[x,y,z\big]\Big)
$$
that is isomorphic to the plane affine curve defined
by the equation
$$
z(y^4+(1-\alpha)z+\alpha^6z^{11}y)=0\subset\mathbb{C}^{2}\cong\mathrm{Spec}\Big(\mathbb{C}\big[y,z\big]\Big).
$$
Therefore, if $\alpha\ne 1$, then the curve $Z_\alpha$ consists of
two irreducible components $L_{xz}$ and $C_\alpha$.
On the
other hand, if $\alpha=1$, then the curve $Z_\alpha$ consists of
three irreducible components $L_{xz}$, $R_y$, and $C_{1}$.
Since $P\not\in C_{x}\cup C_{y}\cup C_z\cup C_{t}$, the point $P$
must be contained in  $C_\alpha$ (including $\alpha=1$). Also, the
curve $C_\alpha$  is smooth at the point $P$. By
Lemma~\ref{lemma:convexity}, we may assume that $\mathrm{Supp}(D)$
does not contain at least one irreducible component of the curve
$Z_{\alpha}$.

Write $D=mC_\alpha+\Gamma$, where $\Gamma$ is an effective $\mathbb{Q}$-divisor whose support contains $C_\alpha$. Suppose that $m\ne 0$. If $\alpha\ne 1$, then we obtain
\[\frac{3}{17\cdot 41}=D\cdot L_{xz}\geqslant mC_\alpha\cdot L_{xz}=\frac{109m}{17\cdot 41}\]
and hence $m\leqslant \frac{3}{109}$. If $\alpha=1$, then one of the inequalities
\[\frac{3}{17\cdot 41}=D\cdot L_{xz}\geqslant mC_1\cdot L_{xz}=\frac{92m}{17\cdot 41}, \ \
\frac{6}{13\cdot 41}=D\cdot R_y\geqslant mC_1\cdot R_y=\frac{11m}{41}\]
must hold, and hence $m\leqslant \frac{6}{11\cdot 13}$.
We also see that
\[D\cdot C_{\alpha}=\left\{%
\aligned
&D\cdot( Z_\alpha-L_{xz})=\frac{531}{13\cdot 17\cdot 41} \text{ if}\ \alpha\ne 1,\\%
&D\cdot (Z_\alpha-L_{xz}-R_y)=\frac{33}{17\cdot 41} \text{ if}\ \alpha=1.\\%
\endaligned\right.
\]
Also, if  $\alpha\ne 1$, then

\[C_{\alpha}^2=Z_\alpha\cdot C_\alpha-L_{xz}\cdot C_\alpha
\geqslant Z_\alpha\cdot C_\alpha-(L_{xz}+R_x)\cdot C_\alpha  = \frac{41}{3}D\cdot C_\alpha .\]
If $\alpha =1$,
\[C_1^2=Z_\alpha\cdot C_1-(L_{xz}+R_y)\cdot C_1
\geqslant Z_\alpha\cdot C_1-(L_{xz}+R_x+L_{yt}+R_y)\cdot C_1  =
8D\cdot C_1.\] In both cases, we have $C_\alpha^2>0$. Since
\[(D-mC_\alpha)\cdot C_{\alpha}\leqslant D\cdot C_\alpha<\frac{24}{65}\]
Lemma~\ref{lemma:handy-adjunction} gives us a contradiction.
Therefore, the point $P$ must be the point $O_t$.

If $L_{xz}$ is not contained in the support of $D$, then the
inequality $$\mult_{O_t}(D)\leqslant 41D\cdot
L_{xz}=\frac{3}{17}<\frac{24}{65}$$ is a contradiction. Therefore,
the  irreducible component $L_{xz}$ must be contained in the
support of $D$, and hence the curve $R_x$ is not contained in the
support of $D$. Put $D=aL_{xz}+bR_y+\Delta$, where $\Delta$ is an
effective $\mathbb{Q}$-divisor whose support contains neither
$L_{xz}$ nor $R_y$. Then
\[\frac{4}{9\cdot 41}=D\cdot R_x\geqslant aL_{xz}\cdot R_x+\frac{\mult_{O_t}(D)-a}{41}>\frac{3a}{41}+\frac{24}{41\cdot 65}\]
and hence $a\leqslant \frac{44}{585}$.
If $b\ne 0$, then $L_{yt}$ is not contained in the support of $D$. Therefore,
\[\frac{1}{9\cdot 13}=D\cdot L_{yt}\geqslant bR_y\cdot L_{yt}=\frac{2b}{13},\]
and hence $b\leqslant \frac{1}{18}$.

Let $\pi\colon\bar{X}\to X$ be the weighted blow up at the point
$O_t$ with weights $(1,4)$ and let $F$ be the exceptional curve of
the morphism $\pi$. Then $F$ contains one singular point $Q_{4}$
of $\bar{X}$ such that $Q_{4}$ is a singular point of type
$\frac{1}{4}(3,1)$. Then
\[
K_{\bar{X}}\qlineq\pi^*(K_X)-\frac{36}{41}F,\ \
\bar{L}_{xz}\qlineq\pi^*(L_{xz})-\frac{4}{41}F, \ \
\bar{R}_y\qlineq\pi^*(R_y)-\frac{1}{41}F,\ \
\bar{\Delta}\qlineq\pi^*(\Delta)-\frac{c}{41}F,
\]
where $\bar{L}_{xz}$, $\bar{R}_y$ and $\bar{\Delta}$ are the
proper transforms of $L_{xz}$, $R_y$ and $\Delta$ by $\pi$,
respectively, and $c$ is a non-negative rational number. Note that
$F\cap\bar{R}_y=\{Q_{4}\}$.

The log pull-back of the log pair $(X,\frac{65}{24}D)$ by $\pi$ is
the log pair
$$
\left(\bar{X},\ \frac{65a}{24}\bar{L}_{xz}+ \frac{65b}{24}\bar{R}_y+
\frac{65}{24}\bar{\Delta}+\theta_1 F\right),%
$$
where
$$\theta_1=\frac{864+65(4a+b+c)}{24\cdot 41}.$$
This  is not log canonical at some point $Q\in F$.
We have
\[0\leqslant\bar{\Delta}\cdot\bar{L}_{xz}
=\frac{3+55a}{17\cdot 41}-\frac{b}{41}-\frac{c}{41}.\]%
This inequality shows $b+c\leqslant \frac{1}{17}(3+55a)$. Since
$a\leqslant \frac{44}{585}$, we obtain
\[\theta_{1}=\frac{864+260a}{24\cdot 41}+\frac{65(b+c)}{24\cdot 41}\leqslant
\frac{864+260a}{24\cdot 41}+\frac{65(3+55a)}{17\cdot 24\cdot 41}
=\frac{121+65a}{8\cdot17}<1.\]

Suppose that the point $Q$ is neither $Q_4$ nor the intersection
point of $F$ and $\bar{L}_{xz}$. Then, the point $Q$ is not in
$\bar{L}_{xz}\cup \bar{R}_y$. Therefore, the pair $ \left(\bar{X},
\frac{65}{24}\bar{\Delta}+F\right) $ is not log canonical at the
point $Q$, and hence
$$1<\frac{65}{24}\bar{\Delta}\cdot F=\frac{65c}{4\cdot 24}.$$
But $c\leqslant b+c\leqslant \frac{1}{17}(3+55a)<\frac{4\cdot
24}{65}$ since $a\leqslant \frac{44}{585}$. Therefore, the point
$Q$ is either $Q_4$ or the intersection point of $F$ and
$\bar{L}_{xz}$.

Suppose that the point $Q$ is the intersection point of $F$ and
$\bar{L}_{xz}$. Then the point $Q$ is in $\bar{L}_{xz}$ but not in
$\bar{R}_y$. Therefore, the pair $ \left(\bar{X},
\bar{L}_{xz}+\frac{65}{24}\bar{\Delta}+\theta_1 F\right) $ is not
log canonical at the point $Q$. Then
\[1<\left(\frac{65}{24}\bar{\Delta}+\theta_1
F\right)\cdot \bar{L}_{xz}
=\frac{65}{24}\left(\frac{3+55a}{17\cdot 41}- \frac{b+c}{
41}\right)+\theta_1 =\frac{121+65a}{8\cdot 17}.\] However, this is
impossible since $a\leqslant \frac{44}{585}$. Therefore, the point
$Q$ must be the point $Q_4$.


Let $\psi\colon\tilde{X}\to \bar{X}$ be the weighted  blow up at
the point $Q_4$ with weights $(3,1)$ and let $E$ be the
exceptional curve of the morphism $\psi$. The exceptional curve
$E$ contains one singular point $O_3$ of $\tilde{X}$. This
singular point is of type $\frac{1}{3}(1,2)$. Then
\[
K_{\tilde{X}}\qlineq\psi^*(K_{\bar{X}}),\ \
\tilde{R}_{y}\qlineq\psi^*(\bar{R}_{y})-\frac{3}{4}E, \ \
\tilde{F}\qlineq\psi^*(F)-\frac{1}{4}E,
\ \
\tilde{\Delta}\qlineq\psi^*(\bar{\Delta})-\frac{d}{4}E,
\]
where $\tilde{R}_{y}$, $\tilde{F}$ and $\tilde{\Delta}$ are the
proper transforms of $\bar{R}_{y}$, $F$ and $\bar{\Delta}$ by
$\psi$, respectively, and $d$ is a non-negative rational number.

The log pull-back of the log pair $(X,\frac{65}{24}D)$ by
$\pi\circ\psi$ is the log pair
$$
\left(\tilde{X},\ \frac{65a}{24}\tilde{L}_{xz}+ \frac{65b}{24}\tilde{R}_y+
\frac{65}{24}\tilde{\Delta}+\theta_1 \tilde{F}+\theta_2E\right),%
$$
where
$$\theta_2=\frac{65(3b+d)}{4\cdot 24}+\frac{1}{4}\theta_1.$$
This  is not log canonical at some point $O\in E$.

We have
\[0\leqslant\tilde{\Delta}\cdot\tilde{R}_{y}=\bar{\Delta}\cdot\bar{R}_{y}-\frac{d}{4}= \frac{6+48b}{13\cdot
41}- \frac{4a+c}{4\cdot 41}-\frac{d}{4},\] and hence
$4a+c+41d\leqslant \frac{4}{13}(6+48b)$. Therefore, this
inequality together with $b\leqslant \frac{1}{18}$ gives us
\[\begin{split}\theta_2 &=\frac{65(3b+d)}{4\cdot 24}+
\frac{864+65(4a+b+c)}{4\cdot 24\cdot 41}=\\
&=\frac{864+8060b+65(4a+c+41d)}{4\cdot 24\cdot 41}\leqslant \\
&\leqslant \frac{6+55b}{24}<1.\\
\end{split}
\]

Suppose that the point $O$ is in the outside of $\tilde{R}_{y}$
and $\tilde{F}$. Then the log pair $(E,
\frac{65}{24}\tilde{\Delta}|_{E})$ is not log canonical at the
point $O$, and hence
\[1<\frac{65}{24}\tilde{\Delta}\cdot E=\frac{65d}{72}.\]
However,
\[41d\leqslant 4a+c+41d\leqslant \frac{4}{13}(6+48b)<\frac{41\cdot 72}{65}\]
since $b\leqslant\frac{1}{18}$. This is a contradiction.

Suppose that the point $O$ belongs to $\tilde{R}_{y}$ Then the log
pair $\left(E,  \left(\frac{65b}{24}\tilde{R}_{y}+
\frac{65}{24}\tilde{\Delta}\right)\Big|_{E}\right)$ is not log canonical at the
point $O$, and hence
\[1<\left(\frac{65b}{24}\tilde{R}_{y}+
\frac{65}{24}\tilde{\Delta}\right)\cdot E=\frac{65}{24}\left(b+\frac{d}{3}\right).\]
However,
\[ \frac{65}{24}\left(b+\frac{d}{3}\right)\leqslant
\frac{65}{24}\left(b+\frac{4}{3\cdot 13\cdot
41}\left(6+48b\right)\right)<1\] since $b\leqslant\frac{1}{18}$.
This is a contradiction. Therefore, the point $O$ is the point $O_3$ which is the
intersection point of $E$ and $\tilde{F}$.


Let $\xi\colon\hat{X}\to \tilde{X}$ be the weighted blow up at the
point $O$ with weights $(1,2)$ and let $H$ be the exceptional
divisor of $\xi$. The exceptional divisor $H$ contains a singular point of $\hat{X}$. This singular point is of type $\frac{1}{2}(1,1)$.  We have
$$
K_{\hat{X}}\qlineq\xi^*(K_{\tilde{X}}),\
\hat{E}\qlineq\xi^*(E)-\frac{1}{3}H,\
\hat{F}\qlineq\xi^*(\tilde{F})-\frac{2}{3}H,\
\hat{\Delta}\qlineq\xi^*(\tilde{\Delta})-\frac{e}{3}H,%
$$
where $\hat{E}$, $\hat{F}$, $\hat{\Delta}$,
 be the proper transforms of $E$,
$\tilde{F}$, $\tilde{\Delta}$ by $\xi$, respectively, and
 $e$ is a non-negative rational number. The log pull-back of the
log pair $(X, \frac{65}{24}D)$ via $\pi\circ \psi\circ\xi$ is
$$
\left(\hat{X},\frac{65a}{24}\hat{L}_{xz}+\frac{65b}{24}\hat{R}_y+\frac{65}{24}\hat{\Delta}
+\theta_1\hat{F}+\theta_2 \hat{E}+\theta_{3}H\right),%
$$
where $\hat{L}_{xz}$ and  $\hat{R}_y$ are the proper transforms of $\tilde{L}_{xz}$ and  $\tilde{R}_y$ by $\xi$, respectively, and
$$
\theta_3=\frac{1}{3}(2\theta_{1}+\theta_{2})+\frac{65e}{3\cdot 24}.
$$
This log pair is not log canonical at some point $A\in H$.  We
have
\[0\leqslant \hat{\Delta}\cdot \hat{F}=\bar{\Delta}\cdot F-\frac{d}{12}-\frac{e}{3}=\frac{c}{4}-\frac{d}{12}-\frac{e}{3},\]
and hence $d+4e\leqslant 3c$. Then
\[\begin{split}\theta_3=\frac{1}{3}(2\theta_{1}+\theta_{2})+\frac{65e}{3\cdot
24}=
&=\frac{3}{4}\theta_1+\frac{65(3b+d)}{3\cdot 4\cdot 24}+\frac{65e}{3\cdot 24}\leqslant\\
&=\frac{2592+65(12a+44b+3c)}{3\cdot32\cdot 41}+\frac{65(d+4e)}{3\cdot 4\cdot 24}\leqslant\\
&\leqslant \frac{2592+65(12a+44b+3c)}{3\cdot32\cdot 41}+\frac{65c}{4\cdot 24}=  \\
&= \frac{2592+65(12a+44b+44c)}{3\cdot32\cdot 41}\leqslant
 \\
&\leqslant \frac{216+65a}{8\cdot 41}+\frac{65\cdot 11(3+55a)}{3\cdot8\cdot 17\cdot 41}
=\frac{321+1040a}{3\cdot 8\cdot 17}<1\\
 \end{split}\]
 since $b+c\leqslant \frac{1}{17}(3+55a)$ and $a\leqslant
 \frac{44}{585}$.

Suppose that $A\not\in\hat{F}\cup\hat{E}$. Then the log pair
 $
\left(\hat{X},\frac{65}{24}\hat{\Delta} +\theta_{3}H\right) $ is
not log canonical at the point $A$. Applying
Lemma~\ref{lemma:adjunction}, we get
$$
1<\frac{65}{24}\hat{\Delta}\cdot H=\frac{65e}{48}.
$$
However, $$e\leqslant \frac{1}{4}(d+4e)\leqslant
\frac{3c}{4}\leqslant \frac{3}{4}(b+c)\leqslant
\frac{3(3+55a)}{4\cdot 17}< \frac{48}{65}.$$ Therefore, the point
$A$ must be either in $\hat{F}$ or in $\hat{E}$.

Suppose that $A\in\hat{F}$. Then the log pair $
\left(\hat{X},\frac{65}{24}\hat{\Delta}
+\theta_1\hat{F}+\theta_{3}H\right) $ is not log canonical at the
point $A$. Applying Lemma~\ref{lemma:adjunction}, we get
\[
1<\left(\frac{65}{24}\hat{\Delta}+\theta_{3}
H\right)\cdot\hat{F}=\frac{65}{24}\left(\frac{c}{4}-\frac{d}{12}-\frac{e}{3}\right)+
\theta_{3}=\frac{2592+65(12a+44b+44c)}{3\cdot32\cdot 41}.
\]
However,
\[\frac{2592+65(12a+44b+44c)}{3\cdot32\cdot 41}
\leqslant \frac{321+1040a}{3\cdot 8\cdot 17}<1.
\]
Therefore, the point $A$ is the intersection point of $H$ and
$\hat{E}$. Then the log pair $
\left(\hat{X},\frac{65}{24}\hat{\Delta} +\theta_2
\hat{E}+\theta_{3}H\right)$ is not log canonical at the point $A$.
From Lemma~\ref{lemma:adjunction}, we obtain
\[
1<2\left(\frac{65}{24}\hat{\Delta}+\theta_{3}H\right)\cdot\hat{E}=\frac{65}{24}\left(\frac{2d}{3}-\frac{e}{3}\right)+\theta_{3}=\frac{2592+65(12a+44b+3c)}{3\cdot32\cdot
41}+\frac{65d}{32}
\]
However,
\[\frac{2592+65(12a+44b+3c)}{3\cdot32\cdot 41}+\frac{65d}{32}=\frac{648+715b}{3\cdot8\cdot 41}+\frac{65(4a+c+41d)}{32\cdot 41}\leqslant\frac{12186+21515b}{17\cdot 24\cdot 41}<1
\]
since $b\leqslant\frac{1}{18}$ and $4a+c+41d\leqslant
\frac{4}{13}(6+48b)$. The obtained contradiction completes the
proof.
\end{proof}

\begin{lemma}
\label{lemma:I-3-W-13-27-61-98-D-196} Let $X$ be a quasismooth
hypersurface of degree $196$ in $\mathbb{P}(13,27,61,98)$. Then
$\lct(X)=\frac{91}{30}$.
\end{lemma}

\begin{proof}
 The surface $X$ can be defined by the
quasihomogeneous equation
$$
t^{2}+y^{5}z+xz^{3}+x^{13}y=0.
$$
The surface $X$ is singular at the points $O_x$, $O_{y}$ and
$O_z$. The curves $C_x$ and $C_y$ are irreducible. We have
$$
\frac{91}{30}=\lct\left(X, \frac{3}{13}C_x\right)
<\lct\left(X, \frac{3}{27}C_y\right)=\frac{15}{2}.%
$$
Therefore, $\lct(X)\leqslant\frac{91}{30}$.

Suppose that $\lct(X)<\frac{91}{30}$. Then there is an effective
$\Q$-divisor $D\qlineq -K_X$ such that the pair
$(X,\frac{91}{30}D)$ is not log canonical at some point $P$. By
Lemma~\ref{lemma:convexity}, we may assume that the support of the
divisor $D$ contains neither the curve $C_x$ nor the curve $C_y$.
Then the inequalities
\[61D\cdot C_x=\frac{2}{9}< \frac{30}{91}, \ \ \ 13D\cdot
C_y=\frac{6}{61}<\frac{30}{91}\] show that the point $P$ is a
smooth point  in the outside of $C_x$. However, $H^0(\P,
\mathcal{O}_{\P}(793))$ contains $x^{61}$, $y^{26}x^{7}$,
$y^{13}x^{34}$ and $z^{13}$, it follows from
Lemma~\ref{lemma:Carolina} that the point $P$ must be
a singular point of $X$ or a point on $C_x$. This is a
contradiction.
\end{proof}

\begin{lemma}
\label{lemma:I-3-W-15-19-43-74-D-148} Let $X$ be a quasismooth
hypersurface of degree $148$ in $\mathbb{P}(15,19,43,74)$. Then
$\lct(X)=\frac{57}{14}$.
\end{lemma}

\begin{proof}
The surface $X$ can be defined by the quasihomogeneous equation
$$
t^{2}+yz^{3}+xy^{7}+x^{7}z=0.
$$
The surface $X$ is singular at the points $O_x$, $O_y$ and $O_z$.
The curves $C_x$, $C_y$ and $C_z$ are irreducible. We can see that
$$
\lct\left(X, \frac{3}{19}C_y\right)=\frac{57}{14}<\lct\left(X, \frac{3}{15}C_x\right)=\frac{25}{6}<\lct\left(X, \frac{3}{43}C_z\right)=\frac{129}{14}.%
$$
Therefore, $\lct(X)\leqslant \frac{57}{14}$.

Suppose that $\lct(X)<\frac{57}{14}$. Then there is an effective
$\Q$-divisor $D\qlineq -K_X$ such that the pair
$(X,\frac{57}{14}D)$ is not log canonical at some point $P$. By
Lemma~\ref{lemma:convexity}, we may assume that the support of the
divisor $D$ contains none of $C_x$, $C_y$, $C_z$. Note that the curve
$C_y$ is singular at the point $O_z$. The inequalities
\[19D\cdot C_x =\frac{6}{43}<\frac{14}{57},
\ \ \ \frac{43}{2}D\cdot C_y =\frac{1}{5}<\frac{14}{57}, \ \ \ D\cdot C_z=\frac{2}{95}<\frac{14}{57}\] show
that the point $P$ is located in the outside of $C_x\cup C_y\cup C_z$.

Now we consider the pencil $\mathcal{L}$ on $X$ defined by the
equations $\lambda z^3+\mu xy^6=0$,
$[\lambda:\mu]\in\mathbb{P}^1$. Then there is a unique member $C$
in $\mathcal{L}$ passing through the point $P$. Since the point
$P$ is located in the outside of $C_x\cup C_y\cup C_z$, the curve
$C$ is cut out by the equation of the form $xy^6+\alpha z^3=0$,
where $\alpha$ is a non-zero constant. Since the curve $C$ is a
double cover of the curve defined by the equation $xy^6+\alpha
z^3=0$ in $\mathbb{P}(15,19,43)$, we have  $\mult_P(C)\leqslant
2$. Therefore, we may assume that the support of $D$ does not
contain at least one irreducible component. If $\alpha\ne 1$, then
the curve $C$ is irreducible, and hence the inequality
\[\mult_P(D)\leqslant D\cdot C=\frac{6}{5\cdot 19}<\frac{14}{57}\]
is a contradiction.
If $\alpha=1$, then the curve $C$ consists of two distinct irreducible and reduced curve $C_1$ and $C_2$.  We have
\[D\cdot C_1=D\cdot C_2=\frac{3}{5\cdot 19}, \ \ C_1^2=C_2^2=\frac{11}{19}.\]
Put $D=a_1C_1+a_2C_2+\Delta$, where $\Delta$ is an effective
$\mathbb{Q}$-divisor whose support contains neither $C_1$ nor
$C_2$. Since the pair $(X, \frac{57}{14}D)$ is log canonical at
$O_x$, both $a_1$ and $a_2$ are at most $\frac{14}{57}$. Then a
contradiction follows from Lemma~\ref{lemma:handy-adjunction}
since
\[(D-a_iC_i)\cdot C_i\leqslant D\cdot C_i=\frac{3}{5\cdot 19}<\frac{14}{57}\]
for each $i$.
\end{proof}

\section{Sporadic cases with $I=4$}
\label{section:index-4}

\begin{lemma}
\label{lemma:I-4-W-5-6-8-9-D-24} Let $X$ be a quasismooth
hypersurface of degree $24$ in $\mathbb{P}(5,6,8,9)$. Then
$\lct(X)=1$.
\end{lemma}

\begin{proof}
The surface $X$ can be defined by the
quasihomogeneous equation
$$
z^{3}+yt^{2}-y^{4}+\epsilon x^{2}yz+x^{3}t=0,
$$
where $\epsilon\in\mathbb{C}$. The surface $X$ is singular at the
points $O_x$, $O_t$, $Q_{2}=[0:1:1:0]$ and $Q_{3}=[0:1:0 :1]$.

The curves $C_x$, $C_y$, $C_{z}$ and $C_{t}$ are all irreducible.
We have
$$
1=\lct\left(X,
\frac{4}{6}C_y\right)<\lct\left(X,
\frac{4}{5}C_x\right)=\frac{5}{4}<\lct\left(X,
\frac{4}{8}C_z\right)=2
$$
and $\lct\left(X, \frac{4}{9}C_t\right)>1$.  Therefore,
$\lct(X)\leqslant 1$.

Suppose that $\lct(X)<1$. Then there is an effective $\Q$-divisor
$D\qlineq -K_X$ such that the pair $(X,D)$ is not log canonical at
some point $P$. By Lemma~\ref{lemma:convexity}, we may assume that
the support of the divisor $D$ contains none of the curves $C_x$,
$C_y$, $C_{z}$ and $C_{t}$. Also, the curve $C_y$ is singular at
the point $O_t$ with multiplicity $3$ and the curve $C_t$ is
singular at the point $O_x$. Then the following intersection
numbers show that the point $P$ is located in the outside of the
set $C_x\cup C_y\cup C_z\cup C_t$:
\[3D\cdot C_x=\frac{2}{3}<1, \ \ \frac{9}{3}D\cdot C_y=\frac{4}{5}<1,
\ \ D\cdot C_z=\frac{16}{45}<1, \ \
\frac{5}{2}D\cdot C_t=1.\]

Now we consider the pencil $\mathcal{L}$ on $X$ defined by the
equations $\lambda xt+\mu yz=0$, where $[\lambda: \mu]\in
\mathbb{P}^1$. There is a unique member $Z$ in the pencil
$\mathcal{L}$ passing through the point $P$. Since $P\not\in
C_{x}\cup C_{y}\cup C_{z}\cup C_{t}$, the divisor $Z$ is defined
by an equation of the form
$$
xt=\alpha yz,
$$
where $\alpha$ is non-zero constant. Note that the curve $C_x$ is
not contained in the support of $Z$. The open subset $Z\setminus
C_{x}$ of the curve $Z$ is a $\mathbb{Z}_{5}$-quotient of the
affine curve
$$
t-\alpha yz=z^{3}+yt^{2}+y^{4}+\epsilon
yz+t=0\subset\mathbb{C}^{3}\cong\mathrm{Spec}\Big(\mathbb{C}\big[y,z,t\big]\Big),
$$
that is isomorphic to the plane affine quintic curve  $Z'$ given
by the equation
$$
z^{3}+\alpha^{2}y^{3}z^{2}+y^{4}+(\epsilon+\alpha)yz=0
\subset\mathbb{C}^{2}\cong\mathrm{Spec}\Big(\mathbb{C}\big[y,z\big]\Big).
$$
This affine plane curve $Z'$ is irreducible and hence the curve
$Z$ is also irreducible. The multiplicity of $Z$ at the point $P$
is at most $3$ since the quintic $Z'$ is singular at the origin.
This implies that the log pair $(X, \frac{4}{14}Z)$ is log
canonical at the point $P$. Thus, we may assume that
$\mathrm{Supp}(D)$ does not contain the curve $Z$ by
Lemma~\ref{lemma:convexity}. Then we obtain a contradictory
inequality
$$
\frac{28}{45}=D\cdot Z\geqslant\mult_{P}\big(D\big)>1.%
$$
\end{proof}

\begin{lemma}
\label{lemma:I-4-W-5-6-8-15-D-30} Let $X$ be a quasismooth
hypersurface of degree $30$ in $\mathbb{P}(5,6,8,15)$. Then
$\lct(X)=\nlb 1$.
\end{lemma}

\begin{proof}
 The surface $X$ can be defined by the
quasihomogeneous equation
\[t(t-x^3)-y^5+yz^3+\epsilon x^2y^2z=0.\]
The surface $X$ is singular at the points $O_x$, $O_z$,
$Q_5=[1:0:0:1]$, $Q_3=[0:1:0:1]$ and $Q_2=[0:1:1:0]$.

The curve $C_x$ is irreducible. However, the curve $C_y$ consists
of two irreducible curves $L_{yt}$ and $L=\{y=t-x^3=0\}$. It is
easy to check
$$
1=\lct\left(X, \frac{4}{6}C_y\right)
<\lct\left(X, \frac{4}{5}C_x\right)=\frac{5}{4}.%
$$
Therefore, $\lct(X)\leqslant 1$.

Suppose that $\lct(X)<1$. Then there is an effective $\Q$-divisor
$D\qlineq -K_X$ such that the pair $(X,D)$ is not log canonical at
some point $P$. By Lemma~\ref{lemma:convexity}, we may assume that
the support of the divisor $D$ does not contain the curve $C_x$.
Similarly, we may assume that the support of $D$ does not contain
either $L_{yt}$ or $L$.

We have the following intersection numbers for $L_{yt}$ and $L$:
\[L_{yt}^2=L^2=-\frac{9}{40}, \ \ \ L_{yt}\cdot L=\frac{3}{8}.\]

Since $H^0(\P, \mathcal{O}_{\P}(30))$ contains the monomials
$y^{5}$, $yz^{3}$ and $t^{2}$, it follows from
Lemma~\ref{lemma:Carolina} that the point $P$ is
either a singular point of $X$ or a point on $C_y$. However, since
$3D\cdot C_x=\frac{1}{2}<1$, the point $P$ must belong to the
curve $C_y$.

Since the support of $D$ does not contain either $L_{yt}$ or $L$,
one of the inequalities
\[\mult_{O_z}(D)\leqslant 8D\cdot L_{yt}=\frac{4}{5}<1,
\ \ \mult_{O_z}(D)\leqslant 8D\cdot L=\frac{4}{5}<1\] must hold,
and hence the point $P$ cannot be the point $O_z$.

We put $D=kL+mL_{yt}+\Delta$, where $\Delta$ is an effective
$\mathbb{Q}$-divisor whose support contains neither $L$ nor
$L_{yt}$. If $k\ne 0$, then $m=0$ and
\[\frac{1}{10}=D\cdot L_{yt}\geqslant
kL\cdot L_{yt}=\frac{3k}{8}.\] Therefor, $k\leqslant
\frac{4}{15}$. By the same way, we can also obtain $m\leqslant
\frac{4}{15}$. Then, by Lemma~\ref{lemma:handy-adjunction}, the
inequalities
\[5(D-kL)\cdot L=\frac{4+9k}{8}<1, \ \ 5(D-mL_{yt})\cdot
L_{yt}=\frac{4+9m}{8}<1\] show that the point $P$ cannot belong to
the curve $C_y$. This is a contradiction.
\end{proof}

\begin{lemma}
\label{lemma:I-4-W-9-11-12-17-D-45}Let $X$ be a quasismooth
hypersurface of degree $45$ in $\mathbb{P}(9,11,12,17)$. Then
$\lct(X)=\frac{77}{60}$.
\end{lemma}

\begin{proof}
 We may assume that the surface $X$ is defined by
the quasihomogeneous equation
\[t^2y+y^3z+xz^3+x^5=0.\]
It is singular at the points $O_y$, $O_z$, $O_t$, and the point
$Q_3=[1:0:-1:0]$. The curve $C_x$ consists of two irreducible and
reduced  curves $L_{xy}$ and $R_x=\{x=t^2+y^2z=0\}$. The curve
$C_y$ consists of two irreducible and reduced  curves $L_{xy}$ and
$R_y=\{y=z^3+x^4=0\}$. The curves $C_z$ and $C_t$ are irreducible
and reduced. It is easy to check that
$\lct(X,\frac{4}{11}C_y)=\frac{77}{60}$ is less than each of the
numbers $\lct(X,\frac{4}{9}C_x)$, $\lct(X,\frac{4}{12}C_z)$ and
$\lct(X,\frac{4}{17}C_t)$.

Suppose that $\lct(X)<\frac{77}{60}$. Then there is an effective
$\Q$-divisor $D\qlineq -K_X$ such that the pair $(X,
\frac{77}{60}D)$ is not log canonical at some point $P$. By
Lemma~\ref{lemma:convexity} we may assume that the support of $D$
contains neither $C_z$ nor $C_t$. Similarly, we may assume that
the support of $D$ does not contain either $L_{xy}$ or $R_x$.
Also, we may assume that the support of $D$ does not contain
either $L_{xy}$ or $R_y$. Then in each of the following pairs of
inequalities, at least one of two must hold:
\[\mult_{O_z}(D) \leqslant 12D\cdot L_{xy}=\frac{4}{17}<\frac{60}{77}, \ \
\mult_{O_z}(D) \leqslant 12D\cdot
R_x=\frac{8}{11}<\frac{60}{77};\]
\[\mult_{O_t}(D) \leqslant 17D\cdot L_{xy}=\frac{1}{3}<\frac{60}{77}, \ \
\mult_{O_t}(D) \leqslant \frac{17}{3}D\cdot
R_y=\frac{4}{9}<\frac{60}{77}.\] Therefore, the point $P$ can be
neither $O_z$ nor $O_t$. The curve $C_z$ is singular at the point
$O_y$. Then the inequalities \[\frac{11}{2} D\cdot
C_z=\frac{10}{17}<\frac{60}{77},\ \  3D\cdot
C_t=\frac{5}{11}<\frac{60}{77}\] imply that the point $P$ cannot
belong to $C_z\cup C_t$.

We can see that
\[L_{xy}\cdot D=\frac{1}{17\cdot 3}, \ \ R_{x}\cdot D=\frac{2}{33}, \ \ R_{y}\cdot D=\frac{4}{3\cdot 17},
 \ \ L_{xy}\cdot R_x=\frac{1}{6},\]
 \[L_{xy}\cdot R_y=\frac{3}{17}, \ \ L_{xy}^2=-\frac{25}{12\cdot 17}, \ \ R_x^2=-\frac{1}{33}, \ \ R_y^2=\frac{2}{3\cdot 17}.\]

If we write $D=nL_{xy}+\Delta$, where $\Delta$ is an effective
$\Q$-divisor whose support does not contain the curve $L_{xy}$,
then we can see that $n\leqslant\nlb\frac{4}{11}$ since $D\cdot
R_x\geqslant nR_x\cdot L_{xy}$ for $n\ne 0$. By
Lemma~\ref{lemma:handy-adjunction} the inequality
$$
(L_{xy}\cdot D-nL^2_{xy})=\frac{4+25n}{12\cdot17}< \frac{60}{77}
$$ implies that the point $P$ cannot
belong to the curve $L_{xy}$. By the same method, we see that the point
$P$ must be in the outside of $R_x$.

If we write $D=mR_y+\Omega$, where $\Omega$ is an effective
$\Q$-divisor whose support does not contain the curve $R_y$, then
we can see that $0\leqslant\nlb m\leqslant\nlb \frac{1}{9}$ since $D\cdot
L_{xy}\geqslant mR_y\cdot L_{xy}$ for $m\ne 0$. By
Lemma~\ref{lemma:handy-adjunction} the inequality
$$
(R_y\cdot D-mR^2_y)\leqslant R_y\cdot
D<\frac{60}{77}
$$ implies that the point $P$ cannot belong to the curve $R_y$.

Now we consider the pencil $\mathcal{L}$ on $X$ cut out by
$\lambda t^2+\mu y^2z=0$. The base locus of the pencil
$\mathcal{L}$ consists of three points $O_y$, $O_z$, and $Q$. Let
$F$ be the member in $\mathcal{L}$ defined by $t^2+y^2z=0$. The
divisor $F$ consists of two irreducible and reduced curves $R_x$
and $E=\{t^2+y^2z=x^4+z^3=0\}$. The   curve $E$ is smooth in the
outside of the base points. We have
\[E\cdot D=\frac{8}{33}.\]
Since
\[E^2=F\cdot E-R_x\cdot E\geqslant F\cdot E-(L_{xy}+R_x)\cdot E=\frac{25}{4}D\cdot E,\]
the self-intersection $E^2$ is positive.
We write $D=kE+\Gamma$, where $\Gamma$ is an effective
$\Q$-divisor whose support does not contain the curve $E$. Since
$(X, \frac{77}{60}D)$ is log canonical at the point $O_y$, the
non-negative number $k$ is at most $\frac{60}{77}$. By
Lemma~\ref{lemma:handy-adjunction}, the inequality
$$
(E\cdot D-kE^2)\leqslant E\cdot D=\frac{8}{33}<\frac{60}{77}
$$
implies that the point $P$ cannot belong to the curve $E$.

So far we have seen that the point $P$ must lie in the outside of
$C_x\cup C_y\cup C_z\cup C_t\cup E$. In particular, it is a smooth
point. There is a unique member $C$ in $\mathcal{L}$ which passes
through the point $P$. Then the curve $C$ is cut out by
$t^2=\alpha y^2z$ where $\alpha$ is a constant different from $0$
and $-1$. The curve $C$ is isomorphic to the curve defined by
$y^3z+xz^3+x^5=0$ and $t^2=y^2z$. The curve $C$ is smooth in the
outside of the base points and the singular locus of $X$  by the
Bertini theorem, since it is isomorphic to a general curve in the
pencil $\mathcal{L}$. We claim that the curve $C$ is irreducible.
If so then we may assume that the support of $D$ does not contain
the curve $C$ and hence we obtain
\[\mult_{P}(D)\leqslant C\cdot D=\frac{10}{33}<\frac{60}{77}.\]
This is a contradiction.

For the irreducibility of the curve $C$, we may consider the curve
$C$ as a surface in $\mathbb{C}^4$ defined by the equations
$y^3z+xz^3+x^5=0$ and $t^2=y^2z$. Then, we consider the surface in
$\P^4$ defined by the equations $y^3zw+xz^3w+x^5=0$ and
$t^2w=y^2z$.  We take the affine piece defined by $t\ne 0$.
This affine piece is isomorphic to the surface defined by the
equation $y^3zw+xz^3w+x^5=0$ and $w=y^2z$ in $\mathbb{C}^4$. It is
isomorphic to the irreducible hypersurface $y^5z^2+xy^2z^5+x^5=0$ in
$\mathbb{C}^3$. Therefore, the curve $C$ is irreducible.
\end{proof}

\begin{lemma}
\label{lemma:I-4-W-10-13-25-31-D-75}Let $X$ be a quasismooth
hypersurface of degree $75$ in $\mathbb{P}(10,13,25,31)$. Then
$\lct(X)=\frac{91}{60}$.
\end{lemma}

\begin{proof}
 We may assume that the surface $X$ is defined by
the quasihomogeneous equation
\[t^2y+z^3+xy^5+x^5z=0.\]
It has singular points at $O_x$, $O_y$, $O_t$ and $Q=[-1:0:1:0]$.
The curve $C_x$ and $C_t$  are irreducible and reduced. The curve
$C_y$ (resp. $C_z$) consists of two irreducible reduced curves
$L_{yz}$ and $R_y=\{y=z^2+x^5=0\}$
(resp. $R_z=\{y=t^2+xy^4=0\}$). It is easy to see that
\[\lct(X, \frac{4}{13}C_y)=\frac{91}{60}<\lct(X, \frac{4}{10}C_x)<\lct(X, \frac{4}{25}C_z)<\lct(X, \frac{4}{31}C_t).\]

Also, we have the following intersection numbers:
\[-K_X\cdot L_{yz}=\frac{2}{5\cdot 31}, \ \ -K_X\cdot R_y =\frac{4}{5\cdot 31}, \ \ -K_X\cdot R_z =\frac{4}{5\cdot 13},\]
\[L_{yz}\cdot R_y=\frac{5}{31}, \ \ L_{yz}\cdot R_z=\frac{1}{5}, \ \ L_{yz}^2 =-\frac{37}{10\cdot 31},
 \ \ R_y^2=-\frac{12}{5\cdot 31} , \ \ R_z^2=\frac{12}{5\cdot 13}.\]

Suppose that $\lct(X)<\frac{91}{60}$. Then, there is an effective
$\Q$-divisor $D\qlineq -K_X$ such that the log pair $(X,
\frac{91}{60}D)$ is not log canonical at some point $P\in X$.
Since the curves $C_x$ and $C_t$ are irreducible we may assume
that the support of $D$ contains none of them. The inequalities
\[13D\cdot C_x<\frac{60}{91}, \ \ \ 5D\cdot C_t<\frac{60}{91}\]
show that the point $P$ must lie in the outside of $C_x\cup
C_t\setminus \{O_x, O_t\}$.

By Lemma~\ref{lemma:convexity}, we may assume that the support of
$D$ does not contain either $L_{yz}$ or $R_y$. If the support of
$D$ does not contain~$L_{yz}$, then the inequality
\[31 D\cdot L_{yz}=\frac{2}{5}<\frac{60}{91}\]
shows that the point $P$ cannot be $O_t$. On the other hand, if
the support of $D$ does not contain~$R_y$, then the inequality
\[\frac{31}{2} D\cdot R_{y}=\frac{2}{5}<\frac{60}{91}\]
shows that the point $P$ cannot be $O_t$. Note that the curve $R_y$ is singular at the point $O_t$. We use the same method
for $C_{z}=R_z+L_{yz}$ so that we can see that the point $P$
cannot be $O_x$.

We write $D=mR_y+\Omega$, where $\Omega$ is an effective
$\Q$-divisor whose support does not contain the curve $R_y$. Then
we see $m\leqslant \frac{2}{25}$ since the support of $D$ does not
contain either $L_{yz}$ or $R_y$ and  $D\cdot L_{yz} \geqslant
mR_y\cdot L_{yz}$. Since $R_y\cdot D-mR_y^2<\frac{60}{91}$,
Lemma~\ref{lemma:handy-adjunction} implies that the point $P$ is
located in the outside of $R_y$. Using the same argument for
$L_{yz}$ , we can also see that the point $P$ is located in the
outside of $L_{yz}$. Also, the same method shows that the point
$P$ is located in the outside of $R_z$. Consequently, the point
$P$ must lie in the outside of $C_x\cup C_y\cup C_z\cup C_t$.

Now we consider the pencil $\mathcal{L}$ on $X$ cut out by
$\lambda t^2+\mu xy^4=0$. The base locus of the pencil
$\mathcal{L}$ consists of three points $O_x$, $O_y$, and $Q$. Let
$F$ be the member of $\mathcal{L}$ defined by $t^2+xy^4=0$. The
divisor $F$ consists of two irreducible and reduced curves $R_z$
and $E=\{t^2+xy^4= z^2+x^5=0\}$. The   curve $E$ is smooth in the
outside of $\mathrm{Sing}(X)$.
We have
\[E\cdot D=\frac{8}{5\cdot 13}.\]
Since
\[E^2=F\cdot E-R_z\cdot E\geqslant F\cdot E-(L_{yz}+R_z)\cdot E=\frac{37}{4}D\cdot E,\]
the self-intersection $E^2$ is positive.
We write $D=kE+\Gamma$, where $\Gamma$ is an effective
$\Q$-divisor whose support does not contain the curve $E$. Since
$(X, \frac{91}{60}D)$ is log canonical at the point $O_y$, the
non-negative number $k$ is at most $\frac{60}{91}$. By
Lemma~\ref{lemma:handy-adjunction}, the inequality
$$
(E\cdot D-kE^2)\leqslant E\cdot D=\frac{8}{5\cdot 13}<\frac{60}{91}
$$
implies that the point $P$ cannot belong to the curve $E$.

So far we have seen that the point $P$ must lie in the outside of
$C_x\cup C_y\cup C_z\cup C_t\cup E$. In particular, it is a smooth
point. There is a unique member $C$ in $\mathcal{L}$ which passes
through the point $P$. Then the curve $C$ is cut out by
$t^2=\alpha xy^4$ where $\alpha$ is a constant different from $0$
and $-1$. The curve $C$ is isomorphic to the curve defined by
$xy^5+z^3+x^5z=0$ and $t^2=xy^4$. The curve $C$ is smooth in the
outside of the base points and the singular locus of $X$ by Bertini theorem, since it is
isomorphic to a general curve in the pencil $\mathcal{L}$. We
claim that the curve $C$ is irreducible. If so then we may assume
that the support of $D$ does not contain the curve $C$ and hence
we obtain
\[\mult_{P}(D)\leqslant C\cdot D=\frac{12}{5\cdot 13}<\frac{60}{91}.\]
This is a contradiction.

For the irreducibility of the curve $C$, we may consider the curve
$C$ as a surface in $\mathbb{C}^4$ defined by the equations
$xy^5+z^3+x^5z=0$ and $t^2=xy^4$. Then, we consider the surface in
$\P^4$ defined by the equations $xy^5+w^3z^3+x^5z=0$ and
$t^2w^3=xy^4$.  We then take the affine piece defined by $y\ne 0$.
This affine piece is isomorphic to the surface defined by the
equation $x+w^3z^3+x^5z=0$ and $t^2w^3=x$ in $\mathbb{C}^4$. It is
isomorphic to the hypersurface defined by
$t^2w^3+w^3z^3+t^{10}w^{15}z=0$ in~$\mathbb{C}^3$. It has two
irreducible components $w=0$ and $t^2+z^3+t^{10}w^{12}z=0$. The
former component originates from the hyperplane at infinity in
$\P^4$. Therefore, the curve $C$ must be irreducible.
\end{proof}

\begin{lemma}
\label{lemma:I-4-W-11-17-20-27-D-71}Let $X$ be a quasismooth
hypersurface of degree $71$ in $\mathbb{P}(11,17,20,27)$. Then
$\lct(X)=\frac{11}{6}$.
\end{lemma}

\begin{proof}
We may assume that the surface $X$ is defined by the quasihomogeneous equation
\[t^2y+y^3z+xz^3+x^4t=0.\]
The surface $X$ is singular at the points $O_x$, $O_y$, $O_z$, $O_t$. Each of the divisors $C_x$, $C_y$, $C_z$, and $C_t$ consists of two irreducible and reduced components. The divisor $C_x$ (resp. $C_y$, $C_z$, $C_t$) consists of $L_{xy}$ (resp. $L_{xy}$, $L_{zt}$, $L_{zt}$)
and $R_x=\{x=y^2z+t^2=0\}$ (resp. $R_y=\{y=x^3t+z^3=0\}$, $R_z=\{z=x^4+yt=0\}$, $R_t=\{t=y^3+xz^2=0\}$).
Also, we see that
\[L_{xy}\cap R_x=\{O_z\}, \ L_{xy}\cap R_y=\{O_t\}, \ L_{zt}\cap R_z=\{O_y\}, \ L_{zt}\cap R_t=\{O_x\}.\]
One can easily check that $\lct(X, \frac{11}{4}C_x)=\frac{11}{6}$
is less than each of the numbers $\lct(X, \frac{17}{4}C_y)$,
$\lct(X, \frac{20}{4}C_z)$ and $\lct(X, \frac{27}{4}C_t)$.
Therefore, $\lct(X)\leqslant \frac{11}{6}$. Suppose
$\lct(X)<\frac{11}{6}$. Then, there is an effective $\Q$-divisor
$D\qlineq -K_X$ such that the log pair $(X, \frac{11}{6}D)$ is not
log canonical at some point $P\in X$.

The intersection numbers among the divisors $D$, $L_{xy}$, $L_{zt}$, $R_x$, $R_y$, $R_z$, $R_t$ are as follows:

\[D\cdot L_{xy}=\frac{1}{5\cdot 27}, \ \ D\cdot R_x=\frac{2}{5\cdot 17}, \ \ D\cdot R_y=\frac{4}{9\cdot 11}, \]
\[D\cdot L_{zt}=\frac{4}{11\cdot 17}, \ \ D\cdot R_z=\frac{16}{17\cdot 27}, \ \ D\cdot R_t=\frac{3}{5\cdot 11}, \]
\[L_{xy}\cdot R_x=\frac{1}{10}, \ \ L_{xy}\cdot R_y=\frac{1}{9}, \ \ L_{zt}\cdot R_z=\frac{4}{17}, \ \ L_{zt}\cdot R_t=\frac{3}{11},\]

\[L_{xy}^2=-\frac{43}{20\cdot 27}, \ \ R_x^2=-\frac{3}{5\cdot 17}, \ \ R_y^2=\frac{2}{3\cdot 11},\]

\[L_{zt}^2=-\frac{24}{11\cdot 17}, \ \ R_z^2=-\frac{28}{17\cdot 27}, \ \ R_t^2=\frac{21}{20\cdot 11}.\]
By Lemma~\ref{lemma:convexity} we may assume that the support of $D$ does not contain at least one component of each divisor $C_x$, $C_y$, $C_z$, $C_t$.
Since the curve $R_t$ is singular at the point $O_x$ and the curve $R_y$ is singular at the point $O_t$ with multiplicity $3$, in
each of the following pairs of inequalities, at least one of two
must hold:
\[\mult_{O_x}(D)\leqslant 11D\cdot L_{zt}=\frac{4}{17}<\frac{6}{11}, \ \ \  \mult_{O_x}(D)\leqslant \frac{11}{2}D\cdot R_t=\frac{3}{10}<\frac{6}{11};\]
\[\mult_{O_z}(D)\leqslant 20D\cdot L_{xy}=\frac{4}{27}<\frac{6}{11}, \ \ \  \mult_{O_z}(D)\leqslant 20D\cdot R_x=\frac{8}{17}<\frac{6}{11};\]
\[\mult_{O_t}(D)\leqslant 27D\cdot L_{xy}=\frac{1}{5}<\frac{6}{11}, \ \ \  \mult_{O_t}(D)\leqslant \frac{27}{3}D\cdot R_y=\frac{4}{11}<\frac{6}{11}.\]
Therefore,
the point $P$ can be none of $O_x$, $O_z$, $O_t$.

Suppose that the point $P$ is the point $O_y$. We then put $D=mL_{zt}+\Delta$, where $\Delta$
is an effective $\Q$-divisor whose support does not contain the
curve $L_{zt}$.  If $m=0$, then
\[\mult_{O_y}(D)\leqslant 17D\cdot L_{zt}=\frac{4}{11}<\frac{6}{11}.\]
This is a contradiction. Therefore, $m>0$, and hence the support of $D$ does not contain the curve $R_z$.
Since
$$
\frac{16}{17\cdot 27}=D\cdot R_z\geqslant \frac{4m}{17}+\frac{\mult_{O_y}(D)-m}{17}>\frac{3m}{17}+\frac{6}{11\cdot 17}%
$$
we obtain $m<\frac{14}{3\cdot 11\cdot 27}$. However, we obtain
\[17(D-mL_{zt})\cdot L_{zt}=\frac{4+24m}{11}>\frac{6}{11}\]
from Lemma~\ref{lemma:handy-adjunction}. This is a contradiction. Therefore, the point $P$ is a smooth point of $X$.

We write
$D=a_0L_{xy}+a_1L_{zt}+a_2R_x+a_3R_y+a_4R_z+a_5R_t+\Omega$, where  $\Omega$
is an effective $\Q$-divisor whose support contains none of the
curves $L_{xy}$, $L_{zt}$, $R_x$, $R_y$, $R_z$, $R_t$. Since the pair $(X,
\frac{11}{6}D)$ is log canonical at the points $O_x$,
$O_z$, $O_t$, the numbers $a_i$ are at most $\frac{6}{11}$. Then
by Lemma~\ref{lemma:handy-adjunction} the following inequalities
enable us to conclude that the point $P$ is in the outside of
$C_x\cup C_y\cup C_z\cup C_t$:
\[(D-a_0L_{xy})\cdot L_{xy}=\frac{4+43a_0}{20\cdot 27}\leqslant \frac{6}{11},\ \ (D-a_1L_{zt})\cdot L_{zt}=\frac{4+24a_1}{11\cdot 17}\leqslant \frac{6}{11}, \]
\[(D-a_2R_{x})\cdot R_{x}=\frac{2+3a_2}{5\cdot 17}\leqslant \frac{6}{11},
\ \ (D-a_3R_{y})\cdot R_{y}=\frac{4-6a_3}{9\cdot 11}\leqslant \frac{6}{11}, \]
\[(D-a_4R_{z})\cdot R_{z}=\frac{16+28a_4}{17\cdot 27}\leqslant \frac{6}{11},
\ \ (D-a_5R_{t})\cdot R_{t}=\frac{12-21a_5}{20\cdot 11}\leqslant \frac{6}{11}. \]

We consider the pencil $\mathcal{L}$ defined by $\lambda ty+\mu
x^4=0$, $[\lambda:\mu]\in\P^1$. The base locus of the pencil
$\mathcal{L}$ consists of the curve $L_{xy}$ and the point $O_y$.
Let $E$ be the unique divisor in $\mathcal{L}$ that passes through
the point $P$. Since $P\not\in C_x\cup C_y\cup C_z\cup C_t$, the
divisor $E$ is defined by the equation $ty=\alpha x^4$, where
$\alpha\ne 0$.

Suppose that $\alpha\ne -1$. Then the curve $E$ is isomorphic to the curve defined by the equations $ty=x^4$ and $x^4t+y^3z+xz^3=0$. Since the curve $E$ is isomorphic to a general curve in $\mathcal{L}$, it is smooth at the point $P$. The affine piece of $E$ defined by $t\ne 0$ is the curve given by $x(x^2+x^{11}z+z^3)=0$. Therefore, the divisor $E$ consists of two irreducible and reduced curves $L_{xy}$ and $C$. We have
\[D\cdot C=D\cdot E-D\cdot L_{xy}=\frac{267}{5\cdot 17\cdot 27}, \]
\[C^2=E\cdot C- L_{xy}\cdot C\geqslant E\cdot C- L_{xy}\cdot C- R_x\cdot C =\frac{33}{4}D\cdot C>0.\]
By Lemma~\ref{lemma:handy-adjunction} the inequality $D\cdot C<\frac{6}{11}$ gives us a contradiction.

Suppose that $\alpha=-1$. Then divisor $E$ consists of three irreducible and reduced curves $L_{xy}$, $R_z$, and $M$.  Note that the curve $M$ is different from the curves $R_x$ and $L_{zt}$. Also, it is smooth at the point $P$. We have
\[D\cdot M=D\cdot E- D\cdot L_{xy}-D\cdot R_z=\frac{11}{5\cdot 27},\]
\[M^2=E\cdot M-L_{xy}\cdot M- R_z\cdot M \geq E\cdot M-C_x\cdot M- C_z\cdot M=\frac{13}{4}D\cdot M>0.\]
By Lemma~\ref{lemma:handy-adjunction} the inequality $D\cdot M<\frac{6}{11}$ gives us a contradiction.
\end{proof}

\begin{lemma}
\label{lemma:I-4-W-11-17-24-31-D-79}Let $X$ be a quasismooth
hypersurface of degree $79$ in $\mathbb{P}(11,17,24,31)$. Then
$\lct(X)=\frac{33}{16}$.
\end{lemma}

\begin{proof}
We may assume that the surface $X$ is defined by the
quasihomogeneous equation
\[t^2y+tz^2+xy^4+x^5z=0.\]
The surface $X$ is singular at the points $O_x$, $O_y$, $O_z$,
$O_t$. Each of the divisors $C_x$, $C_y$, $C_z$, and $C_t$
consists of two irreducible and reduced components. The divisor
$C_x$ (resp. $C_y$, $C_z$, $C_t$) consists of $L_{xt}$ (resp.
$L_{yz}$, $L_{yz}$, $L_{xt}$) and $R_x=\{x=yt+z^2=0\}$ (resp.
$R_y=\{y=zt+x^5=0\}$, $R_z=\{z=xy^3+t^2=0\}$,
$R_t=\{t=y^4+x^4z=0\}$). Also, we see that
\[L_{xt}\cap R_x=\{O_y\}, \ L_{yz}\cap R_y=\{O_t\}, \ L_{yz}\cap R_z=\{O_x\}, \ L_{xt}\cap R_t=\{O_z\}.\]
One can easily check that $\lct(X, \frac{4}{11}C_x)=\frac{33}{16}$
is less than each of the numbers $\lct(X, \frac{4}{17}C_y)$,
$\lct(X, \frac{4}{24}C_z)$ and $\lct(X, \frac{4}{31}C_t)$.
Therefore, $\lct(X)\leqslant \frac{33}{16}$. Suppose
$\lct(X)<\frac{33}{16}$. Then, there is an effective $\Q$-divisor
$D\qlineq -K_X$ such that the log pair $(X, \frac{33}{16}D)$ is
not log canonical at some point $P\in X$.

The intersection numbers among the divisors $D$, $L_{xt}$,
$L_{yz}$, $R_x$, $R_y$, $R_z$, $R_t$ are as follows:

\[D\cdot L_{xt}=\frac{1}{6\cdot 17}, \ \ D\cdot R_x=\frac{8}{17\cdot 31}, \ \ D\cdot R_y=\frac{5}{6\cdot 31}, \]
\[D\cdot L_{yz}=\frac{4}{11\cdot 31}, \ \ D\cdot R_z=\frac{8}{11\cdot 17}, \ \ D\cdot R_t=\frac{2}{3\cdot 11}, \]
\[L_{xt}\cdot R_x=\frac{2}{17}, \ \ L_{yz}\cdot R_y=\frac{5}{31}, \ \ L_{yz}\cdot R_z=\frac{2}{11}, \ \ L_{xt}\cdot R_t=\frac{1}{6},\]

\[L_{xt}^2=-\frac{37}{17\cdot 24}, \ \ R_x^2=-\frac{40}{17\cdot 31}, \ \ R_y^2=-\frac{35}{24\cdot 31},\]

\[L_{yz}^2=-\frac{38}{11\cdot 31}, \ \ R_z^2=\frac{14}{11\cdot 17}, \ \ R_t^2=\frac{10}{3\cdot 11}.\]
By Lemma~\ref{lemma:convexity} we may assume that the support of
$D$ does not contain at least one component of each divisor $C_x$,
$C_y$, $C_z$, $C_t$. The inequalities
\[17D\cdot L_{xt}=\frac{1}{6}<\frac{16}{33}, \ \ \  17D\cdot R_x=\frac{8}{31}<\frac{16}{33}\]
imply that $P\ne O_y$. The inequalities
\[11D\cdot L_{yz}=\frac{4}{31}<\frac{16}{33}, \ \ \  11D\cdot R_z=\frac{8}{17}<\frac{16}{33}\]
imply that $P\ne O_x$. Since the curve $R_t$ is singular at the
point $O_z$ with multiplicity $4$ the inequalities
\[24D\cdot L_{xt}=\frac{ 24}{6\cdot 17}<\frac{16}{33}, \ \ \  \frac{24}{4}D\cdot R_t=\frac{4}{11}<\frac{16}{33}\]
imply that $P\ne O_z$.

We write
$D=a_1L_{xt}+a_2L_{yz}+a_3R_x+a_4R_y+a_5R_z+a_6R_t+\Omega$, where
$\Omega$ is an effective $\Q$-divisor whose support contains none
of the curves $L_{xt}$, $L_{yz}$, $R_x$, $R_y$, $R_z$, $R_t$.
Since the pair $(X, \frac{33}{16}D)$ is log canonical at the
points $O_x$, $O_y$, $O_z$, the numbers $a_i$ are at most
$\frac{16}{33}$. Then by Lemma~\ref{lemma:handy-adjunction} the
following inequalities enable us to conclude that either the point
$P$ is in the outside of $C_x\cup C_y\cup C_z\cup C_t$ or $P=O_t$:
\[\frac{33}{16}D\cdot L_{xt}-L_{xt}^2=\frac{181}{3\cdot 17\cdot 32}<1, \ \ \
\frac{33}{16}D\cdot R_x-R_x^2=\frac{113}{2\cdot 17\cdot 31}<1, \ \ \ \frac{33}{16}D\cdot R_y-R_y^2=\frac{25}{3\cdot 31}<1,\]
\[\frac{33}{16}D\cdot L_{yz}-L_{xt}^2=\frac{185}{4\cdot 11\cdot 31}<1, \ \ \ \frac{33}{16}D\cdot R_z-R_z^2=\frac{5}{2\cdot 11\cdot 17}<1, \ \ \ \frac{33}{16}D\cdot R_t-R_t^2=\frac{-47}{3\cdot 8\cdot 11}<1.\]

Suppose that $P\ne O_t$. Then we consider the pencil $\mathcal{L}$
defined by $\lambda yt+\mu z^2=0$, $[\lambda:\mu]\in\P^1$. The
base locus of the pencil $\mathcal{L}$ consists of the curve
$L_{yz}$ and the point $O_y$. Let $E$ be the unique divisor in
$\mathcal{L}$ that passes through the point $P$. Since $P\not\in
C_x\cup C_y\cup C_z\cup C_t$, the divisor $E$ is defined by the
equation $z^2=\alpha yt$, where $\alpha\ne 0$.

Suppose that $\alpha\ne -1$. Then the curve $E$ is isomorphic to
the curve defined by the equations $yt=z^2$ and
$t^2y+xy^4+x^5z=0$. Since the curve $E$ is isomorphic to a general
curve in $\mathcal{L}$, it is smooth at the point $P$. The affine
piece of $E$ defined by $t\ne 0$ is the curve given by
$z(z+xz^7+x^5)=0$. Therefore, the divisor $E$ consists of two
irreducible and reduced curves $L_{yz}$ and $C$. We have the
intersection numbers
\[D\cdot C=D\cdot E-D\cdot L_{yz}=\frac{564}{11\cdot 17\cdot 31}, \ \ C\cdot L_{yz}=E\cdot L_{yz}-L_{yz}^2=\frac{2}{11}.\]
Also, we see
\[C^2=E\cdot C- C\cdot L_{yz}>0.\]
By Lemma~\ref{lemma:handy-adjunction} the inequality $D\cdot
C<\frac{16}{33}$ gives us a contradiction.

Suppose that $\alpha=-1$. Then divisor $E$ consists of three
irreducible and reduced curves $L_{yz}$, $R_x$, and $M$.  Note
that the curve $M$ is different from the curves $R_y$ and
$L_{xt}$. Also, it is smooth at the point $P$. We have
\[D\cdot M=D\cdot E- D\cdot L_{yz}-D\cdot R_x=\frac{4\cdot 119}{11\cdot 17\cdot 31},\]
\[M^2=E\cdot M-L_{yz}\cdot M- R_x\cdot M \geq E\cdot M-C_y\cdot M- C_x\cdot M=5D\cdot M>0.\]
By Lemma~\ref{lemma:handy-adjunction} the inequality $D\cdot
M<\frac{16}{33}$ gives us a contradiction. Therefore, $P=O_t$.

 We write $D=aL_{yz}+bR_{x}+\Delta$, where $\Delta$ is an effective divisor whose support contains neither $L_{yz}$ nor $R_x$.
Note that we already assumed that the support of $D$ cannot
contain either $L_{yz}$ or $R_y$. If the support of $D$ contains
$R_y$, then it does not contain $L_{yz}$. However, the inequality
$31D\cdot L_{yz}=\frac{4}{11}<\frac{16}{33}$ shows that $P\ne
O_t$. Therefore, the support of $D$ does not contain the curve
$R_y$. The inequality $D\cdot L_{xt}\geq b R_x \cdot L_{xt}$
implies $b\leqslant \frac{1}{12}$. On the other hand, we have
\[\frac{5}{6\cdot 31}=D\cdot R_y\geq \frac{5a}{31}+\frac{b}{31}+
\frac{\mult_{O_t}(D)-a-b}{31}>\frac{4a}{31}+\frac{16}{31\cdot 33},\]
and hence $a<\frac{23}{4\cdot 66}$.

Let $\pi\colon\bar{X}\to X$ be the weighted blow up of $O_{t}$
with weights $(7,4)$ and let $F$ be the exceptional curve of
$\pi$. Then
$$
K_{\bar{X}}\qlineq \pi^{*}(K_{X})-\frac{20}{31}F,\ %
\bar{L}_{yz}\qlineq \pi^{*}(L_{yz})-\frac{4}{31}F,\ %
\bar{R}_{x}\qlineq \pi^{*}(R_{x})-\frac{7}{31}F,\ %
\bar{\Delta}\qlineq \pi^{*}(\Delta)-\frac{c}{31}F,
$$
where $\bar{\Delta}$, $\bar{L}_{yz}$, $\bar{R}_{x}$ are the proper
transforms of $\Delta$, $L_{yz}$, $R_{x}$, respectively, and $c$
is a non-negative rational number. The curve $F$ contains two
singular points $Q_7$ and $Q_4$ of $\bar{X}$. The point $Q_7$ is a
singular point of type $\frac{1}{7}(1,1)$ and the point $Q_4$ is
of type $\frac{1}{4}(1,3)$.  Note that the curve $\bar{R}_x$
passes through the point $Q_4$ but not the point $Q_7$. The curve
$\bar{L}_{yz}$  passes through the point $Q_7$ but not the point
$Q_4$.

The log pull-back of the log pair $(X,\frac{33}{16}D)$ by $\pi$ is
the log pair
$$
\left(\bar{X},\
\frac{33a}{16}\bar{L}_{yz}+ \frac{33b}{16}\bar{R}_x+ \frac{33}{16}\bar{\Delta}+
\theta_1 F\right),%
$$
where
$$
\theta_1=\frac{33(4a+7b+c)+320}{16\cdot 31}.
$$
This pair is not log canonical at some point $Q\in F$. We have
\[0\leqslant\bar{\Delta}\cdot\bar{R}_x=\frac{8+40b}{17\cdot 31}-\frac{a}{31}
-\frac{c}{4\cdot 31}.\] This inequality shows
$4a+c\leqslant\frac{4}{17}(8+40b)$. Then
\[\theta_1=\frac{33(4a+c)+231b+320}{16\cdot 31}
\leqslant \frac{6496+9207b}{16\cdot 17\cdot 31}<1\] since
$b\leqslant\frac{1}{12}$.

Suppose that  the point $Q$ is neither the point $Q_7$ nor the
point $Q_4$. Then the log pair
$\left(\bar{X},\frac{33}{16}\bar{\Delta}+F\right)$ is not log
canonical at the point $Q$. Then
$$
\frac{33c}{16\cdot 28}=\frac{33}{16}\bar{\Delta}\cdot F>1
$$
by Lemma~\ref{lemma:adjunction}. However, $c\leqslant
4a+c\leqslant\frac{4}{17}(8+40b)$. This is a contradiction since
$b\leqslant \frac{1}{12}$. Therefore, the point $Q$ is either the
point $Q_7$ or the point $Q_4$.

Suppose that the point $Q$ is the point $Q_4$. This point is the
intersection point of $F$ and $\bar{R}_x$. Then the log pair
$\left(\bar{X},\frac{33b}{16}\bar{R}_{x}+
\frac{33}{16}\bar{\Delta}+\theta_1 F\right)$ is not log canonical
at the point $Q$. It then follows from
Lemma~\ref{lemma:adjunction} that
\[1<4\left(\frac{33}{16}\bar{\Delta}+\theta_{1}F\right)\cdot\bar{R}_{x}
=\frac{33\cdot 4}{16}\left(\frac{8+40b}{17\cdot 31}-\frac{a}{31}
-\frac{c}{4\cdot 31}\right)+\theta_{1}.\]
However,
\[\frac{33\cdot 4}{16}\left(\frac{8+40b}{17\cdot 31}-\frac{a}{31}
-\frac{c}{4\cdot 31}\right)+\theta_{1}
=\frac{6496+9207b}{16\cdot 17\cdot 31}<1.\] Therefore,
the point $Q$ is the point $Q_7$. This point is the intersection
point of $F$ and $\bar{L}_{yz}$.

Let $\phi\colon\tilde{X}\to \bar{X}$ be the  blow up at the point
$Q_{7}$. Let $G$ be the exceptional divisor of the morphism
$\phi$. The surface $\tilde{X}$ is smooth along the exceptional
divisor $G$. Let $\tilde{L}_{yz}$, $\tilde{R}_x$, $\tilde{\Delta}$
and $\tilde{F}$ be the proper transforms of $L_{yz}$, $R_x$,
$\Delta$ and $F$ by $\phi$, respectively. We have
$$
K_{\tilde{X}}\qlineq\phi^*(K_{\bar{X}})-\frac{5}{7}G,
\ \tilde{L}_{yz}\qlineq\phi^*(\bar{L}_{yz})-\frac{1}{7}G,
\ \tilde{F}\qlineq\phi^*(F)-\frac{1}{7}G,
\ \tilde{\Delta}\qlineq\phi^*(\bar{\Delta})-\frac{d}{7}G,%
$$
where $d$ is a non-negative rational number. The log pull-back of
the log pair $(X, \frac{33}{16}D)$ via $\pi\circ \phi$ is
$$
\left(\tilde{X},
\frac{33a}{16}\tilde{L}_{yz}+\frac{33b}{16}\tilde{R}_x+\frac{33}{16}\tilde{\Delta}+\theta_1\tilde{F}+\theta_2 G\right),%
$$
where
$$\theta_2=\frac{33}{7\cdot 16}(a+d)+\frac{\theta_1}{7}+\frac{5}{7}
=\frac{2800+33(35a+7b+c+31d)}{7\cdot 16\cdot 31}.$$
This log pair is not log canonical at some point $O\in G$. We have
\[0\leqslant \tilde{\Delta}\cdot\tilde{L}_{yz}
=\frac{4+38a}{11\cdot 31}-\frac{b}{31}-\frac{c}{7\cdot
31}-\frac{d}{7}.\] We then obtain $7b+c+31d\leqslant
\frac{7}{11}(4+38a)$. Since $a\leqslant\frac{23}{264}$, we see
\[\theta_2=\frac{2800+33(35a+7b+c+31d)}{7\cdot 16\cdot 31}\leqslant
\frac{4532+3069a}{11\cdot 16\cdot 31}<1. \]

Suppose that $O\not\in\tilde{F}\cup\tilde{L}_{yz}$. The log pair
$\left(\tilde{X},\frac{13}{8}\tilde{\Delta}+G\right)$  is not log
canonical at the point $O$. Applying Lemma~\ref{lemma:adjunction},
we get
$$
1<\frac{33}{16}\tilde{\Delta}\cdot G=\frac{33d}{16},%
$$
and hence $d>\frac{16}{33}$. However, $d\leqslant
\frac{1}{31}(7b+c+31d)\leqslant \frac{7}{11\cdot 31}(4+38a)$. This
is a contradiction since $a\leqslant\frac{23}{264}$. Therefore,
the point $O$ is either the intersection point of $G$ and
$\tilde{F}$ or the intersection point of $G$ and $\tilde{L}_{yz}$.
In the latter case, the pair $
\left(\tilde{X},\frac{33a}{16}\tilde{L}_{yz}+\frac{33}{16}\tilde{\Delta}+\theta_2 G\right)%
$ is not log canonical at the point $O$. Then, applying
Lemma~\ref{lemma:adjunction}, we get
\[
1<\left(\frac{33}{16}\tilde{\Delta}+\theta_{2}G\right)\cdot\tilde{L}_{yz}=\frac{33}{16}\left(\frac{4+38a}{11\cdot
31}-\frac{b}{31}-\frac{c}{7\cdot
31}-\frac{d}{7}\right)+\theta_{2}.\] However,
\[\frac{33}{16}\left(\frac{4+38a}{11\cdot 31}-\frac{b}{31}-\frac{c}{7\cdot
31}-\frac{d}{7}\right)+\theta_{2} =\frac{4532+3069a}{11\cdot
16\cdot 31}<1.\] Therefore, the point $O$ must be the intersection
point of $G$ and $\tilde{F}$.

Let $\xi\colon\hat{X}\to \tilde{X}$ be the blow up at the
point $O$  and let $H$ be the exceptional
divisor of $\xi$. We also let $\hat{L}_{yz}$, $\hat{R}_x$, $\hat{\Delta}$,
$\hat{G}$, and $\hat{F}$ be the proper transforms of $\tilde{L}_{yz}$,
$\tilde{R}_x$, $\tilde{\Delta}$,  $G$ and $\tilde{F}$ by $\xi$, respectively. Then
$\hat{X}$ is smooth along the exceptional divisor $H$. We have
$$
K_{\hat{X}}\qlineq\xi^*(K_{\tilde{X}})-H,\
\hat{G}\qlineq\xi^*(G)-H,\
\hat{F}\qlineq\xi^*(\tilde{F})-H,\
\hat{\Delta}\qlineq\xi^*(\tilde{\Delta})-eH,%
$$
where $e$ is a non-negative rational number. The log pull-back of the
log pair $(X, \frac{33}{16}D)$ via $\pi\circ \phi\circ \xi$ is
$$
\left(\hat{X},\frac{33a}{16}\hat{L}_{yz}+\frac{33b}{16}\hat{R}_x+\frac{33}{16}\hat{\Delta}+\theta_1\hat{F}+\theta_2 \hat{G}+\theta_{3}H\right),%
$$
where
$$
\theta_3=\theta_1+\theta_2+\frac{33e}{16}-1=\frac{1568+33(63a+56b+8c+31d+217e)}{7\cdot 16\cdot 31}.$$
This log pair is not log canonical at some point $A\in H$. We have
\[
\frac{c}{28}-\frac{d}{7}-e=\hat{\Delta}\cdot\hat{F}\geqslant 0.\]
Therefore, $4d+28e\leqslant c$.

Then
\[\begin{split}\theta_3&=\frac{1568+33(63a+56b+8c)}{7\cdot 16\cdot 31}+\frac{33\cdot 31(d+7e)}{7\cdot 16\cdot 31}\leqslant\\
&\leqslant \frac{6272+33(252a+224b+63c)}{4\cdot7\cdot 16\cdot 31}=\\
&=\frac{6272+7392b}{4\cdot7\cdot 16\cdot 31}+\frac{33\cdot 63(4a+c)}{4\cdot7\cdot 16\cdot 31}\leqslant\\
&\leqslant \frac{28+33b}{2\cdot 31}+\frac{9\cdot 33(1+5b)}{2\cdot 17 \cdot 31}=\frac{773+2046b}{2\cdot17\cdot 31}<1\\
\end{split}
\]
since $b\leqslant \frac{1}{12}$ and $4a+c\leqslant
\frac{4}{17}(8+40b)$. In particular, $\theta_3$ is a positive
number.

Suppose that $A\not\in\hat{F}\cup\hat{G}$. Then the log pair $
\left(\hat{X},\frac{33}{16}\hat{\Delta}+\theta_{3}H\right)
$  is not log canonical at the point $A$. Applying
Lemma~\ref{lemma:adjunction}, we get
$$
1<\frac{33}{16}\hat{\Delta}\cdot H=\frac{33e}{16}.
$$
However, $$e\leqslant \frac{1}{28}(4d+28e)\leqslant \frac{c}{28}\leqslant \frac{1}{28}(4a+c)\leqslant
\frac{4(8+40b)}{17\cdot 28}\leqslant \frac{4}{11}.$$
Therefore, the point $A$ must be either in $\hat{F}$ or in $\hat{G}$.

Suppose that $A\in\hat{F}$. Then the log pair  $
\left(\hat{X},\frac{33}{16}\hat{\Delta}+\theta_1\hat{F}+\theta_{3}H\right)%
$ is not log canonical at the point $A$. Applying Lemma~\ref{lemma:adjunction},
we get
\[
1<\left(\frac{33}{16}\hat{\Delta}+\theta_{3}
H\right)\cdot\hat{F}=\frac{33}{16}\left(\frac{c}{28}-\frac{d}{7}-e\right)+
\theta_{3}=\frac{6272+33(252a+224b+63c)}{4\cdot7\cdot 16\cdot 31}.
\]
However,
\[\frac{6272+33(252a+224b+63c)}{4\cdot7\cdot 16\cdot 31}\leqslant \frac{773+2046b}{2\cdot17\cdot 31}<1.
\]
Therefore, the point $A$ is the intersection point of $H$ and
$\hat{G}$. Then the log pair $
\left(\hat{X},\frac{33}{16}\hat{\Delta}+\theta_2
\hat{G}+\theta_{3}H\right) $ is not log canonical at the point
$A$. From Lemma~\ref{lemma:adjunction}, we obtain
\[
1<\left(\frac{33}{16}\hat{\Delta}+\theta_{3}H\right)\cdot\hat{G}=\frac{33}{16}\left(d-e\right)+\theta_{3}=\frac{1568+33(63a+56b+8c+248d)}{7\cdot
16\cdot 31}.
\]
However,
\[\frac{1568+33(63a+56b+8c+248d)}{7\cdot 16\cdot 31}=\frac{224+297a}{16\cdot 31}
+\frac{33(7b+c+31d)}{2\cdot 7\cdot 31}\leqslant\frac{320+1209a}{16\cdot 31}<1
\]
since $a<\frac{23}{4\cdot 66}$ and $7b+c+31d\leqslant \frac{7}{11}(4+38a)$.
The obtained contradiction completes the
proof.
\end{proof}

\begin{lemma}
\label{lemma:I-4-W-11-31-45-83-D-166} Let $X$ be a quasismooth
hypersurface of degree $166$ in $\mathbb{P}(11,31,45,83)$. Then
$\lct(X)=\frac{55}{24}$.
\end{lemma}

\begin{proof}
 The surface $X$ can be defined by the
quasihomogeneous equation
$$
t^{2}+yz^{3}+xy^{5}+x^{11}z=0.
$$
The surface $X$ is singular only at the points $O_x$, $O_{y}$ and
$O_z$. The curves $C_x$ and $C_y$ are irreducible. We have
$$
\frac{55}{24}=\lct\left(X, \frac{4}{11}C_x\right)
<\lct\left(X, \frac{4}{31}C_y\right)=\frac{13\cdot 31}{88}.%
$$
Therefore, $\lct(X)\leqslant \frac{55}{24}$.

Suppose that $\lct(X)<\frac{55}{24}$. Then there is an effective
$\Q$-divisor $D\qlineq -K_X$ such that the pair
$(X,\frac{55}{24}D)$ is not log canonical at some point $P$. By
Lemma~\ref{lemma:convexity}, we may assume that the support of the
divisor $D$ contains neither $C_x$ nor $C_y$. Then the
inequalities
\[45D\cdot C_x=\frac{8}{31}<\frac{24}{55}, \ \ 11D\cdot
C_y=\frac{8}{45}<\frac{24}{55}\] show that the point $P$ is a
smooth point in the outside of $C_x$. However, $H^0(\P,
\mathcal{O}_{\P}(495))$ contains the monomials $x^{45}$,
$y^{11}x^{14}$ and $z^{11}$, it follows from
Lemma~\ref{lemma:Carolina} that the point $P$ is
either a singular point of $X$ or a point on $C_x$. This is a
contradiction.
\end{proof}

\begin{lemma}
\label{lemma:I-4-W-13-14-19-29-D-71}
 Let $X$ be a quasismooth
hypersurface of degree $71$ in $\mathbb{P}(13,14,19,29)$. Then
$\lct(X)=\frac{65}{36}$.
\end{lemma}

\begin{proof}
We may assume that the surface $X$ is defined by the quasihomogeneous equation
\[ty^3+yz^3+xt^2+x^4z=0.\]
The surface $X$ is singular at the points $O_x$, $O_y$, $O_z$, $O_t$. Each of the divisors $C_x$, $C_y$, $C_z$, and $C_t$ consists of two irreducible and reduced components. The divisor $C_x$ (resp. $C_y$, $C_z$, $C_t$) consists of $L_{xy}$ (resp. $L_{xy}$, $L_{zt}$, $L_{zt}$)
and $R_x=\{x=z^3+ty^2=0\}$ (resp. $R_y=\{y=x^3z+t^2=0\}$,
$R_z=\{z=y^3+xt=0\}$, $R_t=\{t=x^4+yz^2=0\}$).
Also, we see that
\[L_{xy}\cap R_x=\{O_t\}, \ L_{xy}\cap R_y=\{O_z\}, \ L_{zt}\cap R_z=\{O_x\}, \ L_{zt}\cap R_t=\{O_y\}.\]
One can easily check that $\lct(X, \frac{4}{13}C_x)=\frac{65}{36}$
is less than each of the numbers $\lct(X, \frac{4}{14}C_y)$,
$\lct(X, \frac{4}{19}C_z)$ and $\lct(X, \frac{4}{29}C_t)$.
Therefore, $\lct(X)\leqslant \frac{65}{36}$. Suppose
$\lct(X)<\frac{65}{36}$. Then, there is an effective $\Q$-divisor
$D\qlineq -K_X$ such that the log pair $(X, \frac{65}{36}D)$ is
not log canonical at some point $P\in X$.

The intersection numbers among the divisors $D$, $L_{xy}$, $L_{zt}$, $R_x$, $R_y$, $R_z$, $R_t$ are as follows:

\[D\cdot L_{xy}=\frac{4}{19\cdot 29}, \ \ D\cdot R_x=\frac{6}{7\cdot 29}, \ \ D\cdot R_y=\frac{8}{13\cdot 19}, \]
\[D\cdot L_{zt}=\frac{2}{7\cdot 13}, \ \ D\cdot R_z=\frac{12}{13\cdot 29}, \ \ D\cdot R_t=\frac{8}{7\cdot 19}, \]
\[L_{xy}\cdot R_x=\frac{3}{29}, \ \ L_{xy}\cdot R_y=\frac{2}{19}, \ \ L_{zt}\cdot R_z=\frac{3}{13}, \ \ L_{zt}\cdot R_t=\frac{2}{7},\]

\[L_{xy}^2=-\frac{44}{19\cdot 29}, \ \ R_x^2=-\frac{3}{14\cdot 29}, \ \ R_y^2=\frac{2}{13\cdot 19},\]

\[L_{zt}^2=-\frac{23}{13\cdot 14}, \ \ R_z^2=-\frac{30}{13\cdot 29}, \ \ R_t^2=\frac{20}{7\cdot 19}.\]
By Lemma~\ref{lemma:convexity} we may assume that the support of $D$ does not contain at least one component of each divisor $C_x$, $C_y$, $C_z$, $C_t$.
Since the curve $R_t$ is singular at the point $O_y$ and the curve $R_y$ is singular at the point $O_z$, in
each of the following pairs of inequalities, at least one of two
must hold:
\[\mult_{O_x}(D)\leqslant 13D\cdot L_{zt}=\frac{2}{7}<\frac{36}{65}, \ \ \ \mult_{O_x}(D)\leqslant  13D\cdot R_z=\frac{12}{29}<\frac{36}{65};\]
\[\mult_{O_y}(D)\leqslant 14D\cdot L_{zt}=\frac{4}{13}<\frac{36}{65}, \ \ \ \mult_{O_y}(D)\leqslant  \frac{14}{2}D\cdot R_t=\frac{8}{19}<\frac{36}{65};\]
\[\mult_{O_z}(D)\leqslant 19D\cdot L_{xy}=\frac{4}{29}<\frac{36}{65}, \ \ \  \mult_{O_z}(D)\leqslant \frac{19}{2}D\cdot R_y=\frac{4}{13}<\frac{36}{65};\]
\[\mult_{O_t}(D)\leqslant 29D\cdot L_{xy}=\frac{4}{19}<\frac{36}{65}, \ \ \  \mult_{O_t}(D)\leqslant \frac{29}{2}D\cdot R_x=\frac{3}{7}<\frac{36}{65}.\]
Therefore,
the point $P$ can be none of $O_x$, $O_y$, $O_z$, $O_t$.

We write
$D=a_0L_{xy}+a_1L_{zt}+a_2R_x+a_3R_y+a_4R_z+a_5R_t+\Omega$, where
$\Omega$ is an effective $\Q$-divisor whose support contains none
of the curves $L_{xy}$, $L_{zt}$, $R_x$, $R_y$, $R_z$, $R_t$.
Since the pair $(X, \frac{65}{36}D)$ is log canonical at the
points $O_x$, $O_y$, $O_z$, $O_t$, the numbers $a_i$ are at most
$\frac{36}{65}$. Then by Lemma~\ref{lemma:handy-adjunction} the
following inequalities enable us to conclude that the point $P$
must be located in the outside of $C_x\cup C_y\cup C_z\cup C_t$:
\[(D-a_0L_{xy})\cdot L_{xy}=\frac{4+44a_0}{19\cdot 29}\leqslant \frac{36}{65},\ \ (D-a_1L_{zt})\cdot L_{zt}=\frac{4+23a_1}{13\cdot 14}\leqslant \frac{36}{65}, \]
\[(D-a_2R_{x})\cdot R_{x}=\frac{12+3a_2}{14\cdot 29}\leqslant \frac{36}{65},
\ \ (D-a_3R_{y})\cdot R_{y}=\frac{8-2a_3}{13\cdot 19}\leqslant \frac{36}{65}, \]
\[(D-a_4R_{z})\cdot R_{z}=\frac{12+30a_4}{13\cdot 29}\leqslant \frac{36}{65},
\ \ (D-a_5R_{t})\cdot R_{t}=\frac{8-20a_5}{7\cdot 19}\leqslant \frac{36}{65}. \]

We consider the pencil $\mathcal{L}$ defined by $\lambda tx+\mu y^3=0$, $[\lambda:\mu]\in\P^1$. The base locus of the pencil consists of the curve $L_{xy}$ and the point $O_x$.
Let $E$ be the unique divisor in $\mathcal{L}$ that passes through the point $P$. Since $P\not\in C_x\cup C_y\cup C_z\cup C_t$, the divisor $E$ is defined by the equation $tx=\alpha y^3$, where $\alpha\ne 0$.

Suppose that $\alpha\ne -1$. Then the curve $E$ is isomorphic to the curve defined by the equations $tx=y^3$ and $xt^2+yz^3+x^4z=0$. Since the curve $E$ is isomorphic to a general curve in $\mathcal{L}$, it is smooth at the point $P$. The affine piece of $E$ defined by $t\ne 0$ is the curve given by $y(y^2+y^{11}z+z^3)=0$. Therefore, the divisor $E$ consists of two irreducible and reduced curves $L_{xy}$ and $C$. We have
\[D\cdot C=D\cdot E-D\cdot L_{xy}=\frac{800}{13\cdot 19\cdot 29}.\]
Also, we see
\[C^2=E\cdot C- C\cdot L_{xy}\geq E\cdot C -C_x\cdot C>0.\]
By Lemma~\ref{lemma:handy-adjunction} the inequality $D\cdot C<\frac{36}{65}$ gives us a contradiction.

Suppose that $\alpha=-1$. Then divisor $E$ consists of three irreducible and reduced curves $L_{xy}$, $R_z$, and $M$.  Note that the curve $M$ is different from the curves $R_x$ and $L_{zt}$. Also, it is smooth at the point $P$. We have
\[D\cdot M=D\cdot E- D\cdot L_{xy}-D\cdot R_z=\frac{572}{13\cdot 19\cdot 29},\]
\[M^2=E\cdot M-L_{xy}\cdot M- R_z\cdot M \geq E\cdot M-C_x\cdot M- C_z\cdot M=\frac{5}{2}D\cdot M>0.\]
By Lemma~\ref{lemma:handy-adjunction} the inequality $D\cdot
M<\frac{36}{65}$ gives us a contradiction.
\end{proof}

\begin{lemma}
\label{lemma:I-4-W-13-14-23-33-D-79}
Let $X$ be a quasismooth
hypersurface of degree $79$ in $\mathbb{P}(13,14,23,33)$. Then
$\lct(X)=\frac{65}{32}$.
\end{lemma}

\begin{proof}
 The surface $X$ can be defined by the
quasihomogeneous equation
$$
z^{2}t+y^{4}z+xt^{2}+x^{5}y=0.
$$
The surface $X$ is singular at $O_x$, $O_y$, $O_z$ and $O_{t}$.
We have
$$
\lct\left(X, \frac{4}{13}C_x\right)=\frac{65}{32}<\lct\left(X, \frac{4}{13}C_x\right)=\frac{21}{8}<\lct\left(X, \frac{5}{25}C_t\right)=\frac{33}{10}<\lct\left(X, \frac{4}{23}C_z\right)=\frac{69}{20}.%
$$
In particular, $\lct(X)\leqslant \frac{65}{32}$.

 Each of the divisors $C_x$, $C_y$, $C_z$, and $C_t$ consists of two irreducible and reduced components. The divisor $C_x$ (resp. $C_y$, $C_z$, $C_t$) consists of $L_{xz}$ (resp. $L_{yt}$, $L_{xz}$, $L_{yt}$)
and $R_x=\{x=y^4+zt=0\}$ (resp. $R_y=\{y=z^2+xt=0\}$, $R_z=\{z=x^4y+t^2=0\}$, $R_t=\{t=x^5+y^3z=0\}$).
The curve $L_{xz}$ intersects $R_x$ (resp. $R_z$) only at the point $O_t$ (resp. $O_y$). The curve $L_{yt}$ intersects $R_y$ (resp. $R_t$) only at the point $O_x$ (resp. $O_z$).

We suppose that $\lct(X)<\frac{65}{32}$. Then there is an effective
$\Q$-divisor $D\qlineq -K_X$ such that the log pair $(X,
\frac{65}{32}D)$ is not log canonical at some point $P\in X$.

The intersection numbers among the divisors $D$, $L_{xz}$, $L_{yt}$, $R_x$, $R_y$, $R_z$, $R_t$ are as follows:
$$
L_{xz}^2=-\frac{43}{14\cdot 33},\ R_x^2=-\frac{40}{23\cdot 33},\ L_{xz}\cdot R_{x}=\frac{4}{33},\ D\cdot L_{xz}=\frac{4}{14\cdot 33},\ D\cdot R_{x}=\frac{16}{23\cdot 33},%
$$
$$
L_{yt}^2=-\frac{32}{13\cdot 23},\ R_{y}^2=-\frac{38}{13\cdot 33},\ \ L_{yt}\cdot R_{y}=\frac{2}{13},\ D\cdot L_{yt}=\frac{4}{13\cdot 23},\ D\cdot R_{y}=\frac{8}{13\cdot 33},%
$$
$$
R_{z}^2=\frac{20}{13\cdot 14},\ L_{xz}\cdot R_{z}=\frac{2}{14},\ D\cdot R_{z}=\frac{8}{13\cdot 14},%
$$
$$
R_{t}^2=\frac{95}{14\cdot 13},\ L_{yt}\cdot R_{t}=\frac{5}{23},\ D\cdot R_{t}=\frac{20}{14\cdot 23}.%
$$
By Lemma~\ref{lemma:convexity} we may assume that the support of $D$ does not contain at least one component of each divisor $C_x$, $C_y$, $C_z$, $C_t$.
Since the curve $R_t$ is singular at the point $O_z$ with multiplicity $3$ and the curve $R_z$ is singular at the point $O_y$, in
each of the following pairs of inequalities, at least one of two
must hold:
\[\mult_{O_x}(D)\leqslant 13D\cdot L_{yt}=\frac{4}{23}<\frac{32}{65}, \ \ \ \mult_{O_x}(D)\leqslant  13D\cdot R_y=\frac{8}{33}<\frac{32}{65};\]
\[\mult_{O_y}(D)\leqslant 14D\cdot L_{xz}=\frac{4}{33}<\frac{32}{65}, \ \ \ \mult_{O_y}(D)\leqslant  \frac{14}{2}D\cdot R_z=\frac{4}{13}<\frac{32}{65};\]
\[\mult_{O_z}(D)\leqslant 23D\cdot L_{yt}=\frac{4}{13}<\frac{32}{65}, \ \ \  \mult_{O_z}(D)\leqslant \frac{23}{3}D\cdot R_t=\frac{10}{21}<\frac{32}{65}.\]
Therefore,
the point $P$ can be none of $O_x$, $O_y$, $O_z$.

Put $D=m_0L_{xz}+m_1L_{yt}+m_2R_x+m_3R_y+m_4R_z+m_5R_t+\Omega$, where  $\Omega$ is an effective $\mathbb{Q}$-divisor whose support contains none of  $L_{xz}$, $L_{yt}$, $R_x$, $R_y$, $R_z$, $R_t$. Since the pair $(X, \frac{65}{32}D)$ is log canonical at the points $O_x$, $O_y$, $O_z$, we have $m_i\leqslant \frac{32}{65}$ for each $i$. Since
\[(D-m_0L_{xz})\cdot L_{xz}=\frac{4+43m_0}{14\cdot 33}\leqslant \frac{32}{65},\ \
(D-m_1L_{yt})\cdot L_{yt}=\frac{4+32m_1}{13\cdot 23}\leqslant \frac{32}{65},\]
\[(D-m_2R_{x})\cdot R_{x}=\frac{16+40m_2}{23\cdot 33}\leqslant \frac{32}{65},\ \
(D-m_3R_{y})\cdot R_{y}=\frac{8+38m_3}{13\cdot 33}\leqslant \frac{32}{65},\]
\[(D-m_4R_{z})\cdot R_{z}=\frac{8-20m_4}{13\cdot 14}\leqslant \frac{32}{65},\ \
(D-m_5R_{t})\cdot R_{t}=\frac{20-95m_5}{14\cdot 23}\leqslant \frac{32}{65}\]
Lemma~\ref{lemma:handy-adjunction}
implies that the point $P$ cannot be a smooth point of $X$ on $C_x\cup C_y\cup C_z\cup C_t$. Therefore, the point $P$ is either
a point  in the outside of  $C_x\cup C_y\cup C_z\cup C_t$ or the point $O_t$.

Suppose that $P\not\in C_{x}\cup C_{y}\cup C_{z}\cup C_{t}$. Then we consider the pencil $\mathcal{L}$ on $X$
defined by the equations $\lambda xt+\mu z^2=0$, $[\lambda :\mu]\in\mathbb{P}^1$. There
is a unique curve $Z_\alpha$ in the pencil passing through the point $P$. This curve is cut out by
$$
xt+\alpha z^{2}=0,
$$
 where $\alpha$ is a non-zero constant.

 The curve
$Z_{\alpha}$ is reduced. But it is always reducible. Indeed, one
can easily check that
$$
Z_{\alpha}=C_{\alpha}+L_{xz}
$$
where $C_{\alpha}$ is a reduced curve whose support contains no
$L_{xy}$. Let us prove that $C_{\alpha}$ is irreducible if
$\alpha\ne 1$.

Any component of the curve $C_t$ is not contained in the curve $Z_\alpha$.
The open subset $Z_{\alpha}\setminus  C_{t}$ of
the curve $Z_{\alpha}$ is a $\mathbb{Z}_{33}$-quotient of the
affine curve
$$
x+\alpha
z^{2}=z^{2}+y^{4}z+x+x^5y=0\subset\mathbb{C}^{3}\cong\mathrm{Spec}\Big(\mathbb{C}\big[x,y,z\big]\Big),
$$
that is isomorphic to a plane affine curve defined by the
equation
$$
z\left((\alpha-1)z+y^{4}-\alpha^5yz^{9}\right)=0\subset\mathbb{C}^{2}\cong\mathrm{Spec}\Big(\mathbb{C}\big[y,z\big]\Big).
$$
Thus, if $\alpha\ne 1$, then the curve $Z_\alpha$ consists of two
irreducible and reduced curves $L_{xz}$ and $C_\alpha$. If
$\alpha=1$, then the curve $Z_\alpha$ consists of three
irreducible and reduced curves $L_{xz}$, $R_y$, and $C_1$. In both
cases, the curve $C_\alpha$ (including $\alpha=1$) is smooth at
the point $P$. By Lemma~\ref{lemma:convexity}, we may assume that
$\mathrm{Supp}(D)$ does not contain at least one irreducible
component of the curve $Z_{\alpha}$.

If $\alpha\ne 1$, then
$$
 D\cdot C_{\alpha}=\frac{8}{13\cdot 14},%
$$
$$C_\alpha^2=Z_\alpha\cdot C_\alpha-L_{xz}\cdot C_\alpha\geqslant Z_\alpha\cdot C_\alpha-(R_x+L_{xz})\cdot C_\alpha=\frac{33}{4}D\cdot C_\alpha>0.$$
If $\alpha=1$, then
$$
 D\cdot C_{1}=\frac{152}{13\cdot 14\cdot 33},%
$$
$$C_1^2=Z_1\cdot C_1-(L_{xz}+R_y)\cdot C_1\geqslant Z_1\cdot C_1-(R_x+L_{xz})\cdot C_1-(L_{yt}+R_y)\cdot C_1=\frac{19}{4}D\cdot C_1>0.$$

We put $D=mC_\alpha+\Delta_\alpha$, where $\Delta_{\alpha}$ is an
effective $\mathbb{Q}$-divisor such that
$C_{\alpha}\not\subset\mathrm{Supp}(\Delta_{\alpha})$. Since
$C_{\alpha}$ intersects the curve $C_t$ and the pair $(X,
\frac{65}{32}D)$ is log canonical along the curve $C_t$, we obtain
$m\leqslant \frac{32}{65}$. Then, the inequality
\[(D-mC_\alpha)\cdot C_\alpha\leqslant D\cdot C_\alpha<\frac{32}{65}\]
implies that the pair $(X, \frac{65}{32}D)$  is log canonical at the point $P$ by Lemma~\ref{lemma:handy-adjunction}. The obtained contradiction conclude that the point $P$ must be the point $O_t$.

If $L_{xz}$ is not contained in the support of $D$, then the
inequality $$\mult_{O_t}(D)\leqslant 33D\cdot
L_{xz}=\frac{2}{7}<\frac{32}{65}$$ is a contradiction. Therefore,
the  curve  $L_{xz}$ must be contained in the support of $D$, and
hence the curve $R_x$ is not contained in the support of $D$. Put
$D=aL_{xz}+bR_y+\Delta$, where $\Delta$ is an effective
$\mathbb{Q}$-divisor whose support contains neither $L_{xz}$ nor
$R_y$. Then
\[\frac{16}{23\cdot 33}=D\cdot R_x\geqslant aL_{xz}\cdot R_x+\frac{\mult_{O_t}(D)-a}{33}>\frac{3a}{33}+\frac{32}{33\cdot 65}\]
and hence $a<\frac{304}{3\cdot 23\cdot 65}$.
If $b\ne 0$, then $L_{yt}$ is not contained in the support of $D$. Therefore,
\[\frac{4}{13\cdot 23}=D\cdot L_{yt}\geqslant bR_y\cdot L_{yt}=\frac{2b}{13},\]
and hence $b\leqslant \frac{2}{23}$.

Let $\pi\colon\bar{X}\to X$ be the weighted blow up at the point
$O_t$ with weights $(13,19)$ and let $F$ be the exceptional curve
of the morphism $\pi$. Then $F$ contains two singular points
$Q_{13}$ and $Q_{19}$ of $\bar{X}$ such that $Q_{13}$ is a singular point of type
$\frac{1}{13}(1,1)$, and $Q_{19}$ is a singular point of type
$\frac{1}{19}(3,7)$. Then
\[
K_{\bar{X}}\qlineq\pi^*(K_X)-\frac{1}{33}F,\ \
\bar{L}_{xz}\qlineq\pi^*(L_{xz})-\frac{19}{33}F, \ \
\bar{R}_y\qlineq\pi^*(R_y)-\frac{13}{33}F,\ \
\bar{\Delta}\qlineq\pi^*(\Delta)-\frac{c}{33}F,
\]
where $\bar{L}_{xz}$, $\bar{R}_y$ and $\bar{\Delta}$
are the proper transforms of $L_{xz}$, $R_y$ and $\Delta$
by $\pi$, respectively, and $c$ is a non-negative rational number.
Note that $F\cap\bar{R}_y=\{Q_{19}\}$ and $F\cap\bar{L}_{xz}=\{Q_{13}\}$.

The log pull-back of the log pair $(X,\frac{65}{32}D)$ by $\pi$ is
the log pair
$$
\left(\bar{X},\ \frac{65a}{32}\bar{L}_{xz}+ \frac{65b}{32}\bar{R}_y+
\frac{65}{32}\bar{\Delta}+\theta_1 F\right),%
$$
where
$$\theta_1=\frac{32+65(19a+13b+c)}{32\cdot 33}.$$
This  is not log canonical at some point $Q\in F$.
We have
\[0\leqslant\bar{\Delta}\cdot\bar{L}_{xz}=\frac{4+43a}{14\cdot 33}-\frac{b}{33}-\frac{c}{13\cdot33}.\]%
This inequality shows $13b+c\leqslant \frac{13}{14}(4+43a)$. Since
$a\leqslant \frac{304}{3\cdot 23\cdot 65}$, we obtain
$$\theta_{1}=\frac{32+1235a}{32\cdot 33}+\frac{65(13b+c)}{32\cdot 33}\leqslant
\frac{32+1235a}{32\cdot 33}+\frac{13\cdot 65(4+43a)}{14\cdot 32\cdot 33}<1.$$

Suppose that the point $Q$ is neither $Q_{13}$ nor $Q_{19}$. Then, the
point $Q$ is not in $\bar{L}_{xz}\cup \bar{R}_y$. Therefore, the
pair $ \left(\bar{X},  \frac{65}{32}\bar{\Delta}+F\right) $ is not
log canonical at the point $Q$, and hence
$$1<\frac{65}{32}\bar{\Delta}\cdot F=\frac{65c}{13\cdot 19\cdot 32}.$$
But $c\leqslant 13b+c\leqslant \frac{13}{14}(4+43a)<\frac{13\cdot 19\cdot
32}{65}$ since $a\leqslant \frac{304}{3\cdot 23\cdot 65}$. Therefore, the point
$Q$ is either $Q_{13}$ or $Q_{19}$.

Suppose that the point $Q$ is $Q_{13}$. Then the point $Q$ is in
$\bar{L}_{xz}$ but not in $\bar{R}_y$. Therefore, the pair $
\left(\bar{X},  \bar{L}_{xz}+\frac{65}{32}\bar{\Delta}+\theta_1
F\right) $ is not log canonical at the point $Q$. However, this is impossible since
\[\begin{split} 13\left(\frac{65}{32}\bar{\Delta}+\theta_1
F\right)\cdot \bar{L}_{xz}
&=\frac{13\cdot 65}{32}\left(\frac{4+43a}{14\cdot 33}-\frac{b}{33}-\frac{c}{13\cdot33}\right)+\theta_1= \\
& =\frac{32+1235a}{32\cdot 33}+\frac{13\cdot 65(4+43a)}{14\cdot
32\cdot 33}<1.\\ \end{split}\]  Therefore, the point $Q$ must be
the point $Q_{19}$.

Let $\psi\colon\tilde{X}\to \bar{X}$ be the weighted  blow up at
the point $Q_{19}$ with weights $(3,7)$ and let $E$ be the
exceptional curve of the morphism $\psi$. The exceptional curve
$E$ contains two singular points $O_3$ and $O_7$ of $\tilde{X}$. The point $O_3$
is of type $\frac{1}{3}(1,2)$ and the point $O_7$ is of type $\frac{1}{7}(4,5)$. Then
\[
K_{\tilde{X}}\qlineq\psi^*(K_{\bar{X}})-\frac{9}{19}E,\ \
\tilde{R}_{y}\qlineq\psi^*(\bar{R}_{y})-\frac{3}{19}E, \ \
\tilde{F}\qlineq\psi^*(F)-\frac{7}{19}E,
\ \
\tilde{\Delta}\qlineq\psi^*(\bar{\Delta})-\frac{d}{19}E,
\]
where $\tilde{R}_{y}$, $\tilde{F}$ and $\tilde{\Delta}$ are the
proper transforms of $\bar{R}_{y}$, $F$ and $\bar{\Delta}$ by
$\psi$, respectively, and $d$ is a non-negative rational number.

The log pull-back of the log pair $(X,\frac{65}{32}D)$ by
$\pi\circ\psi$ is the log pair
$$
\left(\tilde{X},\ \frac{65a}{32}\tilde{L}_{xz}+ \frac{65b}{32}\tilde{R}_y+
\frac{65}{32}\tilde{\Delta}+\theta_1 \tilde{F}+\theta_2E\right),%
$$
where $\tilde{L}_{xz}$ is the proper transform of $\bar{L}_{xz}$ by $\psi$ and
$$\theta_2=\frac{65(3b+d)}{19\cdot 32}+\frac{7}{19}\theta_1+\frac{9}{19}=\frac{9728+65(133a+190b+7c+33d)}{19\cdot 32\cdot 33}.$$
This  is not log canonical at some point $O\in E$.

We have
\[0\leqslant\tilde{\Delta}\cdot\tilde{R}_{y}=\bar{\Delta}\cdot\bar{R}_{y}-\frac{d}{7\cdot 19}= \frac{8+38b}{13\cdot
33}- \frac{19a+c}{19\cdot 33}-\frac{d}{7\cdot 19},\] and hence
$133a+7c+33d\leqslant \frac{133}{13}(8+38b)$. Therefore, this
inequality together with $b<\frac{2}{23}$ gives us
\[\begin{split}\theta_2
&=\frac{9728+65\cdot 190b}{19\cdot 32\cdot 33}+\frac{65(133a+7c+33d)}{19\cdot 32\cdot 33}\leqslant\\
&\leqslant\frac{9728+65\cdot 190b}{19\cdot 32\cdot 33}+\frac{65\cdot 7(8+38b)}{13\cdot 32\cdot 33}<1.\\
\end{split}
\]

Suppose that the point $O$ is in the outside of $\tilde{R}_{y}$
and $\tilde{F}$. Then the log pair $(E,
\frac{65}{32}\tilde{\Delta}|_{E})$ is not log canonical at the
point $O$ and hence
\[1<\frac{65}{32}\tilde{\Delta}\cdot E=\frac{65d}{3\cdot 7\cdot 32}.\]
However,
\[d\leqslant \frac{1}{33}(133a+7c+33d)\leqslant \frac{133}{13\cdot 33}(8+38b)<\frac{3\cdot 7\cdot 32}{65}\]
since $b\leqslant\frac{2}{23}$. This is a contradiction.

Suppose that the point $O$ belongs to $\tilde{R}_{y}$. Then the log
pair $
\left(\tilde{X}, \frac{65b}{32}\tilde{R}_y+
\frac{65}{32}\tilde{\Delta}+\theta_2E\right)
$
 is not log canonical at the
point $O$ and hence
\[1<7\left(
\frac{65}{32}\tilde{\Delta}+\theta_2 E\right)\cdot \tilde{R}_x=\frac{7\cdot 65}{32}\left(\frac{8+38b}{13\cdot
33}- \frac{19a+c}{19\cdot 33}-\frac{d}{7\cdot 19}\right)+\theta_2.\]
However,
\[ \frac{7\cdot 65}{32}\left(\frac{8+38b}{13\cdot
33}- \frac{19a+c}{19\cdot 33}-\frac{d}{7\cdot 19}\right)+\theta_2=\frac{9728+65\cdot 190b}{19\cdot 32\cdot 33}+\frac{65\cdot 7(8+38b)}{13\cdot 32\cdot 33}<1.\]
This is a contradiction. Therefore, the point $O$ is the point $O_3$.

Suppose that the point $O$ belongs to $\tilde{F}$. Then the log
pair $
\left(\tilde{X},
\frac{65}{32}\tilde{\Delta}+\theta_1\tilde{F}+\theta_2E\right)
$
 is not log canonical at the
point $O$ and hence
\[1<3\left(
\frac{65}{32}\tilde{\Delta}+\theta_2 E\right)\cdot \tilde{F}=\frac{3\cdot 65}{32}\left(\frac{c}{13\cdot 19}-\frac{d}{3\cdot 19}\right)+\theta_2.\]
However,
\[ \begin{split} \frac{3\cdot 65}{32}\left(\frac{c}{13\cdot 19}-\frac{d}{3\cdot 19}\right)+\theta_2=&\frac{3\cdot 65 c}{13\cdot 19\cdot 32}+\frac{9728+65(133a+190b+7c)}{19\cdot 32\cdot 33}=\\
&=\frac{512+455a}{ 32\cdot 33}+\frac{65\cdot 190(13b+c)}{13\cdot19\cdot 32\cdot 33}\leqslant\\
&\leqslant \frac{512+455a}{ 32\cdot 33}+\frac{65\cdot 190(4+43a)}{14\cdot19\cdot 32\cdot 33}<1
\end{split}\]
since $13b+c\leqslant \frac{13}{14}(4+43a)$ and $a\leqslant \frac{304}{3\cdot 23\cdot 65}$. This is a contradiction.
\end{proof}

\begin{lemma}
\label{lemma:I-4-W-13-23-51-83-D-166} Let $X$ be a quasismooth
hypersurface of degree $166$ in $\mathbb{P}(13,23,51,83)$. Then
$\lct(X)=\frac{91}{40}$.
\end{lemma}

\begin{proof}
 The surface $X$ can be defined by the
quasihomogeneous equation
$$
t^{2}+y^{5}z+xz^{3}+x^{11}y=0.
$$
The surface $X$ is singular only at the points $O_x$, $O_{y}$ and
$O_z$. The curves $C_x$ and $C_y$ are irreducible. We have
$$
\frac{91}{40}=\lct\left(X, \frac{4}{13}C_x\right)<\lct\left(X, \frac{4}{23}C_y\right)=\frac{115}{24},%
$$
and hence $\lct(X)\leqslant \frac{91}{40}$.

Suppose that $\lct(X)<\frac{91}{40}$. Then there is an effective
$\Q$-divisor $D\qlineq -K_X$ such that the pair
$(X,\frac{91}{40}D)$ is not log canonical at some point $P$. By
Lemma~\ref{lemma:convexity}, we may assume that the support of the
divisor $D$ contains neither $C_x$ nor $C_y$. Then the
inequalities
\[51D\cdot C_x=\frac{8}{23}<\frac{40}{91}, \ \ 13D\cdot
C_y=\frac{8}{51}<\frac{40}{91}\] show that the point $P$ is a
smooth point of $X$ in the outside of $C_x$. However, $H^0(\P,
\mathcal{O}_{\P}(663))$ contains $x^{51}$, $y^{13}x^{28}$,
$y^{26}x^{5}$ and $z^{13}$,  and hence it follows from
Lemma~\ref{lemma:Carolina} that the point $P$ is
either a singular point of $X$ or a point on $C_x$. This is a
contradiction.
\end{proof}

\section{Sporadic cases with $I=5$}
\label{subsection:index-5}

\begin{lemma}
\label{lemma:I-5-W-11-13-19-25-D-63}
Let $X$ be a quasismooth
hypersurface of degree $63$ in $\mathbb{P}(11,13,19,25)$. Then
$\lct(X)=\frac{13}{8}$.
\end{lemma}
\begin{proof}
 The surface $X$ can be defined by the
quasihomogeneous equation
$$
z^{2}t+yt^{2}+xy^{4}+x^{4}z=0,
$$
and $X$ is singular at $O_x$, $O_y$, $O_z$ and $O_{t}$.

The curve $C_x$ (resp. $C_y$, $C_z$, $C_t$) consists of two
irreducible and reduced curves $L_{xt}$ (resp $L_{yz}$, $L_{yz}$,
$L_{xt}$) and $R_x=\{x=z^2+yt=0\}$ (resp. $R_y=\{y=x^4+zt=0\}$,
$R_z=\{z=t^2+xy^3=0\}$, $R_t=\{t=y^4+x^3z=0\}$). The curve
$L_{xt}$ intersects $R_x$ (resp. $R_t$) only at the point $O_y$
(resp. $O_z$). The curve $L_{yz}$ intersects $R_y$ (resp. $R_z$)
 only at the point $O_t$ (resp. $O_x$).

We have the following intersection numbers
$$D\cdot L_{xt}=\frac{5}{13\cdot 19},\ \ D\cdot L_{yz}=\frac{1}{5\cdot 11}, \ \
 D\cdot R_{x}=\frac{2}{5\cdot 13},
\ \ D\cdot R_{y}=\frac{4}{5\cdot 19}, \ \ D\cdot
R_{z}=\frac{10}{11\cdot 13},
 $$
$$D\cdot R_{t}=\frac{20}{11\cdot 19}, \ \
 L_{xt}\cdot R_{x}=\frac{2}{13}, \ \ L_{xt}\cdot R_{t}=\frac{4}{19},\ \ L_{yz}\cdot R_{y}=\frac{4}{25},\
\ L_{yz}\cdot R_{z}=\frac{2}{11},
$$
$$
L_{xt}^2=-\frac{27}{13\cdot 19},\ \ L_{yz}^2=-\frac{31}{11\cdot 25}, \ \ R_x^2=-\frac{28}{13\cdot 25},
 \ \ R_y^2=-\frac{24}{19\cdot 25},
\ \ R_z^2=\frac{12}{11\cdot 13},\ \ R_t^2=\frac{56}{11\cdot 19}.
$$

We have
$$
\lct\left(X, \frac{5}{13}C_y\right)=\frac{13}{8}<\lct\left(X, \frac{5}{11}C_x\right)=\frac{33}{20}<\lct\left(X, \frac{5}{25}C_t\right)=\frac{35}{16}<\lct\left(X, \frac{5}{19}C_z\right)=\frac{19}{8}.%
$$
In particular, we have $\lct(X)\leqslant \frac{13}{8}$.

We suppose that $\lct(X)<\frac{13}{8}$. Then there is an effective
$\Q$-divisor $D\qlineq -K_X$ such that the log pair $(X,
\frac{13}{8}D)$ is not log canonical at some point $P\in X$.

Suppose that the point $P$ is located in the outside of $C_x\cup
C_y\cup C_z\cup C_t$. We consider the pencil $\mathcal{L}$ on $X$
defined by the equations $\lambda x^4+ \mu zt=0$, where $[\lambda:
\mu]\in\mathbb{P}^1$. The curve $L_{xt}$ is the unique base
component of the pencil $\mathcal{L}$. There is a unique member
$Z$ in the pencil $\mathcal{L}$ passing through the point $P$.
Since the point $P$ is in the outside of $C_{x}\cup C_{y}\cup
C_{z}\cup C_{t}$, the curve $Z$ is defined by an equation of the
form
$$\alpha x^4+zt=0,$$
where $\alpha$ is a non-zero constant.

The open subset $Z\setminus C_{z}$ of the curve $Z$ is a
$\mathbb{Z}_{19}$-quotient of the affine curve
$$
\alpha
x^4+t=t+yt^2+xy^{4}+x^4=0\subset\mathbb{C}^{3}\cong\mathrm{Spec}\Big(\mathbb{C}\big[x,y,z\big]\Big),
$$
that is isomorphic to the affine  curve
 given by the equation
$$
x\left(\left(1-\alpha\right)x^3+\alpha^2x^7y+y^4\right)=0
\subset\mathbb{C}^{2}\cong\mathrm{Spec}\Big(\mathbb{C}\big[y,z\big]\Big).
$$

If $\alpha\ne 1$, the divisor $Z$ consists of two irreducible and reduced curves $L_{xt}$ and $Z_{\alpha}$. On the other hand, if $\alpha=1$, then the divisor $Z$ consists of three irreducible and reduced curves $L_{xt}$,
$R_y$ and $Z_1$. Since $P\not\in C_{x}\cup C_{y}\cup C_z\cup C_{t}$, the point $P$ must be contained in  $Z_\alpha$ (including $\alpha=1$). Also, the curve $Z_\alpha$  is smooth at the point $P$.

Write $D=nZ_\alpha+\Gamma$, where $\Gamma$ is an effective $\mathbb{Q}$-divisor whose support contains $Z_\alpha$. Since $Z_\alpha$ passing through the point $O_t$ and the pair $(X, \frac{13}{8}D)$ is log canonical at the point $O_z$, we have $n\leqslant \frac{8}{13}$.
We can easily check
\[D\cdot Z_{\alpha}=\left\{%
\aligned
&D\cdot( Z-L_{xt})=\frac{227}{5\cdot 13\cdot 19} \text{ if}\ \alpha\ne 1,\\%
&D\cdot (Z-L_{xt}-R_y)=\frac{35}{13\cdot 19} \text{ if}\ \alpha=1.\\%
\endaligned\right.
\]
Also, if  $\alpha\ne 1$, then

\[Z_{\alpha}^2=Z\cdot Z_\alpha-L_{xt}\cdot Z_\alpha
\geqslant Z\cdot Z_\alpha-(L_{xt}+R_x)\cdot Z_\alpha  = \frac{33}{5}D\cdot Z_\alpha .\]
If $\alpha =1$,
\[Z_{\alpha}^2=Z\cdot Z_\alpha-(L_{xt}+R_y)\cdot Z_\alpha
\geqslant Z\cdot Z_\alpha-(L_{xt}+R_x+L_{yz}+R_y)\cdot Z_\alpha  =
4D\cdot Z_\alpha.\] In both cases, we have $Z_\alpha^2>0$. Since
\[(D-nZ_\alpha)\cdot Z_{\alpha}\leqslant D\cdot Z_\alpha<\frac{8}{13}\]
Lemma~\ref{lemma:handy-adjunction} shows that the pair $(X,
\frac{13}{8}D)$ is log canonical at the point $P$. This is a
contradiction. Therefore, the point $P$ must belong to the set
$C_x\cup C_y\cup C_z\cup C_t$.

It follows from Lemma~\ref{lemma:convexity} that we may assume
that $\mathrm{Supp}(D)$ does not contain at least one irreducible
component of the curves $C_{x}$, $C_{y}$, $C_{z}$, $C_{t}$. Since
the curve $R_t$ is singular at the point $O_z$ with multiplicity
$3$ and the support of $D$ does not contain either $L_{xt}$ or
$R_t$, one of the inequalities
\[\mult_{O_z}(D)\leqslant 19D\cdot L_{xt}=\frac{5}{13}<\frac{8}{13}, \ \
\mult_{O_z}(D)\leqslant \frac{19}{3}D\cdot R_t=\frac{20}{3\cdot 11}<\frac{8}{13}\]
must hold, and hence the point $P$ cannot be the point $O_z$.
Similarly, we see that the point $P$ can be neither $O_x$ nor
$O_y$.

Now we write
$D=m_0L_{xt}+m_1L_{yz}+m_2R_x+m_3R_y+m_4R_z+m_5R_t+\Omega$, where
$\Omega$ is  an effective $\mathbb{Q}$-divisor whose support
contains none of $L_{xt}$, $L_{yz}$, $R_x$, $R_y$, $R_z$, $R_t$.
Since the pair $(X, \frac{13}{8}D)$ is log canonical at the points
$O_x$, $O_y$, $O_z$, we must have $m_i\leqslant \frac{8}{13}$.
Then the inequalities
\[\left.%
\aligned
&(D-m_0L_{xt})\cdot L_{xt}=\frac{5+27m_0}{13\cdot 19}\\%
&(D-m_1L_{yz})\cdot L_{yz}=\frac{5+31m_1}{11\cdot 25}\\%
&(D-m_2R_{x})\cdot R_{x}=\frac{10+28m_2}{25\cdot 13}\\%
&(D-m_3R_{y})\cdot R_{y}=\frac{20+24m_3}{25\cdot 19}\\
&(D-m_4R_{z})\cdot R_{z}=\frac{10-12m_4}{11\cdot 13}\\
&(D-m_5R_{t})\cdot R_{t}=\frac{20-56m_5}{11\cdot 19}\\
\endaligned\right\}\leqslant\frac{8}{13}\]
imply that the point $P$ must be the point $O_t$.

Put $D=aL_{yz}+bR_x+\Delta$, where $\Delta$ is an effective
$\mathbb{Q}$-divisor whose support contains neither the curve
$L_{yz}$ nor $R_x$. If $a=0$, then we obtain
\[\mult_{O_t}(D)\leqslant 25D\cdot L_{yz}=\frac{5}{11}<\frac{8}{13}.\]
This is a contradiction. Therefore, $a>0$, and hence the support
of $D$ dose not contain the curve $R_y$. Since
$$
\frac{4}{5\cdot 19}=D\cdot R_{y}\geqslant
aL_{yz}\cdot R_{y}+\frac{\mult_{O_t}(D)-a}{25}
>\frac{3a}{25}+\frac{8}{13\cdot 25},%
$$
and hence $a\leqslant \frac{36}{247}$.  If $b>0$, then
\[\frac{5}{13\cdot 19}=D\cdot L_{xt}\geqslant bR_x\cdot
L_{xt}=\frac{2b}{13},\] and hence $b\leqslant\frac{5}{38}$.

Let $\pi\colon\bar{X}\to X$ be the weighted blow up of $O_{t}$
with weights $(7,3)$ and let $F$ be the exceptional curve of
$\pi$. Then
$$
K_{\bar{X}}\qlineq \pi^{*}(K_{X})-\frac{15}{25}F,\ %
\bar{L}_{yz}\qlineq \pi^{*}(L_{yz})-\frac{3}{25}F,\ %
\bar{R}_{x}\qlineq \pi^{*}(R_{x})-\frac{7}{25}F,\ %
\bar{\Delta}\qlineq \pi^{*}(\Delta)-\frac{c}{25}F,
$$
where $\bar{\Delta}$, $\bar{L}_{yz}$, $\bar{R}_{x}$ are the proper
transforms of $\Delta$, $L_{yz}$, $R_{x}$, respectively, and $c$
is a non-negative rational number. The curve $F$ contains two
singular points $Q_7$ and $Q_3$ of $\bar{X}$. The point $Q_7$ is a
singular point of type $\frac{1}{7}(1,1)$ and the point $Q_3$ is
of type $\frac{1}{3}(2,1)$.  Note that the curve $\bar{R}_x$
passes through the point $Q_3$ but not the point $Q_7$. The curve
$\bar{L}_{yz}$  passes through the point $Q_7$ but not the point
$Q_3$.

The log pull-back of the log pair $(X,\frac{13}{8}D)$ by $\pi$ is
the log pair
$$
\left(\bar{X},\
\frac{13a}{8}\bar{L}_{yz}+ \frac{13b}{8}\bar{R}_x+ \frac{13}{8}\bar{\Delta}+
\theta_1 F\right),%
$$
where
$$
\theta_1=\frac{13(3a+7b+c)+120}{8\cdot 25}.
$$
This pair is not log canonical at some point $Q\in F$. We have
\[0\leqslant\bar{\Delta}\cdot\bar{R}_x=\frac{10+28b}{13\cdot 25}-\frac{a}{25}
-\frac{c}{3\cdot 25}.\] This inequality shows
$3a+c\leqslant\frac{3}{13}(10+28b)$. Then
\[\theta_1=\frac{13(3a+c)+91b+120}{8\cdot 25}
\leqslant\frac{6+7b}{8}<1\] since
$b\leqslant\frac{5}{38}$.

Suppose that  the point $Q$ is neither the point $Q_7$ nor the
point $Q_3$. Then the log pair
$\left(\bar{X},\frac{13}{8}\bar{\Delta}+F\right)$ is not log
canonical at the point $Q$. Then
$$
\frac{13c}{8\cdot 21}=\frac{13}{8}\bar{\Delta}\cdot F>1
$$
by Lemma~\ref{lemma:adjunction}. However, $c\leqslant
3a+c\leqslant\frac{3}{13}(10+28b)$. This is a contradiction since
$b\leqslant \frac{5}{38}$. Therefore, the point $Q$ is either the
point $Q_7$ or the point $Q_3$.

Suppose that the point $Q$ is the point $Q_3$. This point is the
intersection point of $F$ and $\bar{R}_x$. Then the log pair
$\left(\bar{X},\frac{13b}{8}\bar{R}_{x}+
\frac{13}{8}\bar{\Delta}+\theta_1 F\right)$ is not log canonical
at the point $Q$. It then follows from
Lemma~\ref{lemma:adjunction} that
\[1<3\left(\frac{13}{8}\bar{\Delta}+\theta_{1}F\right)\cdot\bar{R}_{x}
=\frac{13\cdot 3}{8}\left(\frac{10+28b}{13\cdot 25}-\frac{a}{25}-\frac{c}{3\cdot 25}\right)+\theta_{1}.\]
However,
\[\frac{13\cdot 3}{8}\left(\frac{10+28b}{13\cdot 25}-\frac{a}{25}-\frac{c}{3\cdot 25}\right)+\theta_{1}
=\frac{6+7b}{8}<1.\]  Therefore,
the point $Q$ is the point $Q_7$. This point is the intersection
point of $F$ and $\bar{L}_{yz}$.

Let $\phi\colon\tilde{X}\to \bar{X}$ be the  blow up at the point
$Q_{7}$. Let $G$ be the exceptional divisor of the morphism
$\phi$. The surface $\tilde{X}$ is smooth along the exceptional
divisor $G$. Let $\tilde{L}_{yz}$, $\tilde{R}_x$, $\tilde{\Delta}$
and $\tilde{F}$ be the proper transforms of $L_{yz}$, $R_x$,
$\Delta$ and $F$ by $\pi\circ\phi$, respectively. We have
$$
K_{\tilde{X}}\qlineq\phi^*(K_{\bar{X}})-\frac{5}{7}G,
\ \tilde{L}_{yz}\qlineq\phi^*(\bar{L}_{yz})-\frac{1}{7}G,
\ \tilde{F}\qlineq\phi^*(F)-\frac{1}{7}G,
\ \tilde{\Delta}\qlineq\phi^*(\bar{\Delta})-\frac{d}{7}G,%
$$
where $d$ is a non-negative rational number. The log pull-back of
the log pair $(X, \frac{13}{8}D)$ via $\pi\circ \phi$ is
$$
\left(\tilde{X},
\frac{13a}{8}\tilde{L}_{yz}+\frac{13b}{8}\tilde{R}_x+\frac{13}{8}\tilde{\Delta}+\theta_1\tilde{F}+\theta_2 G\right),%
$$
where
$$\theta_2=\frac{13}{7\cdot 8}(a+d)+\frac{\theta_1}{7}+\frac{5}{7}
=\frac{1120+13(28a+7b+c+25d)}{7\cdot 8\cdot 25}.$$
This log pair is not log canonical at some point $O\in G$.
We have
\[0\leqslant \tilde{\Delta}\cdot\tilde{L}_{yz}
=\frac{5+31a}{11\cdot 25}-\frac{b}{25}-\frac{c}{7\cdot
25}-\frac{d}{7}.\] We then obtain $7b+c+25d\leqslant
\frac{7}{11}(5+31a)$. Since $a\leqslant\frac{36}{247}$, we see
\[\theta_2=\frac{1120+13(28a+7b+c+25d)}{7\cdot 8\cdot 25}\leqslant
\frac{511+273a}{7\cdot 8\cdot 11}<1. \]

Suppose that $O\not\in\tilde{F}\cup\tilde{L}_{yz}$. The log pair
$\left(\tilde{X},\frac{13}{8}\tilde{\Delta}+G\right)$  is not log
canonical at the point $O$. Applying Lemma~\ref{lemma:adjunction},
we get
$$
1<\frac{13}{8}\tilde{\Delta}\cdot G=\frac{13d}{8},%
$$
and hence $d>\frac{8}{13}$. However, $d\leqslant
\frac{1}{25}(7b+c+25d)\leqslant \frac{7}{11\cdot 25}(5+31a)$. This
is a contradiction since $a\leqslant\frac{36}{247}$. Therefore,
the point $O$ is either the intersection point of $G$ and
$\tilde{F}$ or the intersection point of $G$ and $\tilde{L}_{yz}$.
In the latter case, the pair $
\left(\tilde{X},\frac{13a}{8}\tilde{L}_{yz}+\frac{13}{8}\tilde{\Delta}+\theta_2 G\right)%
$ is not log canonical at the point $O$. Then, applying
Lemma~\ref{lemma:adjunction}, we get
\[
1<\left(\frac{13}{8}\tilde{\Delta}+\theta_{2}G\right)\cdot\tilde{L}_{yz}=\frac{13}{8}\left(\frac{5+31a}{11\cdot
25}-\frac{b}{25}-\frac{c}{7\cdot
25}-\frac{d}{7}\right)+\theta_{2}.\] However,
\[\frac{13}{8}\left(\frac{5+31a}{11\cdot 25}-\frac{b}{25}-\frac{c}{7\cdot 25}-\frac{d}{7}\right)+\theta_{2}
=\frac{511+273a}{7\cdot 8\cdot 11}<1.\] Therefore, the point $O$
must be the intersection point of $G$ and $\tilde{F}$.

Let $\xi\colon\hat{X}\to \tilde{X}$ be the blow up at the
point $O$  and let $H$ be the exceptional
divisor of $\xi$. We also let $\hat{L}_{yz}$, $\hat{R}_x$, $\hat{\Delta}$,
$\hat{G}$, and $\hat{F}$ be the proper transforms of $\tilde{L}_{yz}$,
$\tilde{R}_x$, $\tilde{\Delta}$,  $G$ and $\tilde{F}$ by $\xi$, respectively. Then
$\hat{X}$ is smooth along the exceptional divisor $H$. We have
$$
K_{\hat{X}}\qlineq\xi^*(K_{\tilde{X}})-H,\
\hat{G}\qlineq\xi^*(G)-H,\
\hat{F}\qlineq\xi^*(\tilde{F})-H,\
\hat{\Delta}\qlineq\xi^*(\tilde{\Delta})-eH,%
$$
where $e$ is a non-negative rational number. The log pull-back of the
log pair $(X, \frac{13}{8}D)$ via $\pi\circ \phi\circ \xi$ is
$$
\left(\hat{X},\frac{13a}{8}\hat{L}_{yz}+\frac{13b}{8}\hat{R}_x+\frac{13}{8}\hat{\Delta}+\theta_1\hat{F}+\theta_2 \hat{G}+\theta_{3}H\right),%
$$
where
$$
\theta_3=\theta_1+\theta_2+\frac{13e}{8}-1=\frac{560+13(49a+56b+8c+25d+175e)}{7\cdot 8\cdot 25}.$$
This log pair is not log canonical at some point $A\in H$. We have
\[
\frac{c}{21}-\frac{d}{7}-e\hat{\Delta}\cdot\hat{F}\geqslant 0.\]
Therefore, $3d+21e\leqslant c$.

Then
\[\begin{split}\theta_3&=\frac{560+13(49a+56b+8c)}{7\cdot 8\cdot 25}+\frac{13(d+7e)}{7\cdot 8}\leqslant\\
&\leqslant \frac{1680+13(147a+168b+49c)}{3\cdot7\cdot 8\cdot 25}=\\
&=\frac{1680+1284b}{3\cdot7\cdot 8\cdot 25}+\frac{13\cdot 49(3a+c)}{3\cdot7\cdot 8\cdot 25}\leqslant\\
&\leqslant \frac{140+107b}{2\cdot 7\cdot 25}+\frac{7(5+14b)}{4 \cdot 25}=\frac{21+36b}{28}<1\\
\end{split}
\]
since $b\leqslant \frac{5}{38}$ and $3a+c\leqslant
\frac{3}{13}(10+28b)$. In particular, $\theta_3$ is a positive
number.

Suppose that $A\not\in\hat{F}\cup\hat{G}$. Then the log pair $
\left(\hat{X},\frac{13}{8}\hat{\Delta}+\theta_{3}H\right)
$  is not log canonical at the point $A$. Applying
Lemma~\ref{lemma:adjunction}, we get
$$
1<\frac{13}{8}\hat{\Delta}\cdot H=\frac{13e}{8}.
$$
However, $$e\leqslant \frac{1}{21}(3d+21e)\leqslant \frac{c}{21}\leqslant \frac{1}{21}(3a+c)\leqslant
\frac{3(10+28b)}{13\cdot 21}\leqslant \frac{8}{13}.$$
Therefore, the point $A$ must be either in $\hat{F}$ or in $\hat{G}$.

Suppose that $A\in\hat{F}$. Then the log pair  $
\left(\hat{X},\frac{13}{8}\hat{\Delta}+\theta_1\hat{F}+\theta_{3}H\right)%
$ is not log canonical at the point $A$. Applying Lemma~\ref{lemma:adjunction},
we get
\[
1<\left(\frac{13}{8}\hat{\Delta}+\theta_{3}
H\right)\cdot\hat{F}=\frac{13}{8}\left(\frac{c}{21}-\frac{d}{7}-e\right)+
\theta_{3}=\frac{1680+13(147a+168b+49c)}{3\cdot7\cdot 8\cdot 25}.
\]
However,
\[\frac{1680+13(147a+168b+49c)}{3\cdot7\cdot 8\cdot 25}\leqslant \frac{21+36b}{28}<1.
\]
Therefore, the point $A$ is the intersection point of $H$ and
$\hat{G}$. Then the log pair $
\left(\hat{X},\frac{13}{8}\hat{\Delta}+\theta_2
\hat{G}+\theta_{3}H\right) $ is not log canonical at the point
$A$. From Lemma~\ref{lemma:adjunction}, we obtain
\[
1<\left(\frac{13}{8}\hat{\Delta}+\theta_{3}H\right)\cdot\hat{G}=\frac{13}{8}\left(d-e\right)+\theta_{3}=\frac{560+13(49a+56b+8c+200d)}{7\cdot
8\cdot 25}.
\]
However,
\[\frac{560+13(49a+56b+8c+200d)}{7\cdot 8\cdot 25}=\frac{80+91a}{8\cdot 25}
+\frac{13(7b+c+25d)}{7\cdot 25}\leqslant\frac{56+169a}{8\cdot 11}<1
\]
since $a<\frac{36}{247}$ and $7b+c+25d\leqslant \frac{7}{11}(5+31a)$.
The obtained contradiction completes the
proof.
\end{proof}

\begin{lemma}
\label{lemma:I-5-W-11-25-37-68-D-136}
 Let $X$ be a quasismooth
hypersurface of degree $136$ in $\mathbb{P}(11,25,37,68)$. Then
$\lct(X)=\frac{11}{6}$.
\end{lemma}

\begin{proof}  The surface $X$ can be defined by the
quasihomogeneous equation
$$
xy^5 + x^9z + yz^3 + t^2=0.
$$
The surface $X$ is singular at the points $O_x$, $O_y$ and $O_z$.

The curves $C_x$ and $C_y$ are reduced and irreducible. We have
$$
\frac{11}{6}=\lct\left(X, \frac{5}{11}C_x\right)<\lct\left(X,
\frac{5}{25}C_y\right)=\frac{55}{18}.
$$
Thus, $\lct(X) \leqslant \frac{11}{6}$.

Suppose that $\lct(X) < \frac{11}{6}$. Then there is an effective
$\Q$-divisor $D\qlineq -K_X$ such that the pair $(X,
\frac{11}{6}D)$ is not log canonical at some point $P$. By
Lemma~\ref{lemma:convexity} we may assume that the support of $D$
contains neither $C_x$ nor $C_y$. Then two inequalities
\[37D\cdot C_x=\frac{2}{5}<\frac{6}{11}, \ \ 11D\cdot C_y=\frac{10}{37}<\frac{6}{11}\]
imply that the point $P$ is  neither a singular point of $X$ nor a point on $C_x$.
Since $H^0(\P, \mathcal{O}_\P(407))$ contains $x^{37}$, $z^{11}$
and $x^{12}y^{11}$, we see that this cannot happen by
Lemma~\ref{lemma:Carolina}.
\end{proof}

\begin{lemma}
\label{lemma:I-5-W-13-19-41-68-D-136}
 Let $X$ be a quasismooth
hypersurface of degree $136$ in $\mathbb{P}(13,19,41,68)$. Then
$\lct(X)=\frac{91}{50}$.
\end{lemma}

\begin{proof}
 The surface $X$ can be defined by the quasihomogeneous
equation
$$
x^9y + xz^3 + y^5z + t^2=0.
$$
The surface $X$ is singular only at the points $O_x$, $O_y$ and $O_z$.

The curves $C_x$ and $C_y$ are reduced and irreducible. Also, it is easy to check
$$
\frac{91}{50}=\lct\left(X,
\frac{5}{13}C_x\right)<\lct(X, \frac{5}{19}C_y)=\frac{19}{6}.%
$$
Therefore, $\lct(X)\leqslant \frac{50}{91}$.

Suppose that $\lct(X)<\frac{91}{50}$. Then there is an effective
$\Q$-divisor $D\qlineq -K_X$ such that the pair $(X,
\frac{91}{50}D)$ is not log canonical at some point $P$. By
Lemma~\ref{lemma:convexity} we may assume that the support of
$D$ contains neither $C_x$ nor $C_y$. Then two inequalities
\[41D\cdot C_x=\frac{10}{19}<\frac{50}{91}, \ \ 13D\cdot C_y=\frac{10}{41}<\frac{50}{91}\]
imply that the point $P$ is  neither a singular point of $X$ nor a point on $C_x$.
However,  by
Lemma~\ref{lemma:Carolina} this is impossible since $H^0(\P, \mathcal{O}_\P(533))$ contains $x^{41}$, $z^{13}$
and $x^{3}y^{26}$.
\end{proof}

\section{Sporadic cases with $I=6$}
\label{subsection:index-6}

\begin{lemma}
\label{lemma:I-6-W-7-10-15-19-D-45} Let $X$ be a quasismooth
hypersurface of degree $45$ in $\mathbb{P}(7,10, 15,19)$. Then
$\lct(X)=\frac{35}{54}$.
\end{lemma}

\begin{proof}
 The surface $X$ can be defined by the
equation $z^{3}-y^{3}z+xt^{2}+x^{5}y=0$. It is singular at the
points $O_x$, $O_y$, $O_t$ and $Q=[0:1:1:0]$.

The curve $C_x$ consists of two irreducible and reduced curves $L_{xz}$ and $R_x=\{x=z^2-y^3=0\}$. These two curves $L_{xz}$ and $R_x$ meets each other at the point $O_t$. Also,
$$
L_{xz}^2=-\frac{23}{10\cdot 19},\ \ \ R_x^2=-\frac{8}{5\cdot 19},\ \ \ L_{xz}\cdot R_{x}=\frac{3}{19}.%
$$
The curve $R_x$ is singular at the point $O_t$.
The curve $C_y$ is irreducible and
$$
\frac{35}{54}=\lct\left(X, \frac{6}{7}C_x\right)<\lct\left(X, \frac{6}{10}C_y\right)=\frac{25}{18}.%
$$
Therefore, $\lct(X)\leqslant \frac{35}{54}$.

Suppose that $\lct(X)<\frac{35}{54}$. Then there is an effective
$\Q$-divisor $D\qlineq -K_X$ such that the pair $(X,\frac{35}{54}D)$ is
not log canonical at some point $P$. By
Lemma~\ref{lemma:convexity}, we may assume that the support of
the divisor $D$ does not contain the curve $C_y$. Similarly, we
may assume that either $L_{xz}\not\subseteq\mathrm{Supp}(D)$ or
$R_{x}\not\subseteq\mathrm{Supp}(D)$.

Since $H^0(\P, \mathcal{O}_{\P}(105))$ contains the monomials
$x^{15}$, $y^{7}x^{5}$ and $z^{7}$, it follows from
Lemma~\ref{lemma:Carolina} that the point $P$ is either a point on $C_x$ or the singular point $O_x$.

Since either $L_{xz}\not\subseteq\mathrm{Supp}(D)$ or
$R_{x}\not\subseteq\mathrm{Supp}(D)$, one of the inequalities
\[\mult_{O_t}(D)\leqslant 19D\cdot L_{xz}=\frac{3}{5}<\frac{54}{35}, \ \
\mult_{O_t}(D)\leqslant \frac{19}{2}D\cdot R_{x}=\frac{3}{5}<\frac{54}{35}\]
must hold, and hence the point $P$ cannot be the point $O_t$.
On the other hand, the inequality $7D\cdot C_y=\frac{18}{19}<\frac{54}{35}$ shows that the point $P$ cannot be the point $O_x$.

Put $D=mL_{xz}+\Omega$, where
$\Omega$ is an effective $\mathbb{Q}$-divisor such that
$L_{xz}\not\subset\mathrm{Supp}(\Omega)$. If $m\ne 0$, then
$$
\frac{6}{5\cdot 19}=D\cdot R_{x}\geqslant  mL_{xz}\cdot R_{x}=\frac{3m}{19},%
$$
and hence $m\leqslant \frac{2}{5}$.
Then,
\[10(D-mL_{xz})\cdot L_{xz}=\frac{6+23m}{19}\leqslant \frac{54}{35}.\]
Thus it follows from
Lemma~\ref{lemma:handy-adjunction} that
the point $P$ cannot belong to $L_{xz}$.

Now we write $D=\epsilon R_{x}+\Delta$, where
$\Delta$ is an effective $\mathbb{Q}$-divisor such that
$R_{x}\not\subset\mathrm{Supp}(\Delta)$. If $\epsilon\ne 0$, then
$$
\frac{3}{5\cdot 19}=D\cdot L_{xz}\geqslant  \epsilon R_{x}\cdot L_{xz}=\frac{3\epsilon}{19},%
$$
and hence $\epsilon\leqslant \frac{1}{5}$. Then
\[5(D-\epsilon R_{x})\cdot R_{x}=\frac{3+8\epsilon}{19}\leqslant \frac{54}{35}.\]
By Lemma~\ref{lemma:handy-adjunction} the point $P$ cannot be
contained in $R_x$ either. Therefore, the point $P$ is located
nowhere.
\end{proof}

\begin{lemma}
\label{lemma:I-6-W-11-19-29-53-D-106} Let $X$ be a quasismooth
hypersurface of degree $106$ in $\mathbb{P}(11,19, 29,53)$. Then
$\lct(X)=\frac{55}{36}$.
\end{lemma}

\begin{proof}
We may assume that the surface $X$ is defined by the quasihomogeneous equation
$$x^7z+xy^5+yz^3+t^2=0.$$
The surface $X$ is singular at $O_x$, $O_y$ and $O_z$.
The curves $C_x$ and $C_y$  are  irreducible. It is easy to see
$$\lct(X, \frac{6}{11}C_x)=\frac{55}{36}<\lct(X, \frac{6}{19}C_y)=\frac{57}{28}.$$

Suppose that $\lct(X)<\frac{55}{36}$. Then there is an effective $\Q$-divisor $D\qlineq -K_X$ such that
the pair $(X, \frac{55}{36}D)$ is not log canonical.
For a smooth point $P\in X\setminus C_x$, we have
\[\mult_P(D)\leqslant \frac{6\cdot 319\cdot 106}{11\cdot 19\cdot 29\cdot 53}<\frac{36}{55}\]
by Lemma~\ref{lemma:Carolina} since $H^0(\P,
\mathcal{O}_\P(319))$ contains the monomials $x^{29}$, $z^{11}$
and $x^{10}y^{11}$. Therefore, either there is a point $P\in C_x$
such that $\mult_{P}(D)>\frac{36}{55}$ or we have
$\mult_{O_x}(D)>\frac{36}{55}$. Since the pairs $(X, \frac{6\cdot
55}{11\cdot 36}C_x)$ and $(X, \frac{6\cdot 55}{19\cdot 36}C_y)$
are log canonical and the curves $C_x$ and $C_y$ are irreducible,
we may assume that the support of $D$ contains neither the curve
$C_x$ nor the curve $C_y$. Then we can obtain
\[\mult_{O_x}(D)\leqslant 11C_y\cdot D \leqslant \frac{11\cdot 19\cdot 106\cdot 6}{11\cdot 19\cdot 29\cdot 53}<\frac{36}{55}\]
and for any point $P\in C_x$
\[\mult_{P}(D)\leqslant 29C_x\cdot D \leqslant \frac{29\cdot 11\cdot 106\cdot 6}{11\cdot 19\cdot 29\cdot 53}<\frac{36}{55}.\]
This is a contradiction. Therefore, $\lct(X)=\frac{55}{36}$.
\end{proof}

\begin{lemma}
\label{lemma:I-6-W-13-15-31-53-D-106}Let $X$ be a quasismooth
hypersurface of degree $106$ in $\mathbb{P}(13,15, 31,53)$. Then
$\lct(X)=\frac{91}{60}$.
\end{lemma}

\begin{proof}  The surface $X$ can be defined by the
quasihomogeneous equation
$$
x^7y + xz^3 + y^5z + t^2=0.
$$
The surface $X$ is singular at the points $O_x$, $O_y$ and $O_z$.

The curves $C_x$, $C_y$ and $C_z$ are reduced and irreducible. We have
$$
\lct\left(X,
\frac{6}{13}C_x\right)=\frac{91}{60}<\lct\left(X, \frac{6}{15}C_y\right)=\frac{25}{12}<\lct\left(X, \frac{6}{31}C_z\right)=\frac{93}{28}.
$$
Therefore, $\lct(X) \leqslant \frac{91}{60}$.

Suppose that $\lct(X) < \frac{91}{60}$. Then there is an effective
$\Q$-divisor $D\qlineq -K_X$ such that the pair $(X,
\frac{91}{60}D)$ is not log canonical at some point $P$. By
Lemma~\ref{lemma:convexity} we may assume that the support of
$D$ contains none of $C_x$, $C_y$, $C_z$.
Since $C_y$ is singular at the point $O_z$ and $\frac{31}{2}D\cdot C_y=\frac{6}{13}<\frac{60}{91}$, the point $P$ must be in the outside of $C_y$. Furthermore, the point $P$ is in the outside of $C_x\cup C_z$
since $15D\cdot C_x=\frac{12}{31}<\frac{60}{91}$ and $D\cdot C_z=\frac{4}{65}<\frac{60}{91}$.

Now we consider the pencil $\mathcal{L}$ on $X$ defined by the
equations $\lambda z^3+\mu x^6y=0$,
$[\lambda:\mu]\in\mathbb{P}^1$. Then there is a unique member $C$
in $\mathcal{L}$ passing through the point $P$. Since the point
$P$ is located in the outside of $C_x\cup C_y\cup C_z$, the curve
$C$ is cut out by the equation of the form $x^6y+\alpha z^3=0$,
where $\alpha$ is a non-zero constant. Since the curve $C$ is a
double cover of the curve defined by the equation $x^6y+\alpha
z^3=0$ in $\mathbb{P}(13,15,31)$, we have  $\mult_P(C)\leqslant
2$. Therefore, we may assume that the support of $D$ does not
contain at least one irreducible component. If $\alpha\ne 1$, then
the curve $C$ is irreducible, and hence the inequality
\[\mult_P(D)\leqslant D\cdot C=\frac{12}{65}<\frac{60}{91}\]
is a contradiction.
If $\alpha=1$, then the curve $C$ consists of two distinct irreducible and reduced curve $C_1$ and $C_2$.  We have
\[D\cdot C_1=D\cdot C_2=\frac{6}{65}, \ \ C_1^2=C_2^2=\frac{8}{13}.\]
Put $D=a_1C_1+a_2C_2+\Delta$, where $\Delta$ is an effective
$\mathbb{Q}$-divisor whose support contains neither $C_1$ nor
$C_2$. Since the pair $(X, \frac{91}{60}D)$ is log canonical at
$O_x$, both  $a_1$ and $a_2$ are at most $\frac{60}{91}$. Then a
contradiction follows from Lemma~\ref{lemma:handy-adjunction}
since
\[(D-a_iC_i)\cdot C_i\leqslant D\cdot C_i=\frac{12}{65}<\frac{60}{91}\]
for each $i$.
\end{proof}

\section{Sporadic cases with $I=7$}
\label{subsection:index-7}

\begin{lemma}
\label{lemma:I-7-W-11-13-21-38-D-76}
Let $X$ be a quasismooth
hypersurface of degree $76$ in $\mathbb{P}(11,13,21,38)$. Then
$\lct(X)=\frac{13}{10}$.
\end{lemma}

\begin{proof}
 We may assume that the surface $X$ is defined by
the equation $t^{2}+yz^{3}+xy^{5}+x^{5}z=0$. The surface $X$ is
singular at $O_x$, $O_{y}$ and $O_z$. The curves $C_x$, $C_y$ and
$C_z$ are  irreducible. We have
$$
\frac{21}{10}=\lct(X, \frac{7}{21}C_z)>\frac{55}{42}=\lct(X, \frac{7}{11}C_x)>\lct(X, \frac{7}{13}C_y)=\frac{13}{10}.%
$$
Therefore, $\lct(X)\leqslant \frac{13}{10}$.

Suppose that $\lct(X)<\frac{13}{10}$. Then there is an effective
$\Q$-divisor $D\qlineq -K_X$ such that the pair
$(X,\frac{13}{10}D)$ is not log canonical at some point $P$. By
Lemma~\ref{lemma:convexity}, we may assume that the support of $D$
contains none of the curves $C_x$, $C_y$ and $C_z$.

Since the curve $C_y$ is singular at the point $O_z$, the
inequality $11D\cdot C_y=\frac{2}{3}<\frac{10}{13}$ shows that the
point $P$ does not belong to the curve $C_y$. Also, the inequality
$13D\cdot C_x=\frac{2}{3}<\frac{10}{13}$ implies that the point
$P$ cannot belong to $C_x$ either. The inequality $D\cdot C_z=\frac{14}{11\cdot 13}<\frac{10}{13}$ shows that
the point
$P$ cannot belong to $C_z$

Consider the pencil $\mathcal{L}$ on $X$ defined by the equations
$\lambda y^5+\mu x^4z=0$, $[\lambda:\mu]\in\mathbb{P}^1$. There is
a unique member $Z$ in $\mathcal{L}$ passing through the point
$P$. Since $P\not\in C_{x}\cup C_{y}\cup C_{z}$, the curve $Z$ is
defined by an equation of the form $x^{4}z=\alpha y^{5}$, where
$\alpha$ is a non-zero constant.  The open subset $Z\setminus
C_{x}$ of the curve $Z$ is a $\mathbb{Z}_{11}$-quotient of the
affine curve
$$
z-\alpha
y^{5}=t^{2}+yz^{3}+y^{5}+z=0\subset\mathbb{C}^{3}
\cong\mathrm{Spec}\Big(\mathbb{C}\big[y,z,t\big]\Big)
$$
that is isomorphic to the plane affine curve
$C\subset\mathbb{C}^{2}$ defined by the equation
$$
t^{2}+\alpha^{3}y^{16}+(1+\alpha)y^{5}=0\subset\mathbb{C}^{2}
\cong\mathrm{Spec}\Big(\mathbb{C}\big[y,z\big]\Big).
$$
The curve $C$ is irreducible if $\alpha\ne -1$ and reducible if
$\alpha=1$. Since the $C_{x}$ is not contained in the support of
$Z$, the curve $Z$ is irreducible if $\alpha\ne -1$ and reducible
if $\alpha=1$. From the equation of $C$, we can see that the log
pair $(X, \frac{7}{50}Z)$ is log canonical at the point $P$. By
Lemma~\ref{lemma:convexity}, we may assume that $\mathrm{Supp}(D)$
does not contain at least one irreducible component of the curve
$Z$.

Suppose that $\alpha\ne -1$. Then $Z\not\subseteq\mathrm{Supp}(D)$
and
$$
\frac{10}{33}=D\cdot Z\geqslant\mult_{P}\big(D\big)>\frac{10}{13}.%
$$
This is a contradiction. Thus, $\alpha=-1$. Then it follows from
the equation of $C$ that the curve $Z$ consists of two irreducible
and reduced curves $Z_1$ and $Z_2$. Without loss of generality we
may assume that the point $P$ belongs to the curve $Z_1$.

Put $D=mZ_{1}+\Omega$, where $\Omega$ is an effective
$\mathbb{Q}$-divisor such that
$Z_{1}\not\subset\mathrm{Supp}(\Omega)$. Since the pair $(X,
\frac{13}{10}D)$ is log canonical at the point $O_x$, one has
$m\leqslant \frac{10}{13}$.  Then
$$
\big(D-mZ_{1}\big)\cdot Z_{1}<D\cdot Z_1=\frac{5}{33}<\frac{10}{13}.%
$$
since $ Z_{1}^2>0$.  By
Lemma~\ref{lemma:handy-adjunction}, the log pair $(X,
\frac{13}{10}D)$  is log canonical at the point $P$. This is a
contradiction.
\end{proof}

\section{Sporadic cases with $I=8$}
\label{subsection:index-8}

\begin{lemma}
\label{lemma:I-8-W-7-11-13-23-D-46}Let $X$ be a quasismooth
hypersurface of degree $46$ in $\mathbb{P}(7,11,13,23)$. Then
$\lct(X)=\frac{35}{48}$.
\end{lemma}

\begin{proof}
 The surface $X$ can be defined by the
equation $t^{2}+y^{3}z+xz^{3}+x^{5}y=0$. The surface $X$ is
singular at the points $O_x$, $O_y$ and $O_z$. The curves $C_x$,
$C_y$ and $C_{z}$ are irreducible. We have
$$
\frac{35}{48}=\lct\left(X, \frac{8}{7}C_x\right)<\lct\left(X,
\frac{8}{13}C_z\right)=\frac{91}{80}<\lct\left(X,
\frac{8}{11}C_y\right)=\frac{55}{48}.
$$
In particular, $\lct(X)\leqslant \frac{35}{48}$. Suppose that
$\lct(X)<\frac{35}{48}$. Then there is an effective $\Q$-divisor
$D\qlineq -K_X$ such that the pair $(X,\frac{35}{48}D)$ is not log
canonical at some point $P$. By Lemma~\ref{lemma:convexity}, we
may assume that the support of the divisor $D$  contains none of
the curves $C_x$, $C_y$ and $C_{z}$.

Since the curve $C_x$ is singular at the point $O_z$, the
inequality $$11D\cdot C_x=\frac{16}{13}<\frac{48}{35}$$ shows that
the point $P$ cannot belong to $C_x$. Also, the inequality
$$
7D\cdot C_{y}=\frac{16}{13}<\frac{48}{35}
$$
implies that the point $P$ is not in $C_y$. Since
$$
D\cdot C_{z}=\frac{16}{7\cdot 11}<\frac{48}{35},%
$$
the point $P$ cannot be in $C_z$ either.

Consider the pencil $\mathcal{L}$ on $X$ defined by the equations
$\lambda x^4y+\mu z^3=0$, $[\lambda:\mu]\in\mathbb{P}^1$. There is
a unique member $Z$ in $\mathcal{L}$ passing through the point
$P$. Since $P\not\in C_{x}\cup C_{y}\cup C_{z}$, the curve $Z$ is
defined by an equation of the form $x^{4}y=\alpha z^{3}$, where
$\alpha$ is a non-zero constant.  The open subset $Z\setminus
C_{x}$ of the curve $Z$ is a $\mathbb{Z}_{7}$-quotient of the
affine curve
$$
y-\alpha
z^{3}=t^{2}+y^{3}z+z^{3}+y=0\subset\mathbb{C}^{3}
\cong\mathrm{Spec}\Big(\mathbb{C}\big[y,z,t\big]\Big)
$$
that is isomorphic to the plane affine curve
$C\subset\mathbb{C}^{2}$ defined by the equation
$$
t^{2}+\alpha^{3}z^{10}+(1+\alpha)z^{3}=0\subset\mathbb{C}^{2}
\cong\mathrm{Spec}\Big(\mathbb{C}\big[y,z\big]\Big).
$$
The curve $C$ is irreducible if $\alpha\ne -1$ and reducible if
$\alpha=-1$. Since the $C_{x}$ is not contained in the support of
$Z$, the curve $Z$ is irreducible if $\alpha\ne -1$ and reducible
if $\alpha=-1$. From the equation of $C$, we can see that the log
pair $(X, \frac{35}{234}Z)$ is log canonical at the point $P$. By
Lemma~\ref{lemma:convexity}, we may assume that $\mathrm{Supp}(D)$
does not contain at least one irreducible component of the curve
$Z$.

Suppose that $\alpha\ne -1$. Then $Z\not\subseteq\mathrm{Supp}(D)$
and
$$
\frac{48}{77}=D\cdot Z\geqslant\mult_{P}\big(D\big)>\frac{48}{35}.%
$$
This is a contradiction. Thus, $\alpha=-1$. Then it follows from
the equation of $C$ that the curve $Z$ consists of two irreducible
and reduced curves $Z_1$ and $Z_2$. Without loss of generality we
may assume that the point $P$ belongs to the curve $Z_1$.

Put $D=mZ_{1}+\Omega$, where $\Omega$ is an effective
$\mathbb{Q}$-divisor such that
$Z_{1}\not\subset\mathrm{Supp}(\Omega)$. Since the pair $(X,
\frac{48}{35}D)$ is log canonical at the point $O_x$, one has
$m\leqslant \frac{35}{48}$.  Then
$$
\big(D-mZ_{1}\big)\cdot Z_{1}<D\cdot Z_1=\frac{24}{77}<\frac{48}{35}.%
$$
since $ Z_{1}^2>0$.  By
Lemma~\ref{lemma:handy-adjunction}, the log pair $(X,
\frac{48}{35}D)$  is log canonical at the point $P$. This is a
contradiction.
\end{proof}

\begin{lemma}
\label{lemma:I-8-W-7-18-27-37-D-81} Let $X$ be a quasismooth
hypersurface of degree $81$ in $\mathbb{P}(7,18,27,37)$. Then
$\lct(X)=\frac{35}{72}$.
\end{lemma}

\begin{proof}
  The surface $X$ can be defined by the
quasihomogeneous equation
$$
z^{3}-y^{3}z+xt^{2}+x^{9}y=0.
$$
The surface $X$ is singular at the points $O_x$, $O_y$, $O_t$ and
$Q=[0:1:1:0]$.

The curve $C_x$ consists of two irreducible and reduced curves
$L_{xz}$ and $R_x=\{x=z^2-y^3=0\}$. These two curves intersect
each other only at the point $O_t$. Also,
$$
L_{xz}^2=-\frac{47}{18\cdot 37},\ \ \ R_x^2=-\frac{20}{9\cdot 37},\ \ \
 L_{xz}\cdot R_{x}=\frac{3}{37}.%
$$
 The curve $C_y$ is irreducible and
$$
\frac{35}{72}=\lct\left(X, \frac{8}{7}C_x\right)
<\lct\left(X, \frac{8}{18}C_y\right)=\frac{15}{8}.%
$$

Suppose that $\lct(X)<\frac{35}{72}$. Then there is an effective
$\Q$-divisor $D\qlineq -K_X$ such that the pair
$(X,\frac{35}{72}D)$ is not log canonical at some point $P$. By
Lemma~\ref{lemma:convexity}, we may assume that the support of the
divisor $D$ does not contain the curve $C_y$. Similarly, we may
assume that either $L_{xz}\not\subseteq\mathrm{Supp}(D)$ or
$R_{x}\not\subseteq\mathrm{Supp}(D)$.

Since either $L_{xz}\not\subseteq\mathrm{Supp}(D)$ or
$R_{x}\not\subseteq\mathrm{Supp}(D)$, one of the inequalities
\[\mult_{O_t}(D)\leqslant 37D\cdot
L_{xz}=\frac{4}{9}<\frac{72}{35}, \ \ \ \mult_{O_t}(D)\leqslant
37D\cdot R_{x}=\frac{8}{9}<\frac{72}{35}\] must hold, and hence
the point $P$ cannot be $O_t$. Since $\mult_{O_x}(D)\leqslant
7D\cdot C_y=\frac{24}{37}<\frac{72}{35}$, the point $P$ cannot be
the point $O_x$.

Put $D=mL_{xz}+\Omega$, where $\Omega$ is an effective
$\mathbb{Q}$-divisor such that
$L_{xz}\not\subset\mathrm{Supp}(\Omega)$. If $m\ne 0$, then
$$
\frac{16}{18\cdot 37}=D\cdot R_{x}\geqslant  mL_{xz}\cdot R_{x}=\frac{3m}{37},%
$$
and hence $m\leqslant \frac{8}{27}$. Since
$$
18\big(D-mL_{xz}\big)\cdot
L_{xz}=\frac{8+47m}{37}\leqslant\frac{72}{35}
$$
it follows from Lemma~\ref{lemma:handy-adjunction} that the point
$P$ cannot belong to $L_{xz}$.

Now we write  $D=\epsilon Z_{x}+\Delta$, where $\Delta$ is an
effective $\mathbb{Q}$-divisor such that
$Z_{x}\not\subset\mathrm{Supp}(\Delta)$. If $\epsilon\ne 0$, then
$$
\frac{8}{18\cdot 37}=D\cdot L_{xz}\geqslant  \epsilon R_{x}\cdot L_{xz}=\frac{3\epsilon}{37},%
$$
and hence $\epsilon\leqslant \frac{4}{27}$. Since
$$
9\big(D-\epsilon R_{x}\big)\cdot
R_{x}=\frac{8+20\epsilon}{37}\leqslant\frac{72}{35}
$$
it follows from Lemma~\ref{lemma:handy-adjunction} that the point
$P$ cannot belong to $R_{x}$. Consequently, the point $P$ must be
a smooth point in the outside of $C_x$.  However, since $H^0(\P,
\mathcal{O}_{\P}(189))$ contains the monomials $x^{27}$,
$y^{7}x^{9}$ and $z^{7}$, it follows from
Lemma~\ref{lemma:Carolina} that $P$ must be either a
singular point of $X$ or a point on $C_x$. This is a
contradiction.
\end{proof}

\section{Sporadic cases with $I=9$}
\label{subsection:index-9}

\begin{lemma}
\label{lemma:I-9-W-7-15-19-32-D-64} Let $X$ be a quasismooth
hypersurface of degree $64$ in $\mathbb{P}(7,15,19,32)$. Then
$\lct(X)=\frac{35}{54}$.
\end{lemma}

\begin{proof}
 The surface $X$ can be defined by the
quasihomogeneous equation
$$
t^{2}+y^{3}z+xz^{3}+x^{7}y=0.
$$
The surface $X$ is singular only at the points $O_x$, $O_y$ and
$O_z$. The curves $C_x$ and $C_y$ are irreducible, and
$$
\frac{35}{54}=\lct\left(X, \frac{9}{7}C_x\right)
<\lct\left(X, \frac{9}{15}C_y\right)=\frac{25}{18}.%
$$
In particular, $\lct(X)\leqslant \frac{35}{54}$.

Suppose that $\lct(X)<\frac{35}{54}$. Then there is an effective
$\Q$-divisor $D\qlineq -K_X$ such that the pair
$(X,\frac{35}{54}D)$ is not log canonical at some point $P$. By
Lemma~\ref{lemma:convexity}, we may assume that the support of the
divisor $D$ contains neither the curve $C_x$ nor the curve $C_y$.
Then two inequalities $19D\cdot C_x=\frac{6}{5}<\frac{54}{35}$,
$7D\cdot C_y=\frac{18}{19}<\frac{54}{35}$ show that the point $P$
must be a smooth point in the outside of $C_x$.

Note that  $H^0(\P, \mathcal{O}_{\P}(133))$ contains the monomials
$x^{19}$, $y^7x^{4}$ and $z^{7}$ and hence it follows from
Lemma~\ref{lemma:Carolina} that the point $P$ is either a singular
point of $X$ or a point on $C_x$. This is a contradiction.
\end{proof}

\section{Sporadic cases with $I=10$}
\label{subsection:index-10}

\begin{lemma}
\label{lemma:I-10-W-7-19-25-41-D-82} Let $X$ be a quasismooth
hypersurface of degree $82$ in $\mathbb{P}(7,19,25,41)$. Then
$\lct(X)=\frac{7}{12}$.
\end{lemma}

\begin{proof}
 The surface $X$ can be defined by the
quasihomogeneous equation
$$
t^{2}+y^{3}z+xz^{3}+x^{9}y=0.
$$
It is singular at the points $O_{x}$, $O_{y}$ and $O_z$.

The curves $C_{x}$ and $C_y$ are irreducible. We have
$$
\frac{7}{12}=\lct\left(X, \frac{10}{7}C_x\right)<\lct\left(X, \frac{10}{19}C_y\right)=\frac{19}{12},%
$$
and hence $\lct(X)\leqslant \frac{7}{12}$.

Suppose that $\lct(X)<\frac{7}{12}$. Then there is an effective
$\Q$-divisor $D\qlineq -K_X$ such that the pair
$(X,\frac{7}{12}D)$ is not log canonical at some point $P$. By
Lemma~\ref{lemma:convexity}, we may assume that the support of the
divisor $D$ contains neither the curve $C_x$ nor the curve
$C_{y}$. Since $25D\cdot C_x=\frac{20}{19}<\frac{12}{7}$ and
$7D\cdot C_y=\frac{4}{5}<\frac{12}{7}$, the point $P$ must be a
smooth point in the outside of the curve $C_x$. Note that
$H^0(\P, \mathcal{O}_{\P}(175))$ contains the monomials $x^{25}$,
$x^{6}y^{7}$ and $z^{7}$, and hence the point $P$ cannot be a
smooth point in the outside of $C_x$ by
Lemma~\ref{lemma:Carolina}. Consequently, $\lct(X)=\frac{7}{12}$.
\end{proof}

\begin{lemma}
\label{lemma:I-10-W-7-26-39-55-D-117} Let $X$ be a quasismooth
hypersurface of degree $117$ in $\mathbb{P}(7,26,39,55)$. Then
$\lct(X)=\frac{7}{18}$.
\end{lemma}

\begin{proof}
The surface $X$ can be defined by the equation
$z^{3}-y^{3}z+xt^{2}+x^{13}y=0$. It is singular at the points
$O_x$, $O_y$, $O_t$ and $Q=[0:1:1:0]$.

The curve $C_x$ consists of two irreducible curves $L_{xz}$ and
$R_x=\{x=z^2-y^3=0\}$.  These two curves intersect each other only
at the point $O_t$. It is easy to check
$$
L_{xz}^2=-\frac{71}{26\cdot 55},\ \ \ R_{x}^2=-\frac{32}{13\cdot 55},
\ \ \ L_{xz}\cdot R_{x}=\frac{3}{55}.%
$$
On the other hand, the curve $C_y$ is irreducible. We have
$$
\frac{7}{18}=\lct\left(X, \frac{10}{7}C_x\right)<
\lct\left(X, \frac{10}{26}C_y\right)=\frac{13}{6}.%
$$
In particular, $\lct(X)\leqslant \frac{7}{18}$.

Suppose that $\lct(X)<\frac{7}{18}$. Then there is an effective
$\Q$-divisor $D\qlineq -K_X$ such that the pair
$(X,\frac{7}{18}D)$ is not log canonical at some point $P$. By
Lemma~\ref{lemma:convexity}, we may assume that the support of the
divisor $D$ does not contain the curve $C_y$. Similarly, we may
assume that either $L_{xz}\not\subset\mathrm{Supp}(D)$ or
$R_{x}\not\subset\mathrm{Supp}(D)$.

Since $7D\cdot C_y=\frac{6}{11}<\frac{18}{7}$, the point $P$
cannot be the point $O_x$. Meanwhile, since the support of $D$
does not contain at least one components of $C_x$, one of the
inequalities
\[\mult_{O_t}(D)\leqslant 55D\cdot
L_{xz}=\frac{5}{13}<\frac{18}{7},\]
\[\mult_{O_t}(D)\leqslant 55D\cdot
R_{x}=\frac{10}{13}<\frac{18}{7}\] must hold,
and hence the point $P$ cannot be the point $O_t$.

Put $D=mL_{xz}+\Omega$, where $\Omega$ is an effective
$\mathbb{Q}$-divisor such that
$L_{xz}\not\subset\mathrm{Supp}(\Omega)$. If $m\ne 0$, then
$$
\frac{10}{13\cdot 55}=D\cdot R_{x}\geqslant  mL_{xz}\cdot R_{x}=\frac{3m}{55},%
$$
and hence $m\leqslant \frac{10}{39}$. Then
$$
26\big(D-mL_{xz}\big)\cdot L_{xz}=\frac{10+71m}{55}<\frac{18}{7},$$
and hence Lemma~\ref{lemma:handy-adjunction} implies that the
point $P$ cannot belong to $L_{xz}$.

Now we write $D=\epsilon R_{x}+\Delta$, where $\Delta$ is an
effective $\mathbb{Q}$-divisor such that
$R_{x}\not\subset\mathrm{Supp}(\Delta)$. If $\epsilon\ne 0$, then
$$
\frac{10}{26\cdot 55}=D\cdot L_{xz}\geqslant  \epsilon R_{x}\cdot L_{xz}=\frac{3\epsilon}{55},%
$$
and hence $\epsilon\leqslant \frac{5}{39}$. Then
\[13(D-\epsilon
R_{x}\big)\cdot R_{x}=\frac{10+32\epsilon}{55}<\frac{18}{7}.\]
Thus, Lemma~\ref{lemma:adjunction} shows that the point $P$ is not
on $R_x$.

Therefore, the point $P$ must be a smooth point in the outside of
the curve $C_x$. Since $H^0(\P, \mathcal{O}_{\P}(273))$ contains
the monomials $x^{39}$, $y^{7}x^{13}$ and $z^{7}$, it follows from
Lemma~\ref{lemma:Carolina} that $P$ is either a point on $C_x$ or
a singular point of $X$. This is a contradiction.
\end{proof}


\part{The Big
Table} \label{section:table}

The tables contains the following information on del Pezzo surfaces.
\begin{itemize}
\item The first column: the weights
$(a_0,a_1,a_2,a_3)$ of the weighted projective space
$\mathbb{P}$.
\item The second column: the degree of the surface
$X\subset\mathbb{P}$.

\item The third column: the self-intersection number $K_{X}^{2}$ of an anticanonical divisor of $X$.

\item The
fourth column: the rank $\rho$ of the Picard group of the surface $X$.

\item The
fifth column: the global log canonical theeshold $\mathrm{lct}(X)$ of $X$.

\item The sixth column:
the possible monomials in $x$, $y$, $z$, $t$ in the defining
equation $f(x, y, z, t)=0$ of the surface $X$.

\item The seventh column: the information on the singular points
of $X$. We use the standard notation for cyclic quotient
singularities along with the following convention: when we write,
for instance, $O_xO_y=n\times\frac{1}{r}(a, b)$, we mean that
there are $n$ cyclic quotient singularities of type
$\frac{1}{r}(a, b)$ cut out on $X$ by the equations $z=t=0$ that
are different from the point $O_{x}$ in the case when $O_{x}\in X$
and $O_{x}$ is not of type $\frac{1}{r}(a, b)$, and that are
different from the point $O_{y}$ in the case when $O_{y}\in X$ and
$O_{y}$ is not of type $\frac{1}{r}(a, b)$.
\end{itemize}

\setbox1=\hbox{\parbox{1.5\linewidth}{

\begin{center}
 Log del Pezzo surfaces with $I=1$
\label{table:Boyer}

\begin{longtable}{|c|c|c|c|c|c|c|}
\hline
\begin{minipage}[m]{.162\linewidth}\setlength{\unitlength}{.20mm}
\begin{center}
Weights \end{center}
\end{minipage}&
\begin{minipage}[m]{.053\linewidth}\setlength{\unitlength}{.20mm}
\begin{center}
Degree
\end{center}
\end{minipage}&
\begin{minipage}[m]{.031\linewidth}\setlength{\unitlength}{.20mm}
\begin{center}
$K_X^2$
\end{center}
\end{minipage}& $\rho$ &
\begin{minipage}[m]{.019\linewidth}\setlength{\unitlength}{.20mm}
\begin{center}
$\mathrm{lct}$
\end{center}
\end{minipage}&
\begin{minipage}[m]{.121\linewidth}\setlength{\unitlength}{.20mm}
\begin{center}
Monomials
\end{center}
\end{minipage} &
\begin{minipage}[m]{.135\linewidth}\setlength{\unitlength}{.20mm}
\begin{center}
Singular Points
\end{center}
\end{minipage}\\
\hline $\left(2,2n+1,2n+1,4n+1\right)$& $8n+4$
&$\frac{2}{(2n+1)(4n+1)}$ & 8 &1&
 \begin{minipage}[m]{.121\linewidth}\setlength{\unitlength}{.20mm}
\begin{center}$y^4$, $y^3z$, $y^2z^2$, $yz^3$, $z^4$, $xt^2$, $x^{n+1}yt$, $x^{n+1}zt$, $x^{2n+1}y^2$, $x^{2n+1}yz$, $x^{2n+1}z^2$, $x^{4n+2}$ \end{center}\end{minipage}& \begin{minipage}[m]{.135\linewidth}\setlength{\unitlength}{.20mm}

$O_t=\frac{1}{4n+1}(1,1)$

$O_yO_z=\\
\phantom{a}=4\times \frac{1}{2n+1}(1, n)$
\end{minipage}\\
 \hline

$\left(1,2,3,5\right)$ & $10$ &$\frac{1}{3}$&9
&\begin{minipage}[m]{.019\linewidth}\setlength{\unitlength}{.20mm}
 \begin{center}$1^\textnormal{a}$

 $\frac{7}{10}^\textnormal{b}$
 \end{center}\end{minipage}&
 \begin{minipage}[m]{.121\linewidth}\setlength{\unitlength}{.20mm}
\begin{center}$t^{2}$, $yzt$, $y^{2}z^{2}$, $y^{5}$, $xz^{3}$, $xy^{2}t$, $xy^{3}z$, $x^{2}zt$, $x^{2}yz^{2}$, $x^{2}y^{4}$, $
 x^{3}yt$, $x^{3}y^{2}z$, $x^{4}z^{2}$, $x^{4}y^{3}$, $x^{5}t$, $x^{5}yz$, $x^{6}y^{2}$, $x^{7}z$, $x^{8}y$, $x^{10}$\end{center}\end{minipage}&\begin{minipage}[m]{.135\linewidth}\setlength{\unitlength}{.20mm}

$O_z=\frac{1}{3}(1,1)$

\end{minipage}\\
\hline

\end{longtable}
\end{center}

{ \ }\phantom{a: if $C_x$ has an ordin\,} a: if $C_x$ has an
ordinary double point, b: if $C_x$ has a non-ordinary double
point, \\
{  \ \ \ }\phantom{a: if $C_x$ has an ordina} c: if the defining
equation of $X$ contains $yzt$, d: if the defining equation of $X$
does not contain $yzt$. }}

\rotatebox{90}{\box1}

\newpage
  \setbox1=\hbox{\parbox{1.5\linewidth}{
\begin{center}
Log del Pezzo surfaces with $I=1$
\begin{longtable}{|c|c|c|c|c|c|c|}
\hline
\begin{minipage}[m]{.110\linewidth}\setlength{\unitlength}{.20mm}
\begin{center}
Weights \end{center}
\end{minipage}&
\begin{minipage}[m]{.053\linewidth}\setlength{\unitlength}{.20mm}
\begin{center}
Degree
\end{center}
\end{minipage}&
\begin{minipage}[m]{.031\linewidth}\setlength{\unitlength}{.20mm}
\begin{center}
$K_X^2$
\end{center}
\end{minipage}& $\rho$&
\begin{minipage}[m]{.019\linewidth}\setlength{\unitlength}{.20mm}
\begin{center}
$\mathrm{lct}$
\end{center}
\end{minipage}&
\begin{minipage}[m]{.141\linewidth}\setlength{\unitlength}{.20mm}
\begin{center}
Monomials
\end{center}
\end{minipage} &
\begin{minipage}[m]{.177\linewidth}\setlength{\unitlength}{.20mm}
\begin{center}
Singular Points
\end{center}
\end{minipage}\\
\hline

$\left(2,3,5,9\right)$ & $18$
&$\frac{1}{15}$&7&\begin{minipage}[m]{.019\linewidth}\setlength{\unitlength}{.20mm}
 \begin{center}$2^\textnormal{a}$

 $\frac{11}{6}^\textnormal{b}$
 \end{center}\end{minipage}&
 \begin{minipage}[m]{.121\linewidth}\setlength{\unitlength}{.20mm}
\begin{center}

$ t^{2}$, $yz^{3}$, $y^{3}t$, $y^{6}$, $xy^{2}z^{2}$, $x^{2}zt$,
$x^{2}y^{3}z$, $x^{3}yt$, $x^{3}y^{4}$, $ x^{4}z^{2}$, $x^{5}yz$,
$x^{6}y^{2}$, $x^{9}$

\end{center}\end{minipage} &\begin{minipage}[m]{.177\linewidth}\setlength{\unitlength}{.20mm}

$O_z=\frac{1}{5}(1,2)$

$O_yO_t=2\times \frac{1}{3}(1, 1)$
\end{minipage}\\
\hline

 $\left(3,3,5,5\right)$ & $15$ &$\frac{1}{15}$&5&2&
 \begin{minipage}[m]{.121\linewidth}\setlength{\unitlength}{.20mm}
\begin{center}

$ t^{3}$, $zt^{2}$, $z^{2}t$, $z^{3}$, $y^{5}$, $xy^{4}$,
$x^{2}y^{3}$, $x^{3}y^{2}$, $x^{4}y$, $x^{5}$

\end{center}\end{minipage} &\begin{minipage}[m]{.177\linewidth}\setlength{\unitlength}{.20mm}

$O_xO_y=5\times\frac{1}{3}(1,1)$

$O_zO_t=3\times \frac{1}{5}(1, 1)$
\end{minipage}\\
\hline

\end{longtable}
\end{center}

{ \ } \phantom{a: if $C_y$ has a tacnodal p} a: if $C_y$ has a
tacnodal point, b: if $C_y$ has no tacnodal points.

}} \rotatebox{90}{\box1}
\newpage

 \setbox1=\hbox{\parbox{1.5\linewidth}{
\begin{center}
Log del Pezzo surfaces with $I=1$
\begin{longtable}{|c|c|c|c|c|c|c|}

\hline
\begin{minipage}[m]{.110\linewidth}\setlength{\unitlength}{.20mm}
\begin{center}
Weights \end{center}
\end{minipage}&
\begin{minipage}[m]{.053\linewidth}\setlength{\unitlength}{.20mm}
\begin{center}
Degree
\end{center}
\end{minipage}&
\begin{minipage}[m]{.031\linewidth}\setlength{\unitlength}{.20mm}
\begin{center}
$K_X^2$
\end{center}
\end{minipage}& $\rho$ &
\begin{minipage}[m]{.019\linewidth}\setlength{\unitlength}{.20mm}
\begin{center}
$\mathrm{lct}$
\end{center}
\end{minipage}&
\begin{minipage}[m]{.141\linewidth}\setlength{\unitlength}{.20mm}
\begin{center}
Monomials
\end{center}
\end{minipage} &
\begin{minipage}[m]{.177\linewidth}\setlength{\unitlength}{.20mm}
\begin{center}
Singular Points
\end{center}
\end{minipage}\\
\hline

$\left(7,11,27,44\right)$ & $88$ &
$\frac{2}{2079}$&4&$\frac{35}{8}$&
 \begin{minipage}[m]{.141\linewidth}\setlength{\unitlength}{.20mm}
\begin{center}

$ t^{2}$, $y^{4}t$, $y^{8}$, $xz^{3}$, $x^{4}y^{3}z$, $x^{11}y$

\end{center}\end{minipage} &
\begin{minipage}[m]{.177\linewidth}\setlength{\unitlength}{.20mm}

$O_x=\frac{1}{7}(3,1)$

$O_z=\frac{1}{27}(11,17)$

$O_yO_t=2\times \frac{1}{11}(7, 5)$
\end{minipage}\\
\hline

$\left(9,15,17,20\right)$ & $60$
&$\frac{1}{765}$&3&$\frac{21}{4}$&
 \begin{minipage}[m]{.141\linewidth}\setlength{\unitlength}{.20mm}
\begin{center}

$ t^{3}$, $y^{4}$, $xz^{3}$, $x^{5}y$

\end{center}\end{minipage} & \begin{minipage}[m]{.177\linewidth}\setlength{\unitlength}{.20mm}

$O_x=\frac{1}{9}(4,1)$

$O_z=\frac{1}{17}(5,1)$

$O_xO_y=1\times\frac{1}{3}(1,1)$

$O_yO_t=1\times \frac{1}{5}(2, 1)$
\end{minipage}\\
\hline
\end{longtable}
\end{center}

}} \rotatebox{90}{\box1}

\newpage

 \setbox1=\hbox{\parbox{1.5\linewidth}{
\begin{center}
Log del Pezzo surfaces with $I=2$
\begin{longtable}{|c|c|c|c|c|c|c|}

\hline
\begin{minipage}[m]{.162\linewidth}\setlength{\unitlength}{.20mm}
\begin{center}
Weights \end{center}
\end{minipage}&
\begin{minipage}[m]{.053\linewidth}\setlength{\unitlength}{.20mm}
\begin{center}
Degree
\end{center}
\end{minipage}&
\begin{minipage}[m]{.031\linewidth}\setlength{\unitlength}{.20mm}
\begin{center}
$K_X^2$
\end{center}
\end{minipage}& $\rho$ &
\begin{minipage}[m]{.019\linewidth}\setlength{\unitlength}{.20mm}
\begin{center}
$\mathrm{lct}$
\end{center}
\end{minipage}&
\begin{minipage}[m]{.141\linewidth}\setlength{\unitlength}{.20mm}
\begin{center}
Monomials
\end{center}
\end{minipage} &
\begin{minipage}[m]{.177\linewidth}\setlength{\unitlength}{.20mm}
\begin{center}
Singular Points
\end{center}
\end{minipage}\\

 \hline
 $\left(3,3n,3n+1,3n+1\right)$ & $9n+3$ & $\frac{4}{3n(3n+1)}$ &7&$1$&\begin{minipage}[m]{.141\linewidth}\setlength{\unitlength}{.20mm}
\begin{center}
$t^{3}$, $zt^{2}$, $z^{2}t$, $z^{3}$, $xy^{3}$, $x^{n+1}y^{2}$,
$x^{2n+1}y$, $x^{3n+1}$

\end{center}\end{minipage}& \begin{minipage}[m]{.177\linewidth}\setlength{\unitlength}{.20mm}
$O_y=\frac{1}{3n}(1, 1)$

$O_xO_y=3\times\frac{1}{3}(1,1)$

$O_zO_t=3\times\frac{1}{3n+1}(1,n)$
\end{minipage}\\
 \hline
 $\left(3,3n+1,3n+2,3n+2\right)$ &  $9n+6$ & $\frac{4}{(3n+1)(3n+2)}$&5& 1 &\begin{minipage}[m]{.141\linewidth}\setlength{\unitlength}{.20mm}
\begin{center}
$t^{3}$, $zt^{2}$, $z^{2}t$, $z^{3}$, $xy^{3}$, $x^{n+1}yt$,
$x^{n+1}yz$, $x^{3n+2}$
\end{center}\end{minipage}& \begin{minipage}[m]{.177\linewidth}\setlength{\unitlength}{.20mm}
$O_y=\frac{1}{3n+1}(1, 1)$

$O_zO_t=\\
\phantom{a}=3\times\frac{1}{3n+2}(3,3n+1)$
\end{minipage}\\
\hline
\end{longtable}\end{center}

}} \rotatebox{90}{\box1}
\newpage
  \setbox1=\hbox{\parbox{1.5\linewidth}{
\begin{center}
Log del Pezzo surfaces with $I=2$
\begin{longtable}{|c|c|c|c|c|c|c|}

\hline
\begin{minipage}[m]{.110\linewidth}\setlength{\unitlength}{.20mm}
\begin{center}
Weights \end{center}
\end{minipage}&
\begin{minipage}[m]{.053\linewidth}\setlength{\unitlength}{.20mm}
\begin{center}
Degree
\end{center}
\end{minipage}&
\begin{minipage}[m]{.031\linewidth}\setlength{\unitlength}{.20mm}
\begin{center}
$K_X^2$
\end{center}
\end{minipage}& $\rho$ &
\begin{minipage}[m]{.019\linewidth}\setlength{\unitlength}{.20mm}
\begin{center}
$\mathrm{lct}$
\end{center}
\end{minipage}&
\begin{minipage}[m]{.141\linewidth}\setlength{\unitlength}{.20mm}
\begin{center}
Monomials
\end{center}
\end{minipage} &
\begin{minipage}[m]{.177\linewidth}\setlength{\unitlength}{.20mm}
\begin{center}
Singular Points
\end{center}
\end{minipage}\\

\hline

 $\left(2,3,4,5\right)$ &  $12$ &$\frac{2}{5}$&5&\begin{minipage}[m]{.019\linewidth}\setlength{\unitlength}{.20mm}
 \begin{center}$1^\textnormal{a}$

 $\frac{7}{12}^\textnormal{b}$
 \end{center}\end{minipage}
   &\begin{minipage}[m]{.141\linewidth}\setlength{\unitlength}{.20mm}
\begin{center}
$z^{3}$, $yzt$, $y^{4}$, $xt^{2}$, $xy^{2}z$, $x^{2}z^{2}$,
$x^{2}yt$, $x^{3}y^{2}$, $x^{4}z$, $x^{6}$
\end{center}\end{minipage}& \begin{minipage}[m]{.177\linewidth}\setlength{\unitlength}{.20mm}
$O_t=\frac{1}{5}(3,4)$

$O_xO_z=3\times\frac{1}{2}(1,1)$
\end{minipage}\\
 \hline

 $\left(2,3,4,7\right)$ & $14$  &
$\frac{1}{3}$&6&1&\begin{minipage}[m]{.141\linewidth}\setlength{\unitlength}{.20mm}
\begin{center}

$t^{2}$, $yzt$, $y^{2}z^{2}$, $xz^{3}$, $xy^{4}$, $x^{2}yt$,
$x^{2}y^{2}z$, $x^{3}z^{2}$, $x^{4}y^{2}$, $x^{5}z$, $x^{7}$

\end{center}\end{minipage}& \begin{minipage}[m]{.177\linewidth}\setlength{\unitlength}{.20mm}
$O_y=\frac{1}{3}(1,1)$

$O_z=\frac{1}{4}(1,1)$

$O_xO_z=3\times\frac{1}{2}(1,1)$
\end{minipage}\\
\hline

\end{longtable}
\end{center}

{ \ } \phantom{a: if the defining equation} a: if the defining
equation of $X$ contains $yzt$, b: if the defining equation of $X$
contains no $yzt$.

}} \rotatebox{90}{\box1}
\newpage
  \setbox1=\hbox{\parbox{1.5\linewidth}{
\begin{center}
Log del Pezzo surfaces with $I=2$
\begin{longtable}{|c|c|c|c|c|c|c|}

\hline
\begin{minipage}[m]{.110\linewidth}\setlength{\unitlength}{.20mm}
\begin{center}
Weights \end{center}
\end{minipage}&
\begin{minipage}[m]{.053\linewidth}\setlength{\unitlength}{.20mm}
\begin{center}
Degree
\end{center}
\end{minipage}&
\begin{minipage}[m]{.031\linewidth}\setlength{\unitlength}{.20mm}
\begin{center}
$K_X^2$
\end{center}
\end{minipage}& $\rho$ &
\begin{minipage}[m]{.019\linewidth}\setlength{\unitlength}{.20mm}
\begin{center}
$\mathrm{lct}$
\end{center}
\end{minipage}&
\begin{minipage}[m]{.141\linewidth}\setlength{\unitlength}{.20mm}
\begin{center}
Monomials
\end{center}
\end{minipage} &
\begin{minipage}[m]{.177\linewidth}\setlength{\unitlength}{.20mm}
\begin{center}
Singular Points
\end{center}
\end{minipage}\\
 \hline

  $\left( 7,8,19,25  \right)$ &  $ 57$ & $\frac{57}{6650}$&3&$\frac{49}{24}$&\begin{minipage}[m]{.141\linewidth}\setlength{\unitlength}{.20mm}
\begin{center}
$ty^4$, $z^3$, $xt^2$, $x^2y^3z$,  $x^{7}y$
\end{center}\end{minipage}& \begin{minipage}[m]{.177\linewidth}\setlength{\unitlength}{.20mm}
$O_x=\frac{1}{7}(5,4)$

$O_y=\frac{1}{8}(7,3)$

$O_t=\frac{1}{25}(8,19)$
\end{minipage}\\
 \hline

 $\left(  7,8,19,32 \right)$ &  $64 $ &$\frac{1}{133}$& 4&$\frac{35}{16}$ &\begin{minipage}[m]{.141\linewidth}\setlength{\unitlength}{.20mm}
\begin{center}
$t^2$, $ty^4$, $y^8$, $xz^3$, $x^3y^3z$,  $x^{8}y$
\end{center}\end{minipage}& \begin{minipage}[m]{.177\linewidth}\setlength{\unitlength}{.20mm}
$O_x=\frac{1}{7}(5,4)$

$O_z=\frac{1}{19}(8,13)$

$O_yO_t=2\times\frac{1}{8}(7,3)$
\end{minipage}\\
 \hline

  $\left( 9,12,13,16  \right)$ &  $48 $ &$\frac{1}{117}$&3& $\frac{63}{24}$ &\begin{minipage}[m]{.141\linewidth}\setlength{\unitlength}{.20mm}
\begin{center}
$t^3$, $y^4$, $xz^3$, $x^4y$
\end{center}\end{minipage}& \begin{minipage}[m]{.177\linewidth}\setlength{\unitlength}{.20mm}
$O_x=\frac{1}{9}(4,7)$

$O_z=\frac{1}{13}(4,1)$

$O_xO_y=1\times\frac{1}{3}(1,1)$

$O_yO_t=1\times\frac{1}{4}(1,1)$
\end{minipage}\\
\hline

\end{longtable}
\end{center}

}} \rotatebox{90}{\box1}
\newpage
  \setbox1=\hbox{\parbox{1.5\linewidth}{
\begin{center}
Log del Pezzo surfaces with $I=2$
\begin{longtable}{|c|c|c|c|c|c|c|}

\hline
\begin{minipage}[m]{.110\linewidth}\setlength{\unitlength}{.20mm}
\begin{center}
Weights \end{center}
\end{minipage}&
\begin{minipage}[m]{.053\linewidth}\setlength{\unitlength}{.20mm}
\begin{center}
Degree
\end{center}
\end{minipage}&
\begin{minipage}[m]{.031\linewidth}\setlength{\unitlength}{.20mm}
\begin{center}
$K_X^2$
\end{center}
\end{minipage}& $\rho$ &
\begin{minipage}[m]{.019\linewidth}\setlength{\unitlength}{.20mm}
\begin{center}
$\mathrm{lct}$
\end{center}
\end{minipage}&
\begin{minipage}[m]{.141\linewidth}\setlength{\unitlength}{.20mm}
\begin{center}
Monomials
\end{center}
\end{minipage} &
\begin{minipage}[m]{.147\linewidth}\setlength{\unitlength}{.20mm}
\begin{center}
Singular Points
\end{center}
\end{minipage}\\
 \hline

 $\left( 11,25,32,41  \right)$ &  $107 $ &$\frac{107}{90200}$ &3&$\frac{11}{3}$ &\begin{minipage}[m]{.141\linewidth}\setlength{\unitlength}{.20mm}
\begin{center}
$t^2y$, $y^3z$, $xz^3$, $x^6t$
\end{center}\end{minipage}& \begin{minipage}[m]{.147\linewidth}\setlength{\unitlength}{.20mm}
$O_x=\frac{1}{11}(3,10)$

$O_y=\frac{1}{25}(11,16)$

$O_z=\frac{1}{32}(25,9)$

$O_t=\frac{1}{41}(11,32)$
\end{minipage}\\
 \hline

 $\left( 11,25,34,43  \right)$ &  $ 111$ &$\frac{222}{201025}$ &3&$\frac{33}{8}$ &\begin{minipage}[m]{.141\linewidth}\setlength{\unitlength}{.20mm}
\begin{center}
$t^2y$, $z^2t$, $xy^4$, $x^7z$
\end{center}\end{minipage}& \begin{minipage}[m]{.147\linewidth}\setlength{\unitlength}{.20mm}
$O_x=\frac{1}{11}(3,10)$

$O_y=\frac{1}{25}(1,2)$

$O_z=\frac{1}{34}(11,25)$

$O_t=\frac{1}{43}(11,34)$
\end{minipage}\\
 \hline

 $\left(  11,43,61,113 \right)$ &  $226 $ & $\frac{8}{28853}$& 2&$\frac{55}{12}$ &\begin{minipage}[m]{.141\linewidth}\setlength{\unitlength}{.20mm}
\begin{center}
$t^2$, $yz^3$, $xy^5$, $x^{15}z$
\end{center}\end{minipage}& \begin{minipage}[m]{.147\linewidth}\setlength{\unitlength}{.20mm}
$O_x=\frac{1}{11}(10,3)$

$O_y=\frac{1}{43}(2,3)$

$O_z=\frac{1}{61}(11,52)$
\end{minipage}\\
\hline

\end{longtable}
\end{center}

}} \rotatebox{90}{\box1}
\newpage

\setbox1=\hbox{\parbox{1.5\linewidth}{
\begin{center}
Log del Pezzo surfaces with $I=3$
\begin{longtable}{|c|c|c|c|c|c|c|}

\hline
\begin{minipage}[m]{.110\linewidth}\setlength{\unitlength}{.20mm}
\begin{center}
Weights \end{center}
\end{minipage}&
\begin{minipage}[m]{.053\linewidth}\setlength{\unitlength}{.20mm}
\begin{center}
Degree
\end{center}
\end{minipage}&
\begin{minipage}[m]{.031\linewidth}\setlength{\unitlength}{.20mm}
\begin{center}
$K_X^2$
\end{center}
\end{minipage}& $\rho$ &
\begin{minipage}[m]{.019\linewidth}\setlength{\unitlength}{.20mm}
\begin{center}
$\mathrm{lct}$
\end{center}
\end{minipage}&
\begin{minipage}[m]{.141\linewidth}\setlength{\unitlength}{.20mm}
\begin{center}
Monomials
\end{center}
\end{minipage} &
\begin{minipage}[m]{.147\linewidth}\setlength{\unitlength}{.20mm}
\begin{center}
Singular Points
\end{center}
\end{minipage}\\
\hline

 $\left(5,7,11,13  \right)$ & $33$ &$\frac{27}{455}$&3&$\frac{49}{36}$ &
 \begin{minipage}[m]{.141\linewidth}\setlength{\unitlength}{.20mm}
\begin{center}
$t^2y$, $z^3$, $xy^4$, $x^3yz$, $x^4t$
\end{center}\end{minipage}& \begin{minipage}[m]{.147\linewidth}\setlength{\unitlength}{.20mm}
$O_x=\frac{1}{5}(2,1)$

$O_y=\frac{1}{7}(2,3)$

$O_t=\frac{1}{13}(5,11)$
\end{minipage}\\
 \hline

 $\left( 5,7,11,20 \right)$ & $40$ & $\frac{18}{385}$&4&$\frac{25}{18}$&
  \begin{minipage}[m]{.141\linewidth}\setlength{\unitlength}{.20mm}
\begin{center}
$t^2$, $yz^3$, $xy^5$, $x^3y^2z$, $x^4t$, $x^8$
\end{center}\end{minipage}& \begin{minipage}[m]{.147\linewidth}\setlength{\unitlength}{.20mm}
$O_y=\frac{1}{7}(2,3)$

$O_z=\frac{1}{11}(1,4)$

$O_xO_t=2\times\frac{1}{5}(2,1)$
\end{minipage}\\
 \hline

 $\left( 11,21,29,37 \right)$ & $95$ &$\frac{285}{82621}$ &3&$\frac{11}{4}$&
   \begin{minipage}[m]{.141\linewidth}\setlength{\unitlength}{.20mm}
\begin{center}
$t^2y$, $z^2t$, $xy^4$,  $x^6z$
\end{center}\end{minipage}& \begin{minipage}[m]{.147\linewidth}\setlength{\unitlength}{.20mm}
$O_x=\frac{1}{11}(5,2)$

$O_y=\frac{1}{21}(1,2)$

$O_z=\frac{1}{29}(11,21)$

$O_t=\frac{1}{37}(11,29)$
\end{minipage}\\
\hline

\end{longtable}
\end{center}

}} \rotatebox{90}{\box1}
\newpage
  \setbox1=\hbox{\parbox{1.5\linewidth}{

\begin{center}
Log del Pezzo surfaces with $I=4$
\begin{longtable}{|c|c|c|c|c|c|c|}

\hline
\begin{minipage}[m]{.162\linewidth}\setlength{\unitlength}{.20mm}
\begin{center}
Weights \end{center}
\end{minipage}&
\begin{minipage}[m]{.053\linewidth}\setlength{\unitlength}{.20mm}
\begin{center}
Degree
\end{center}
\end{minipage}&
\begin{minipage}[m]{.031\linewidth}\setlength{\unitlength}{.20mm}
\begin{center}
$K_X^2$
\end{center}
\end{minipage}& $\rho$ &
\begin{minipage}[m]{.019\linewidth}\setlength{\unitlength}{.20mm}
\begin{center}
$\mathrm{lct}$
\end{center}
\end{minipage}&
\begin{minipage}[m]{.121\linewidth}\setlength{\unitlength}{.20mm}
\begin{center}
Monomials
\end{center}
\end{minipage} &
\begin{minipage}[m]{.177\linewidth}\setlength{\unitlength}{.20mm}
\begin{center}
Singular Points
\end{center}
\end{minipage}\\
\hline
 $\left(6,6n+3,6n+5,6n+5\right)$ & $18n+15$ &$\frac{8}{(6n+3)(6n+5)}$ &5&1&
 \begin{minipage}[m]{.121\linewidth}\setlength{\unitlength}{.20mm}
\begin{center}

$t^{3}$, $zt^{2}$, $z^{2}t$, $z^{3}$, $xy^{3}$, $x^{2n+2}y$

\end{center}\end{minipage}& \begin{minipage}[m]{.177\linewidth}\setlength{\unitlength}{.20mm}
$O_x=\frac{1}{6}(1,1)$

$O_y=\frac{1}{6n+3}(1,1)$

$O_xO_y=1\times\frac{1}{3}(1,1)$

$O_zO_t=\\
\phantom{a}=3\times\frac{1}{6n+5}(2,2n+1)$
\end{minipage}\\
 \hline
 $\left(6, 6n+5, 12n+8, 18n+9\right)$ & $36n+24$ & $\frac{8}{3(6n+3)(6n+5)}$&3&1&
 \begin{minipage}[m]{.121\linewidth}\setlength{\unitlength}{.20mm}
\begin{center}
$z^{3}$, $y^{3}t$, $xt^{2}$, $x^{2n+1}y^{2}z$, $x^{6n+4}$
\end{center}\end{minipage}& \begin{minipage}[m]{.177\linewidth}\setlength{\unitlength}{.20mm}
$O_y=\frac{1}{6n+5}(2,2n+1)$

$O_t=\\
\phantom{a}=\frac{1}{18n+9}(6n+5, 12n+8)$

$O_xO_z=1\times\frac{1}{2}(1,1)$

$O_xO_t=1\times\frac{1}{3}(1,1)$
\end{minipage}\\
\hline

\end{longtable}
\end{center}

}} \rotatebox{90}{\box1}
\newpage
  \setbox1=\hbox{\parbox{1.5\linewidth}{

\begin{center}
Log del Pezzo surfaces with $I=4$
\begin{longtable}{|c|c|c|c|c|c|c|}

\hline
\begin{minipage}[m]{.110\linewidth}\setlength{\unitlength}{.20mm}
\begin{center}
Weights \end{center}
\end{minipage}&
\begin{minipage}[m]{.053\linewidth}\setlength{\unitlength}{.20mm}
\begin{center}
Degree
\end{center}
\end{minipage}&
\begin{minipage}[m]{.031\linewidth}\setlength{\unitlength}{.20mm}
\begin{center}
$K_X^2$
\end{center}
\end{minipage}& $\rho$ &
\begin{minipage}[m]{.019\linewidth}\setlength{\unitlength}{.20mm}
\begin{center}
$\mathrm{lct}$
\end{center}
\end{minipage}&
\begin{minipage}[m]{.141\linewidth}\setlength{\unitlength}{.20mm}
\begin{center}
Monomials
\end{center}
\end{minipage} &
\begin{minipage}[m]{.147\linewidth}\setlength{\unitlength}{.20mm}
\begin{center}
Singular Points
\end{center}
\end{minipage}\\
\hline $\left(10,13,25,31\right)$ & $75$
&$\frac{24}{2015}$&3&$\frac{91}{60}$
  &\begin{minipage}[m]{.141\linewidth}\setlength{\unitlength}{.20mm}
\begin{center}
$t^2y$, $z^3$, $xy^5$, $x^5z$
\end{center}\end{minipage}& \begin{minipage}[m]{.147\linewidth}\setlength{\unitlength}{.20mm}
$O_x=\frac{1}{10}(3,1)$

$O_y=\frac{1}{13}(12,5)$

$O_t=\frac{1}{31}(2,5)$

$O_xO_z=1\times\frac{1}{5}(3,1)$
\end{minipage}\\
\hline $\left(11,17,20,27\right)$ & $71$
&$\frac{284}{25245}$&3&$\frac{11}{6}$
  &\begin{minipage}[m]{.141\linewidth}\setlength{\unitlength}{.20mm}
\begin{center}
$t^2y$, $y^3z$, $xz^3$, $x^4t$
\end{center}\end{minipage}& \begin{minipage}[m]{.147\linewidth}\setlength{\unitlength}{.20mm}
$O_x=\frac{1}{11}(2,3)$

$O_y=\frac{1}{17}(11,10)$

$O_z=\frac{1}{20}(17,7)$

$O_t=\frac{1}{27}(11,20)$
\end{minipage}\\
\hline

 $\left(11,17,24,31\right)$ & $79$
 &$\frac{158}{17391}$&3&$\frac{33}{16}$
  &\begin{minipage}[m]{.141\linewidth}\setlength{\unitlength}{.20mm}
\begin{center}
$t^2y$, $tz^2$, $xy^4$, $x^5z$
\end{center}\end{minipage}& \begin{minipage}[m]{.147\linewidth}\setlength{\unitlength}{.20mm}
$O_x=\frac{1}{11}(2,3)$

$O_y=\frac{1}{17}(1,2)$

$O_z=\frac{1}{24}(11,17)$

$O_t=\frac{1}{31}(11,24)$
\end{minipage}\\
\hline

\end{longtable}
\end{center}

}} \rotatebox{90}{\box1}
\newpage
  \setbox1=\hbox{\parbox{1.5\linewidth}{
\begin{center}
Log del Pezzo surfaces with $I=5$
\begin{longtable}{|c|c|c|c|c|c|c|}

\hline
\begin{minipage}[m]{.110\linewidth}\setlength{\unitlength}{.20mm}
\begin{center}
Weights \end{center}
\end{minipage}&
\begin{minipage}[m]{.053\linewidth}\setlength{\unitlength}{.20mm}
\begin{center}
Degree
\end{center}
\end{minipage}&
\begin{minipage}[m]{.031\linewidth}\setlength{\unitlength}{.20mm}
\begin{center}
$K_X^2$
\end{center}
\end{minipage}& $\rho$ &
\begin{minipage}[m]{.019\linewidth}\setlength{\unitlength}{.20mm}
\begin{center}
$\mathrm{lct}$
\end{center}
\end{minipage}&
\begin{minipage}[m]{.141\linewidth}\setlength{\unitlength}{.20mm}
\begin{center}
Monomials
\end{center}
\end{minipage} &
\begin{minipage}[m]{.147\linewidth}\setlength{\unitlength}{.20mm}
\begin{center}
Singular Points
\end{center}
\end{minipage}\\
\hline $\left(11,13, 19, 25\right)$ & $63$ &$\frac{63}{2717}$&
3&$\frac{13}{8}$
 &\begin{minipage}[m]{.141\linewidth}\setlength{\unitlength}{.20mm}
\begin{center}
$t^2y$, $tz^2$, $xy^4$, $x^4z$
\end{center}\end{minipage}& \begin{minipage}[m]{.147\linewidth}\setlength{\unitlength}{.20mm}
$O_x=\frac{1}{11}(2,3)$

$O_y=\frac{1}{13}(1,2)$

$O_z=\frac{1}{19}(11,13)$

$O_t=\frac{1}{25}(11,19)$
\end{minipage}\\
 \hline
  $\left(11,25,37,68\right)$ & $136$& $\frac{2}{407}$&2&$\frac{11}{6}$
 &\begin{minipage}[m]{.141\linewidth}\setlength{\unitlength}{.20mm}
\begin{center}
$t^2$, $yz^3$, $xy^5$, $x^9z$
\end{center}\end{minipage}& \begin{minipage}[m]{.147\linewidth}\setlength{\unitlength}{.20mm}
$O_x=\frac{1}{11}(3,2)$

$O_y=\frac{1}{25}(2,3)$

$O_z=\frac{1}{37}(11,31)$
\end{minipage}\\
 \hline
  $\left(13,19,41,68\right)$ & $136$ &$\frac{50}{10127}$&2&$\frac{91}{50}$
  &\begin{minipage}[m]{.141\linewidth}\setlength{\unitlength}{.20mm}
\begin{center}
$t^2$, $y^5z$, $xz^3$, $x^9y$
\end{center}\end{minipage}& \begin{minipage}[m]{.147\linewidth}\setlength{\unitlength}{.20mm}
$O_x=\frac{1}{13}(2,3)$

$O_y=\frac{1}{19}(13,11)$

$O_z=\frac{1}{41}(19,27)$
\end{minipage}\\
 \hline
\end{longtable}
\end{center}

}} \rotatebox{90}{\box1}
\newpage

\setbox1=\hbox{\parbox{1.5\linewidth}{

\begin{center}
Log del Pezzo surfaces with $I=6$
\begin{longtable}{|c|c|c|c|c|c|c|}

\hline
\begin{minipage}[m]{.162\linewidth}\setlength{\unitlength}{.20mm}
\begin{center}
Weights \end{center}
\end{minipage}&
\begin{minipage}[m]{.053\linewidth}\setlength{\unitlength}{.20mm}
\begin{center}
Degree
\end{center}
\end{minipage}&
\begin{minipage}[m]{.031\linewidth}\setlength{\unitlength}{.20mm}
\begin{center}
$K_X^2$
\end{center}
\end{minipage}& $\rho$ &
\begin{minipage}[m]{.019\linewidth}\setlength{\unitlength}{.20mm}
\begin{center}
$\mathrm{lct}$
\end{center}
\end{minipage}&
\begin{minipage}[m]{.101\linewidth}\setlength{\unitlength}{.20mm}
\begin{center}
Monomials
\end{center}
\end{minipage} &
\begin{minipage}[m]{.177\linewidth}\setlength{\unitlength}{.20mm}
\begin{center}
Singular Points
\end{center}
\end{minipage}\\
\hline
 $\left(8,4n+5,4n+7,4n+9\right)$ & $12n+23$ & $\frac{9(12n+23)}{2(4n+5)(4n+7)(4n+9)}$&
 3& $1$&
 \begin{minipage}[m]{.101\linewidth}\setlength{\unitlength}{.20mm}
\begin{center}
$z^{2}t$, $yt^{2}$, $xy^{3}$, $x^{n+2}z$

\end{center}\end{minipage}& \begin{minipage}[m]{.177\linewidth}\setlength{\unitlength}{.20mm}
$O_x=\frac{1}{8}(4n+5,4n+9)$

$O_y=\frac{1}{4n+5}(1,2)$

$O_z=\frac{1}{4n+7}(8,4n+5)$

$O_t=\frac{1}{4n+9}(8,4n+7)$
\end{minipage}\\
 \hline
 $\left(9,3n+8,3n+11,6n+13\right)$ & $12n+35$ & $\frac{4(12n+35)}{(3n+8)(3n+11)(6n+13)}$&3& $1$&
 \begin{minipage}[m]{.101\linewidth}\setlength{\unitlength}{.20mm}
\begin{center}

$z^{2}t$, $y^{3}z$, $xt^{2}$, $x^{n+3}y$

\end{center}\end{minipage}& \begin{minipage}[m]{.177\linewidth}\setlength{\unitlength}{.20mm}
$O_x=\frac{1}{9}(3n+11,6n+13)$

$O_y=\frac{1}{3n+8}(9, 6n+13)$

$O_z=\frac{1}{3n+11}(9, 3n+8)$

$O_t=\\
\phantom{a}=\frac{1}{6n+13}(3n+8, 3n+11)$
\end{minipage}\\
\hline
\end{longtable}
\end{center}

}} \rotatebox{90}{\box1}
\newpage

\setbox1=\hbox{\parbox{1.5\linewidth}{

\begin{center}
Log del Pezzo surfaces with $I=7$
\begin{longtable}{|c|c|c|c|c|c|c|}

\hline
\begin{minipage}[m]{.110\linewidth}\setlength{\unitlength}{.20mm}
\begin{center}
Weights \end{center}
\end{minipage}&
\begin{minipage}[m]{.053\linewidth}\setlength{\unitlength}{.20mm}
\begin{center}
Degree
\end{center}
\end{minipage}&
\begin{minipage}[m]{.031\linewidth}\setlength{\unitlength}{.20mm}
\begin{center}
$K_X^2$
\end{center}
\end{minipage}& $\rho$ &
\begin{minipage}[m]{.019\linewidth}\setlength{\unitlength}{.20mm}
\begin{center}
$\mathrm{lct}$
\end{center}
\end{minipage}&
\begin{minipage}[m]{.141\linewidth}\setlength{\unitlength}{.20mm}
\begin{center}
Monomials
\end{center}
\end{minipage} &
\begin{minipage}[m]{.147\linewidth}\setlength{\unitlength}{.20mm}
\begin{center}
Singular Points
\end{center}
\end{minipage}\\
\hline
 $\left(11,13,21,38\right)$ & $76$ &$\frac{14}{429}$&2&$\frac{13}{10}$
 &\begin{minipage}[m]{.141\linewidth}\setlength{\unitlength}{.20mm}
\begin{center}
$t^2$, $yz^3$, $xy^5$, $x^5z$
\end{center}\end{minipage}& \begin{minipage}[m]{.147\linewidth}\setlength{\unitlength}{.20mm}
$O_x=\frac{1}{11}(2,5)$

$O_y=\frac{1}{13}(2,3)$

$O_z=\frac{1}{21}(11,17)$
\end{minipage}\\
 \hline

\end{longtable}
\end{center}
\begin{center}
Log del Pezzo surfaces with $I=8$
\begin{longtable}{|c|c|c|c|c|c|c|}

\hline
\begin{minipage}[m]{.110\linewidth}\setlength{\unitlength}{.20mm}
\begin{center}
Weights \end{center}
\end{minipage}&
\begin{minipage}[m]{.053\linewidth}\setlength{\unitlength}{.20mm}
\begin{center}
Degree
\end{center}
\end{minipage}&
\begin{minipage}[m]{.031\linewidth}\setlength{\unitlength}{.20mm}
\begin{center}
$K_X^2$
\end{center}
\end{minipage}& $\rho$ &
\begin{minipage}[m]{.019\linewidth}\setlength{\unitlength}{.20mm}
\begin{center}
$\mathrm{lct}$
\end{center}
\end{minipage}&
\begin{minipage}[m]{.141\linewidth}\setlength{\unitlength}{.20mm}
\begin{center}
Monomials
\end{center}
\end{minipage} &
\begin{minipage}[m]{.147\linewidth}\setlength{\unitlength}{.20mm}
\begin{center}
Singular Points
\end{center}
\end{minipage}\\
\hline
 $\left(7,11,13,23\right)$ & $46$ &$\frac{128}{1001}$&2&$\frac{35}{48}$
 &\begin{minipage}[m]{.141\linewidth}\setlength{\unitlength}{.20mm}
\begin{center}
$t^2$, $y^3z$, $xz^3$, $x^5y$
\end{center}\end{minipage}& \begin{minipage}[m]{.147\linewidth}\setlength{\unitlength}{.20mm}
$O_x=\frac{1}{7}(3,1)$

$O_y=\frac{1}{11}(7,1)$

$O_z=\frac{1}{13}(11,10)$

\end{minipage}\\
 \hline
 $\left(7,18,27,37\right)$ & $81$ &$\frac{32}{777}$&3&$\frac{35}{72}$
 &\begin{minipage}[m]{.141\linewidth}\setlength{\unitlength}{.20mm}
\begin{center}
 $y^3z$, $z^3$, $xt^2$, $x^9y$
\end{center}\end{minipage}&
\begin{minipage}[m]{.147\linewidth}\setlength{\unitlength}{.20mm}
$O_x=\frac{1}{7}(3,1)$

$O_y=\frac{1}{18}(7,1)$

$O_t=\frac{1}{37}(2,3)$

$O_yO_z=1\times\frac{1}{9}(7,1)$
\end{minipage}\\
 \hline

\end{longtable}
\end{center}
\begin{center}
Log del Pezzo surfaces with $I=9$
\begin{longtable}{|c|c|c|c|c|c|c|}

\hline
\begin{minipage}[m]{.110\linewidth}\setlength{\unitlength}{.20mm}
\begin{center}
Weights \end{center}
\end{minipage}&
\begin{minipage}[m]{.053\linewidth}\setlength{\unitlength}{.20mm}
\begin{center}
Degree
\end{center}
\end{minipage}&
\begin{minipage}[m]{.031\linewidth}\setlength{\unitlength}{.20mm}
\begin{center}
$K_X^2$
\end{center}
\end{minipage}& $\rho$ &
\begin{minipage}[m]{.019\linewidth}\setlength{\unitlength}{.20mm}
\begin{center}
$\mathrm{lct}$
\end{center}
\end{minipage}&
\begin{minipage}[m]{.141\linewidth}\setlength{\unitlength}{.20mm}
\begin{center}
Monomials
\end{center}
\end{minipage} &
\begin{minipage}[m]{.147\linewidth}\setlength{\unitlength}{.20mm}
\begin{center}
Singular Points
\end{center}
\end{minipage}\\
\hline
 $\left(7,15,19,32\right)$ & $64$ &$\frac{54}{665}$&2&$\frac{35}{54}$
 &\begin{minipage}[m]{.141\linewidth}\setlength{\unitlength}{.20mm}
\begin{center}
$t^2$, $y^3z$, $xz^3$, $x^7y$
\end{center}\end{minipage}& \begin{minipage}[m]{.147\linewidth}\setlength{\unitlength}{.20mm}
$O_x=\frac{1}{7}(5,4)$

$O_y=\frac{1}{15}(7,2)$

$O_z=\frac{1}{19}(15,13)$
\end{minipage}\\
 \hline

\end{longtable}
\end{center}

\begin{center}
Log del Pezzo surfaces with $I=10$

\begin{longtable}{|c|c|c|c|c|c|c|}

\hline
\begin{minipage}[m]{.110\linewidth}\setlength{\unitlength}{.20mm}
\begin{center}
Weights \end{center}
\end{minipage}&
\begin{minipage}[m]{.053\linewidth}\setlength{\unitlength}{.20mm}
\begin{center}
Degree
\end{center}
\end{minipage}&
\begin{minipage}[m]{.031\linewidth}\setlength{\unitlength}{.20mm}
\begin{center}
$K_X^2$
\end{center}
\end{minipage}& $\rho$ &
\begin{minipage}[m]{.019\linewidth}\setlength{\unitlength}{.20mm}
\begin{center}
$\mathrm{lct}$
\end{center}
\end{minipage}&
\begin{minipage}[m]{.141\linewidth}\setlength{\unitlength}{.20mm}
\begin{center}
Monomials
\end{center}
\end{minipage} &
\begin{minipage}[m]{.147\linewidth}\setlength{\unitlength}{.20mm}
\begin{center}
Singular Points
\end{center}
\end{minipage}\\
\hline
 $\left(7,19,25,41\right)$ & $82$ & $\frac{8}{133}$&2&$\frac{7}{12}$
 &\begin{minipage}[m]{.141\linewidth}\setlength{\unitlength}{.20mm}
\begin{center}
$t^2$, $y^3z$, $xz^3$, $x^9y$
\end{center}\end{minipage}& \begin{minipage}[m]{.147\linewidth}\setlength{\unitlength}{.20mm}
$O_x=\frac{1}{7}(2,3)$

$O_y=\frac{1}{19}(7,3)$

$O_z=\frac{1}{25}(19,16)$
\end{minipage}\\
 \hline
 $\left(7,26,39,55\right)$ & $117$ &$\frac{30}{1001}$&3&$\frac{7}{18}$
 &\begin{minipage}[m]{.141\linewidth}\setlength{\unitlength}{.20mm}
\begin{center}
$y^3z$, $z^3$, $xt^2$, $x^{13}y$
\end{center}\end{minipage}& \begin{minipage}[m]{.147\linewidth}\setlength{\unitlength}{.20mm}
$O_x=\frac{1}{7}(2,3)$

$O_y=\frac{1}{26}(7,3)$

$O_t=\frac{1}{55}(26,39)$

$O_yO_z=1\times\frac{1}{13}(7,3)$
\end{minipage}\\
 \hline

\end{longtable}
\end{center}

}} \rotatebox{90}{\box1}

\end{document}